\documentclass[10pt]{article}
\usepackage{amsmath,amsthm,amssymb,ifsym,a4wide}
\usepackage[twoside]{geometry}
\usepackage{graphicx}
\usepackage{epsfig}   
\usepackage[justification=centering]{caption}
\newtheorem{theorem}{Theorem}[section]

\newtheorem{lemma}[theorem]{Lemma}

\numberwithin{equation}{section} 
\numberwithin{figure}{section}  
\newtheorem{conjecture}{Conjecture}[section]
\newtheorem{conclusion}{Conclusion}[section]

\newcommand \bei {\begin{itemize}}
\newcommand \eei {\end{itemize}}
\newcommand \be {\begin{equation}}
\newcommand \bel {\begin{equation}\label}
\newcommand \ee {\end{equation}}

\newcommand \trianglerightNEW \triangleright

\newcommand \sgn {\text{sgn}}

\newcommand \auth {\textsc}

\newcommand \del \partial

\newcommand \eps \epsilon 
\let\oldmarginpar\marginpar
\renewcommand\marginpar[1]{\-\oldmarginpar[\raggedleft\footnotesize #1]%
{\raggedright\footnotesize #1}}

\begin{document} 

\title{\em \Large 
Weakly regular fluid flows with bounded variation 
\\
on the domain of outer communication of a Schwarzschild blackhole spacetime. 
\\
 A numerical study} 

\author{Philippe G. LeFloch\footnote{
\normalsize Laboratoire Jacques-Louis Lions \& Centre National de la Recherche Scientifique, Universit\'e Pierre et Marie Curie (Paris 6), 
4 Place Jussieu, 75252 Paris, France. 
Email : {\sl contact@philippelefloch.org, xiang@ljll.math.upmc.fr}
}
\, and Shuyang Xiang$^*$}

\date{December 2016}

\maketitle

\abstract{
We study the dynamical behavior of compressible fluids evolving on the outer domain of communication of a Schwarz\-schild background. To this end, we design several numerical methods which take the Schwarzschild geometry into account and we treat, both, the relativistic Burgers equation and the relativistic Euler system under the assumption that the flow is spherically symmetric.  All the schemes we construct are proven to be well-balanced and therefore to preserve the family of steady state solutions for both models. They enable us to study the nonlinear stability of fluid equilibria, and in particular to investigate the behavior of the fluid near the blackhole horizon. We state and numerically demonstrate several conjectures about the late-time behavior of perturbations of steady solutions.}


\section{Introduction}

In this paper and the companion papers \cite{PLF-SX-one,PLF-SX-two,PLF-SX-four}, we study numerically compressible fluid flows on a Schwarzschild blackhole background. The present investigation is part of a research project by LeFloch and co-authors on designing numerical methods for relativistic fluid problems posed on curved spacetimes; see \cite{Amorim-L,Ceylan-L,LF-malaga,LF-M,LF-M-O}. Building upon the  numerical analysis in the later papers and on the analytical work performed by the authors in  \cite{PLF-SX-one,PLF-SX-two,PLF-SX-four}, we are able here to design several numerical  schemes for the approximation of shock wave solutions to, both, the relativistic Burgers equation and the compressible Euler system under the assumption that the flow is spherically symmetric. Our schemes are asymptotic preserving and therefore allow us to investigate the late-time asymptotic of solutions. One important challenge addressed here is taking the curved geometry into account at the level of the discretization and handling the behavior of solutions near the horizon of the blackhole. 

The  {\em relativistic Burgers equation on a Schwarzschild background} reads as follows (see \cite{PLF-SX-one} for further details): 
 \bel{Burgers}
\del_t \bigg({v\over (1- {2M /  r})^2}\bigg)+ \del_r  \Bigg({v^2 -1 \over 2 (1-2M/r)} \Bigg)=0, 
\qquad r>2M, 
\ee
where we have normalized the light speed to unit and the unknown is the function $v=v(t,r) \in [-1, 1]$.  This equation can also be put in the following non-conservative form: 
\bel{Burgers'}
\del_t v + \del_ r \Bigg(\Big(1-{2M \over r} \Big) {v^2 -1 \over 2}\Bigg) ={2M \over r^2} \big(v^2 -1 \big), 
\qquad r>2M.  
\ee
Here $M>0$ denotes the mass of the blackhole and, clearly, we recover the standard Burgers equations when the mass vanishes.  

Our main contribution for the relativistic Burgers model above is as follows. First of all, we are going to construct a well-balanced finite volume method as well as a random choice method which, both, are capable to preserve the steady state solutions. We will use these schemes to investigate the following issues and validate and extend our theoretical results (briefly reviewed below in Theorems~\ref{Burgers1} to \ref{Burgers3}): 
\bei

\item The global-in-time existence theory for the generalized Riemann problem generated by an arbitrary initial discontinuity. 

\item The late-time behavior of an initially perturbed steady state solution, possibly containing a stationary shock wave. 

\eei
Furthermore, our study here have led us to the following two conjectures for general initial data.

\begin{conjecture}
\label{con0}
Given any compactly perturbed steady shock as an initially data, the solution to the relativistic Burgers model on a Schwarzschild background \eqref{Burgers} converges to a steady state shock asymptotically in time.  
\end{conjecture}
 
\begin{conjecture}
\label{con1}
Given an initial data $v_0 = v_0 (r) \in [-1, 1]$ defined on $[2M, +\infty)$, the corresponding solution $v= v(t, r)$ to the relativistic Burgers model \eqref{Burgers} is as follows:  
\bei 

\item If $v_0 (2M)=1 $, then there exists a finite time  $t_0>0$ such that, for all $t>t_0$, the solution $v$ is a single shock with left-hand state $1$ and right-hand state $- \sqrt  {2M \over r}$.

\item If  $v_0 (2M)  <  1$ and $\lim \limits_ {r\to +\infty} v_0(r)  >  0$,  then there exists a finite time $t_0>0$ such that, for all $t>t_0$,  the solution is $v(t, r)=- \sqrt  {2M \over r}$ for all  $t>t_0$.

\item If  $v_0(2M)<  1$ and  $\lim \limits_ {r\to +\infty} v_0(r)  \leq 0$, then there exists a finite time $t_0>0$ such that, for all $t>t_0$,  the solution to the relativistic Burgers model satisfies for all   $t>t_0$
$$
v(t, r)= - \sqrt {1-({1-(v_ 0^\infty)}^2)  \big(1-{2M \over r} \big)}, 
\qquad 
\lim \limits_ {r \to +\infty} v_0(r)  =: v_0^\infty \leq 0.  
$$
\eei
\end{conjecture}

We also investigate solutions to the Euler system on a Schwarzschild background, which takes the form: 

\bel{Euler}
\aligned 
& \del_t \Big(r^2 {1 + k^2  v^2 \over 1 - v^2}  \rho \Big) +\del_r \Big(r(r-2M) {1+k^2  \over 1 -  v^2} \rho v\Big) = 0,
\\
& \del_t \Big(r(r-2M) {1+k^2  \over 1 -  v^2} \rho v \Big) + \del_r \Big((r-2M)^2 {v^2+ k^2  \over 1 -  v^2} \rho\Big)
  =3M   \Big(1 - {2M \over r} \Big) {v^2+ k^2  \over 1 - v^2}\rho   -M {r-2M \over r} {1+  k^2  v^2 \over 1 -  v^2}  \rho+2{(r-2M)^2 \over r}k^2 \rho, 
\endaligned 
\ee
where the light speed is normalized to unit and $k \in (0, 1]$ denotes the sound speed. By formally letting $k \to 0$, we can recover the pressureless Euler system, from which in turn we can derive the relativistic Burgers equation above. On the other hand, by letting the blackhole mass $M \to 0$, we recover the relativistic Euler system. Furthermore, we can also write the Euler equations in the following form: 
\bel{Euler1}
\aligned 
& \del_t \Big({1 + k^2  v^2 \over 1 - v^2}  \rho \Big) +\del_r \Big((1-2M/r) {1+k^2  \over 1 -  v^2} \rho v\Big) = - {2\over r} (1-2M/r) {1+k^2  \over 1 -  v^2} \rho v,
\\
& \del_t \Big ({1+k^2  \over 1 -  v^2} \rho v \Big) + \del_r \Big((1-2M/r)  {v^2+ k^2  \over 1 -  v^2} \rho\Big)
  ={- 2 r+5M \over r^2}{v^2+ k^2  \over 1 - v^2}\rho    -{M \over r^2}  {1+  k^2  v^2 \over 1 -  v^2}  \rho+2{r-2M  \over r^2}k^2 \rho. 
\endaligned 
\ee
Our study of the relativistic Euler equations on a Schwarzschild background  \eqref{Euler} is based on  the construction of a finite volume method with second-order accuracy, which preserves the family of steady state solutions. Our numerical study suggests a global-in-time existence theory for the generalized Riemann problem, whose explicit form is not yet known theoretically. In particular, we exhibit here solutions containing up to three steady state components, connected by a $1$-wave and a $2$-wave.

\begin{conjecture}
\label{con2}
Let $(\rho_*, v_*) =(\rho_*, v_*)(r)$, $r>2M$ be a smooth steady state solution to the Euler model above and let $(\rho_0, v_0)= (\rho_0, v_0) (r)= (\rho_*, v_*)(r)+  (\delta_ {\rho}, \delta_v)(r)$ where $ (\delta_ {\rho}, \delta_v)=  (\delta_ {\rho}, \delta_v)(r)$ has compact support. Then, the corresponding solution to the relativistic Euler equation on a Schwarzschild background $(\rho, v)=(\rho, v)(t, r)$ satisfies: 
\bei

\item If $|\int \delta_ {\rho} (r)dr| +| \int  \delta_v (r)dr |= 0$, then there exists a time $t_0>0$ such that $(\rho, v) (t, r)= (\rho_*, v_*)(r)$ for all $t>t_0$.

\item If $|\int \delta_ {\rho} (r)dr | +  |\int  \delta_v (r)dr  |\neq 0$, then there exists a time $t_0 >0$ such that $(\rho, v) (t, r)= (\rho_{**}, v_{**})(r)$ for all $t>t_0$, where $(\rho_{**}, v_{**})$ is a possibly different steady state solution.  
\eei
\end{conjecture}

Using steady shocks (to be defined in Section~\ref{Sec:8}), we also have the following. 

\begin{conjecture}
\label{con3}
Let $(\rho_*, v_*) =(\rho_*, v_*)(r)$, $r>2M$ be a steady shock and let $(\rho_0, v_0)= (\rho_*, v_*)(r)+  (\delta_ {\rho}, \delta_v)(r)$ where $ (\delta_ {\rho}, \delta_v)=  (\delta_ {\rho}, \delta_v)(r)$ is a compactly supported perturbation. Then there exists a finite time $t>t_0$ such that the solution $(\rho, v)= (\rho, v)(t, r)$ 
is a (possibly different) steady shock. 
\end{conjecture}

Our numerical approach on the Glimm scheme is motivated by the approach proposed by Glimm, Marshall, and Plohr \cite{Glimm-P} for quasi-one-dimensional gas flows. We rely on static solutions and on the generalized Riemann problem, which we studied extensively in \cite{PLF-SX-one,PLF-SX-two,PLF-SX-four} for the relativistic models under consideration here. The numerical analysis of hyperbolic problems posed on curved spacetimes was initiated in  \cite{Amorim-L,Ceylan-L,LF-malaga,LF-M,LF-M-O} using the finite volume methodology, and we also recall that hyperbolic conservation laws on curved spaces are also studied by Dziuk and co-authors \cite{Dziuk-E,Dziu-K}. 

This paper is organized as follows. In Section~\ref{Sec:2}, we briefly overview our theoretical results for the relativistic Burgers model. We include a full description of the family of steady state solutions, as well as some outline of the existence theory for the initial data problem and the nonlinear stability of piecewise steady solutions.  In Section~\ref{Sec:3}, we introduce a finite volume method for the relativistic Burgers model \eqref{Burgers}, which is well-balanced and second-order accurate. In Section~\ref{Sec:4}, we apply our scheme in order to study the generalized Riemann problem and to elucidate the late-time behavior of perturbations of steady solutions. 
 
Building on our theoretical results, in Section~\ref{Sec:5} we implement a generalized Glimm scheme for the relativistic Burgers model \eqref{Burgers}. Our numerical method is based on an explicit and accurate solver of the generalized Riemann problem and, therefore, our method preserves all steady state solutions. Numerical experiments are presented in Section~\ref{Sec:6}, in which we are able to validate and expand the theoretical results in Section~\ref{Sec:2}. Our method avoids  to introduce numerical diffusion and provide an efficient approach for computing shock wave solutions. Furthermore, in Section~\ref{Sec:7} we apply both methods to the study of the initial problem for the relativistic Burgers equation when the initial velocity is rather arbitrary and we validate our Conjectures~\ref{con0} and \ref{con1} and, along the way, clarify the behavior of the fluid flow near the blackhole horizon. 
 
Next, in Section~\ref{Sec:8}, we turn our attention to the relativistic Euler model on a Schwarzschild background. We begin by  reviewing some theoretical results, including the existence theory for steady state solutions, the construction of a solver for the generalized Riemann problem, and the existence theory for the initial value problem. We are then in a position, in   Section~\ref{Sec:9}, to construct a finite volume method for the relativistic Euler model. Our method is second-order accuracy and is proven be well-balanced. With the proposed algorithm, in Section~\ref{Sec:10}, we are able to tackle the generalized Riemann problem (which has not yet been solved in a closed form) and we study the nonlinear stability of steady state solutions when the perturbation has compact support. This leads us to numerically demonstrate the validity of Conjectures~\ref{con2} and  \ref{con3} above. 


\section{Overview of the theory for the relativistic Burgers model} 
\label{Sec:2}

An important class of solutions to the relativistic Burgers model \eqref{Burgers} is provided by the {\em steady state solutions}, that is, solutions depending on the space variable $r$ only:  
\bel{static-Burgers}
 \del_r  \Bigg({v^2 -1 \over 2 (1-2M/r)} \Bigg)=0. 
\ee
Clearly, $\Big({v^2 -1 \over 2 (1-2M/r)} \Big)$ is then a constant, and we see that steady state solutions for the Burgers equation are 
\bel{steady-solution}
v(r)= \pm \sqrt {1- K^2 (1-2M/r)}, 
\ee
where $K >0$ is a constant and, clearly, the sign of a steady state cannot change. The following remarks are in order: 
\bei

\item $v=v(r)$ is a uniformly bounded and smooth in $r$ and it admits the finite limit  $\lim\limits_ {r \to 2M} v (r)= \pm 1$ at the  blackhole horizon. 

\item When  $0<K <   1$, one has $\lim \limits_ {r \to +\infty} v (r)= \pm \sqrt {1-K^2}$. 

\item When $K =1$ or equivalently, $v_*^\pm = \pm \sqrt {2M\over r}$, the steady state solution vanishes at infinity. These two solutions are referred to as the {\em critical steady state solutions}.  

\item When $K>1$, the steady state solution vanishes at a finite radius  $r^\natural ={2M K^2 \over 1- K^2}$, which we may refer to as the \emph {vanishing point}. 
\eei 


\begin{figure}[!htb] 
\centering 
\epsfig{figure=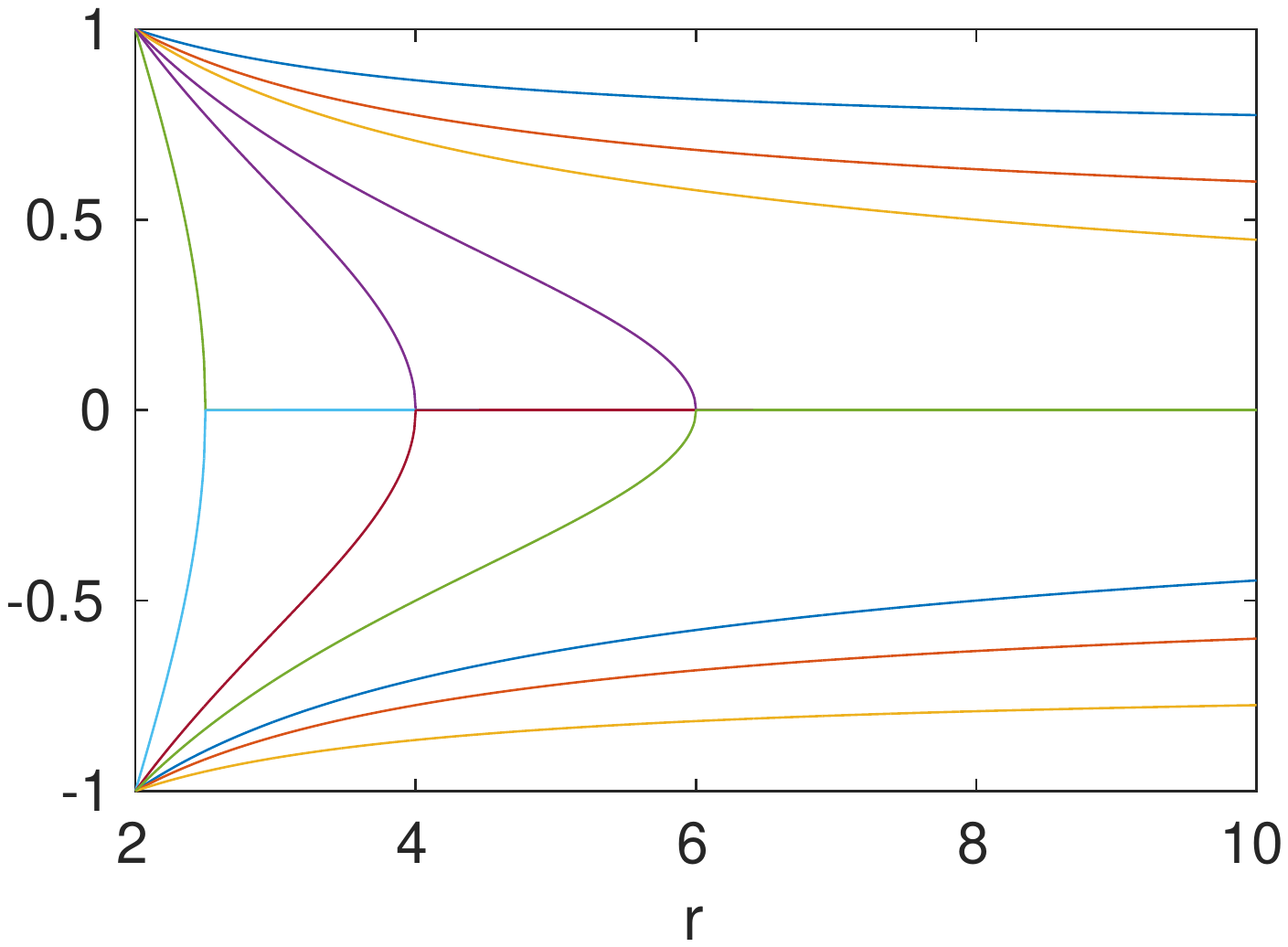,clip, trim=1.2in 3.2in 1.2in 3.2in,height = 2.5in} 
\caption{Steady state solutions for the relativistic Burgers model} 
\label{FG-20}
\end{figure}


In addition to the smooth steady state solutions, we can also define the class of {\em steady shocks for the relativistic Burgers equation}, which are given by 
\bel{Burgers-steady-shock}
v=\begin{cases}
\sqrt {1- K^2 (1-2M/r)}, &  2M<r <r_0, \\
 -\sqrt {1- K^2 (1-2M/r)},  & r>r_0, 
\end{cases}
\ee
where $K $ is a constant and $r_0$ is any given radius. The solution \eqref{Burgers-steady-shock} is time-independent and the  discontinuity point $r=r_0$ does not move when time increases. 
The relevant solutions to the relativistic Burgers equation $v= v(t, r)$ have a range bounded by the light speed, that is, $v\in [-1, 1] $ for all $t>0$ and $r>2M$. An initial problem of particular importance is given by the generalized Riemann problem, associated with initial data made of two steady states separated by a jump discontinuity located at some given radius. 

\begin{theorem}[The generalized Riemann problem for the relativistic Burgers model]
\label{Burgers1}
There exists a unique solution to the generalized Riemann problem defined for all $t>0$ realized by either by a shock wave or a rarefaction wave. Moreover, the following asymptotic behaviors hold: 
\bei 
\item The wave location tends to the blackhole horizon if it initially  converges towards  the blackhole.

\item  The wave location tends to the space infinity if it initially converges away from the blackhole.

\item  The wave location does not change if it is initially steady.  
\eei
\end{theorem}

In connection with the general existence theory for \eqref{Burgers},  we introduce the auxiliary variable
$z: =  \sgn (v)\sqrt{{v^2 -1 \over 1-2M /r}+1}$.
It is obvious that $z$ is a constant if $v$ is a steady state solution. With this notation, we have the following result from \cite{PLF-SX-one}. 

\begin{theorem}[Existence theory for the relativistic  Burgers model]
\label{Burgers2}
Consider the relativistic Burgers equation  \eqref{Burgers} posed on the outer domain of a Schwarzschild blackhole with mass $M$. Then, for any initial velocity $z_0 = z_0(r) \in (-1, 1)$  such that $z_0= z_0 (r)$ has bounded total variation, there exists a corresponding weak solution to \eqref{Burgers}  $z=z(t, r)$ whose total variation is non-increasing with respect to time: 
$$
{TV}\big(z  (s, \cdot)\big) \leq{TV} \big(z (t, \cdot)\big), \qquad 
0 \leq t\leq s. 
$$ 
 \end{theorem}

We are going to design several numerical methods for study these solutions. In particular,   we are interested in the behavior of solutions when the initial data $v_0= v_0 (r)$ is a piecewise smooth and steady state solution, to which we will add a compactly supported perturbation, i.e.  
\bel{initial-steady-pe}
v_0 (r)= \begin{cases}
v_L(r) & 2M<r< r_L, 
\\ 
v_R(r) & r>r_R, 
\end{cases}
\ee
where $v_L= v_L(r)$, $ v_R= v_R(r)$ are two steady state solutions given by \eqref{steady-solution} and $r_L, r_R$ are two fixed points. 

\begin{theorem}[Time-asymptotic properties for the relativistic Burgers model]
\label{Burgers3}
Consider the asymptotic behavior of a relativistic Burgers solution $v=v(t,r)$ on a Schwarzschild background \eqref{Burgers} whose  initial data is composed by steady state solutions $v_L, v_R$ with a compactly supported perturbation.  

\bei
\item If $v_L > v_R$, then the solution $v= v(t, r)$ converges asymptotically to a shock curve generated  by a left-hand state $v_L$ and a right-hand state $ v_R$. 

\item  If $v_L <  v_R$, then a generalized N-wave $N= N(t, r)$ can be defined such that inside a  rarefaction fan,  one has  $|v (t, r)-N(t, r)|= O(t^{-1})$ while in a region supporting of the evolution of the initial data, one has $|v (t, r)-N(t, r)|= O(t^{-1/2})$. Otherwise, one has $v (t, r)=N(t, r)$. 

\item If $v_L =  v_R$, then $||v(t, r)-v_R(t, r) ||_{L^{1}(2M,+\infty)} = O (t^{-1/2})$. 
\eei
\end{theorem}


\section{A finite volume scheme for the relativistic Burgers model}
\label{Sec:3}

\paragraph{The first-order formulation}

In this section, we propose a finite volume method for the relativistic Burgers equation \eqref{Burgers'} which takes the Schwarzschild geometry into consideration.  In order to construct our approximations, we will rely on the solution to the Riemann problem for the standard Burgers equation : 
\bel{standard}
\del_ t v + \del_ x {v^2 \over 2} =0 
\ee
that is,  an initial data problem with $v(t, r)= v_0(r)$ where $v_0= v_0 (r)$ is given as a piecewise constant function
$
v_0 = \begin{cases}
v_L  & r< r_0, \\
v_R & r>r_0, 
\end{cases}
$ 
for some fixed $r_0$ and two constants $v_L$,  $v_R$. The solution to the standard Riemann problem is given as 
\bel{Riemann-sol}
v(t, r)= \begin{cases}
v_L & r < s_L t +r_0,  \\
{r-r_0 \over t} &s_L t +r_0  <r< s_R t +r_0,  \\
v_R  & r> s_R t +r_0, 
\end{cases}
\qquad 
s_L= \begin{cases}
v_L  & v_L< v_R, \\
{v_L +v_R \over 2} & v_L>v_R,  
\end{cases}\qquad 
s_R= \begin{cases}
v_R & v_L< v_R, \\
{v_L +v_R \over 2} & v_L>v_R. 
\end{cases}
\ee

Denote by $\Delta t $, $\Delta r$ 
the mesh lengths in time and in space respectively with the CFL condition 
${\Delta t \over \Delta r} =\Lambda$, 
where $\Lambda$ is  such that $\Lambda |v| \leq 1/2$ in order to avoid wave interaction between two Riemann problems. We set 
$t_n = n \Delta t$ and $r_j=2M+  j\Delta r$. 
Introduce also the mesh point $(t_n, t_j)$, $n\geq 0$, $j\geq 0$ and the rectangle
$R_{nj}= \{t_n \leq t < t_{n+1}, \quad r_{j-1/2} \leq r < r_{j+1/2}\}$. 
Integrate \eqref{Burgers'} from $r_{j- 1/2}$ to $r_{j+1/2}$ in space and from $t_n$ to $t_{n+1}$: 
$$ 
\aligned 
& \int_{r_{j-1/2}}^{r_{j+1/2}} \Big (v (t_{n+1}, r)-v (t_n, r) \Big)dr 
+\int^{t_{n+1}}_{t_n}\Bigg((1-2M/r_{j+1/2}) \bigg({v^2(t, r_{j+1/2}) -1 \over 2} \bigg)
\\ & - (1-2M/r_{j-1/2})\bigg({v^2(t, r_{j-1/2}) -1 \over 2} \bigg)\Bigg) dt 
-   \int_{r_{j-1/2}}^{r_{j+1/2}} \int^{t_{n+1}}_{t_n} {2M \over r^2} (v^2-1) dt dr =0. 
\endaligned 
$$
Denote by $V_j^n= \int_{r_{j-1/2}}^{r_{j+1/2}}  v (t_n, r) dr, $ the average value of the solution in the space interval $(r_{j-1/2},r_{j+1/2})$, 
and introduce the finite volume scheme for the relativistic Burgers equation on a Schwarzschild background: 
\bel{Godunov}
V_j^{n+1}  = V_j^n- 
{\Delta t\over \Delta r}\big(F_{j+1/2} -  F_{j-1/2} \big)- \Delta t {2M \over r_j^2} ({V_j^n}^2 -1), 	   
\ee
where $  F_{j+1/2}$ and $ F_{j-1/2}$ are
$F_{j+1/2}= \mathcal F (r_{_{j+1/2}}, V_j^n, V_{j-1}^n)$, and 
\bel{mathcal-F}
 \mathcal F (r, V_L, V_R)= \Big(1-{2M \over r}  \Big) {q^2(V_L, V_R)-1 \over 2} 
\ee
with  $q (\cdot,\cdot)$ the standard solution to the Riemann problem   centered at $r$ given by \eqref{Riemann-sol}.  Observe that the CFL condition guarantees that the solution to the Riemann problem  does not  to leave the rectangle $R_{n,j}$ within one time step.

We now consider the  boundary condition of our finite volume  scheme. Let $J$ be the number of the space mesh points and we introduce ghost cells at the space boundaries: 
$R_{n,0}=  \{t_n \leq t < t_{n+1}, \quad r_{-1/2} \leq r < r_{1/2}\}$
and $R_{n,J}=  \{t_n \leq t < t_{n+1}, \quad r_{J-1/2} \leq r < r_{J+1/2}\}$. 
We solve the Riemann problem at the two boundaries with initial condition 
$$
V_0 (r)=\begin{cases} 
1 & r<r_0, \\ 
V _0^n & r>r_0, 
\end{cases}
\qquad 
V_0 (r)=\begin{cases}
V _J^n  & r<r_J, \\ 
-1 & r>r_J. 
\end{cases}
$$


\paragraph{A consistency property}


\begin{lemma}
The finite volume method for the relativistic Burgers model introduced in
\eqref{Godunov} satisfies the following properties: 
\bei
\item 
The scheme is well-balanced, that is, it preserves the steady state solution to the Euler equation \eqref{Euler-static}. 
\item 
The scheme is consistent, that is, if  $v=v(t, r)$ is an exact  solution to the relativistic Burgers model given by the ordinary differential equation \eqref{static-Burgers}, then for every fixed point $r>2M$,  
\bel{consistent}
\mathcal F (r_R, V_L, V_R)- \mathcal F (r_L, V_L, V_R) =   {2M \over r^2} (v^2 -1) (r_R- r_L)+O(r_R-r_L)^2 
\ee
holds as $V_L, V_R \to v $ and $r_L, r_R \to r$. 
\eei
\end{lemma}

\begin{proof}
To establish the well-balanced property, we write
$$
\aligned 
F_{j+1/2} -  F_{j-1/2}= & (1-2M /r_{j+ 1/2}) {q (V_j^n, V_{j+1}^n)-1 \over 2} -  (1-2M /r_{j-  1/2})  {q (V_{j-1}^n, V_j^n)-1 \over 2} \\
= & \int_ {j-1/2}^{j+1/2}   {2M \over r^2} (v^2 -1) dr = {2M \over r_j^2} ({V_j^n}^2 -1), 
\endaligned 
$$
and, therefore, $V_j^n= V_j^{n+1}$ holds.  
Next, recall  that  $\mathcal F(r, V_L, V_R)=\Big(1-{2M \over r} \Big) {q(r, V_L, V_R) -1 \over 2} $ is the numerical flux of the scheme determined by the standard the Riemann solution. A Taylor expansion gives
$$  
\aligned 
& 1 -{2M \over r '}=	1-{2M \over r}  + {2 M \over  r^2}  (r- r')  +O(r-r')^2,
 \\ 
 & 
 {q^2 (r ',  V_L,V_R)-1 \over 2} ={v^2 -1 \over 2} +  v  \del_ r v (r-r')+  O (r-r')^2.
  \endaligned 
$$ 
Hence, we have
$$
\aligned 
\mathcal F (r_R, V_L, V_R)- \mathcal F (r_L, V_L, V_R)=&  {2M \over r^2}{v^2 -1 \over 2} {v^2 -1 \over 2} +\Big(1-{2M \over r}\Big) v  \del_ r v  (r_R- r_L)+  O(r_R-r_L)^2 \\
=&  \del_ r \Big(((1-2M/r) {v^2 -1 \over 2}\Big) +  O(r_R-r_L)^2
 =  {2M \over r^2} (v^2 -1) (r_R- r_L)+O(r_R-r_L)^2.  \hfill \qedhere
\endaligned 
$$ 
\end{proof}


\paragraph{A second-order formulation}

We now extend the method to second-order. The solution is now discretized as a piecewise linear function, and we define 
\bel{Different-constant}
\Delta_j^n  V = \begin{cases}
 \min (2| \Delta_ {j -1/2} V^n |, 2| \Delta_ {j +1/2} V^n |, |\Delta_j   V^n|)
&  \text{if $\sgn \Delta_ {j -1/2} V^n = \sgn \Delta_ {j +1/2} V^n  = \sgn \Delta_j   V^n$}, 
\\ 0  & \text{otherwise,}
\end{cases}
\ee
where 
$$
\Delta_j   V^n= {1\over 2} (\Delta  V_ {j+1}^n- \Delta  V_ {j -1}^n), \qquad \Delta_ {j +1/2} V^n =  (\Delta  V_ {j+1}^n- \Delta  V_ j^n), \qquad \Delta_ {j -1/2} V^n =  (\Delta  V_ j^n- \Delta  V_ {j-1}^n). 
$$
Then, our second-order  scheme is stated as 
\bel{second-order-Godunov}
\aligned 
V_j^{n+1}  =&  V_j^n- 
{\Delta t\over \Delta r}\Big(\mathcal F(r_{j+1/2},V_j^{n+1/2, R},V_{j+  1}^{n+1/2, L})
\\ & -\mathcal   F(r_{j-1/2}, V_{j- 1}^{n+1/2, R}, V_j^{n+1/2, L})  \Big)- \Delta t {2M \over r_j^2} (V_j^2 -1),
\endaligned 
\ee
where 
$\mathcal F$ is the numerical flux \eqref{mathcal-F}. Here, the two values $V_{j+  1}^{n+1/2, L}, V_j^{n+1/2, R}$ are given by 
\bel{Riemann-two-value}
\aligned 
& V_j^{n+1/2, L} := V_j^{n, L}  -  {\Delta t \over 2}\Big({(1-2M/ r_j) V_j^n \Delta_j^n  V \over \Delta r} - {2M \over r_j^2} ({V_j^n}^2 -1) \Big),
\\ 
& V_j^{n+1/2, R} :=  V_j^{n, R}  -  {\Delta t \over 2}\Big({(1-2M/ r_j) V_j^n \Delta_j^n   V \over \Delta r} - {2M \over r_j^2} ({V_j^n}^2 -1) \Big),
\endaligned 
\ee
where, with $\Delta_j^n V$ defined by \eqref{Different-constant} and
$V_j^{n, L} = V_j^n -  {\Delta_j^n   V \over  2}$ and $V_j^{n, R}=  V_j^n +  {\Delta_j^n   V \over  2}$.  


\section{Numerical experiments with the finite volume scheme}
\label{Sec:4}

\paragraph{Asymptotic-preserving property}

We now present some numerical  tests with the proposed finite volume method applied to the relativistic Burgers equation  \eqref{Burgers'}. As mentioned earlier, we work within the domain $r>2M$, and the mass parameter $M$ is taken to be $M=1$ in all our tests. We work in the space interval $(r_{\min,}, r_{\max})$ with $r_{\min}= 2M=2$ and $r_{\max}=4$ and we take $256$ points to discreize the space interval. 

We begin by showing  that the method at, both, first-order and second-order accuracy  preserves the steady state solutions.  For positive/negative steady state Burgers solutions $v= \pm \sqrt {{3\over 4}+{1\over 2 r}}  $,  we see that the initial steady states are exactly conserved  by  the scheme. We also show that the following steady state shock is preserved by the scheme: 
$$
v= \begin{cases}
\sqrt {{3\over 4}+{1\over 2 r}}  & 2.0<r<3.0,
\\ 
- \sqrt {{3\over 4}+{1\over 2 r}}  & r>3.0. 
\end{cases}
$$
We obtain that our finite volume scheme preserves three typical forms for the static solutions, as is illustrated in Figures \ref{FIG-51} and {FIG-52}.

\begin{figure}[!htb] 
\centering 
\begin{minipage}[t]{0.3\linewidth}
\centering
\epsfig{figure=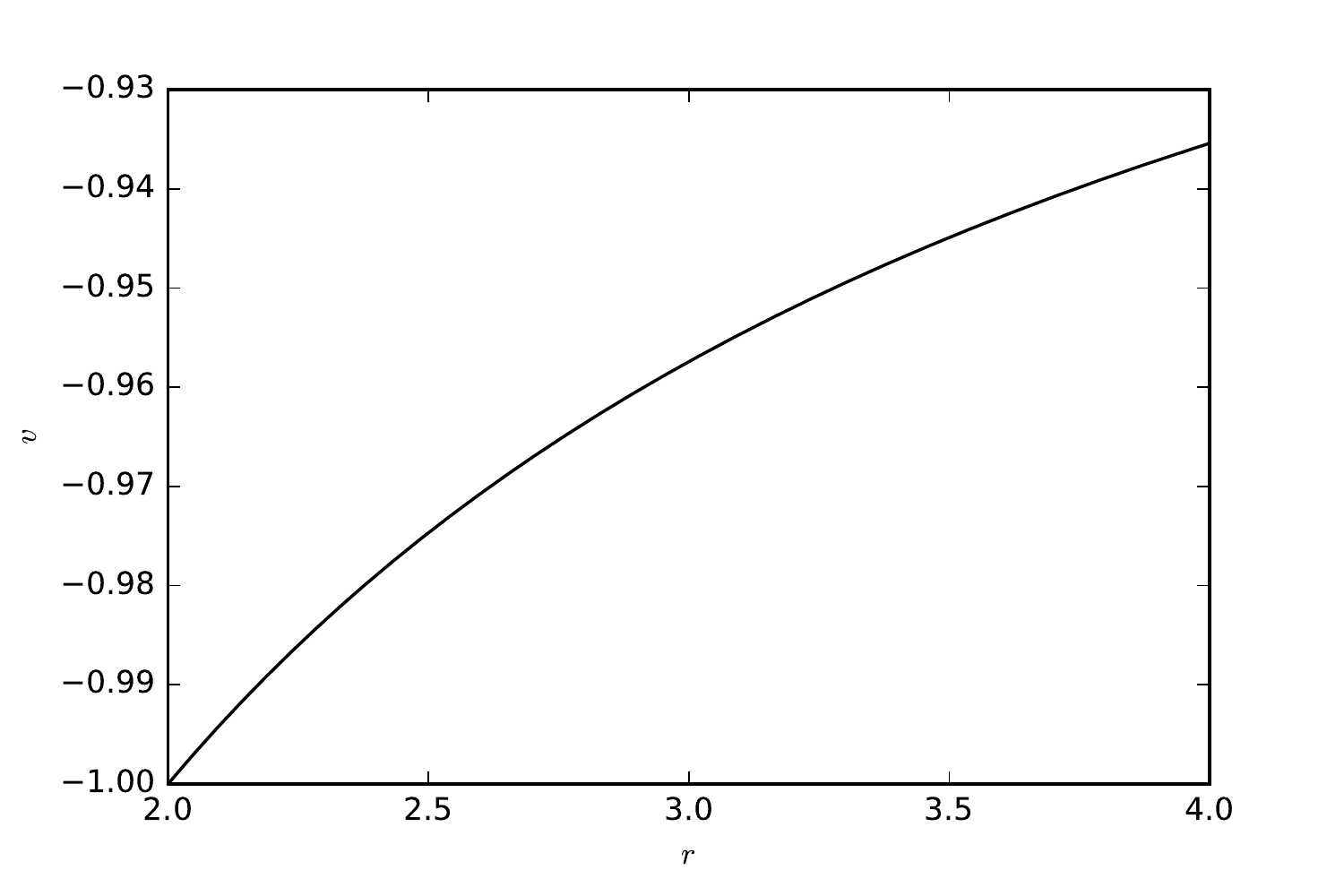,width= 2.5in} 
\end{minipage}
\hspace{0.1in}
\begin{minipage}[t]{0.3\linewidth}
\centering
\epsfig{figure=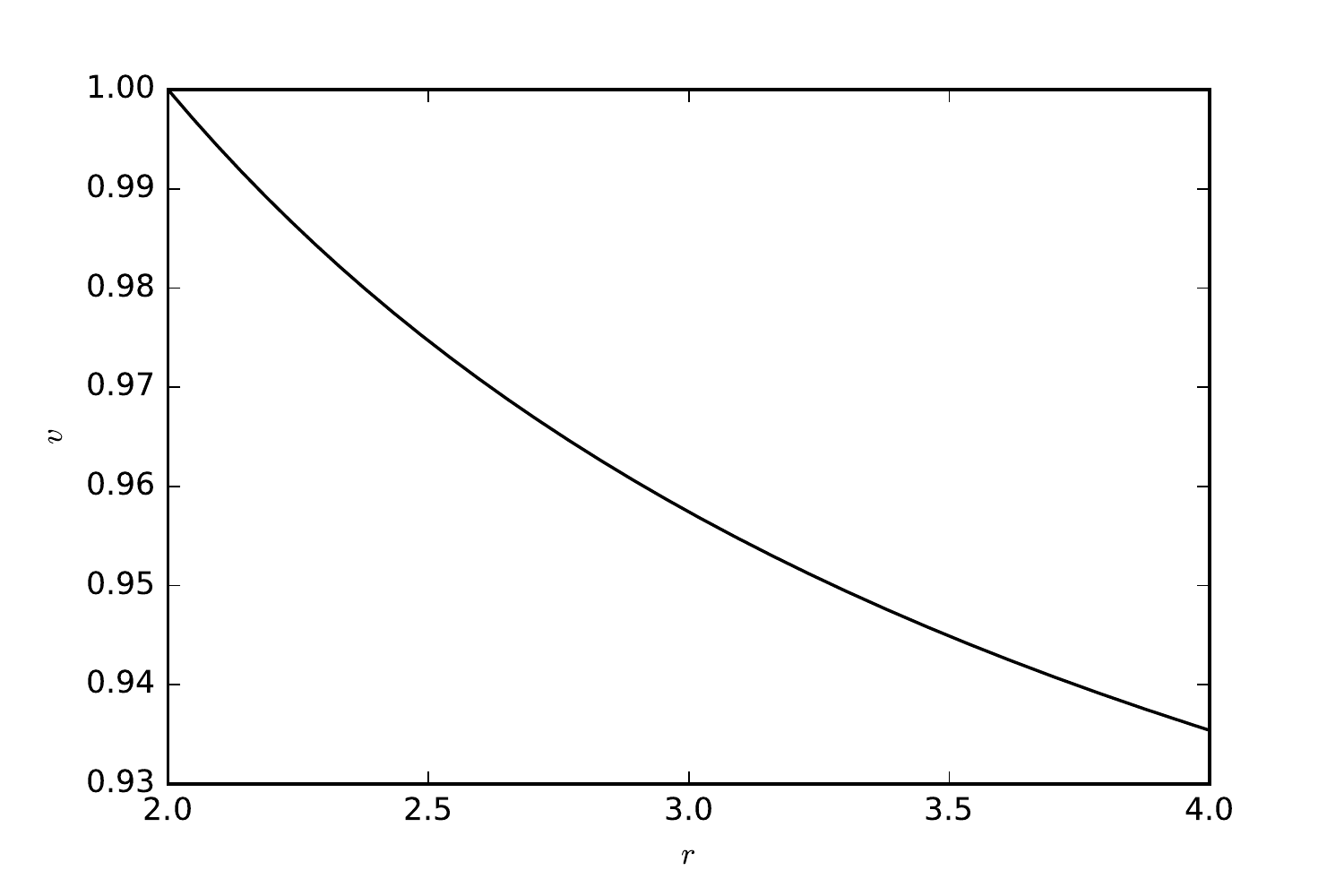,width=2.5in} 
\end{minipage}
\hspace{0.1in}
\begin{minipage}[t]{0.3\linewidth}
\centering
\epsfig{figure=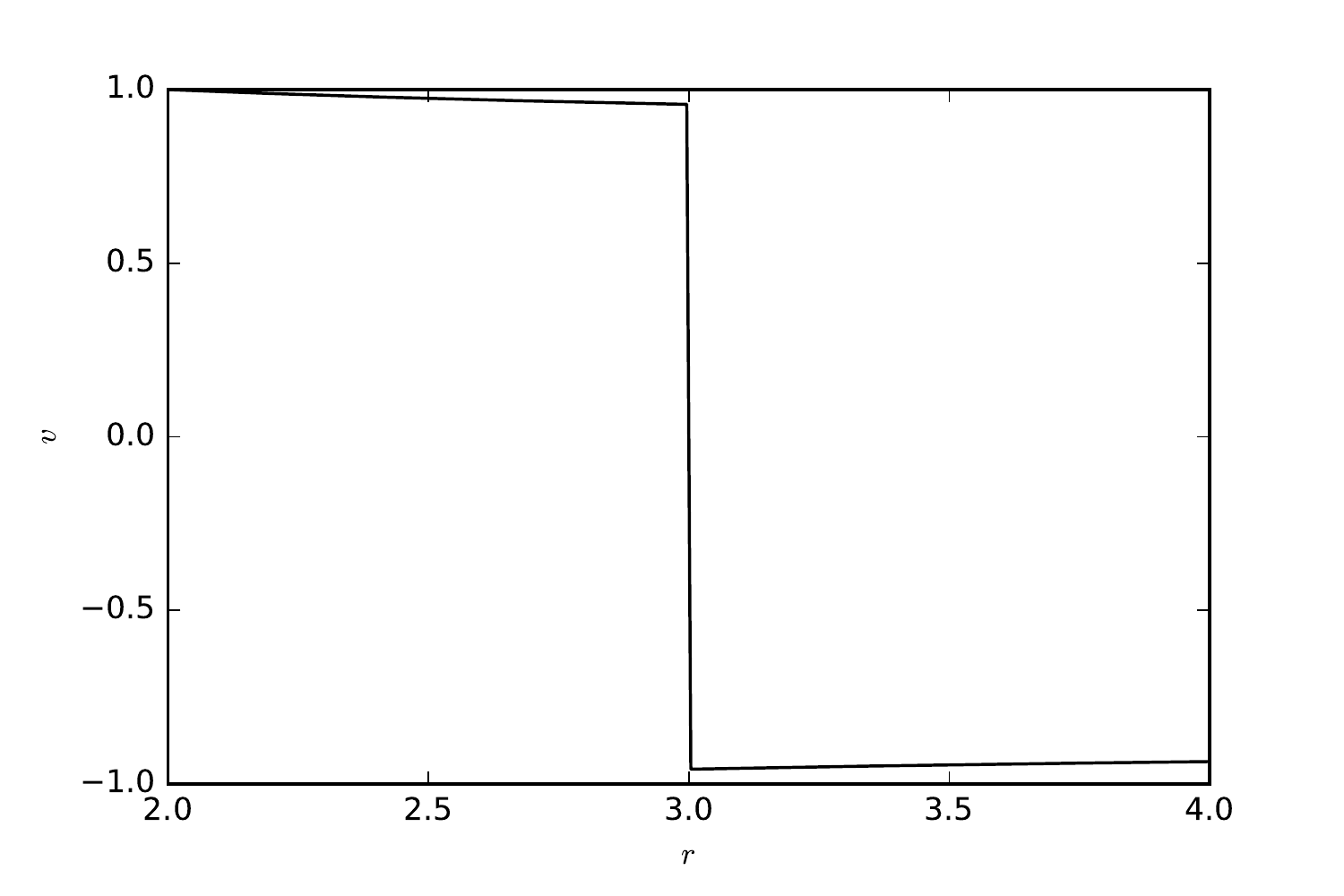,width=2.5in} 
\end{minipage}
\caption{Three static solutions}
\label{FIG-51} 
\end{figure}

\begin{figure}[!htb] 
\centering
\begin{minipage}[t]{0.3\linewidth}
\centering
\epsfig{figure=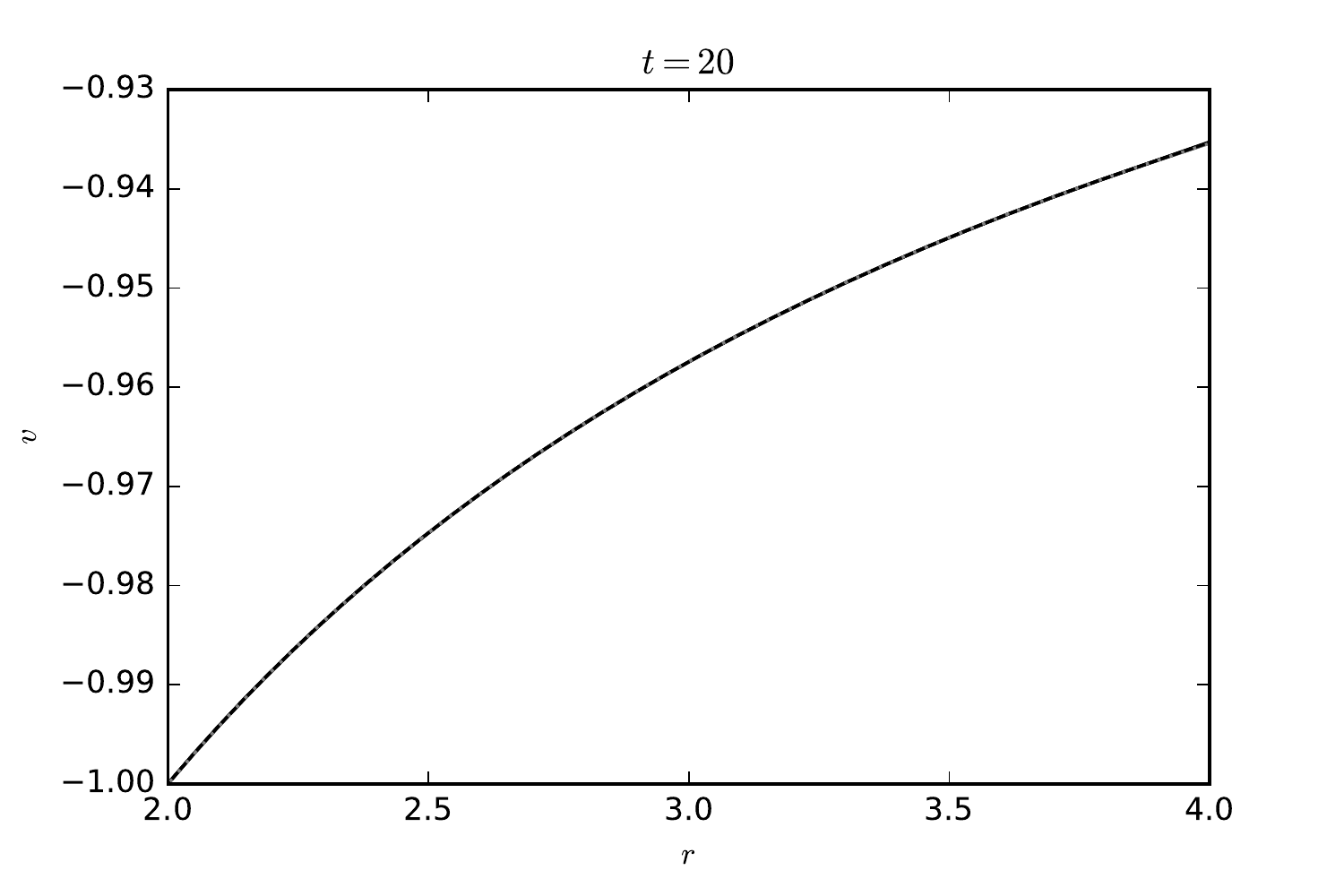,width= 2.5in} 
\end{minipage}
\hspace{0.1in}
\begin{minipage}[t]{0.3\linewidth}
\centering
\epsfig{figure=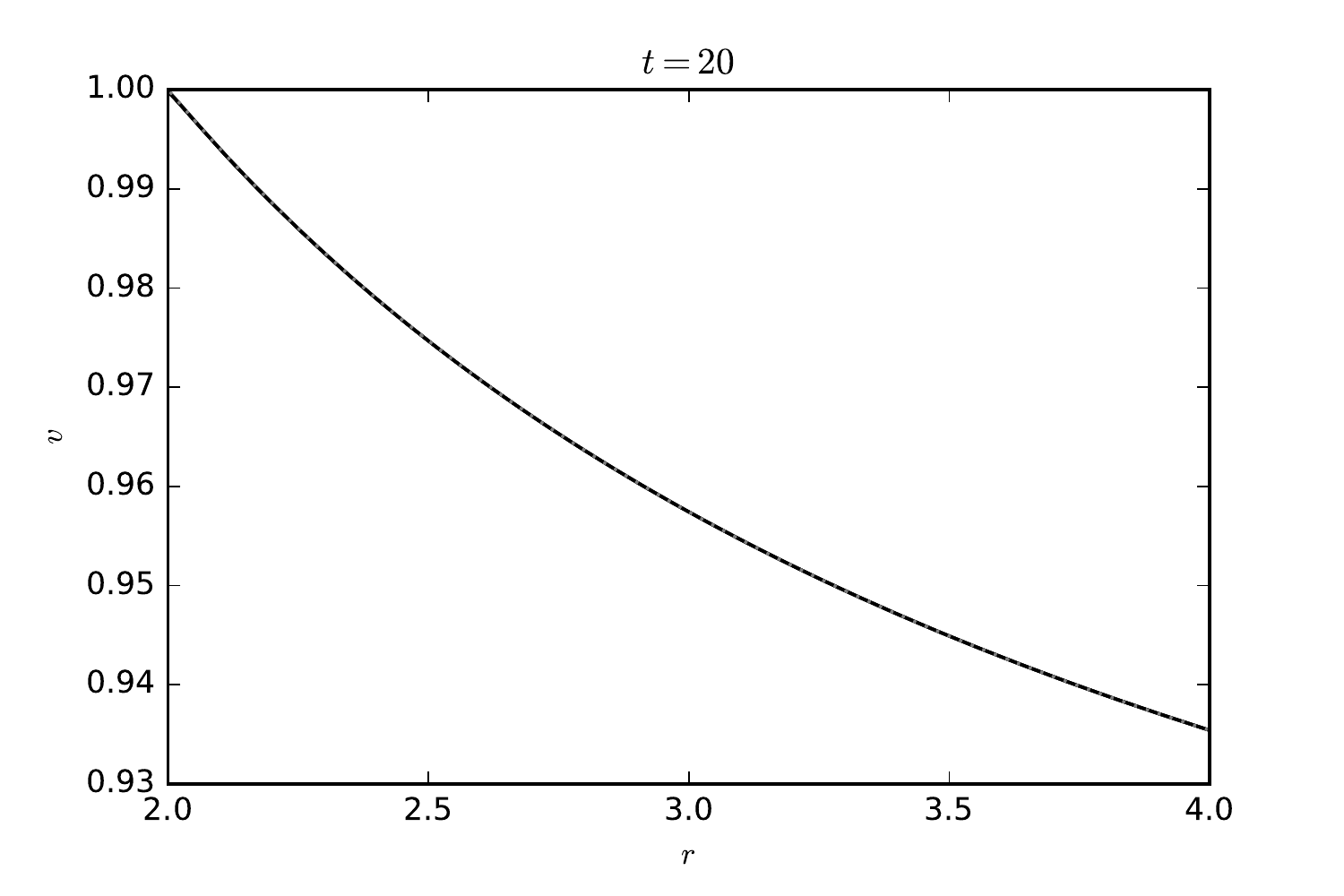,width=2.5in} 
\end{minipage}
\hspace{0.1in}
\begin{minipage}[t]{0.3\linewidth}
\centering
\epsfig{figure=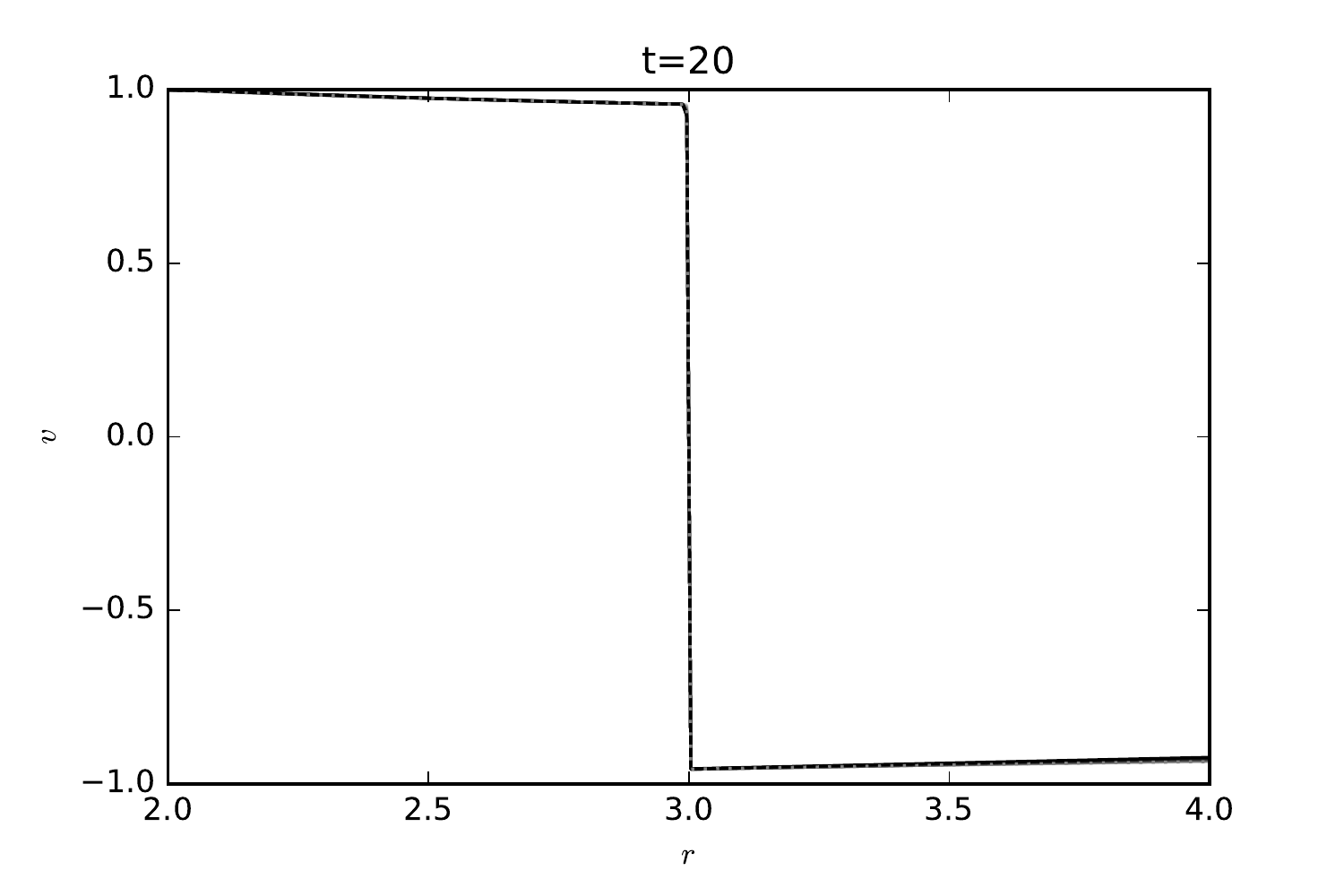,width=2.5in} 
\end{minipage}
\centering
\caption{Solution at time $t=20$ of a steady state, using the second-order finite volume scheme}
\label{FIG-52} 
\end{figure}


\paragraph{A moving shock separating two static solutions}

In view of Theorem~\ref{Burgers1},  whether the solution to the Riemann problem will move towards the blackhole horizon depends only on the behavior of the initial velocity.  We take again the space interval to be $(2.0,4.0) $ with 256 space mesh points.  We take then two kinds of initial data to be 
$$
v= \begin{cases}
\sqrt{{1\over 2}+{1\over r}} & 2.0<r <2.5, \\ 
\sqrt{2\over r} & r >2.5, 
\end{cases}
\qquad 
v= \begin{cases}
- \sqrt{2\over r}& 2.0<r <2.5, \\ 
- \sqrt{{3\over 4}+{1\over 4 r}}  & r >2.5. 
\end{cases}
$$
The behavior of the two shock solutions obtained with the first-order and second-order accurate versions are shown in Figures~\ref{FIG-53}, \ref{FIG-54}, \ref{FIG-55}, and \ref{FIG-56}.

\begin{figure}[!htb] 
\centering 
\begin{minipage}[t]{0.3\linewidth}
\centering
\epsfig{figure=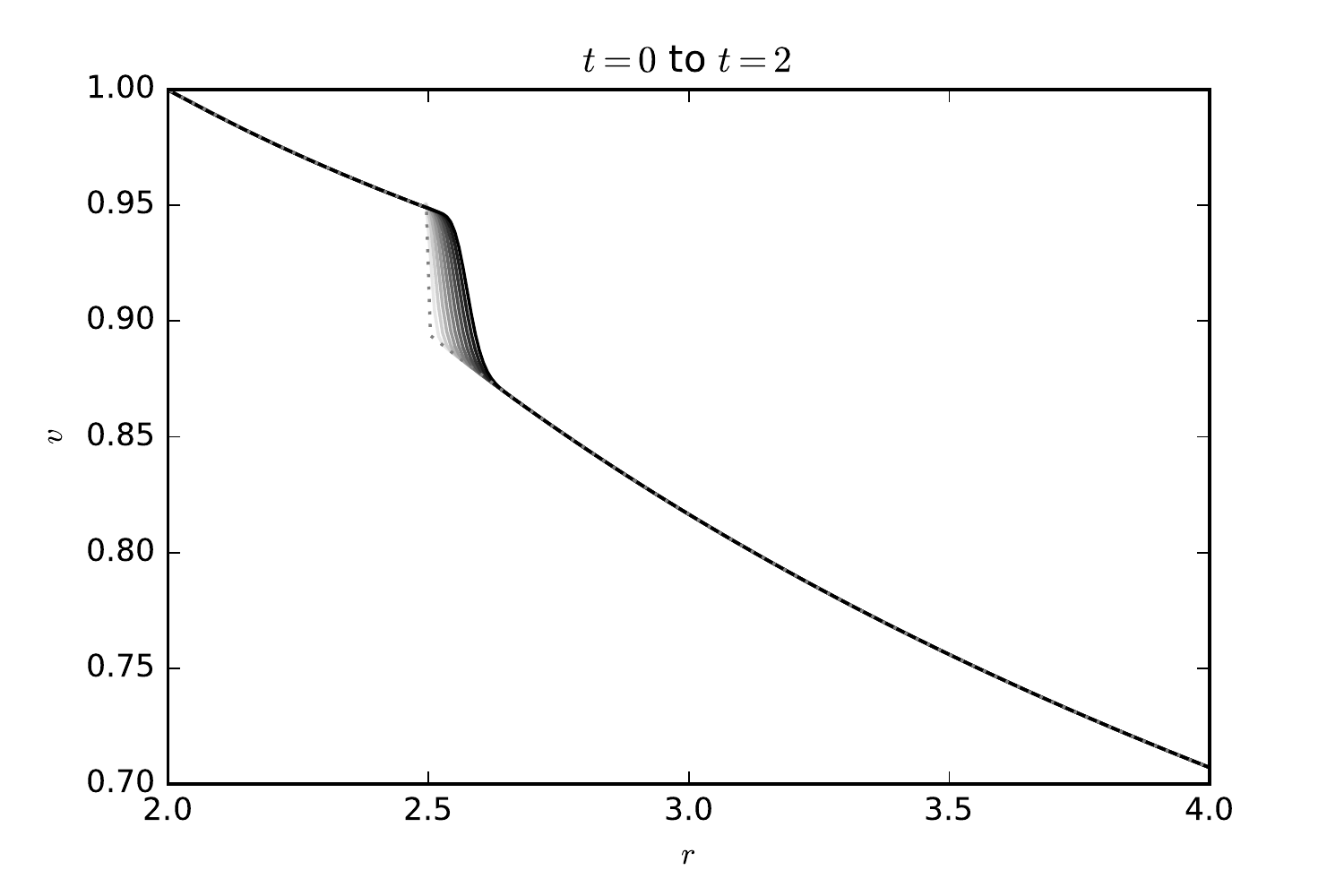,width=2.5in} 
\end{minipage}
\hspace{0.1in}
\begin{minipage}[t]{0.3\linewidth}
\centering
\epsfig{figure=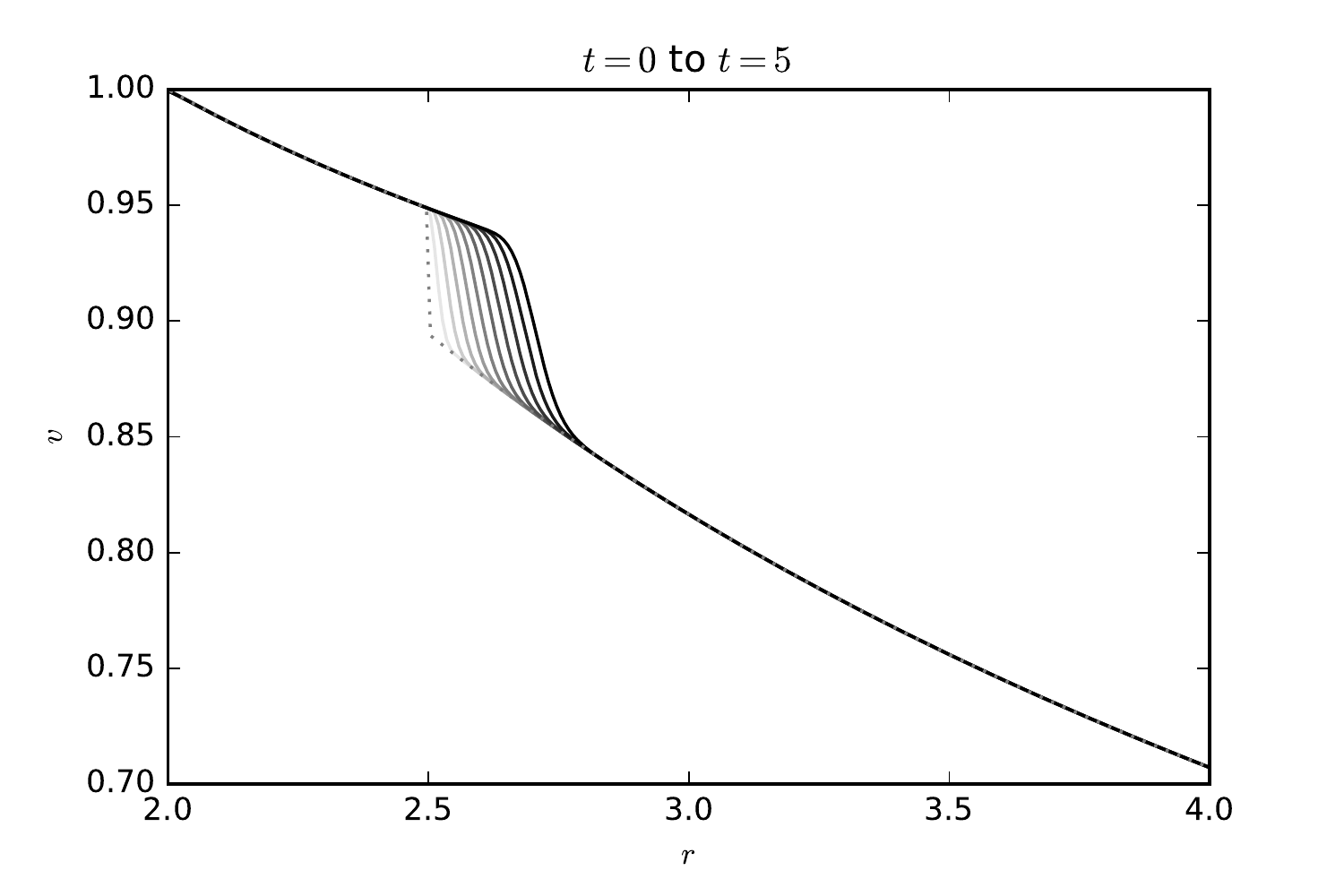,width= 2.5in} 
\end{minipage}
\hspace{0.1in}
\begin{minipage}[t]{0.3\linewidth}
\centering
\epsfig{figure=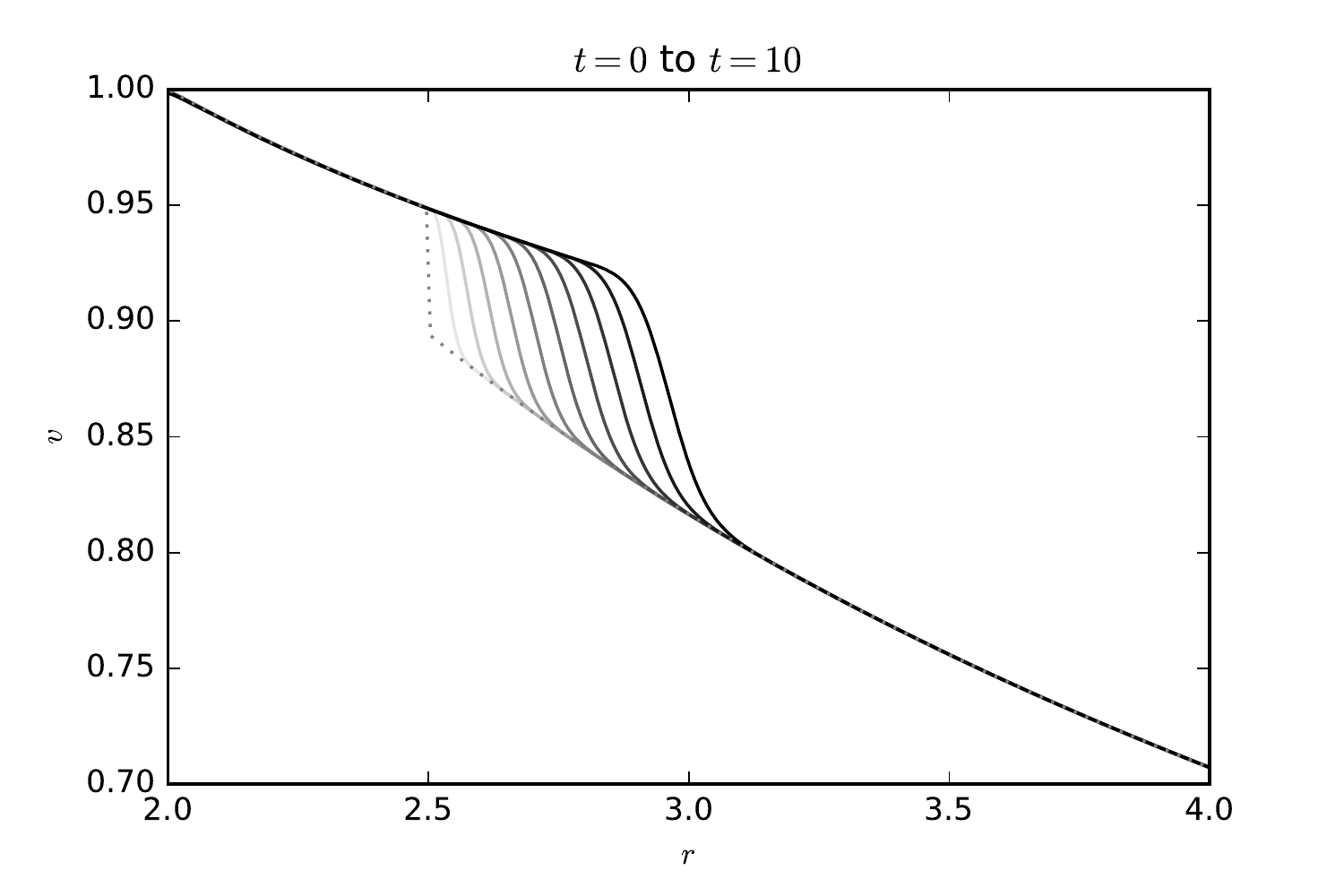,width=2.5in} 
\end{minipage}
\centering
\caption{Static solution with a right-moving shock computed with the first-order finite volume scheme}
\label{FIG-53} 
\end{figure}

\begin{figure}[!htb] 
\centering 
\begin{minipage}[t]{0.3\linewidth}
\centering
\epsfig{figure=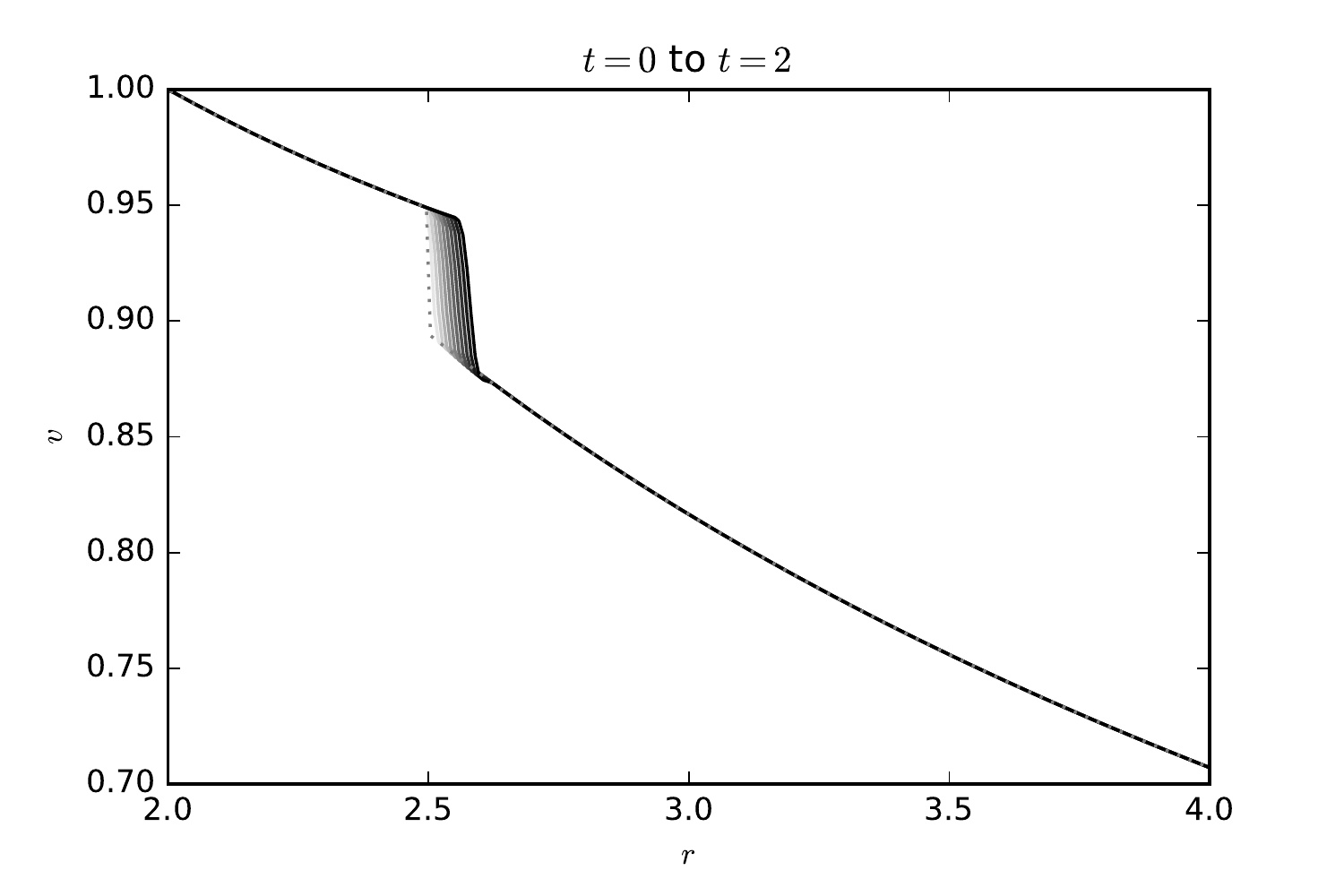,width=2.5in} 
\end{minipage}
\hspace{0.1in}
\begin{minipage}[t]{0.3\linewidth}
\centering
\epsfig{figure=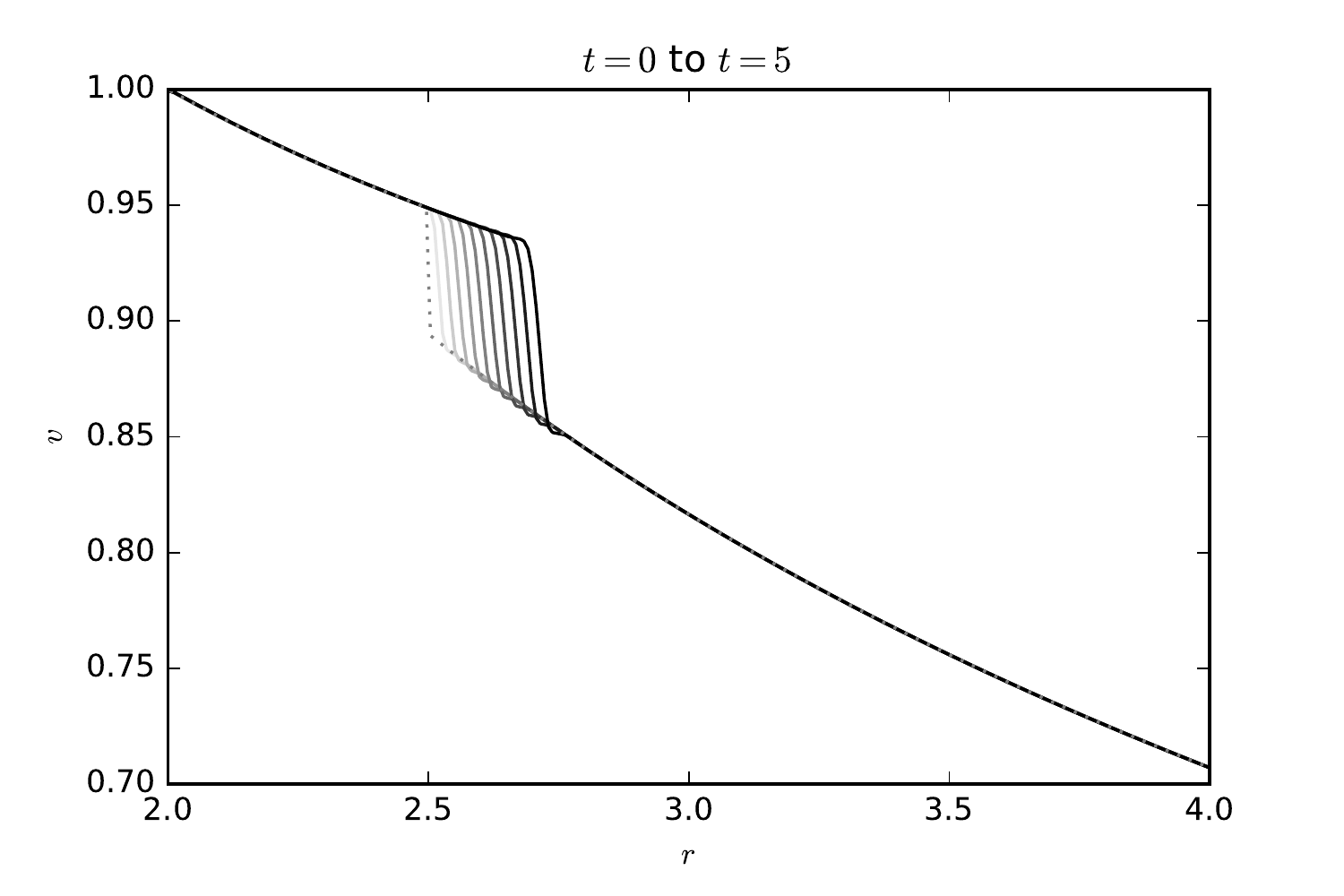,width= 2.5in} 
\end{minipage}
\hspace{0.1in}
\begin{minipage}[t]{0.3\linewidth}
\centering
\epsfig{figure=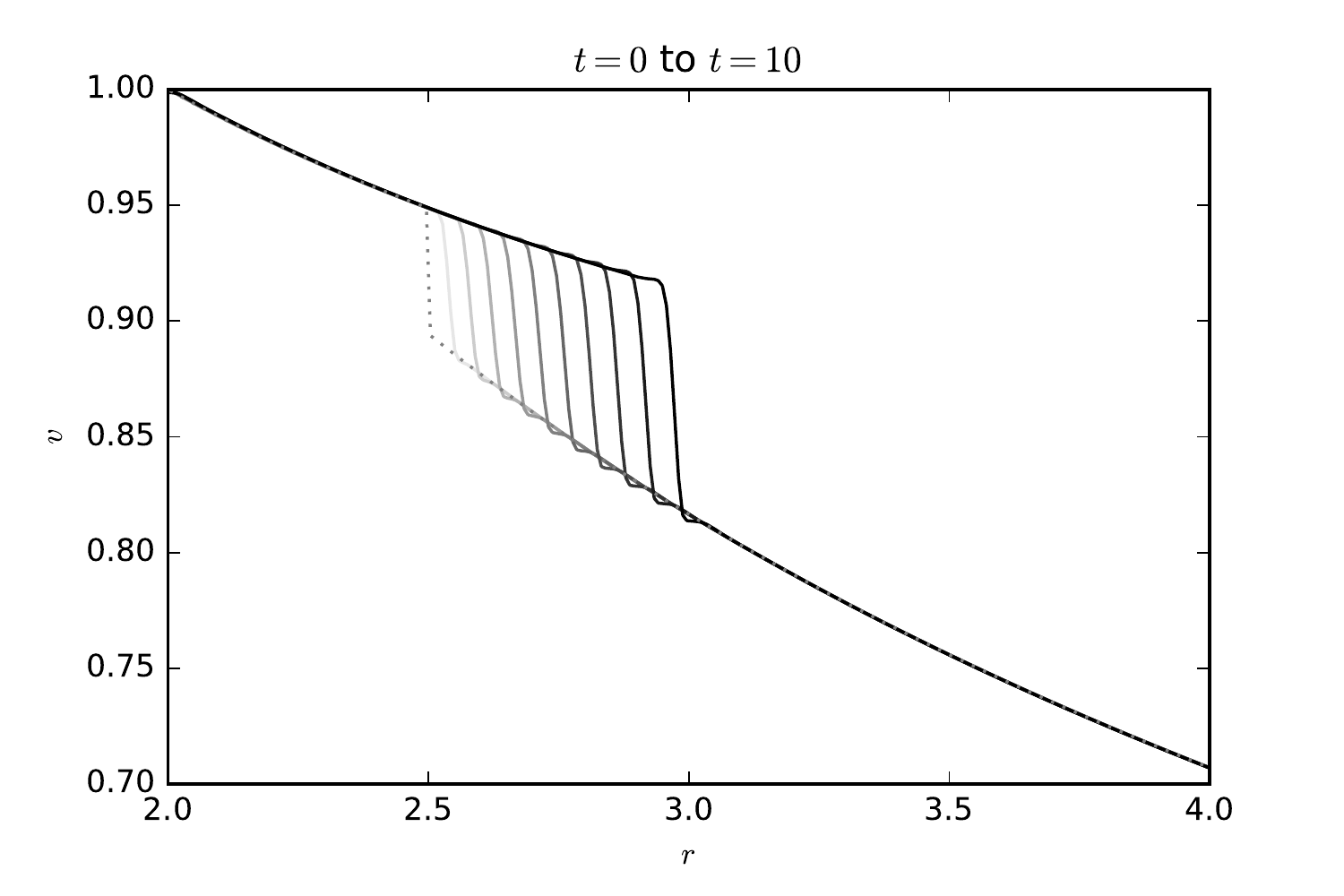,width=2.5in} 
\end{minipage}
\centering
\caption{Static solution with a right-moving shock computed with the second-order finite volume scheme}
\label{FIG-54} 
\end{figure}

\begin{figure}[!htb] 
\centering 
\begin{minipage}[t]{0.3\linewidth}
\centering
\epsfig{figure=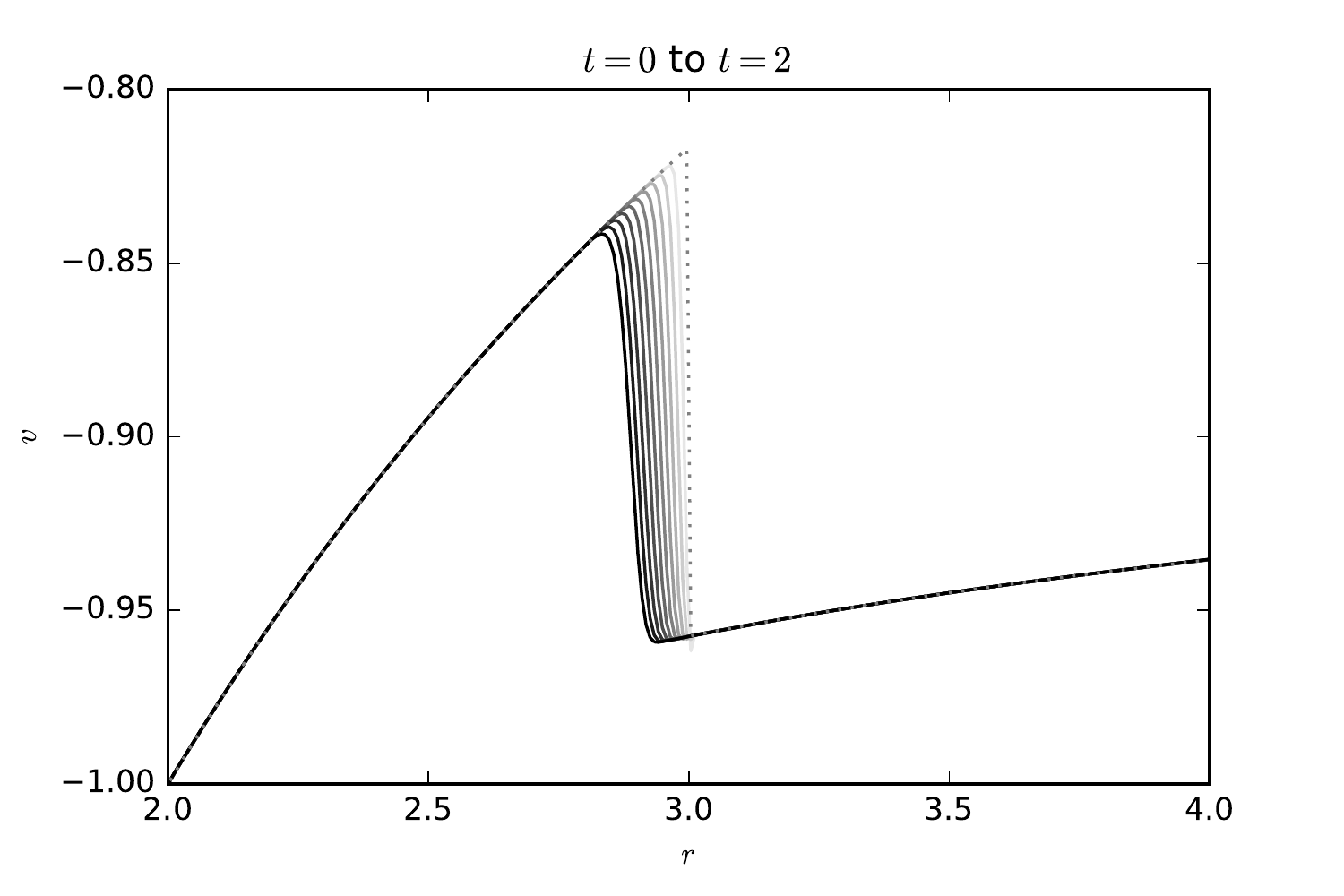,width= 2.5in} 
\end{minipage}
\hspace{0.1in}
\begin{minipage}[t]{0.3\linewidth}
\centering
\epsfig{figure=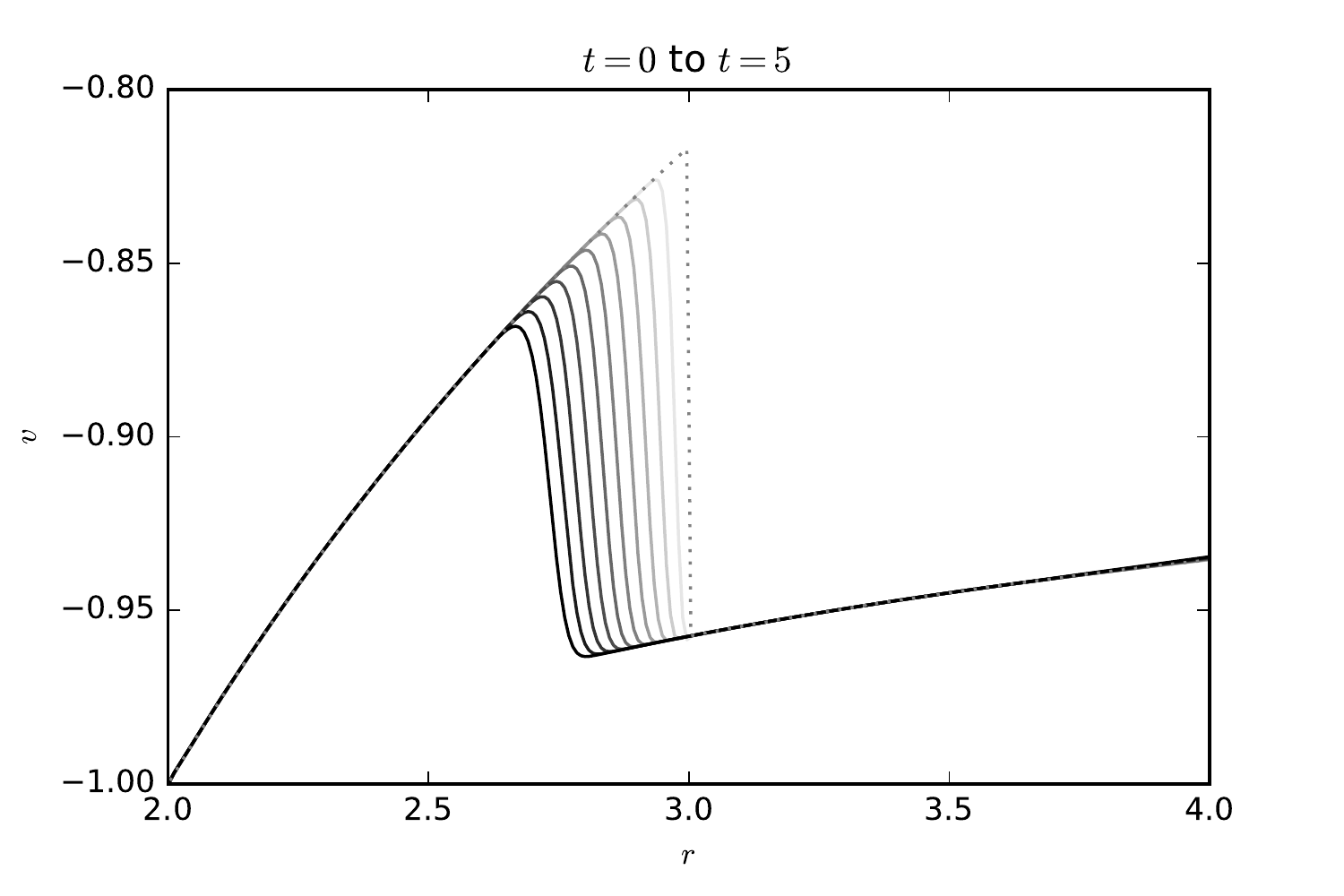,width= 2.5in} 
\end{minipage}
\hspace{0.1in}
\begin{minipage}[t]{0.3\linewidth}
\centering
\epsfig{figure=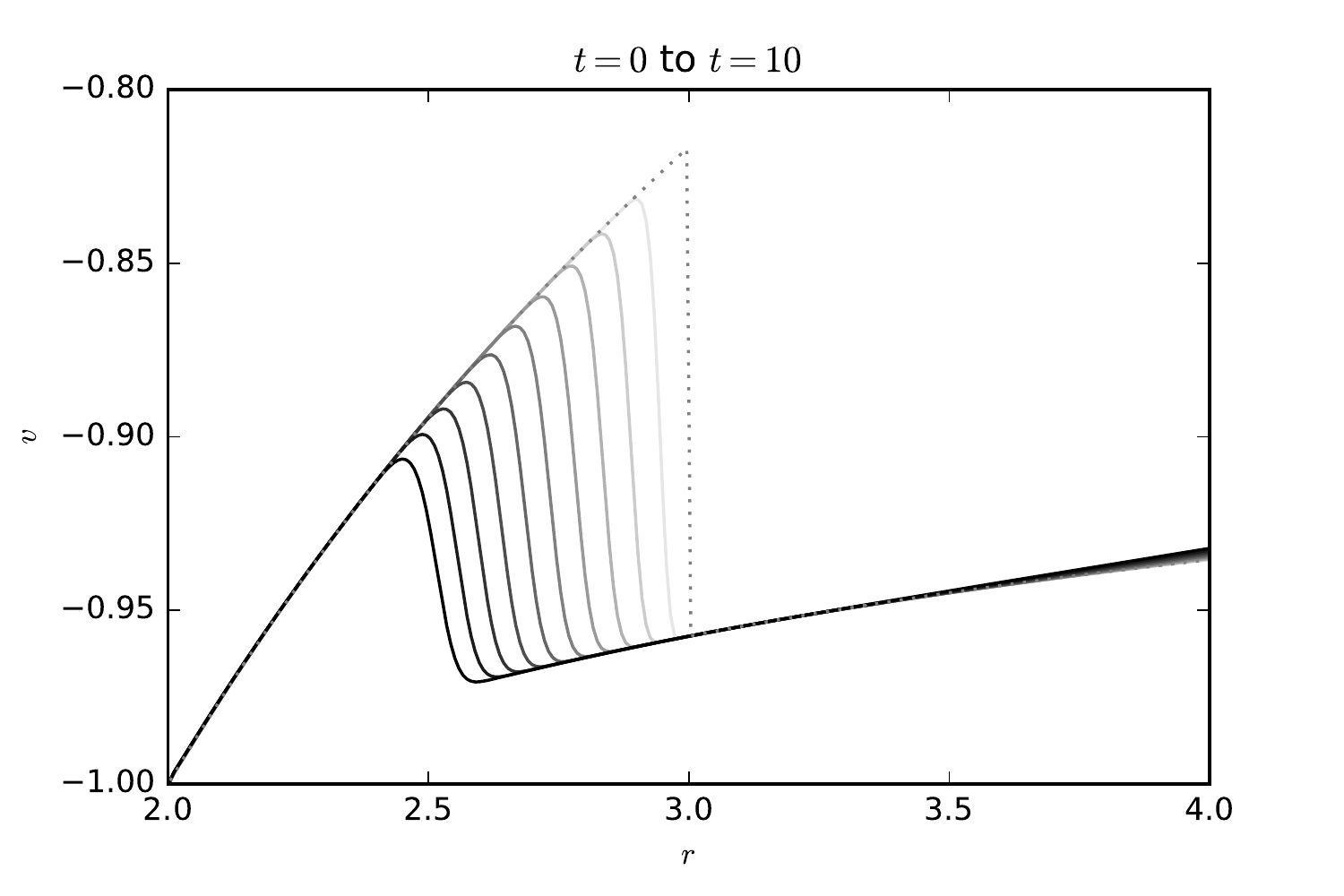,width= 2.5in} 
\end{minipage}
\caption{Static solution with a left-moving shock computed with the first-order finite volume scheme}
\label{FIG-55} 
\end{figure}

\begin{figure}[!htb] 
\centering 
\begin{minipage}[t]{0.3\linewidth}
\centering
\epsfig{figure=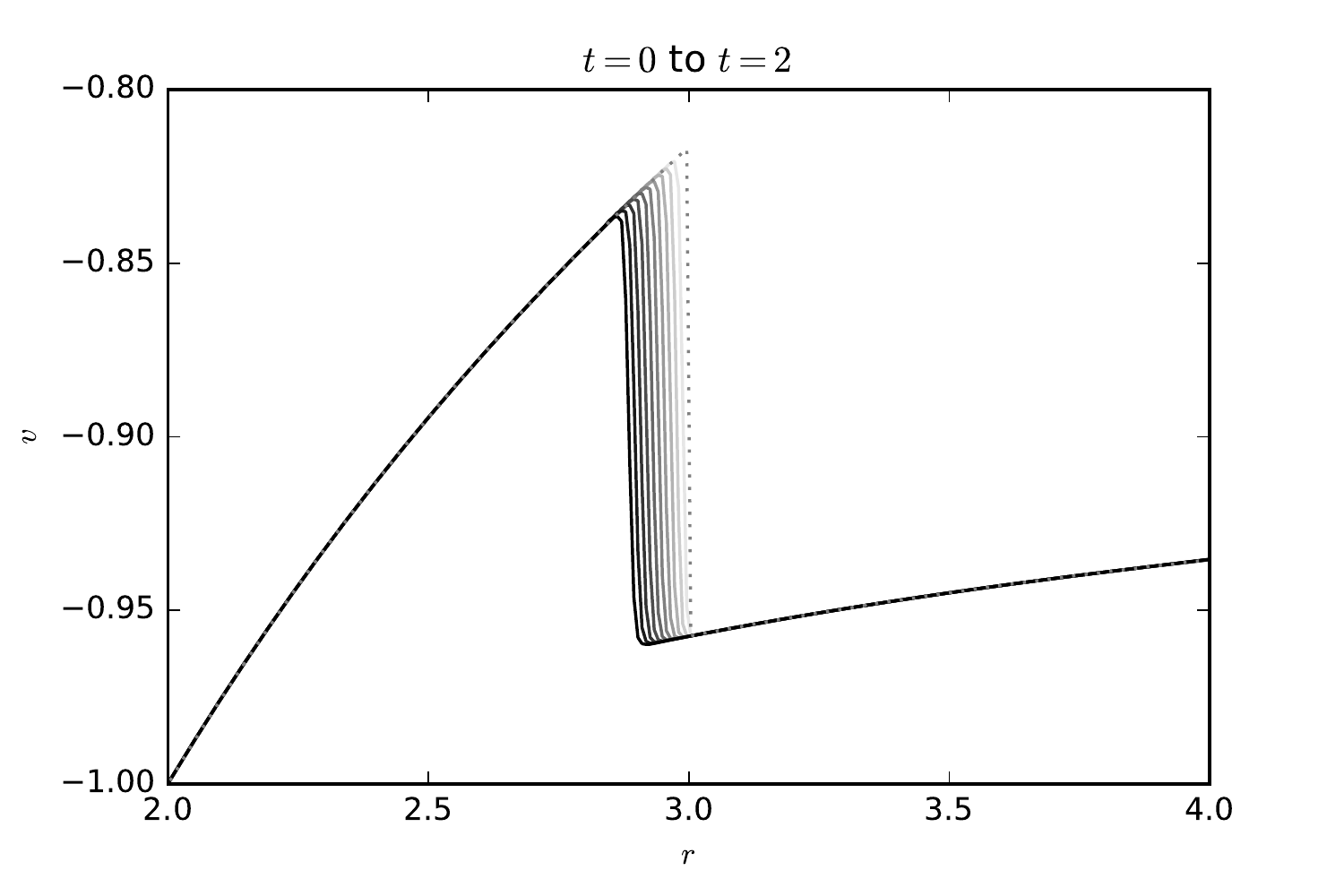,width= 2.5in} 
\end{minipage}
\hspace{0.1in}
\begin{minipage}[t]{0.3\linewidth}
\centering
\epsfig{figure=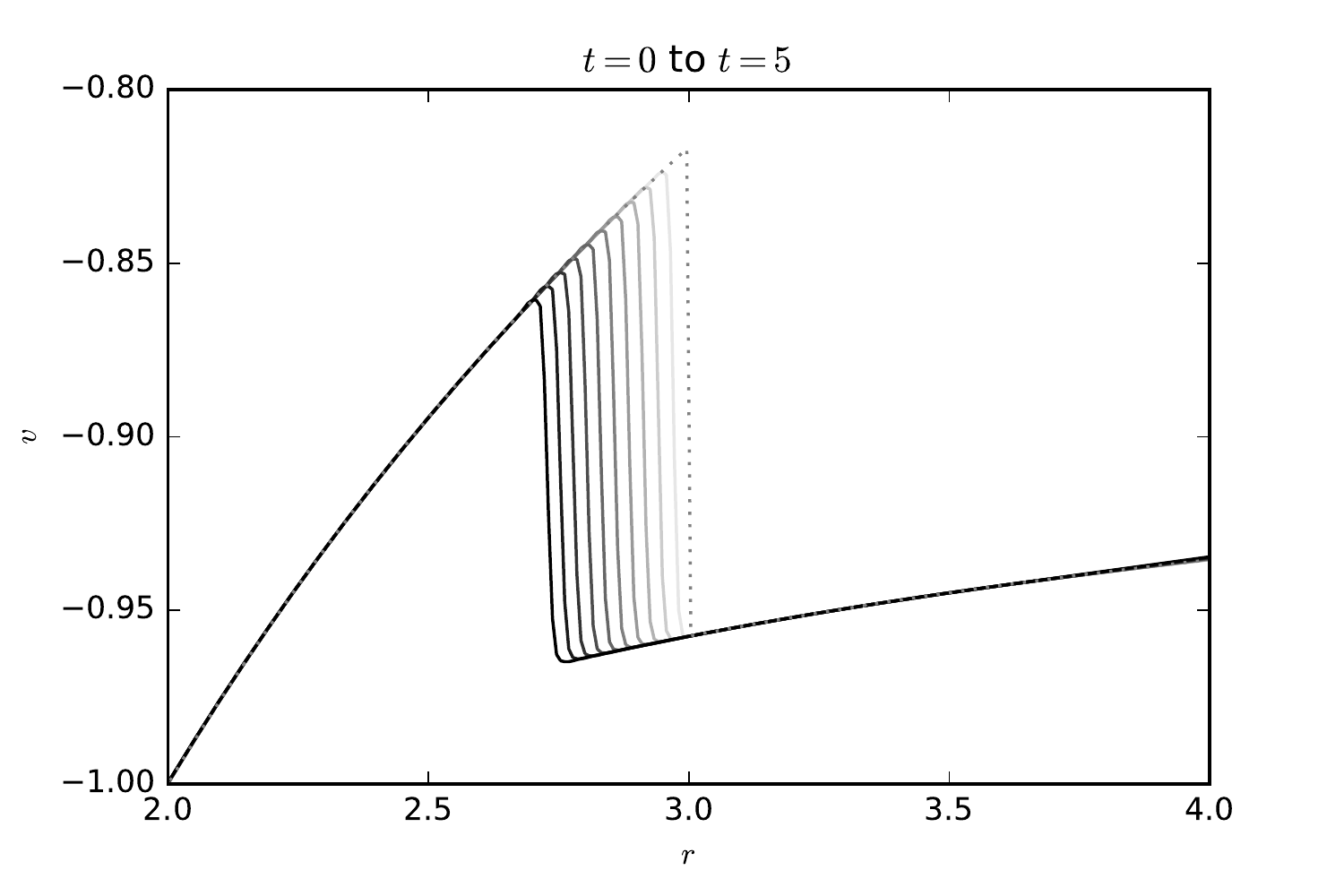,width= 2.5in} 
\end{minipage}
\hspace{0.1in}
\begin{minipage}[t]{0.3\linewidth}
\centering
\epsfig{figure=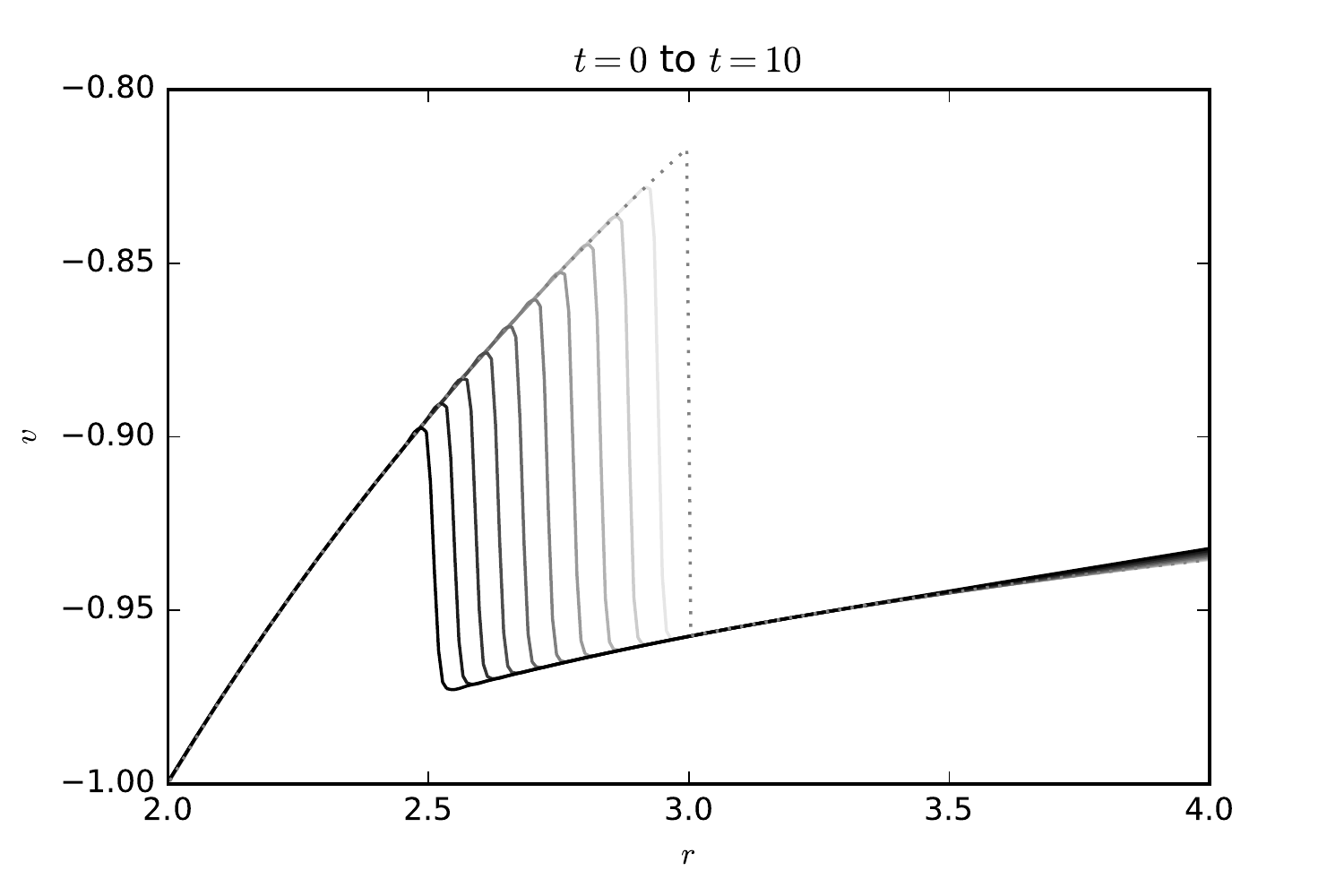,width= 2.5in} 
\end{minipage}
\caption{Static solution with a left-moving shock computed with the second-order finite volume scheme}
\label{FIG-56} 
\end{figure}


\paragraph{Late-time behavior of solutions}

We now study the late-time behavior of solutions whose initial data is given as \eqref{initial-steady-pe}, that is, a piecewise steady state solution with a compactly supported perturbation. We treat the following two kinds of piecewise steady state solutions: 
$$
 v= \sqrt{{1\over 2}+ {1\over r}}, \qquad v= \begin{cases}
\sqrt{{1\over 2}+{1\over r}} & 2.0<r <2.5, \\ 
\sqrt{2\over r} & r >2.5, 
\end{cases}
$$
with compactly supported perturbations.

\begin{figure}[!htb] 
\centering
\begin{minipage}[t]{0.3\linewidth}
\centering
\epsfig{figure=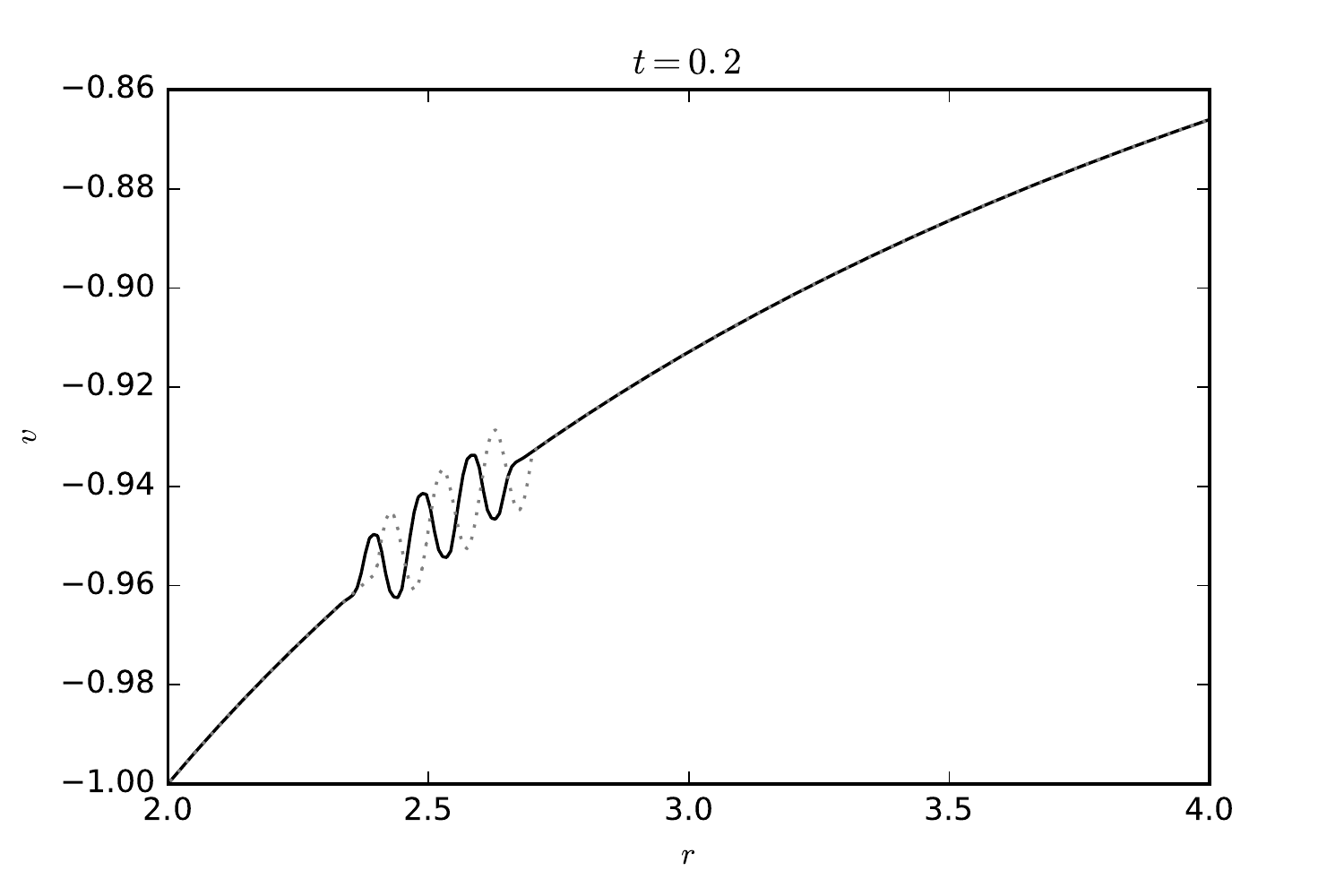,width=2.5in} 
\end{minipage}
\hspace{0.1in}
\begin{minipage}[t]{0.3\linewidth}
\centering
\epsfig{figure=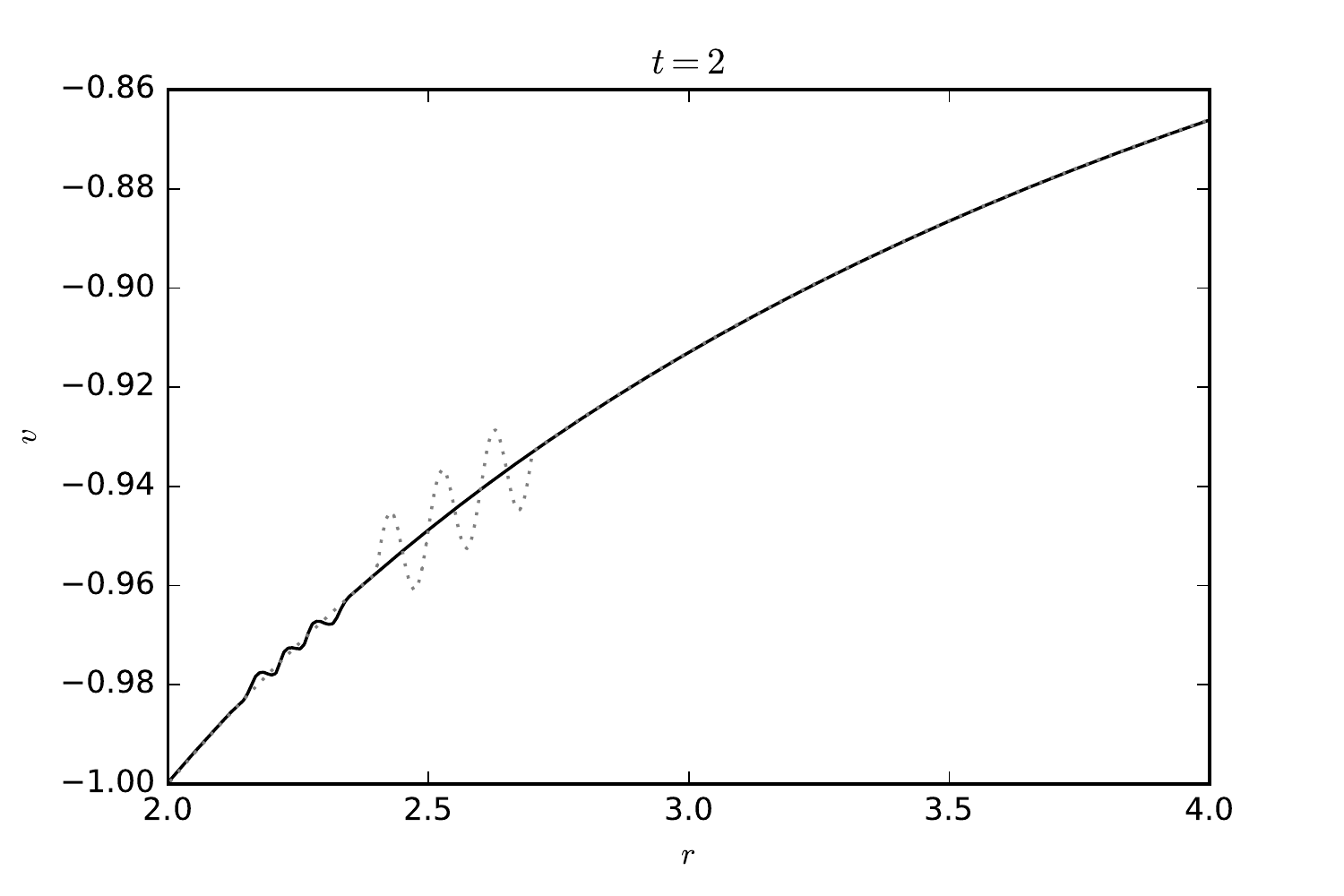,width=2.5in} 
\end{minipage}
\hspace{0.1in}
\begin{minipage}[t]{0.3\linewidth}
\centering
\epsfig{figure=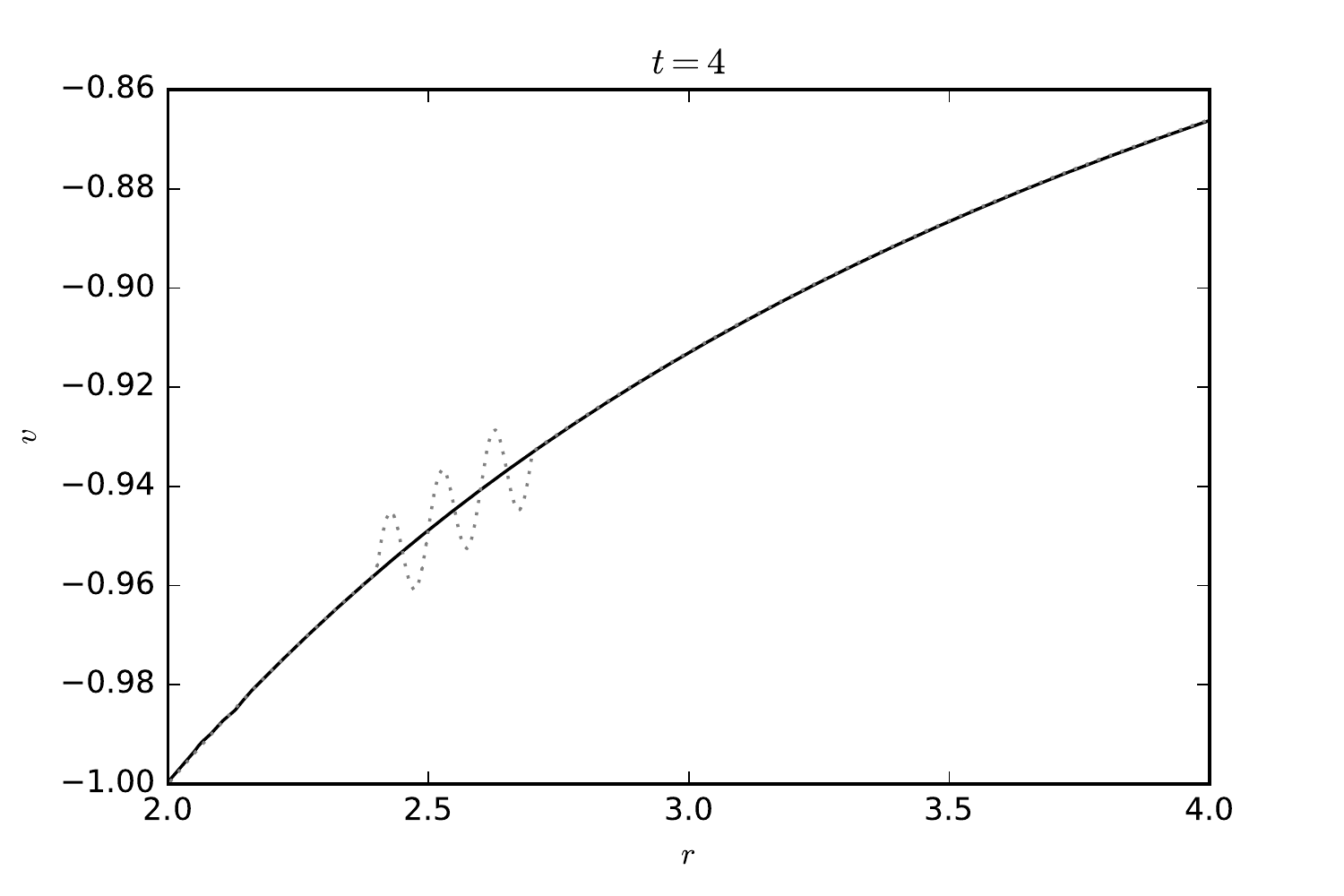,width=2.5in} 
\end{minipage}
\caption{Numerical solution from initially perturbed steady state}
\label{FIG-57} 
\end{figure}

\begin{figure}[!htb] 
\centering
\begin{minipage}[t]{0.3\linewidth}
\centering
\epsfig{figure=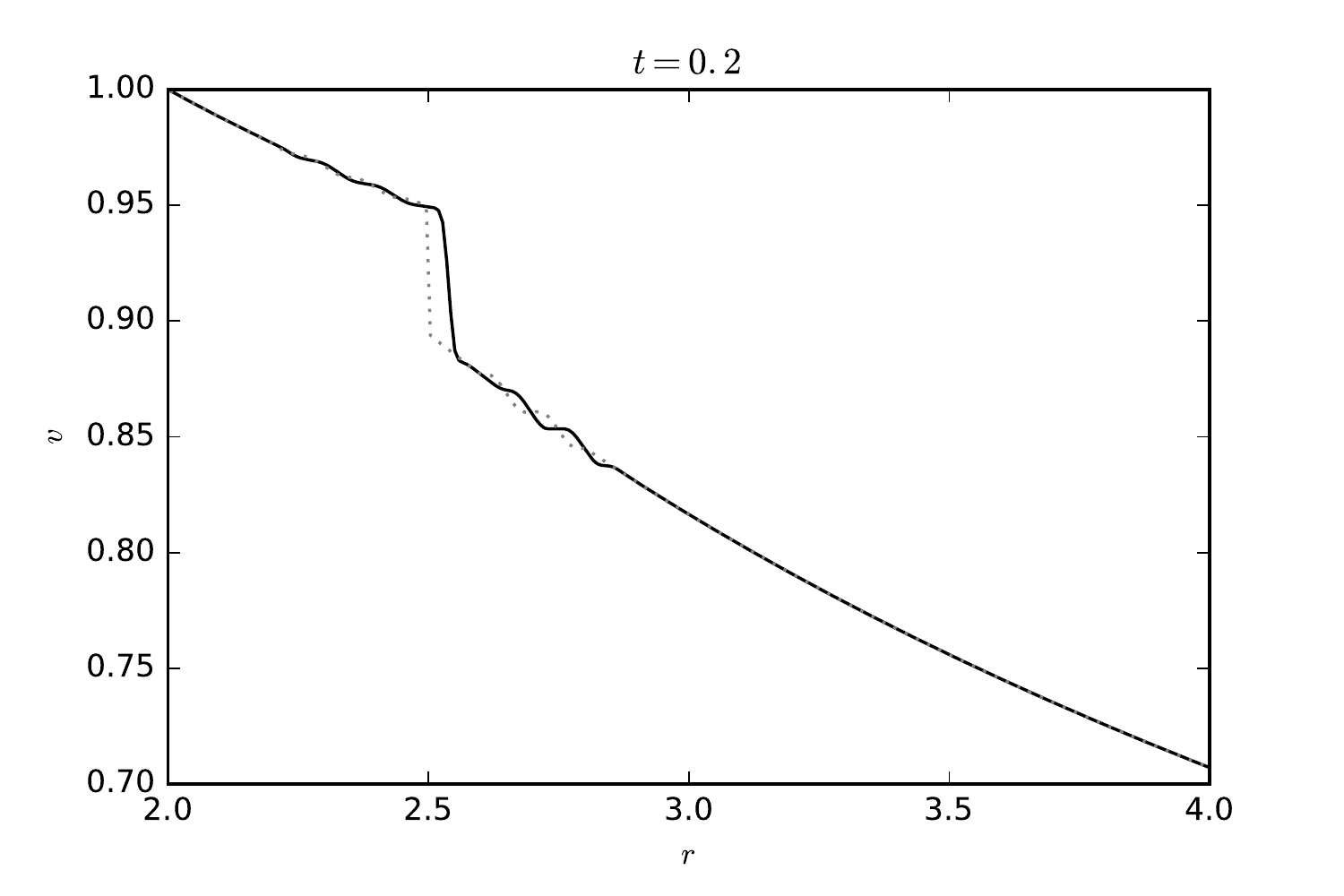,width=2.5in} 
\end{minipage}
\hspace{0.1in}
\begin{minipage}[t]{0.3\linewidth}
\centering
\epsfig{figure=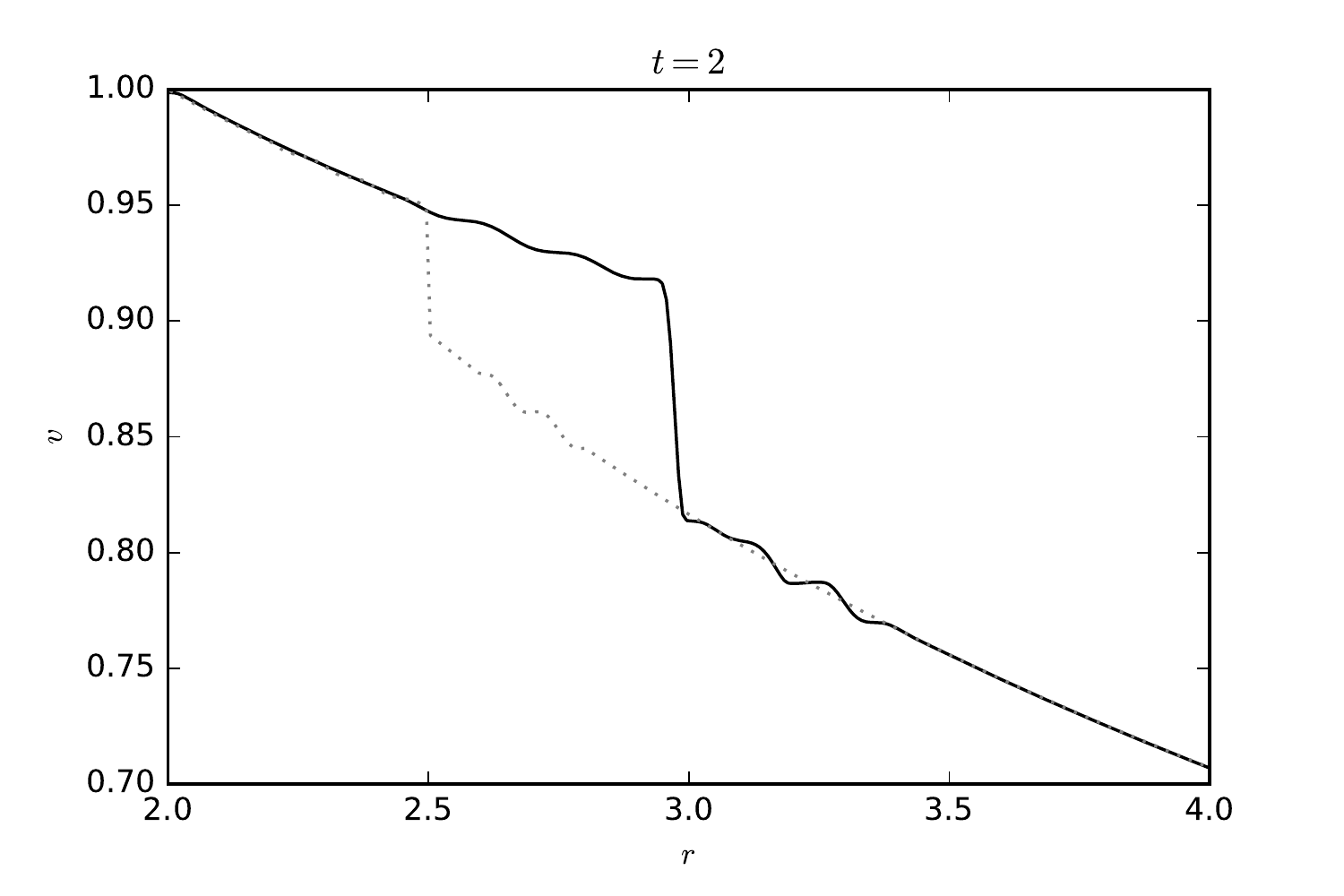,width=2.5in} 
\end{minipage}
\hspace{0.1in}
\begin{minipage}[t]{0.3\linewidth}
\centering
\epsfig{figure=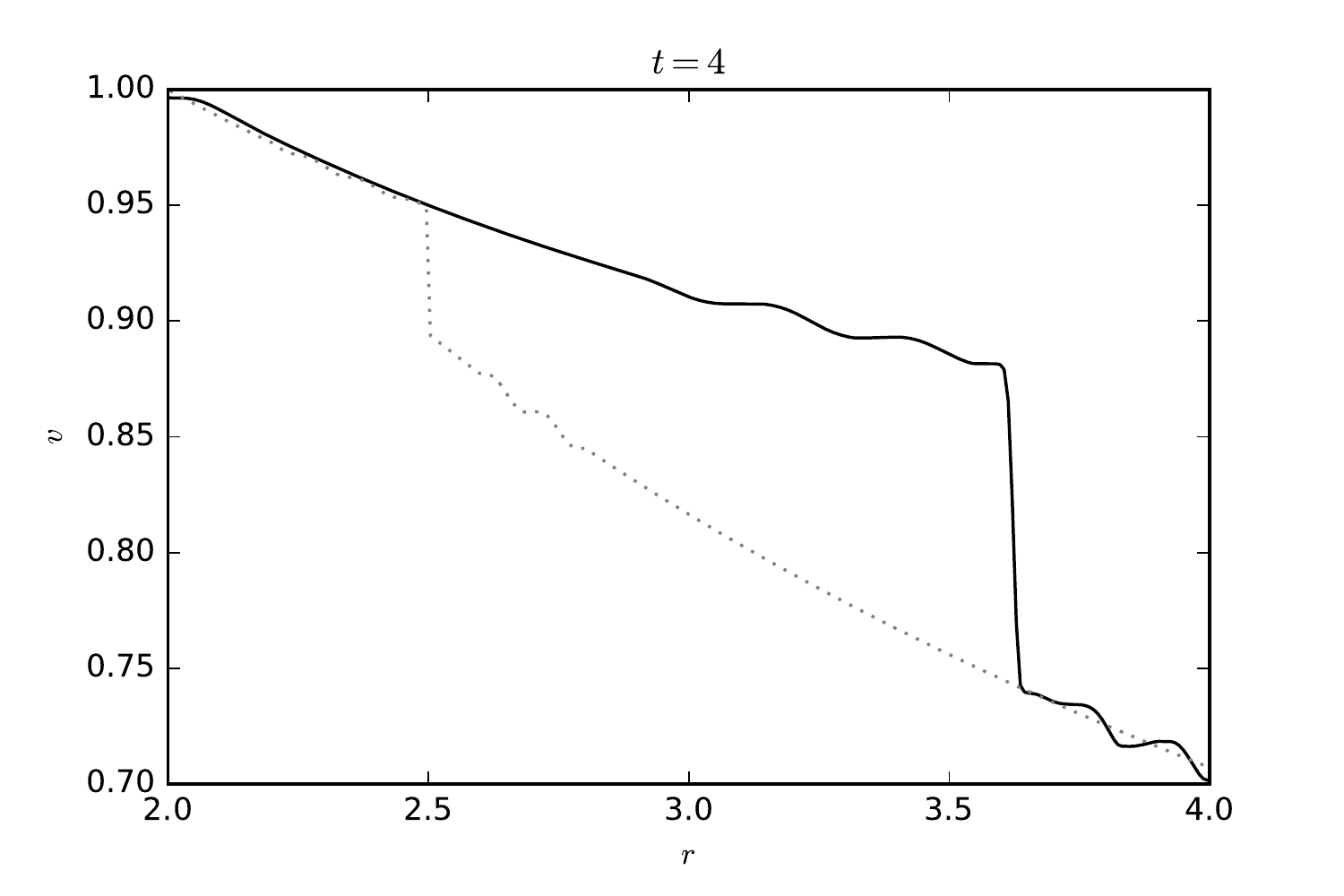,width=2.5in} 
\end{minipage}
\caption{Numerical solution from an initially perturbed shock}
\label{FIG-58} 
\end{figure}


\section{A generalized random choice scheme for the relativistic Burgers model}
\label{Sec:5}

\paragraph{Explicit solution to the generalized Riemann problem}

In order to construct a Glimm method for the relativistic Burgers model, we need first introduce the explicit form of the generalized Riemann problem of the relativistic Burgers equation \eqref{Burgers},which is an initial problem whose initial data $v_0 = v_0 (r)$ is given as 
\bel{initial-Riemann}
v_0 (r)= \begin{cases}
v_L (r) & 2M <r < r_0, \\ 
v_R (r) & r>r_0, 
\end{cases}
\ee
where $r_0$ is a fixed point in space and $v_L =v_L (r)$, $v_R= v_R (r)$ are two steady state solutions  of the Burgers' equation with explicit forms 
\bel{v_L-v_R}
v_L(r)= \sgn(v_L^0) \sqrt{1- K_L^2 \Big(1-{2M\over r} \Big)}, \qquad v_R (r) = \sgn(v_R^0) \sqrt{1- K_R^2 \Big(1-{2M\over r} \Big)},
\ee
where $K_L, K_R>0 $ are two constants and we denote  by $v_L^0 = v_L (r_0), v_R (r_0) =v_R^0$. 
The existence of the generalized Riemann problem is concluded in Theorem~\ref{Burgers1}. More precisely, the solution to the Riemann problem $v= v(t, r)$ can be realized by either a shock wave or a rarefaction wave which is given explicitly by the following form: 
\bel{Rie-sol}
v (t, r) = \begin{cases}
v_L (r) & r< r_L(t), \\
\widetilde v (t, r) & r_L(t)<r < r_R(t), \\ 
v_R (r) & r>r_R(t). 
\end{cases}
\ee
Here, $r_L(t)$ and $r_R(t)$ are bounds of rarefaction regions satisfying 
\bel{rj}
R_j\big(r_j (t)\big)- R_j (r_0) =  t, 
\ee
where $R_j = R_j (r)$ is given by 
\bel{Rj}
R_j(r): = {R^{v_j} (r) \over 2} + \chi_{[v_j^0< v_k^0  ]} (r){R^{v_j} (r) \over 2} +  \chi_{[v_j^0< v_k^0  ]} (r){R^{v_k} (r) \over 2}
\ee
with $j =L, R$,$k = R, L$, 
$$
 \chi_{[v_j^0  \gtrless v_k^0  ]} (r) = \begin{cases}
 1 & \text{if $v_j^0  \gtrless v_k^0$}, \\ 
 0 & \text{otherwise}, 
 \end{cases}
 $$
 and the function $R^v_j = R^v_j(r)$  is 
\bel{R^v_j}
\aligned 
R^{v_j} (r) : = & \sgn(v_j){1  \over (1/\eps^2 - K_j^2)^{3/2}} 
\Bigg (2 M \eps \big({1\over \eps^2}  - K_j^2\big)^{3/2} \ln (r- 2M)
\\  
& -2M  \big({1\over \eps^2}  - K_j^2\big)^{3/2} \ln \Big({2 r \over \eps}  \sqrt {{1 \over \eps^2}- K_j^2 \Big(1- {2M\over r}\Big)} +(2M-r) K_j^2\Big)
\\
& 
+ {1\over \eps} \bigg(r \sqrt {{1\over \eps^2}  -K_j^2}   \sqrt {{1 \over \eps^2}- {K_j}^2\Big(1- {2M\over r}\Big)} 
\\
& + M(2/ \eps^2 -3 K_*^2) \ln \Big(r \sqrt {{1\over \eps}- K_j^2}   \sqrt {{1 \over \eps^2}- K_j^2 \Big(1- {2M\over r}\Big)} + (M-r) K_j^2+{r \over \eps^2}  \Big)
 \bigg) 
 \Bigg). 
\endaligned 
\ee
The function $\widetilde v= \widetilde v(t,r)$ denotes the generalized rarefaction wave 
\bel{curved-ra}
\widetilde v(t, r) = \sgn(r-r_0)  \sqrt{{1\over \eps^2} -K^2 (t,r) \Big(1-{2M \over r}\Big)},  
\ee
where $K = K (t, r)$ is characterized by the condition 
\bel{K}
 \sgn(r-r_0) =   {\widetilde R(r, K)-\widetilde R(r_0, K) \over t},  
\ee
where 
\bel{Rcurve}
\aligned 
\widetilde R(r, K): & ={1  \over (1/\eps^2 - K^2)^{3/2}}  \Bigg (2 M \eps \big({1\over \eps^2}  - K^2\big)^{3/2} \ln (r- 2M)
\\  
& -2M  \big(1/ \eps^2 - K^2\big)^{3/2} \ln \Big({2 r \over \eps}  \sqrt {{1 \over \eps^2}- K^2 \Big(1- {2M\over r}\Big)} +(2M-r) K^2\Big)
\\
& 
+ {1\over \eps} \bigg(r \sqrt {{1\over \eps^2}  - K^2}   \sqrt {{1 \over \eps^2}- K^2\Big(1- {2M\over r}\Big)} 
\\
& + M \big (2/ \eps^2  -3 K^2) \ln \Big(r \sqrt {{1\over \eps}- K^2}   \sqrt {{1 \over \eps^2}- K^2 \Big(1- {2M\over r}\Big)} + (M-r) K^2+{r \over \eps^2}  \Big)
 \bigg) 
 \Bigg). 
\endaligned 
\ee 

Indeed,  referring to~\cite{PLF-SX-two}, the solution constructed by \eqref{Rie-sol} is proven to be unique, satisfying the Rankine-Hugoniot jump condition and the entropy inequality at the same time.  Besides, the solution to the generalized Riemann problem is globally defined both in time and in space. 


\paragraph{A generalized random choice method}

The random choice method is a scheme based on the result of generalized Riemann problem.  We use again the time-space grid where the mesh lengths in time and in space are $\Delta  t, \Delta r$ with $t_n = n \Delta t, r_j =2M+ j \Delta r $ where we recall $2M$  is the blackhole horizon. Denote by $V_j^n$ the numerical solution $V (n\Delta t, 2M + j \Delta r)$.  Let $(w_n)$ be a sequence equidistributed in $(-{1\over 2}, {1\over 2}) $  and write $r _ {n,j} = 2M+  (j+ w_n) \Delta r$. 
We define our Glimm-type appromations as follows: 
\bel{Glimm-method}
V_j^{n+1} =  V_{\mathcal R}^{j, n}  (t_{n+1}, r_{n,j}), 
\ee
where $ V_{\mathcal R}^{j, n}=  V_{\mathcal R}^{j, n}(t, r)$ is the solution to the Riemann problem with the initial data 
\bel{V^nj0}
V^{j,n}_0= \begin{cases}
V_L^{j, n} (r), & r < r_{j+\sgn(w_n)/2}, \\ 
V_R^{j, n} (r), \quad & r > r_{j+\sgn(w_n)/2}, 
\end{cases}
\ee
where the left-hand state  $V_L^{j, n}=V_L^{j, n} (r) $ and the right-hand state $V_R^{j, n}= V_R^{j, n} (r) $ are steady state solutions to \eqref{static-Burgers} with initial conditions: 
$$
\begin{cases}
 V_L^{j, n} (r_j)=  V_j^n, \quad  & w_n \geq 0, \\
  V_L^{j, n} (r_{j-1})=   V_{j-1}^n, &  w_n  < 0, 
 \end{cases} \qquad \qquad 
 \begin{cases}
 V_R^{j, n} (r_j)=  V_j^n,  & w_n > 0, \\
  V_R^{j, n} (r_{j+1})=   V_{j+1}^n, &  w_n \geq 0. 
 \end{cases}
$$
We choose a random number only once at each time level $t=t_n$ rather than at every each mesh point $(t_n, r_j)$.

In order to have an equidistributed sequence, the random values $(w_n)$ are defined by following Chorin~\cite{Chorin}: we give two large prime numbers $p_1< p_2$ and define a sequence of integers $(q_n)$: 
\bel{qn}
\aligned 
& q_0,  \quad \text{given} \quad q_0 < p_2; 
\qquad 
 q_n := (p_1+ q_{n-1}) \mod p_2, \quad n \geq 1. 
\endaligned
\ee
Then we define the sequence 
$w_n' = {q_n + w_n+1/ 2  \over p_2}  -{1\over 2}$, 
which is to be used in our Glimm method instead of  instead of $(w_n)$. It is direct to see that $ w_n' \in \big(-{1\over 2}, {1\over 2}\big)$.


\section{Numerical experiments with the random choice scheme for the relativistic Burgers model}
\label{Sec:6}

\paragraph{Consistency property}

We now presents numerical experiment with the proposed Glimm method for the Burgers equation on a Schwarzschild background \eqref{Burgers}.  Recall that $r>2M$ and we choose again $M=1$ for the blackhole mass.  The space interval in consideration is $(r_{\min,}, r_{\max})$ with $r_{\min}= 2M=2$ and $r_{\max}=4$.  To introduce the random sequence, we fix  two prime integers, specifically  $p_1=937, p_2=  997$ and $q_0= 800$.
Since the solution to every local generalized Riemann problem \eqref{Burgers}, \eqref{initial-Riemann} is exact, the following observation is immediate. 

\begin{lemma}
Consider a given initial velocity $v_0=v_0(r)$ as a steady state solution such that the static Burgers model  \eqref{static-Burgers} holds. Then the approximate solution to the relativistic Burgers equation \eqref{Burgers} constructed by the Glimm method \eqref{Glimm-method} is accurate. 
\end{lemma}

We will still observe the evolution of those three types of solutions shown in Figure~\ref{FIG-51}, that is,  the two steady state solutions $v= \pm \sqrt {{3\over 4}+{1\over 2 r}} $ and  the steady shock: 
$$
v= \begin{cases}
\sqrt {{3\over 4}+{1\over 2 r}},  & 2.0<r<3.0,
\\ 
- \sqrt {{3\over 4}+{1\over 2 r}},  & r>3.0. 
\end{cases}
$$
\begin{figure}[!htb] 
\centering
\begin{minipage}[t]{0.3\linewidth}
\centering
\epsfig{figure=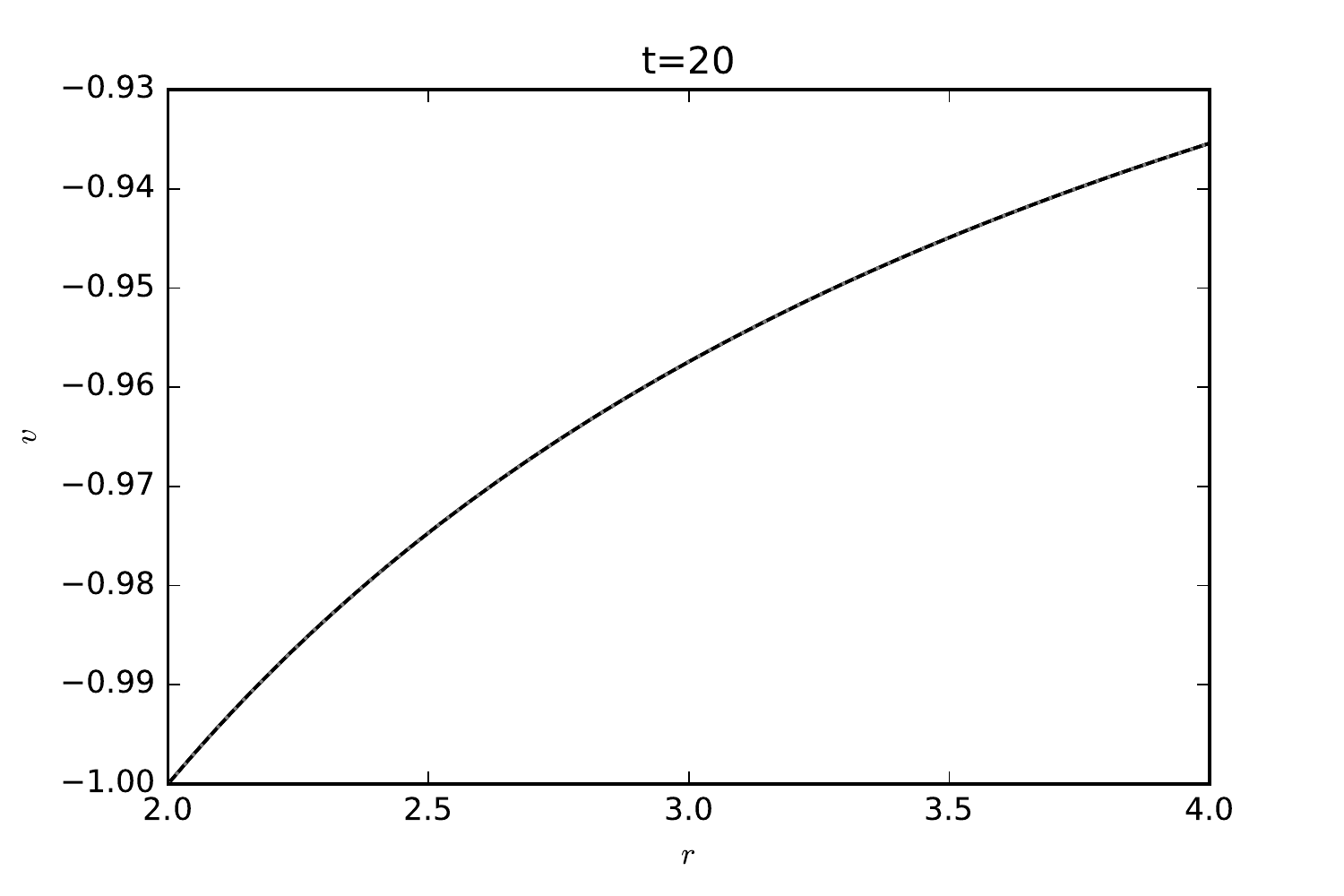,width= 2.5in} 
\end{minipage}
\hspace{0.1in}
\begin{minipage}[t]{0.3\linewidth}
\centering
\epsfig{figure=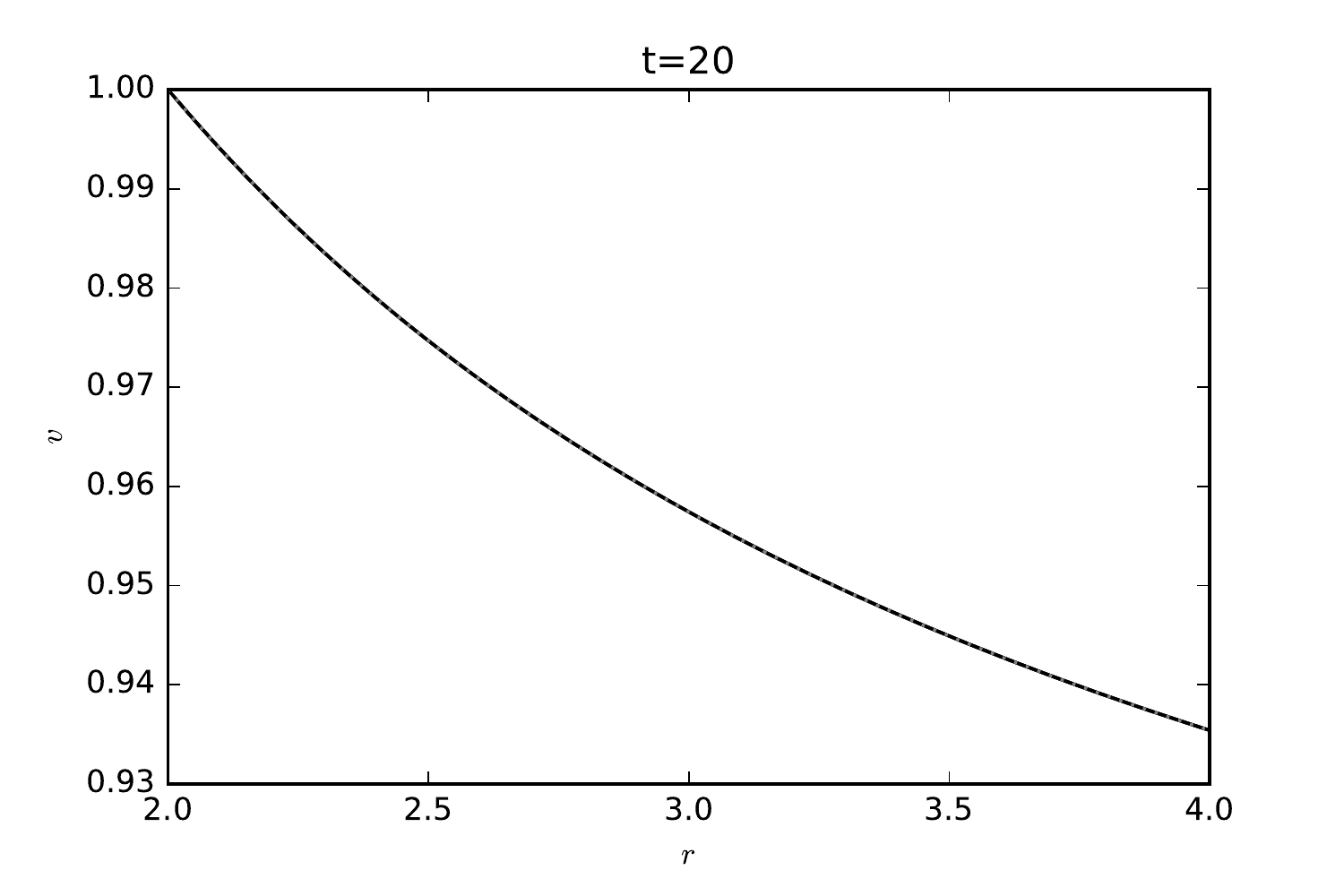,width=2.5in} 
\end{minipage}
\hspace{0.1in}
\begin{minipage}[t]{0.3\linewidth}
\centering
\epsfig{figure=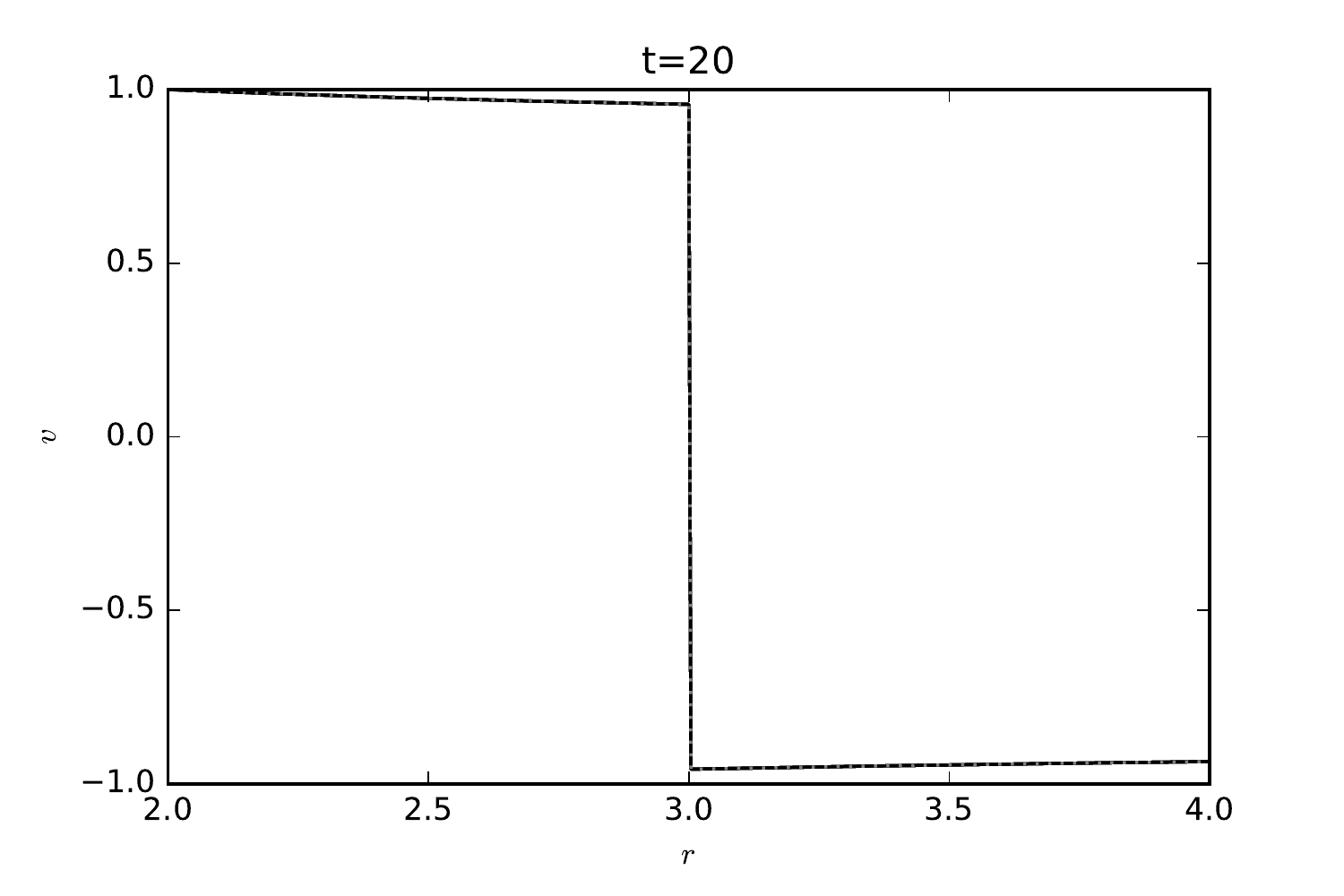,width=2.5in} 
\end{minipage}
\centering
\caption{Solution at time $t=20$ from a steady state initial data, using the Glimm  scheme}
\label{FIG-62} 
\end{figure}

\paragraph{Different types of shocks}

 We consider two different shocks whose initial speed are positive and negative. As was observed by the finite volume method, whether the position of the shock will go toward the blackhole horizon is determined uniquely by their initial behavior. We can recover the same conclusion with the Glimm method. Again, we take two kinds of initial data: 
$$
v= \begin{cases}
\sqrt{{1\over 2}+{1\over r}}, & 2.0<r <2.5, \\ 
\sqrt{2\over r}, & r >2.5, 
\end{cases}
\qquad \qquad 
v= \begin{cases}
- \sqrt{2\over r}, & 2.0<r <2.5, \\ 
\sqrt{{3\over 4}+{1\over 4 r}},  & r >2.5. 
\end{cases}
$$
Since our Riemann solver is exact, the numerical solutions contain no numerical diffusion. 
\begin{figure}[!htb] 
\centering
\begin{minipage}[t]{0.3\linewidth}
\centering
\epsfig{figure=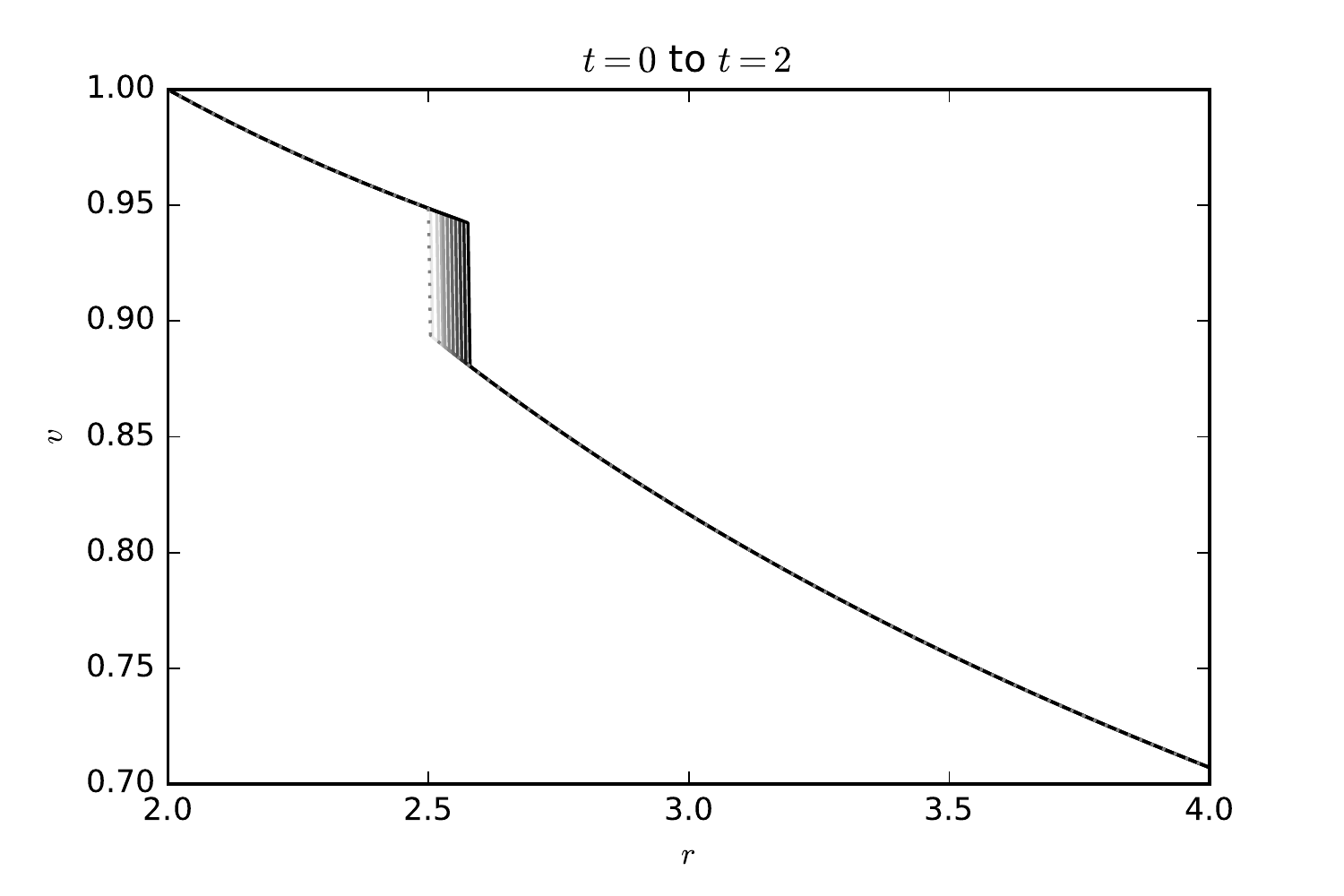,width= 2.5in} 
\end{minipage}
\hspace{0.1in}
\begin{minipage}[t]{0.3\linewidth}
\centering
\epsfig{figure=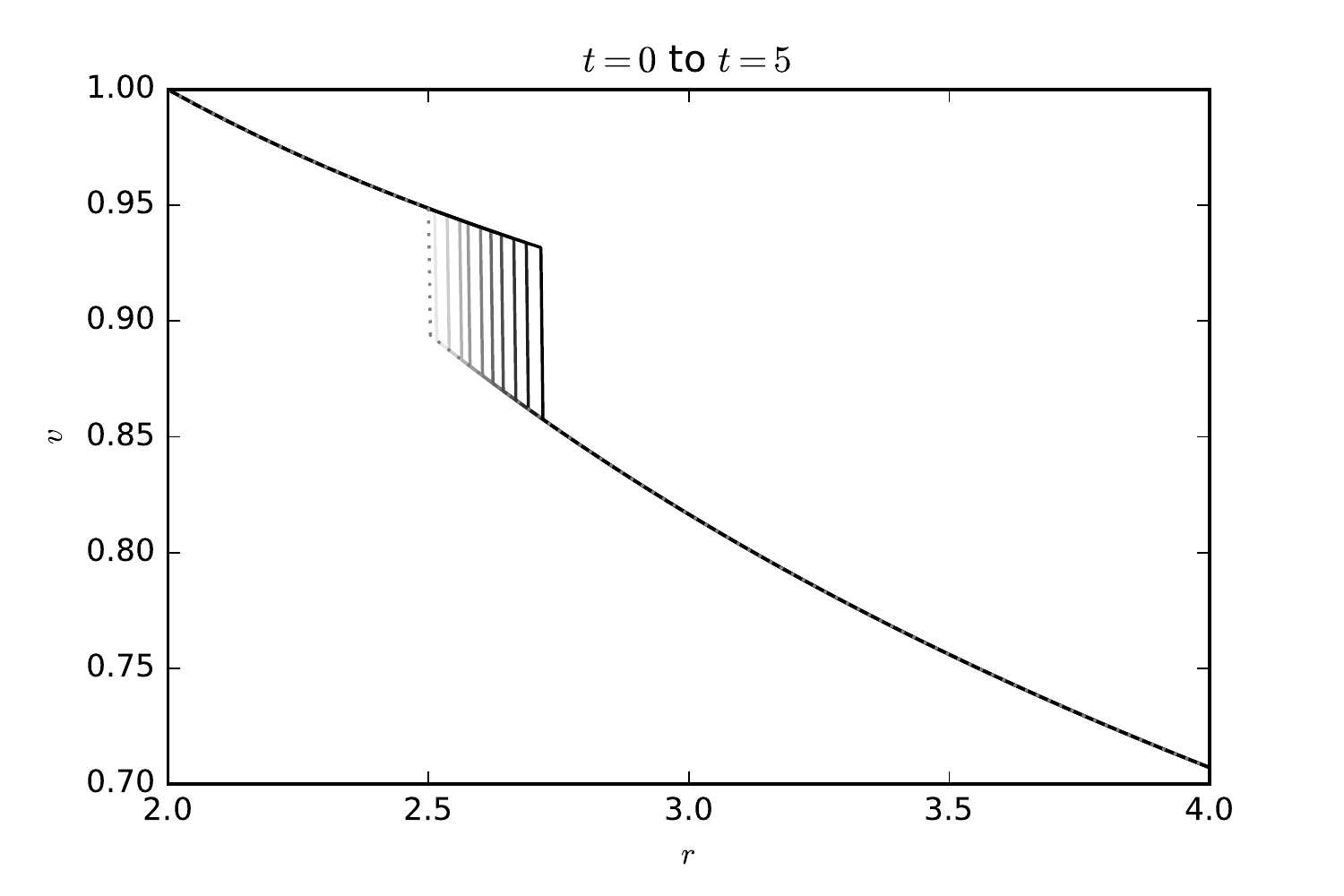,width=2.5in} 
\end{minipage}
\hspace{0.1in}
\begin{minipage}[t]{0.3\linewidth}
\centering
\epsfig{figure=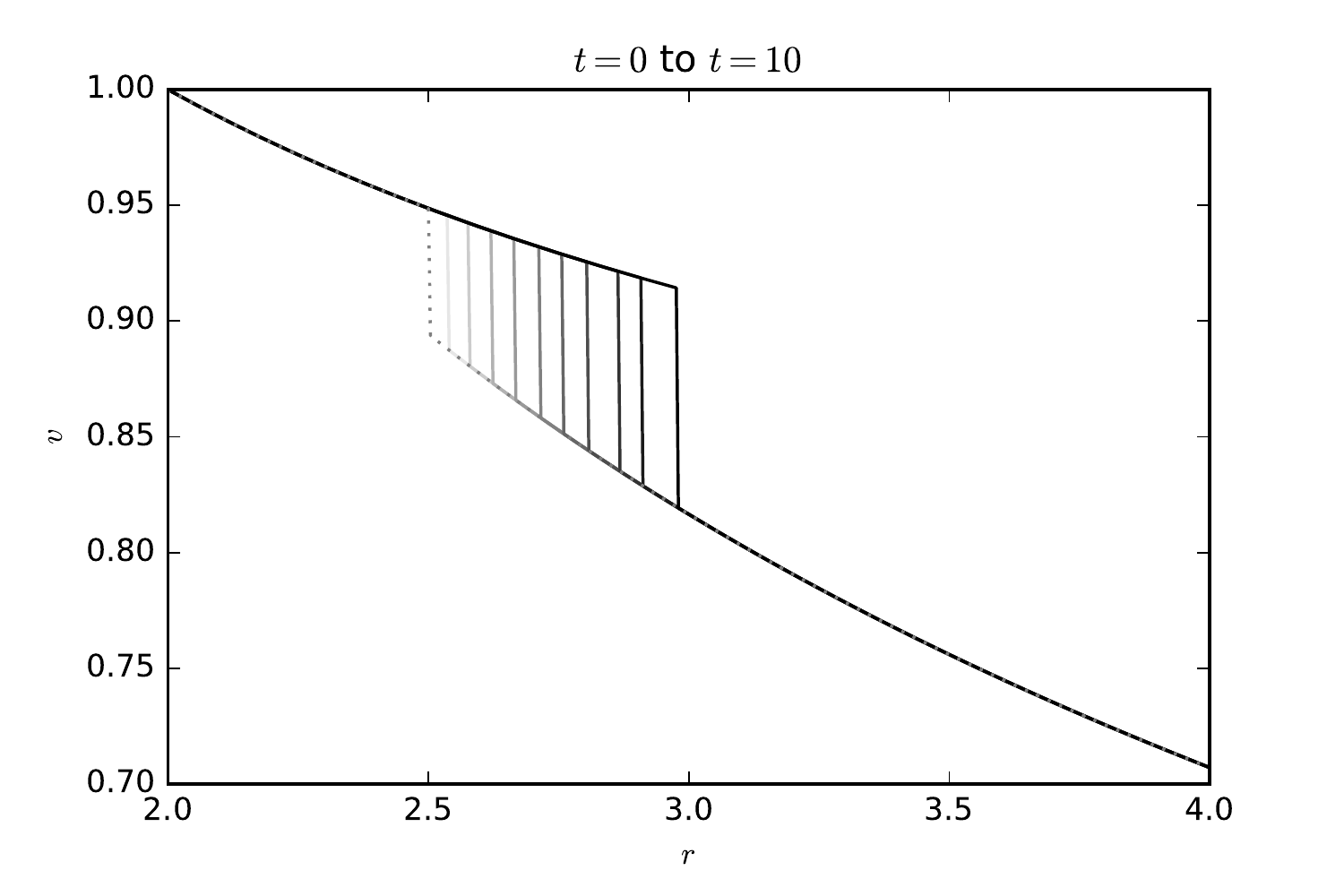,width=2.5in} 
\end{minipage}
\centering
\caption{Static solution with a right-moving shock computed by the Glimm  scheme}
\label{FIG-63} 
\end{figure}

\begin{figure}[!htb] 
\centering
\begin{minipage}[t]{0.3\linewidth}
\centering
\epsfig{figure=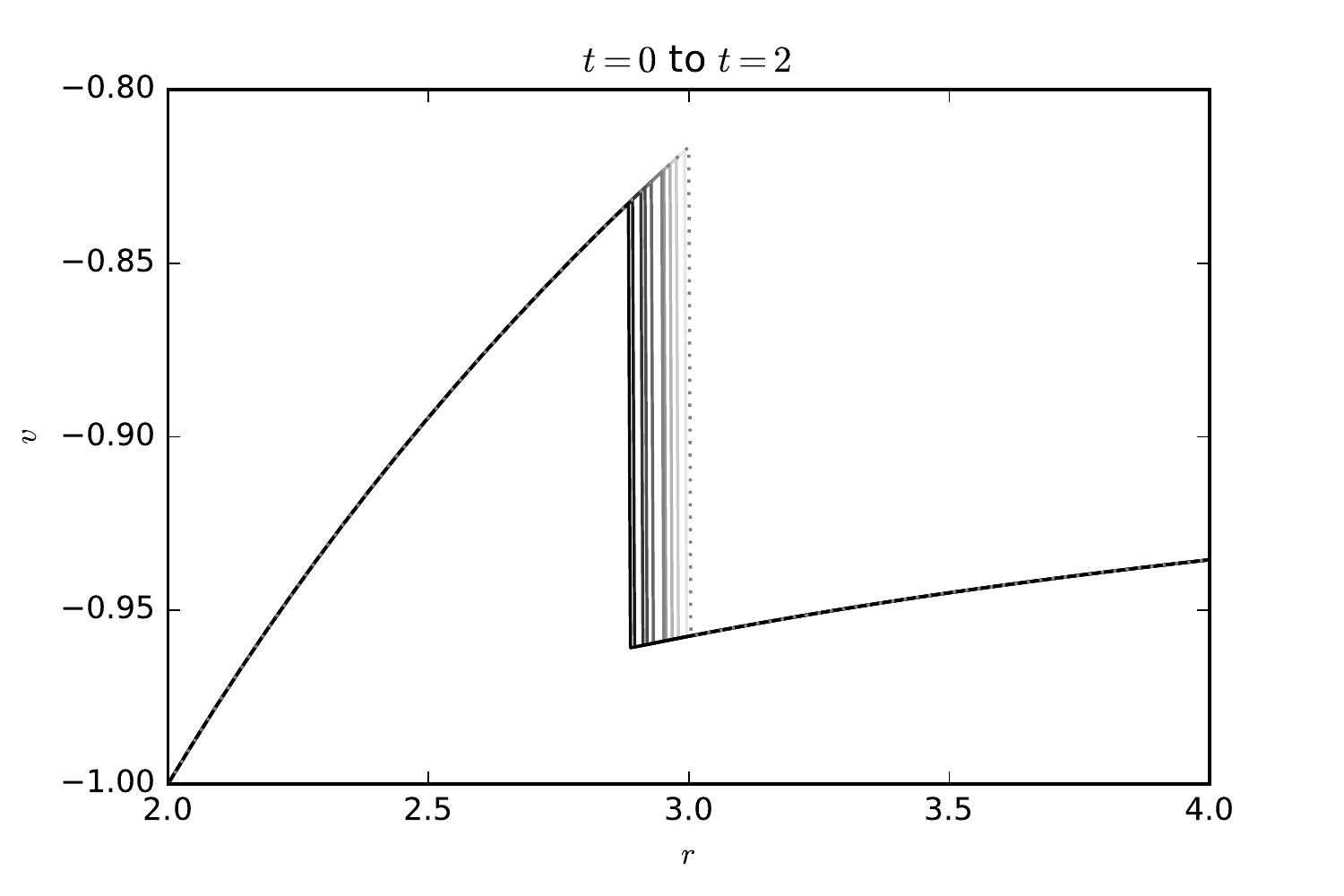,width= 2.5in} 
\end{minipage}
\hspace{0.1in}
\begin{minipage}[t]{0.3\linewidth}
\centering
\epsfig{figure=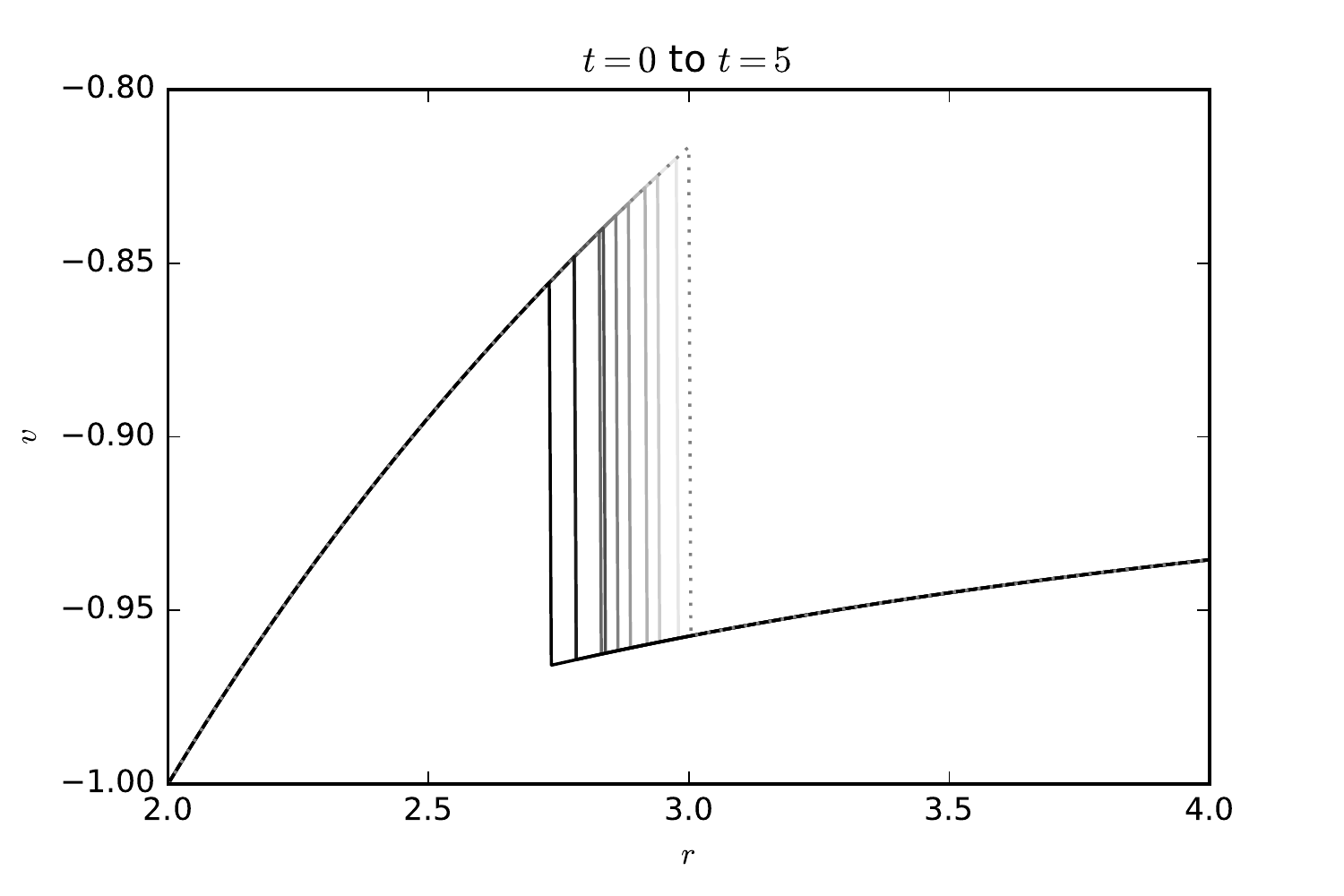,width=2.5in} 
\end{minipage}
\hspace{0.1in}
\begin{minipage}[t]{0.3\linewidth}
\centering
\epsfig{figure=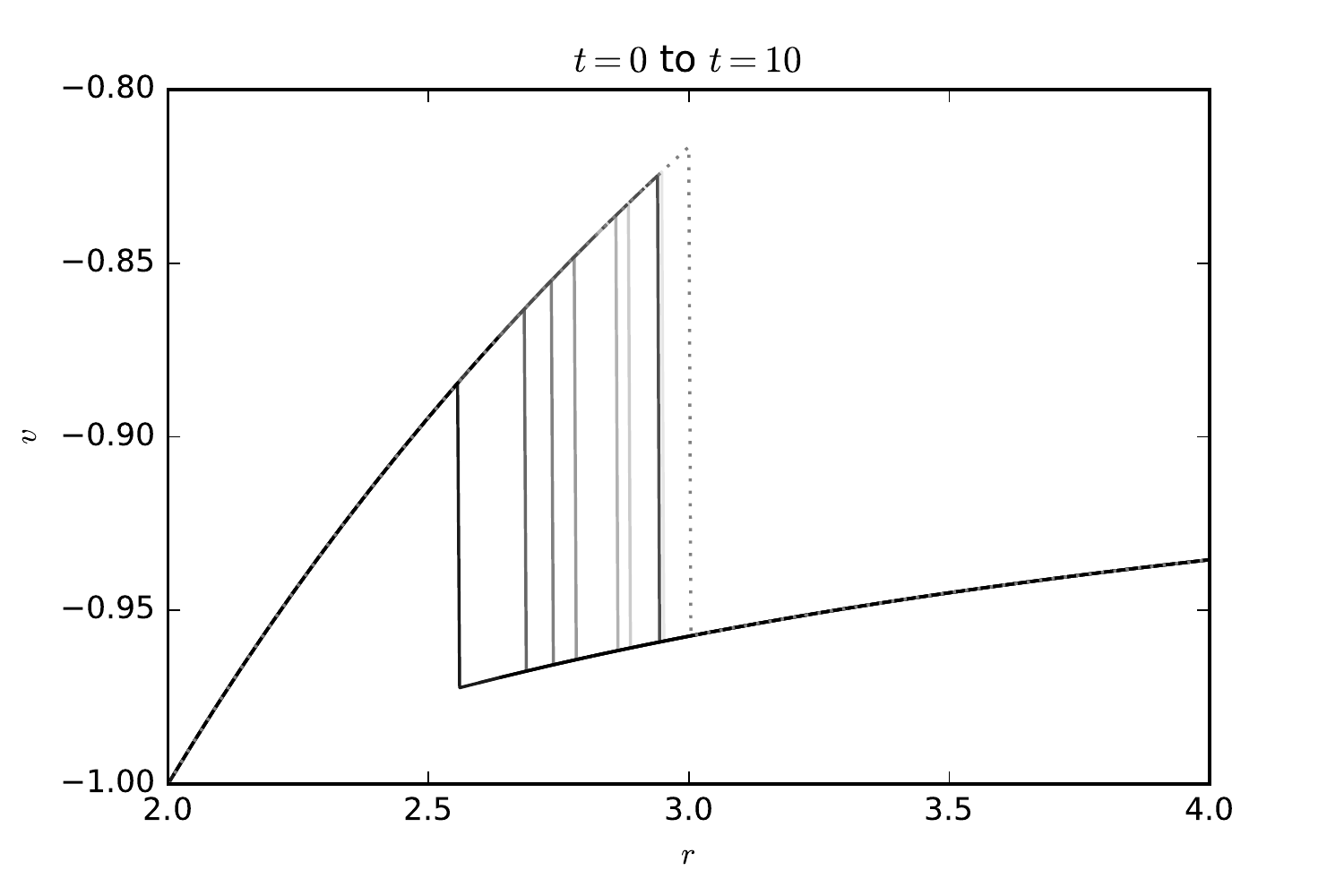,width=2.5in} 
\end{minipage}
\centering
\caption{Static solution with a left-moving shock computed by the Glimm  scheme}
\label{FIG-64} 
\end{figure}


\paragraph{Asymptotic behavior of Burgers solutions}

We are now interested in the evolution of solutions whose initial data is given as piecewise steady state solution satisfying \eqref{static-Burgers}. As was done earlier, we take into account  
 two kinds of initial data:  
$$
 v= \sqrt{{1\over 2}+ {1\over r}}, \qquad v= \begin{cases}
\sqrt{{1\over 2}+{1\over r}} & 2.0<r <2.5, \\ 
\sqrt{2\over r} & r >2.5, 
\end{cases}, 
$$
perturbed by compactly supported functions.

\begin{figure}[!htb] 
\centering
\begin{minipage}[t]{0.3\linewidth}
\centering
\epsfig{figure=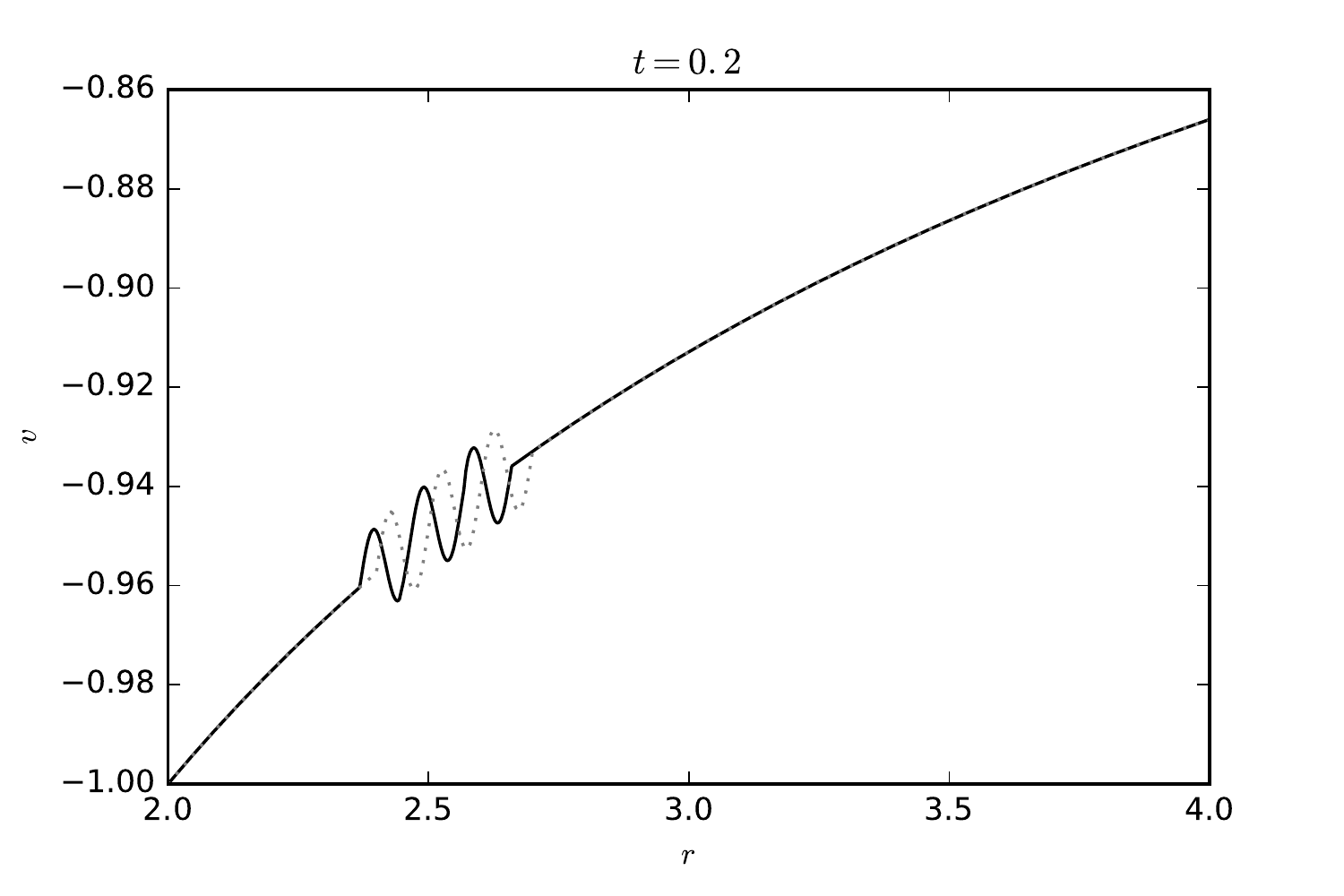,width=2.5in} 
\end{minipage}
\hspace{0.1in}
\begin{minipage}[t]{0.3\linewidth}
\centering
\epsfig{figure=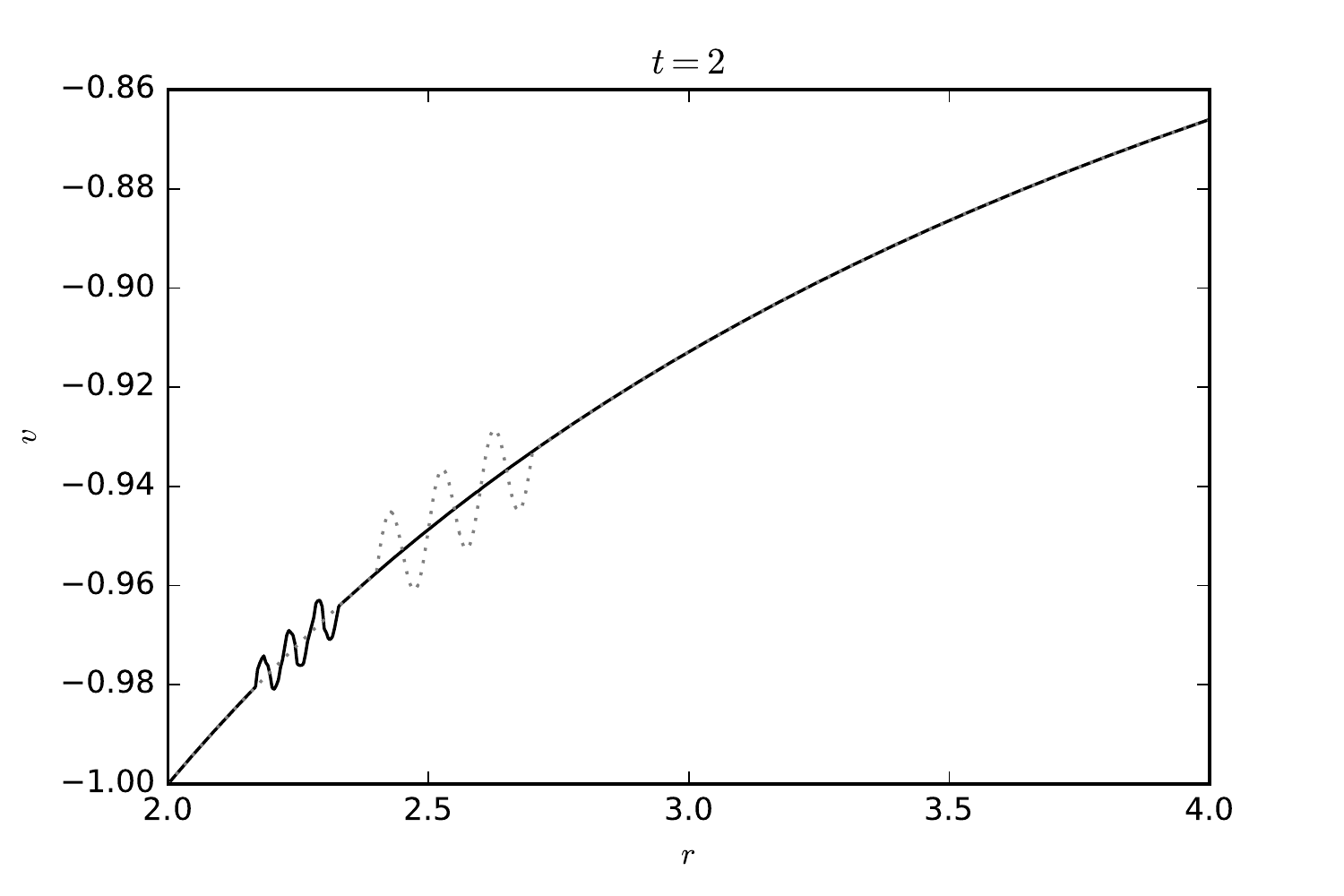,width=2.5in} 
\end{minipage}
\hspace{0.1in}
\begin{minipage}[t]{0.3\linewidth}
\centering
\epsfig{figure=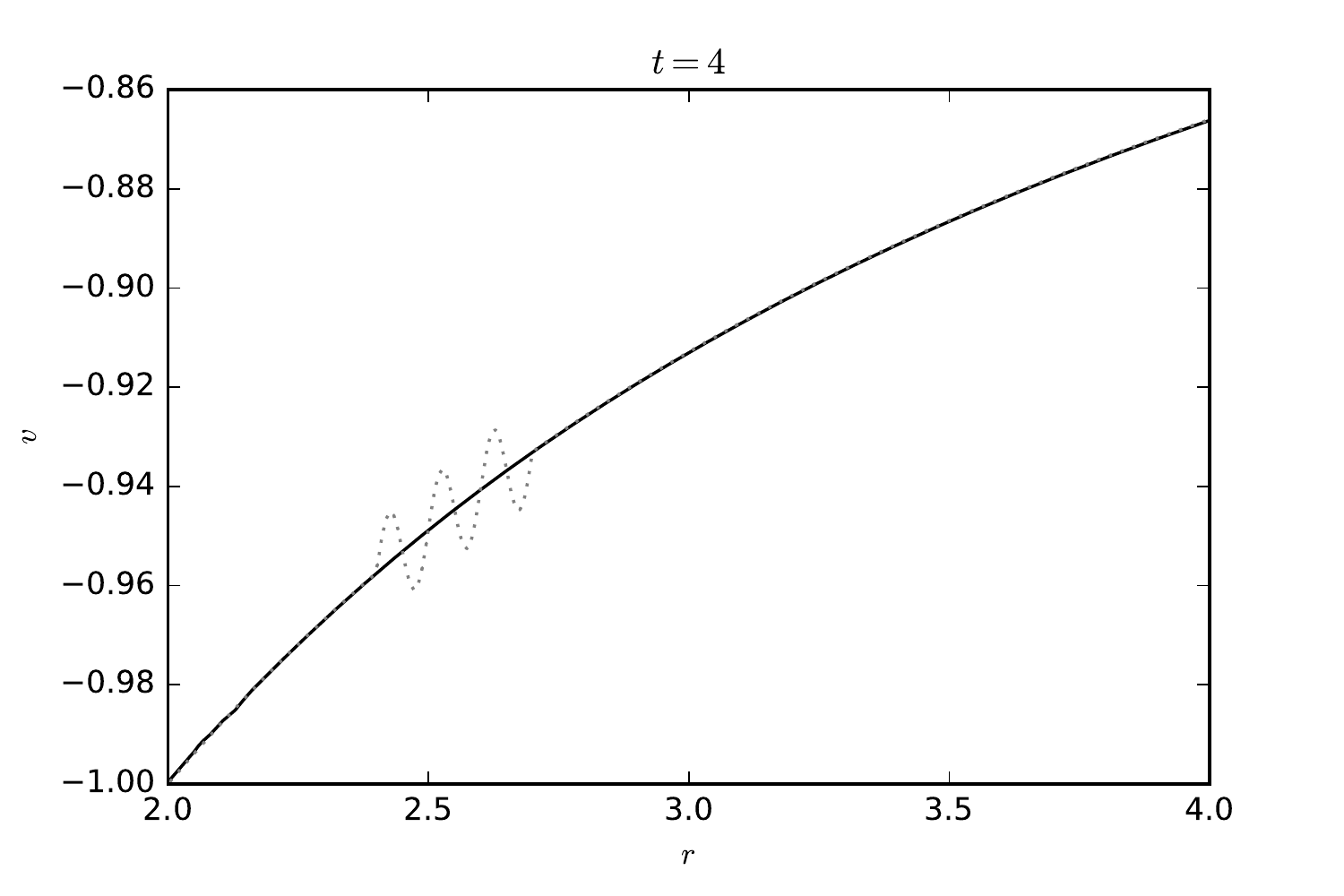,width=2.5in} 
\end{minipage}
\caption{Numerical solution from an initially perturbed steady state, using the Glimm method}
\label{FG-56}
\end{figure}

\begin{figure}[!htb] 
\centering
\hspace{0.1in}
\begin{minipage}[t]{0.3\linewidth}
\centering
\epsfig{figure=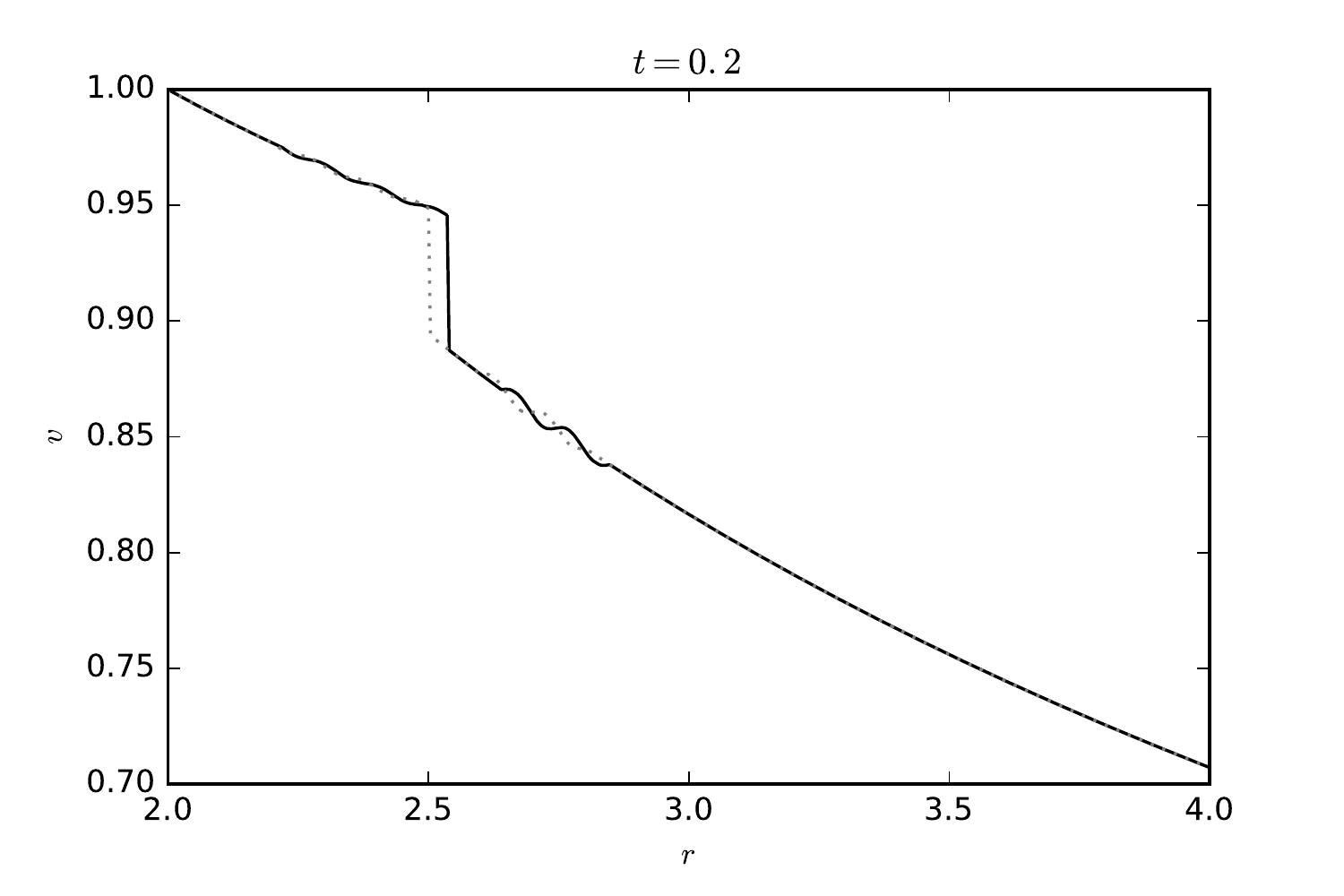,width=2.5in} 
\end{minipage}
\hspace{0.1in}
\begin{minipage}[t]{0.3\linewidth}
\centering
\epsfig{figure=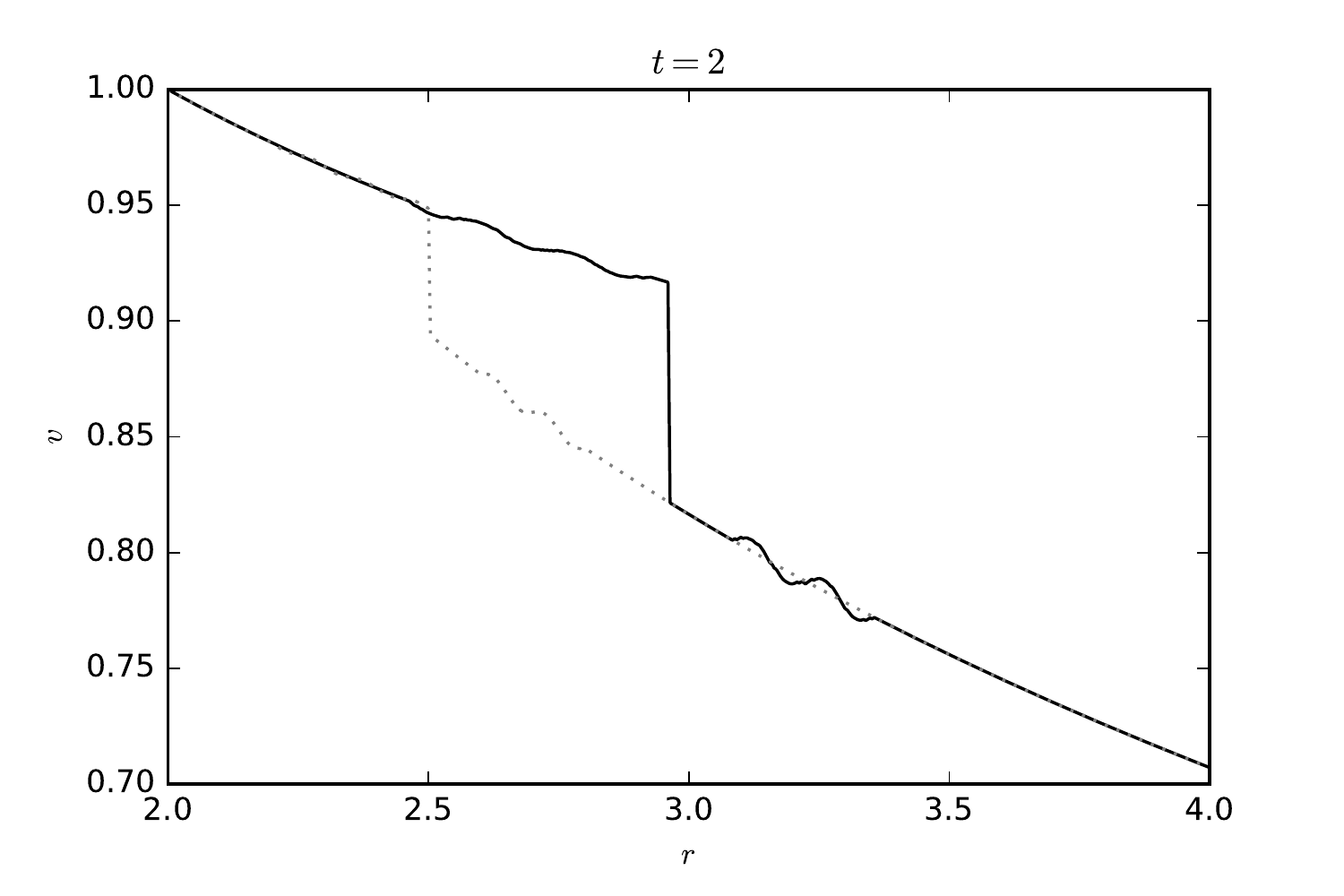,width=2.5in} 
\end{minipage}
\hspace{0.1in}
\begin{minipage}[t]{0.3\linewidth}
\centering
\epsfig{figure=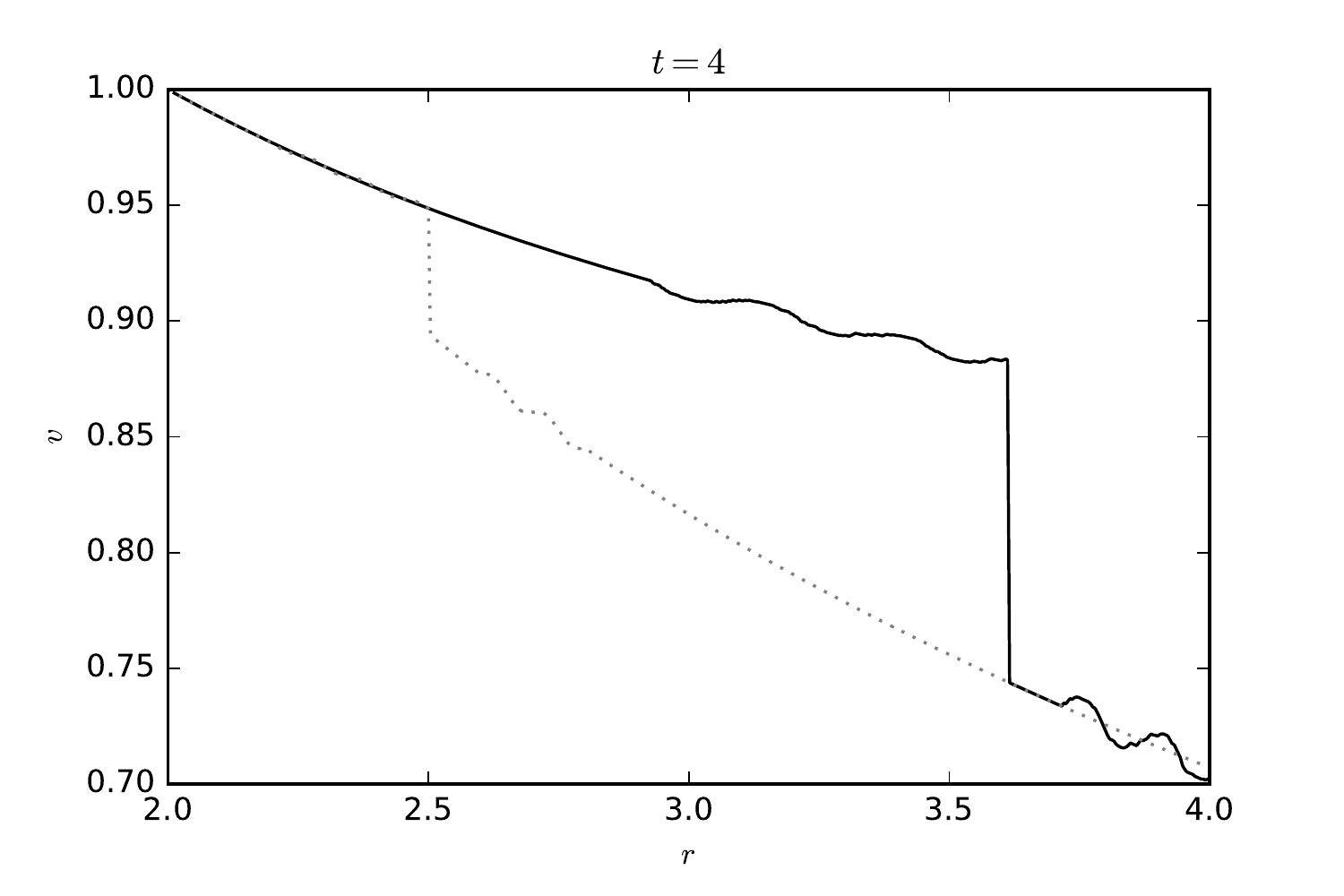,width=2.5in} 
\end{minipage}
\caption{Numerical solution from an initially perturbed shock, using the Glimm method}
\label{FIG-66} 
\end{figure}


\section{General initial data for the relativistic Burgers equation}
\label{Sec:7}

\paragraph{Steady shock with perturbation}

The behavior of a smooth steady state solution to the relativistic Burgers model \eqref{Burgers} perturbed by a function on a compactly supported function is understood both numerically and theoretically: the solution converge to the same initial steady state solution. The steady shock \eqref{Burgers-steady-shock} is  a solution to the static equation \eqref{static-Burgers} in the distribution sense. We are interested in the asymptotic behavior and our numerical results in Figures~\ref{FIG-81-0} and  \ref{FIG-81-1} lead us to the following.

\begin{conclusion}
Consider a perturbed steady shock given as \eqref{Burgers-steady-shock}:
$$
v_0=\begin{cases}
\sqrt {1- K^2 (1-2M/r)} &  2M<r <r_0, \\
 -\sqrt {1- K^2 (1-2M/r)}  & r>r_0, 
\end{cases}
$$
where $K$  is a given constant and $r_0>2M$ is fixed radius out of the Schwarzschild blackhole region. The solution to the relativistic Burgers model \eqref{Burgers} converges at some finite time to a solution of the form (with possibly $r_1\neq r_0$):
$$
v=\begin{cases}
\sqrt {1- K^2 (1-2M/r)} &  2M<r <r_1, \\
 -\sqrt {1- K^2 (1-2M/r)}  & r>r_1, 
\end{cases} 
$$
\end{conclusion}

\begin{figure}[!htb] 
\centering 
\begin{minipage}[t]{0.3\linewidth}
\centering
\epsfig{figure=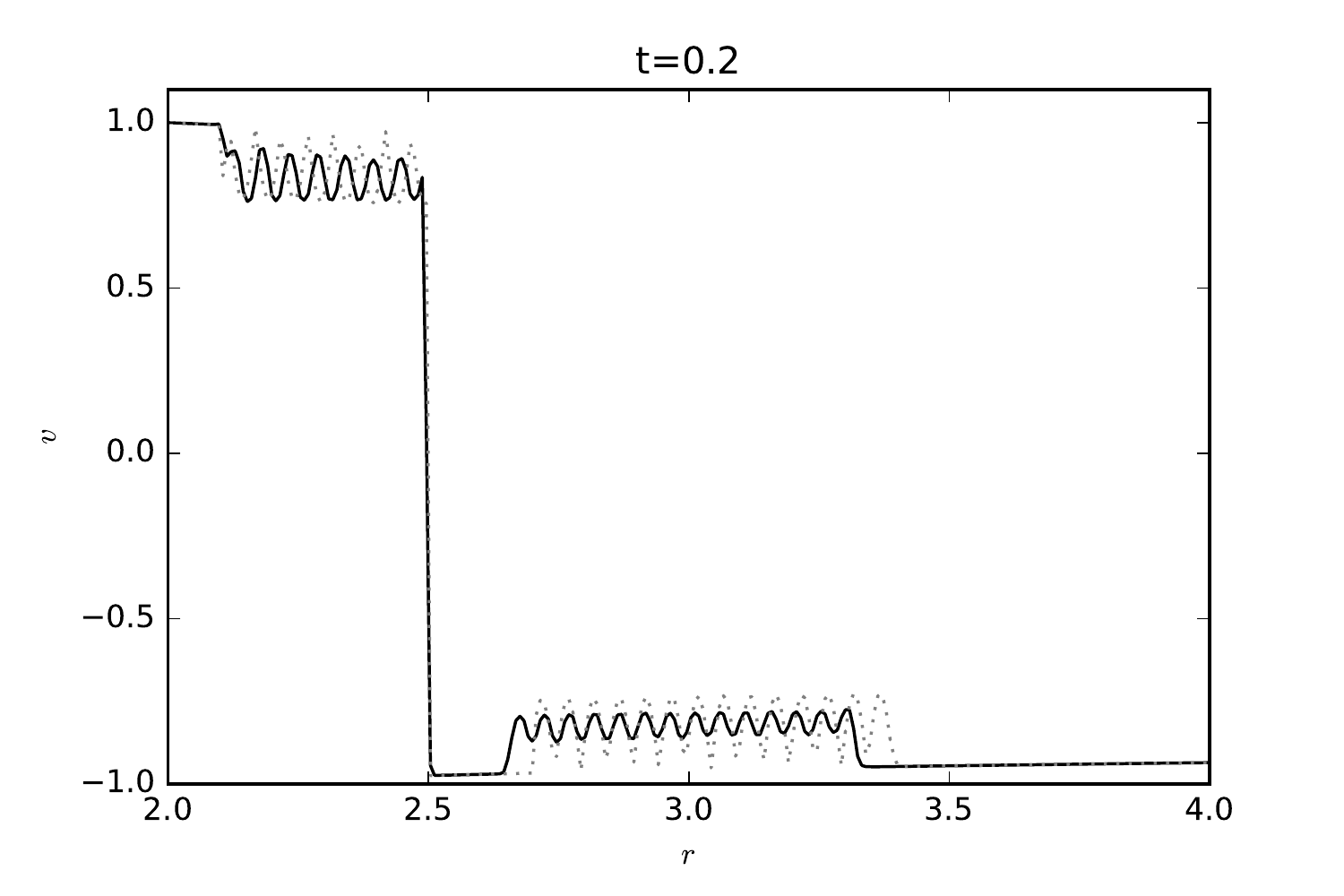,width= 2.5in} 
\end{minipage}
\hspace{0.1in}
\begin{minipage}[t]{0.3\linewidth}
\centering
\epsfig{figure=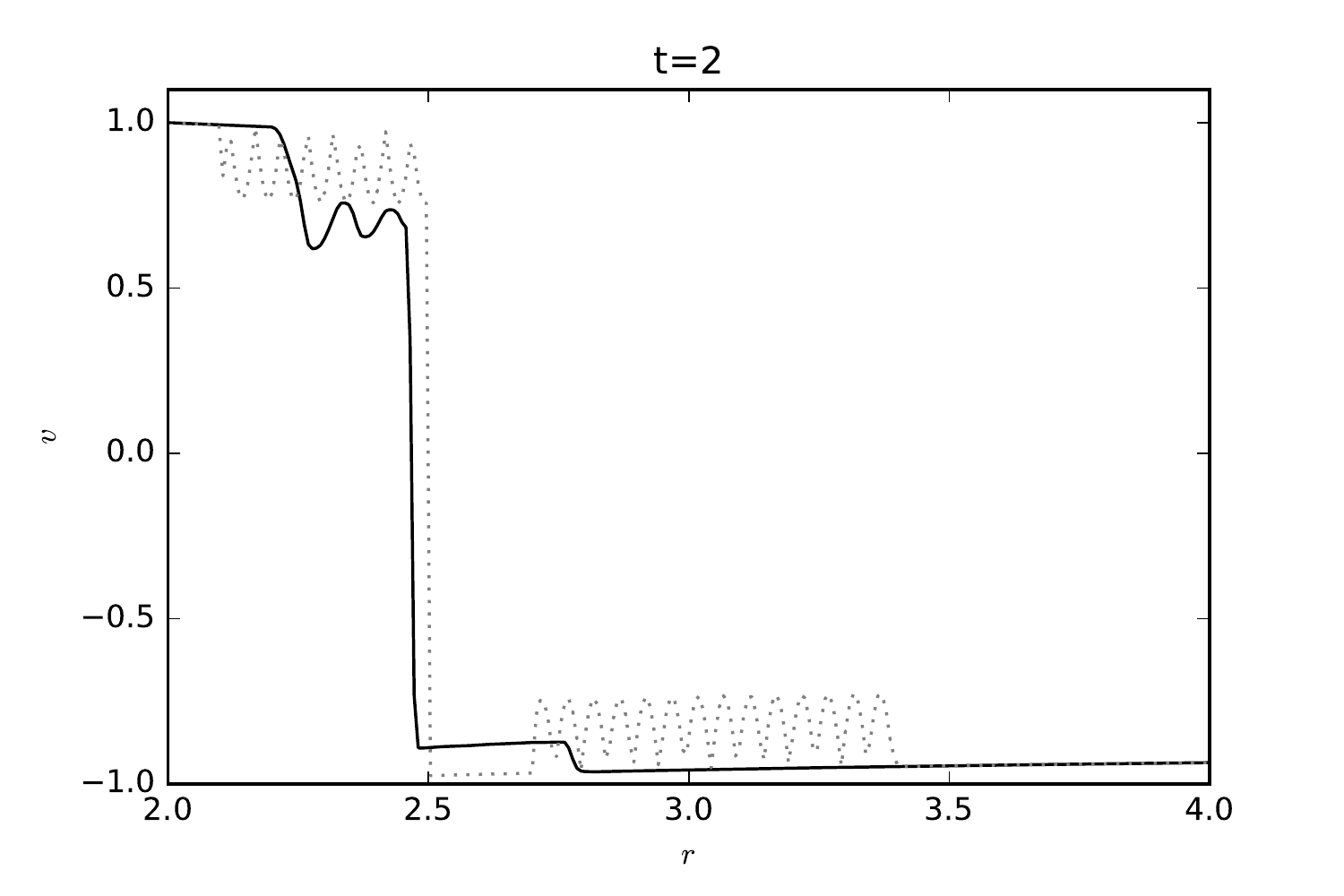,width=2.5in} 
\end{minipage}
\hspace{0.1in}
\begin{minipage}[t]{0.3\linewidth}
\centering
\epsfig{figure=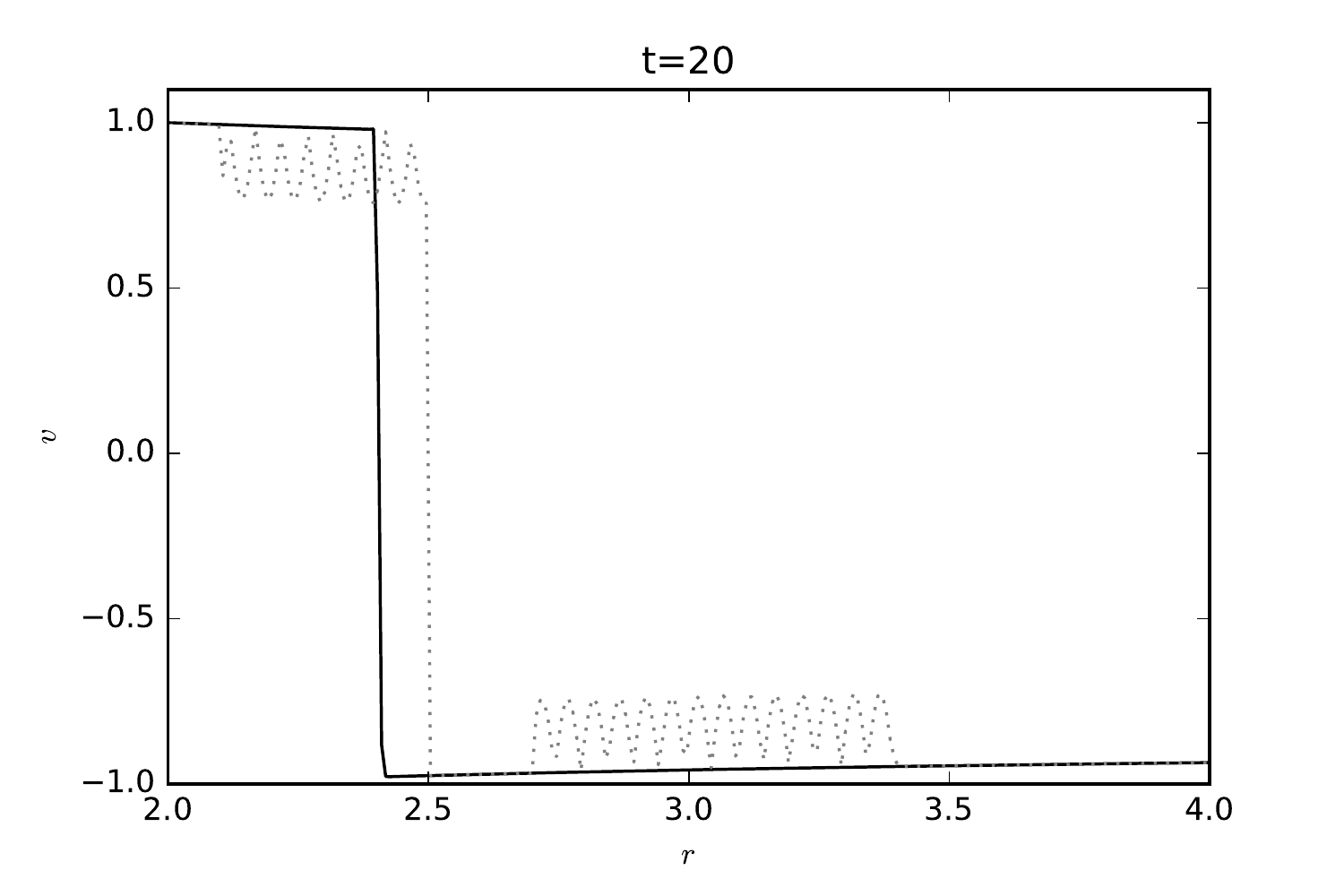,width=2.5in} 
\end{minipage}
\caption{Evolution of a perturbed steady shock, using the finite volume method}
\label{FIG-81-0} 
\end{figure}

\begin{figure}[!htb] 
\centering 
\begin{minipage}[t]{0.3\linewidth}
\centering
\epsfig{figure=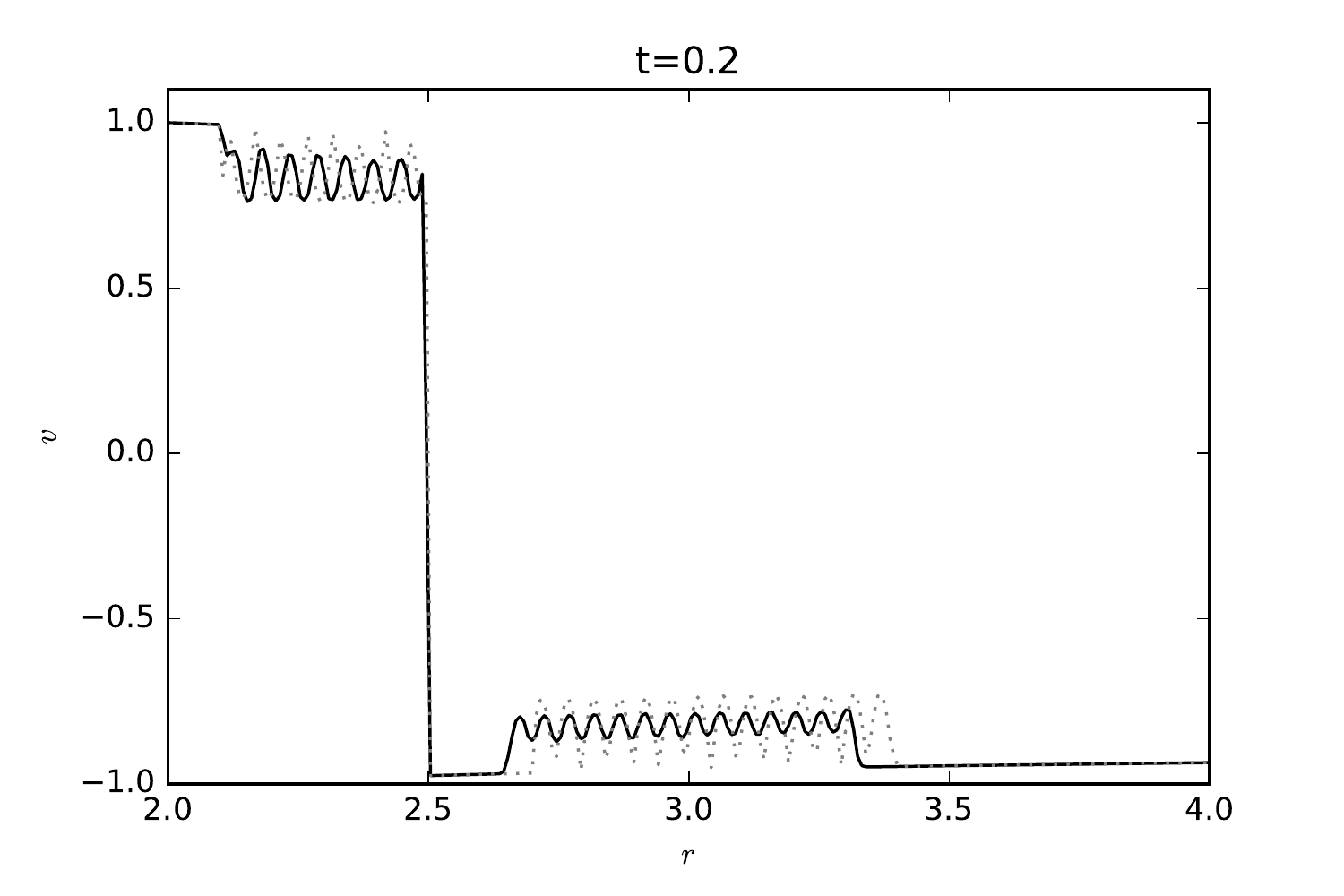,width= 2.5in} 
\end{minipage}
\hspace{0.1in}
\begin{minipage}[t]{0.3\linewidth}
\centering
\epsfig{figure=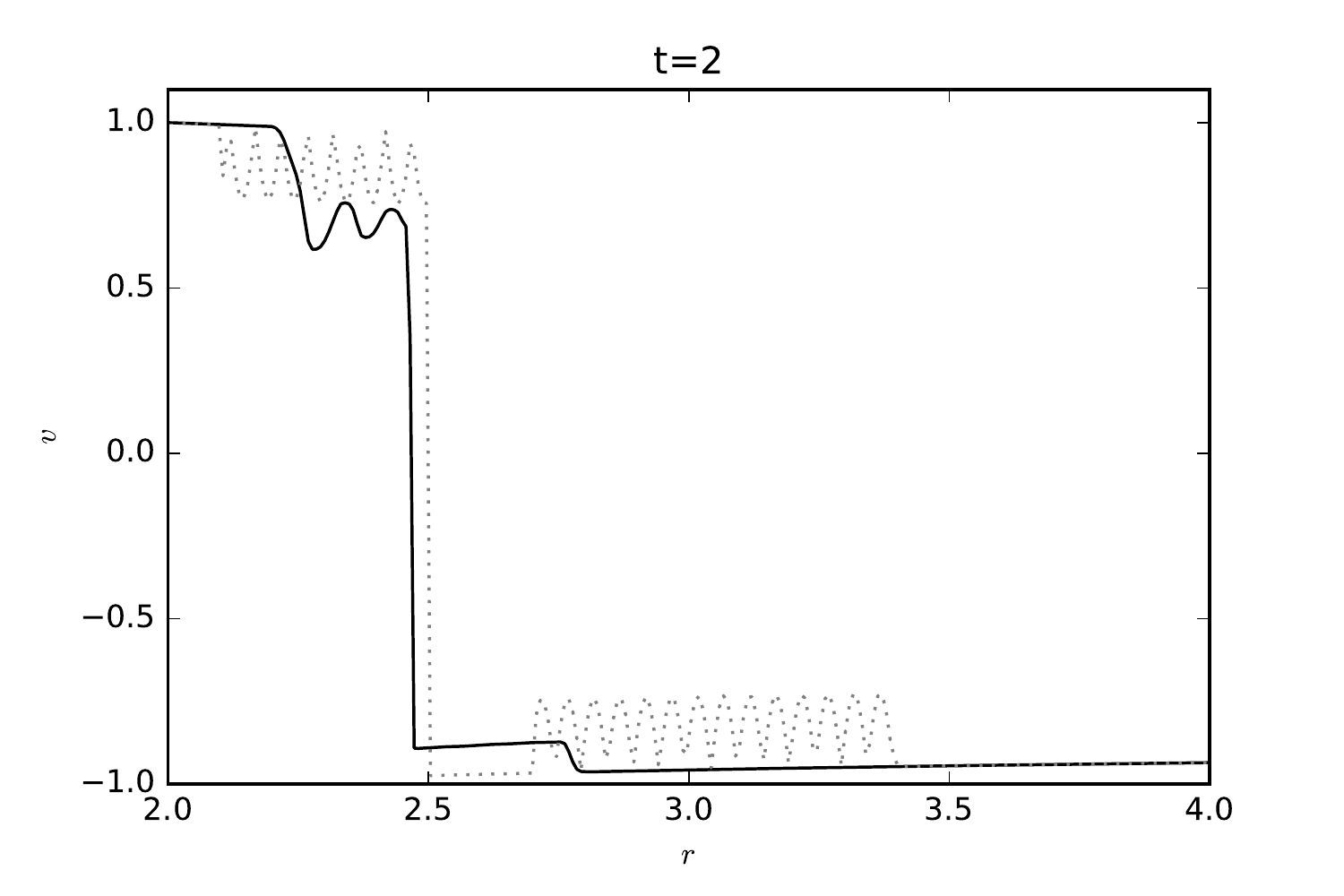,width=2.5in} 
\end{minipage}
\hspace{0.1in}
\begin{minipage}[t]{0.3\linewidth}
\centering
\epsfig{figure=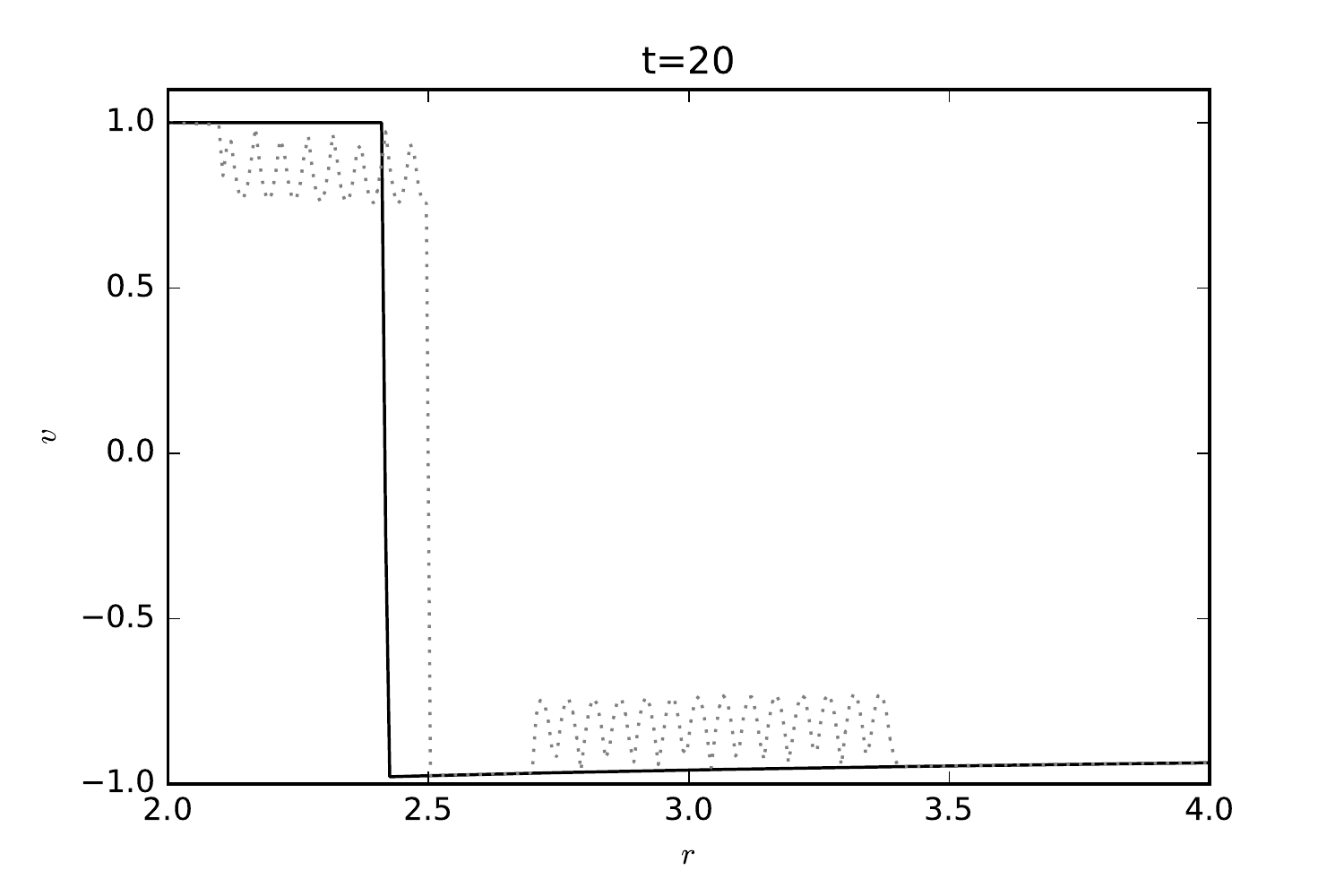,width=2.5in} 
\end{minipage}
\caption{Evolution of a perturbed steady shock, using the Glimm method}
\label{FIG-81-1} 
\end{figure}

\paragraph{Late-time behavior of general solutions}

It is obvious that the steady state solution satisfying \eqref{static-Burgers} serves as a solution to the relativistic Burgers equation on a Schwarzschild background. Notice that on the blackhole horizon $r=2M$, the steady state solution values the light speed, that is, either $1$ or $-1$, which equals exactly the light speed and obviously their boundary values will not change as time evolves.   The value of a steady state solution at infinity is also given explicitly. Observations on the numerical method shows that the asymptotic behavior of Burgers model \eqref{Burgers} is mainly determined by the values of the initial data at the blackhole horizon $r=2M$ and the space infinity $r=+\infty$. More precisely, suppose that a given velocity $v_0 = v_0 (r)$ does not satisfy the static Burgers equation \eqref{static-Burgers}, we have the following conclusion. 

\begin{conclusion}
\begin{enumerate}
\item If the initial velocity $\lim\limits_{r \to 2M} v_0(r)=1$, then the solution to the Burgers equation \eqref{Burgers} satisfies that there exists a time $t>t_0$ such that  for all $t>t_0$ the solution $v= v (t,r)$ is a shock with left-hand state $1$ and right-hand state $v_*^-$ with $v_*^- (r)=-  \sqrt  {2M \over r}$  the negative critical steady solution. 

\item If the initial velocity $\lim\limits_{r \to 2M} v_0 (r)   < 1$ and $\lim \limits_ {r\to +\infty} v_0(r)  >  0$,   there exists a time $t_0>0$ such that the solution to the Burgers equation $ v(t, r)=v_*^-(r)$ for all $t>t_0 $ where $v_*^-(r) =-  \sqrt  {2M \over r}$   is the negative critical steady state solution to the relativistic Burgers model. 

\item If the initial velocity  $\lim\limits_{r \to 2M} v_0(r)<  1$ and  $\lim \limits_ {r\to +\infty} v_0(r)  \leq 0$, then the solution to the relativistic Burgers model satisfies that   $ v(t, r)= - \sqrt {1-({1-v_ 0^\infty}^2)  (1-{2M \over r})}$  for $t>t_0$ for a time $t_0 >0$ where $0 \geq v_0^\infty = \lim \limits_ {r\to +\infty} v_0(r)$. 

\end{enumerate}
\end{conclusion}


\begin{figure}[!htb] 
\centering 
\begin{minipage}[t]{0.3\linewidth}
\centering
\epsfig{figure=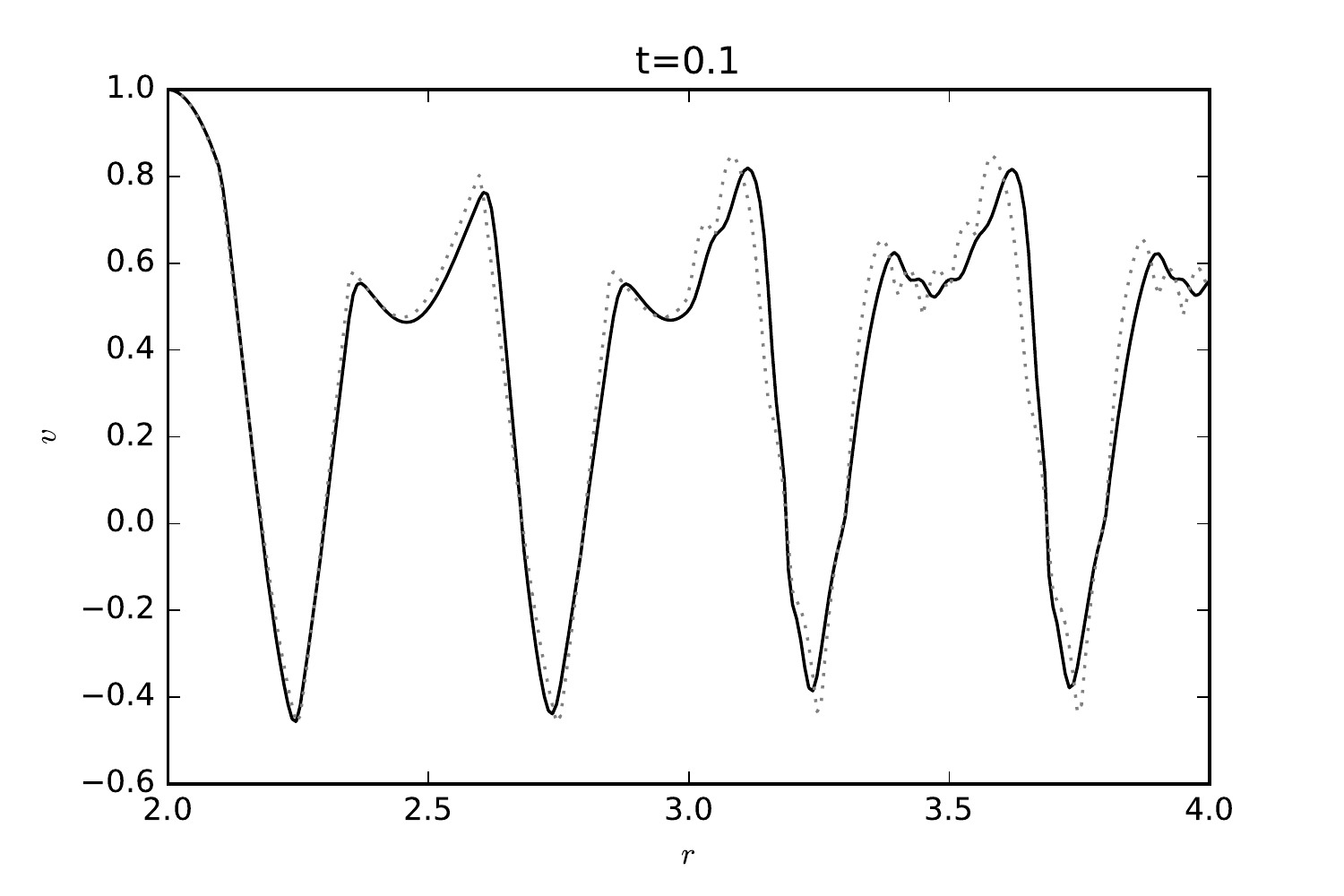,width=2.5in} 
\end{minipage}
\hspace{0.1in}
\begin{minipage}[t]{0.3\linewidth}
\centering
\epsfig{figure=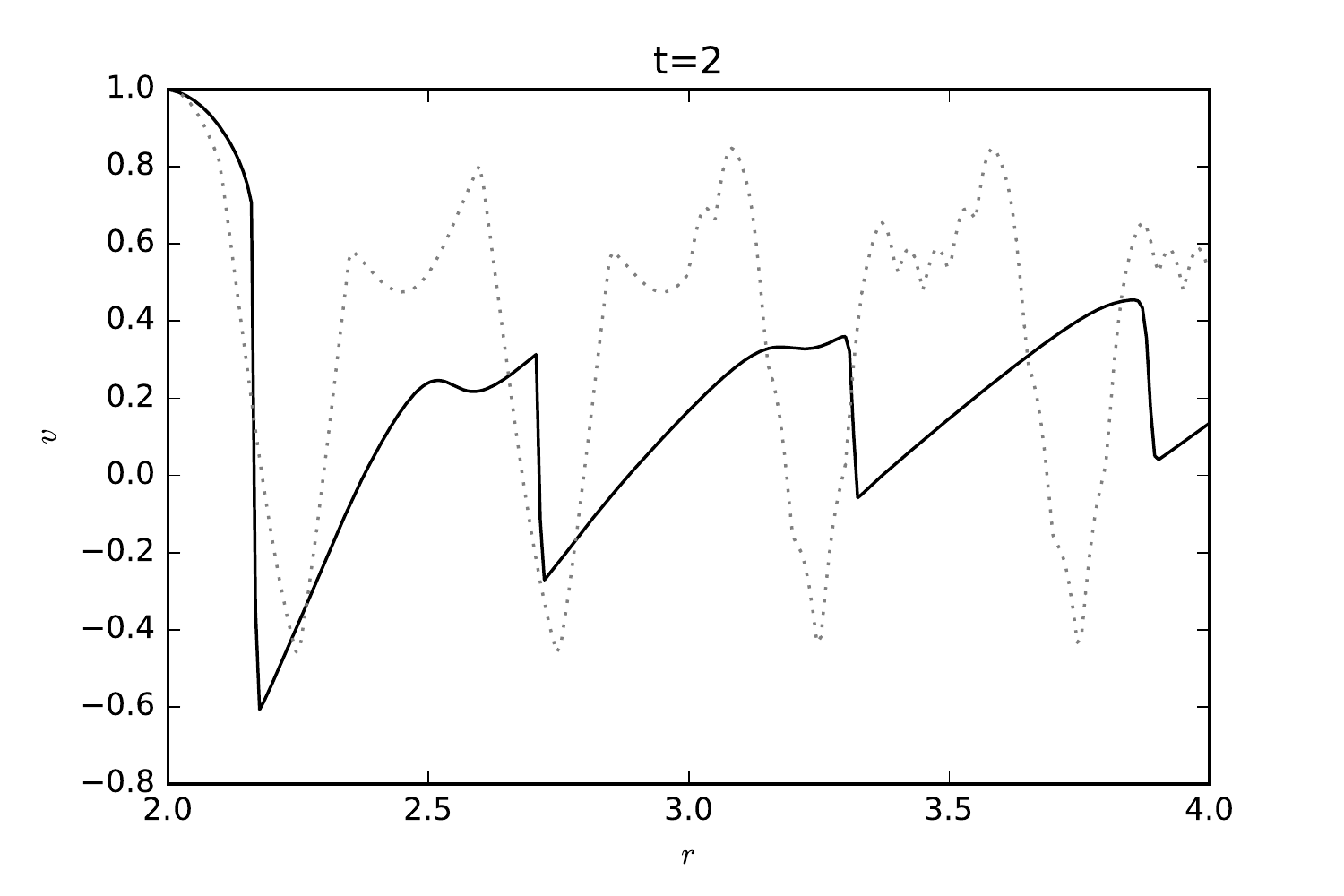,width=2.5in} 
\end{minipage}
\hspace{0.1in}
\begin{minipage}[t]{0.3\linewidth}
\centering
\epsfig{figure=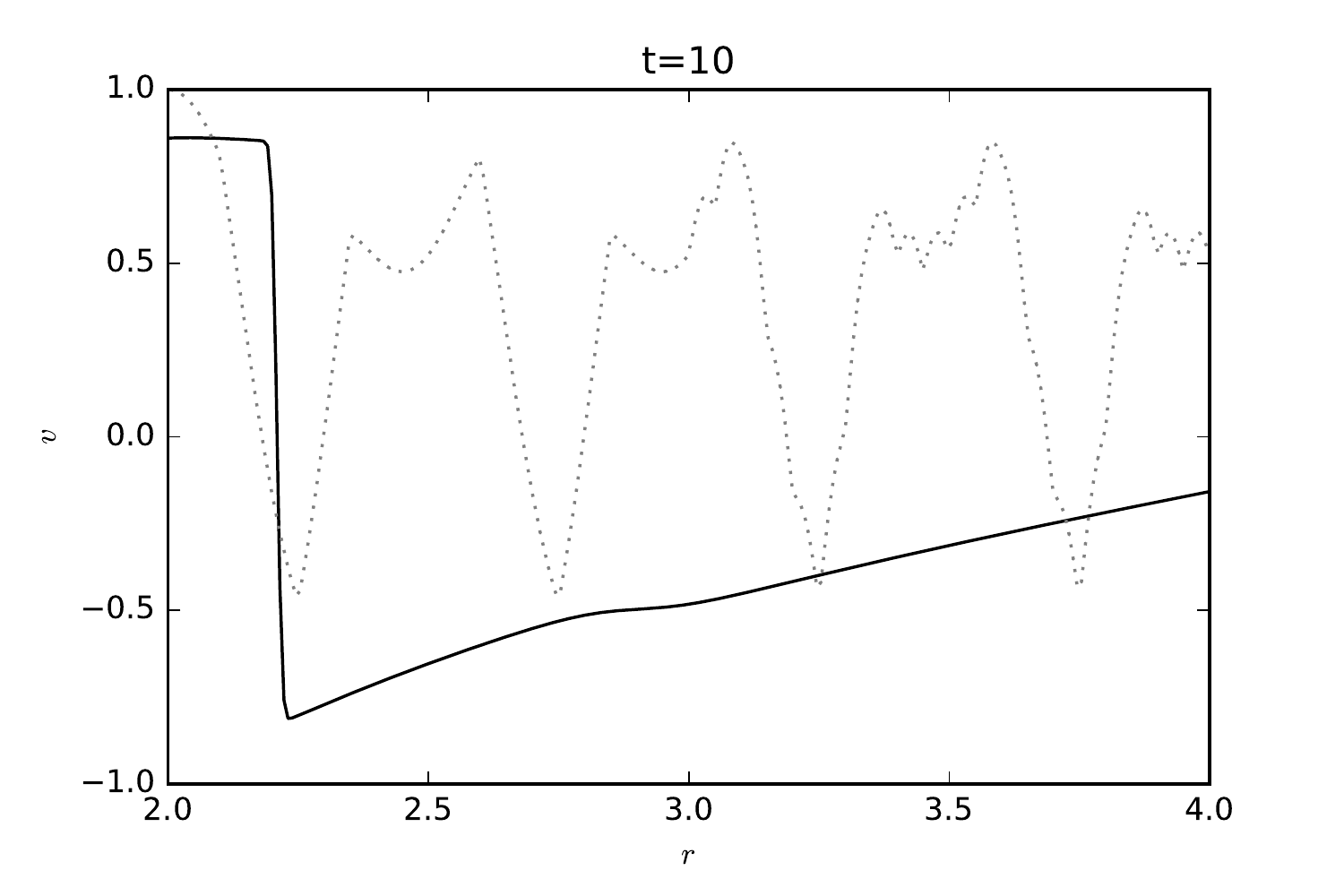,width=2.5in} 
\end{minipage}\\
\begin{minipage}[t]{0.3\linewidth}
\centering
\epsfig{figure=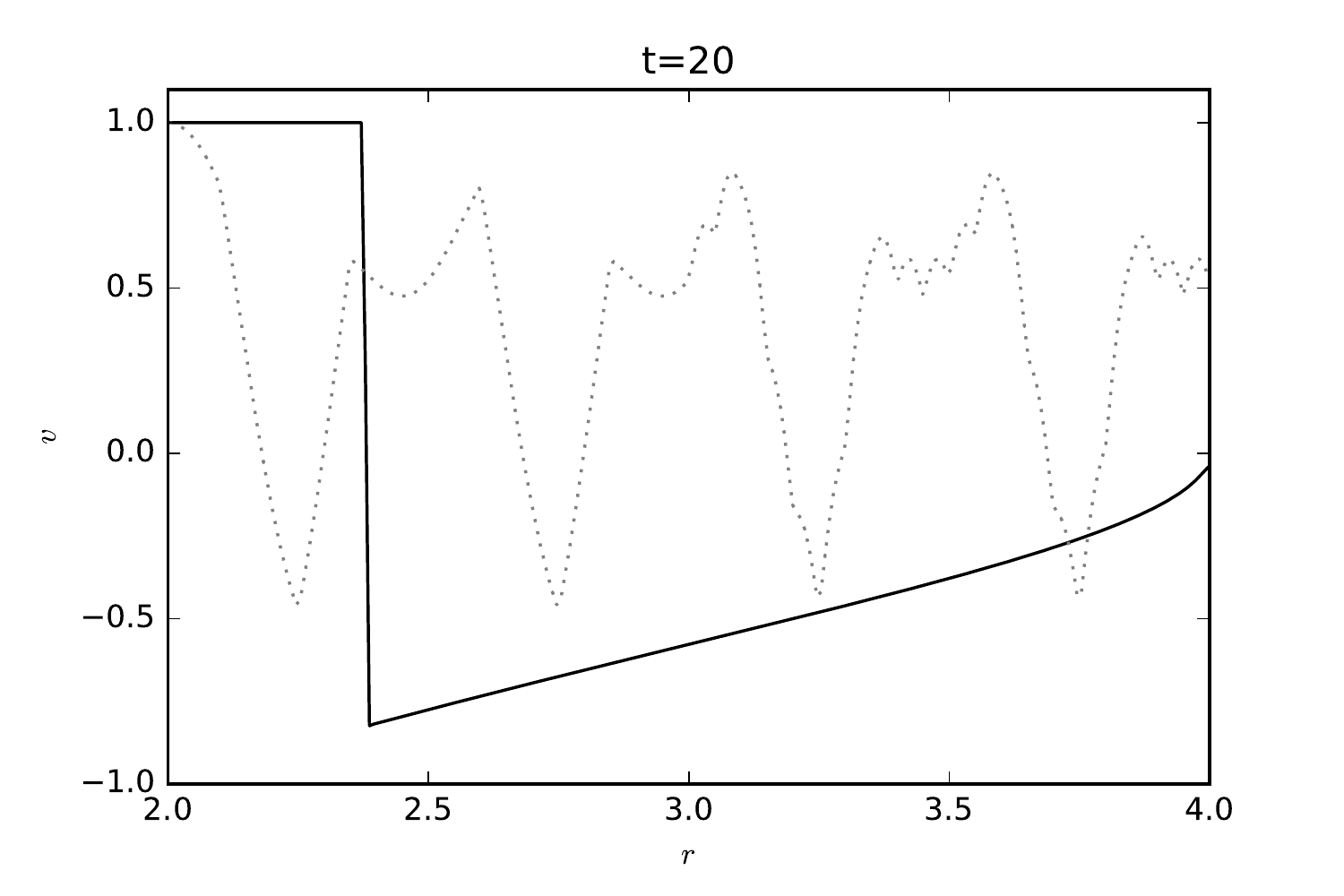,width=2.5in} 
\end{minipage}
\hspace{0.1in}
\begin{minipage}[t]{0.3\linewidth}
\centering
\epsfig{figure=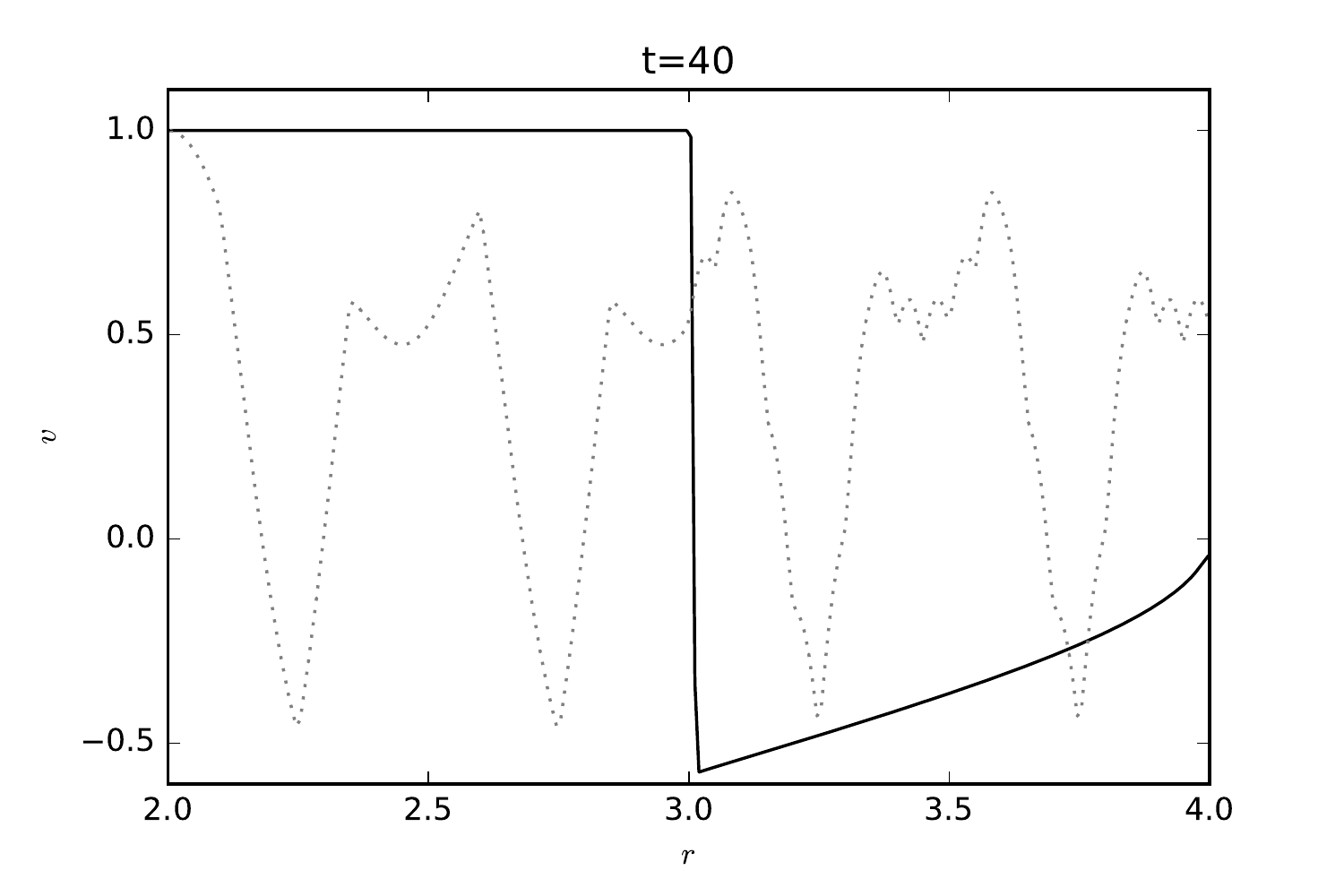,width=2.5in} 
\end{minipage}
\hspace{0.1in}
\begin{minipage}[t]{0.3\linewidth}
\centering
\epsfig{figure=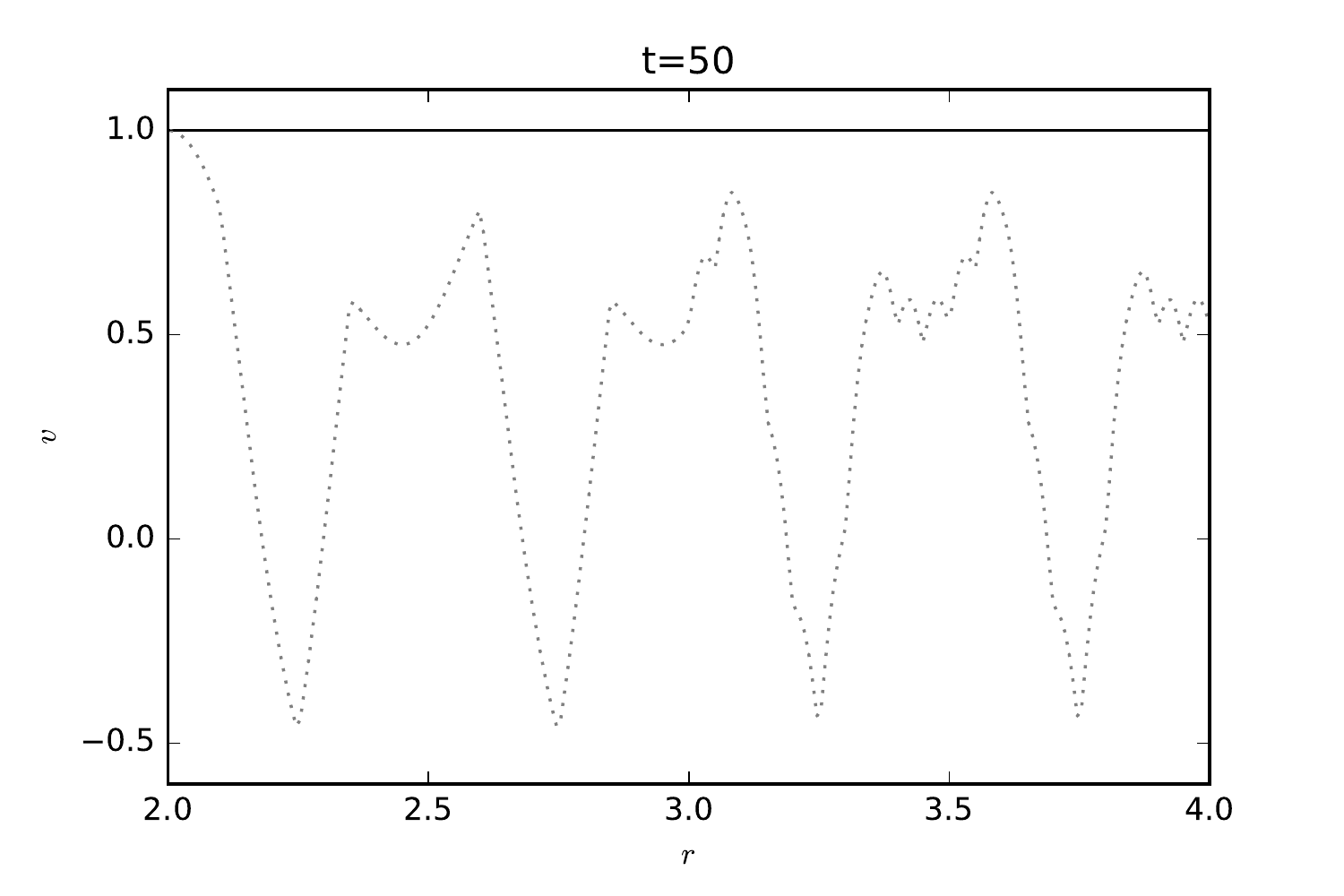,width=2.5in} 
\end{minipage}

\caption{Numerical solution with velocity $1$ at $r=2M$ and $r=+\infty$, using the finite volume scheme}
\label{FIG-82-0} 
\end{figure}

\begin{figure}[!htb] 
\centering 
\begin{minipage}[t]{0.3\linewidth}
\centering
\epsfig{figure=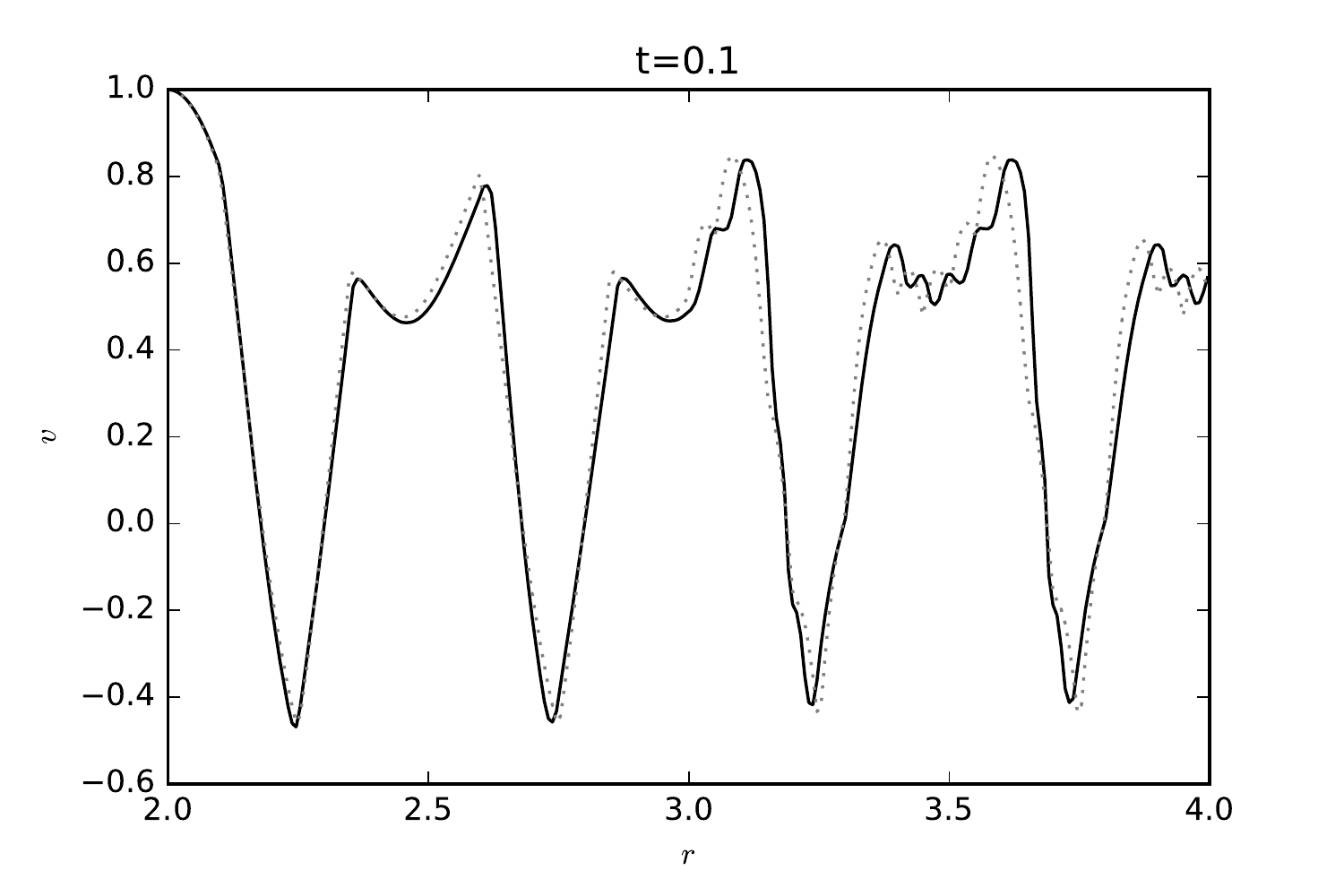,width=2.5in} 
\end{minipage}
\hspace{0.1in}
\begin{minipage}[t]{0.3\linewidth}
\centering
\epsfig{figure=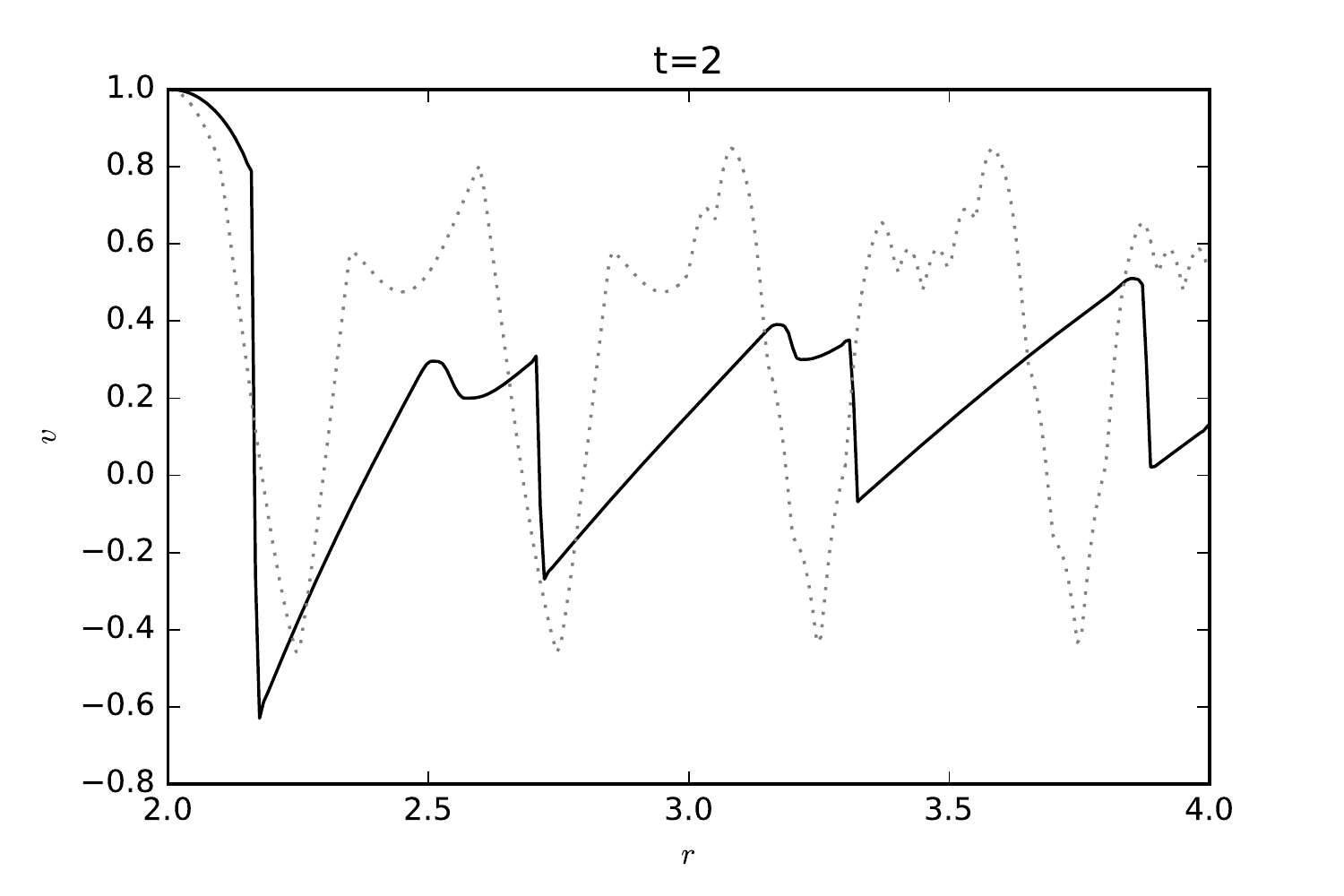,width=2.5in} 
\end{minipage}
\hspace{0.1in}
\begin{minipage}[t]{0.3\linewidth}
\centering
\epsfig{figure=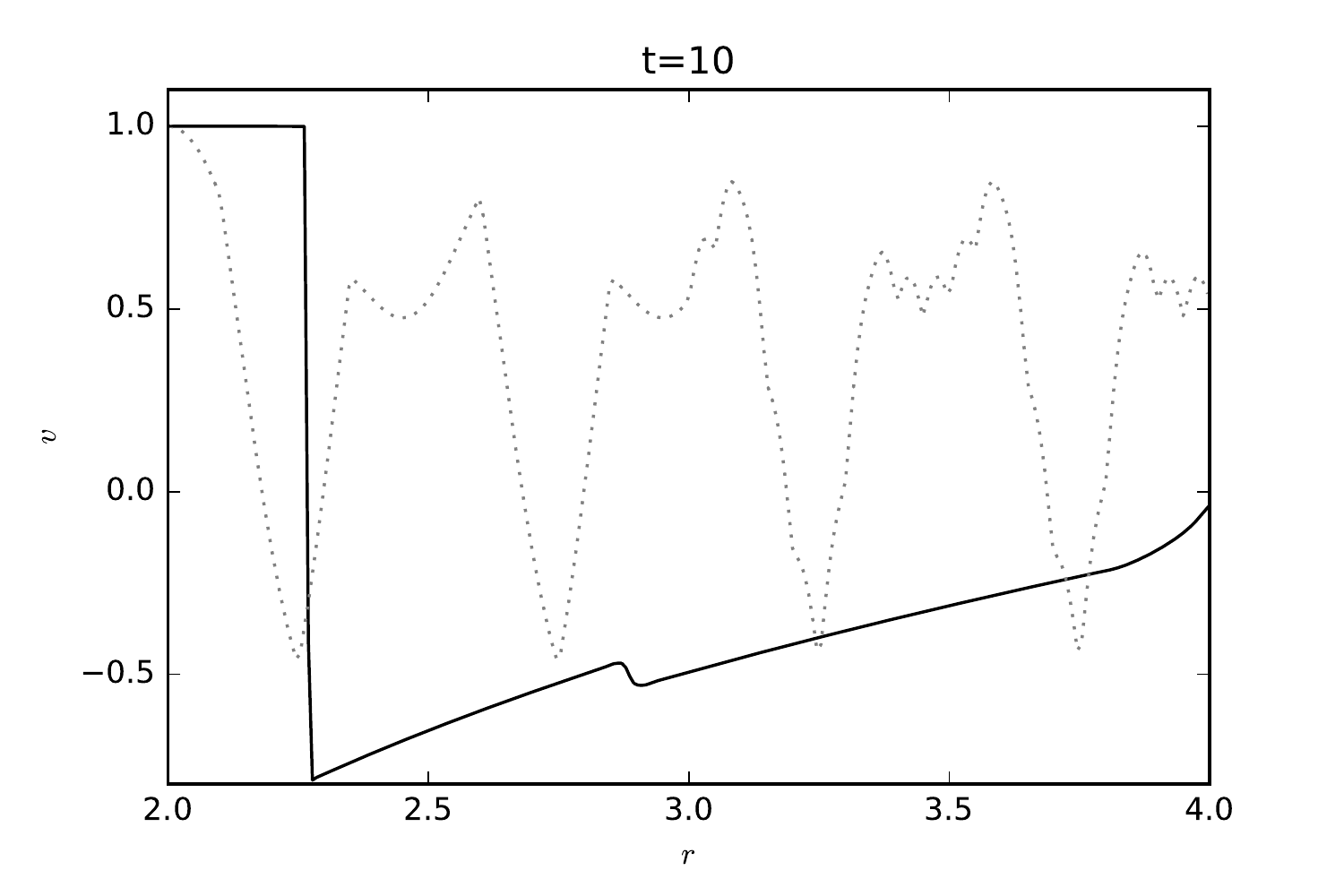,width=2.5in} 
\end{minipage}\\
\begin{minipage}[t]{0.3\linewidth}
\centering
\epsfig{figure=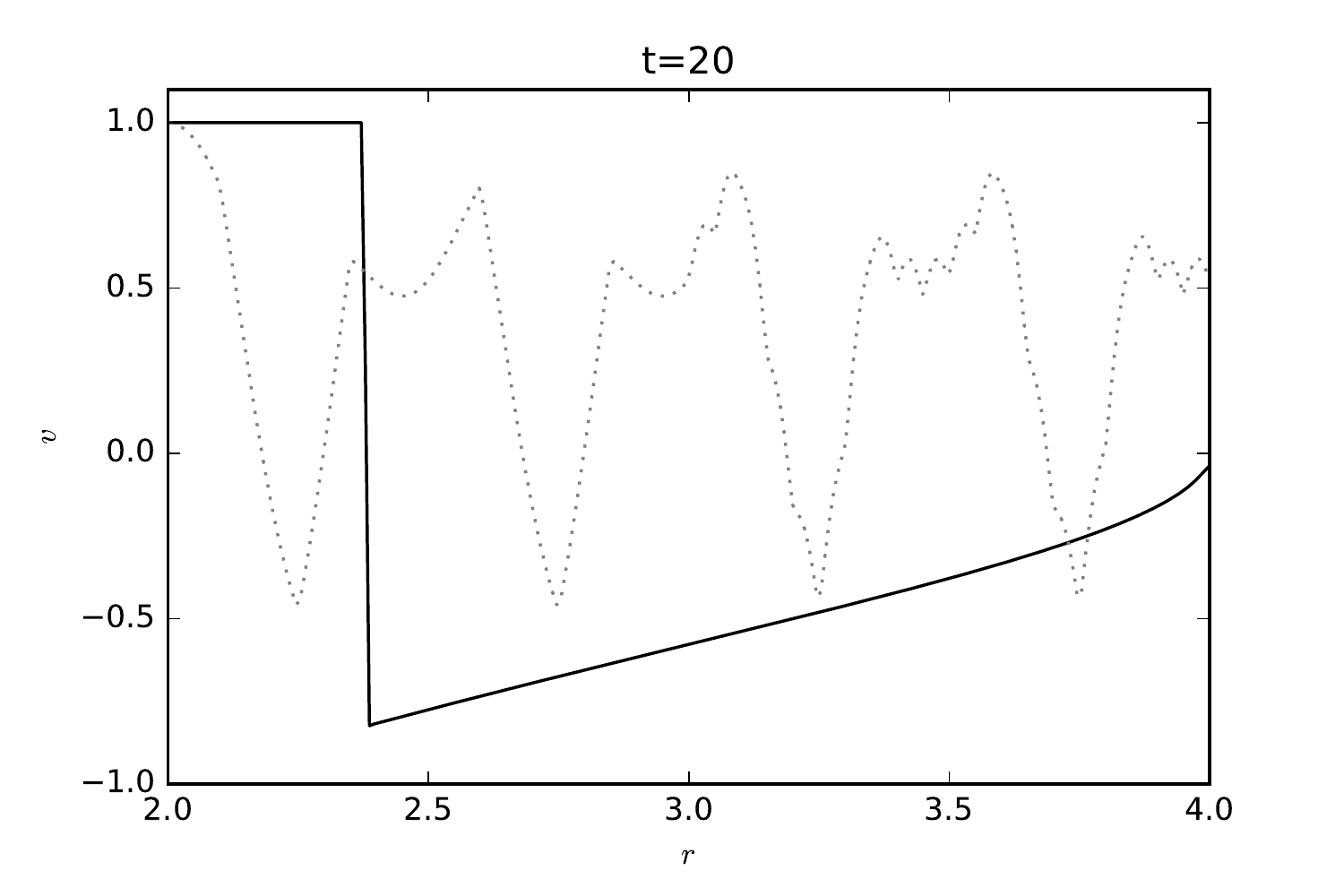,width=2.5in} 
\end{minipage}
\hspace{0.1in}
\begin{minipage}[t]{0.3\linewidth}
\centering
\epsfig{figure=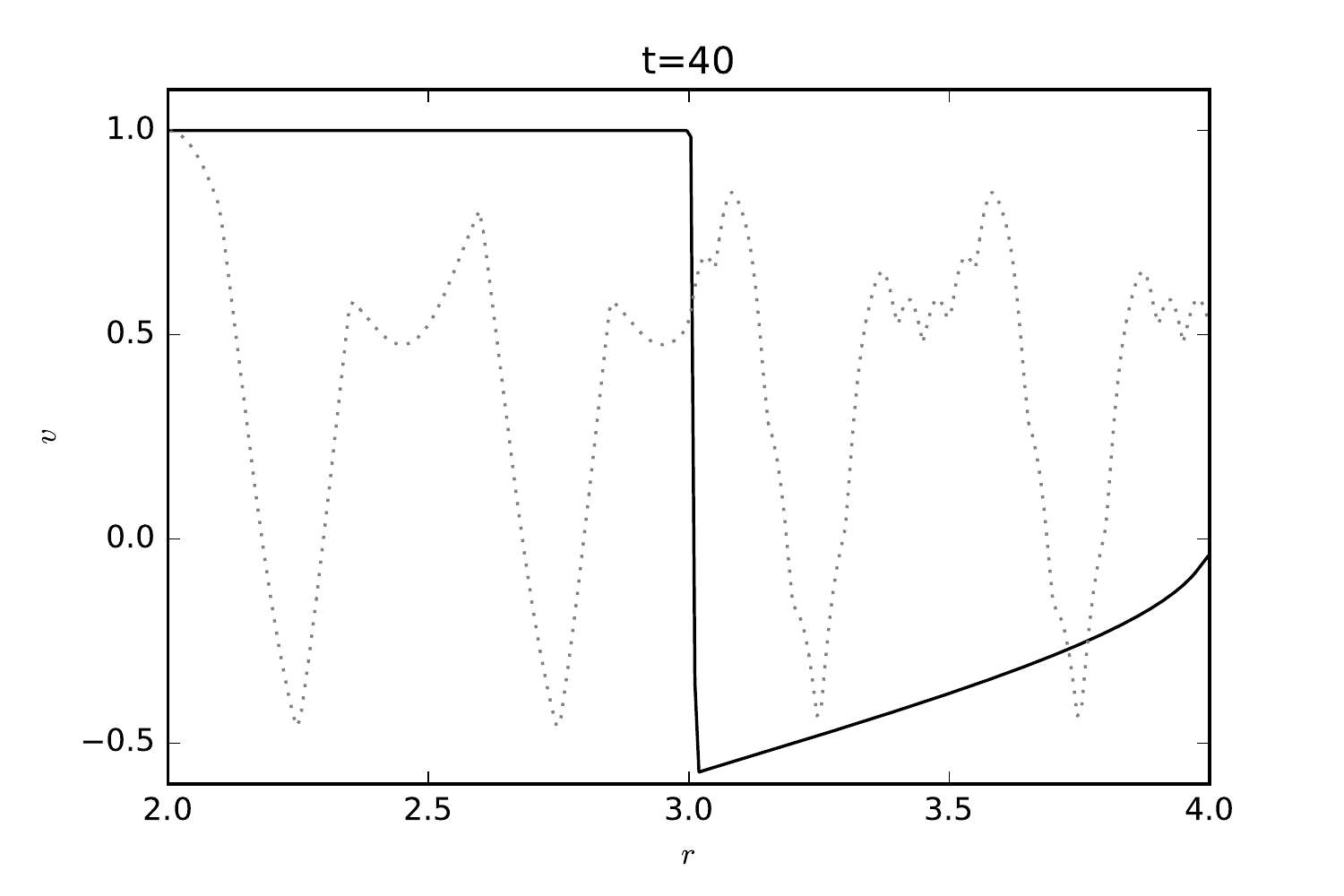,width=2.5in} 
\end{minipage}
\hspace{0.1in}
\begin{minipage}[t]{0.3\linewidth}
\centering
\epsfig{figure=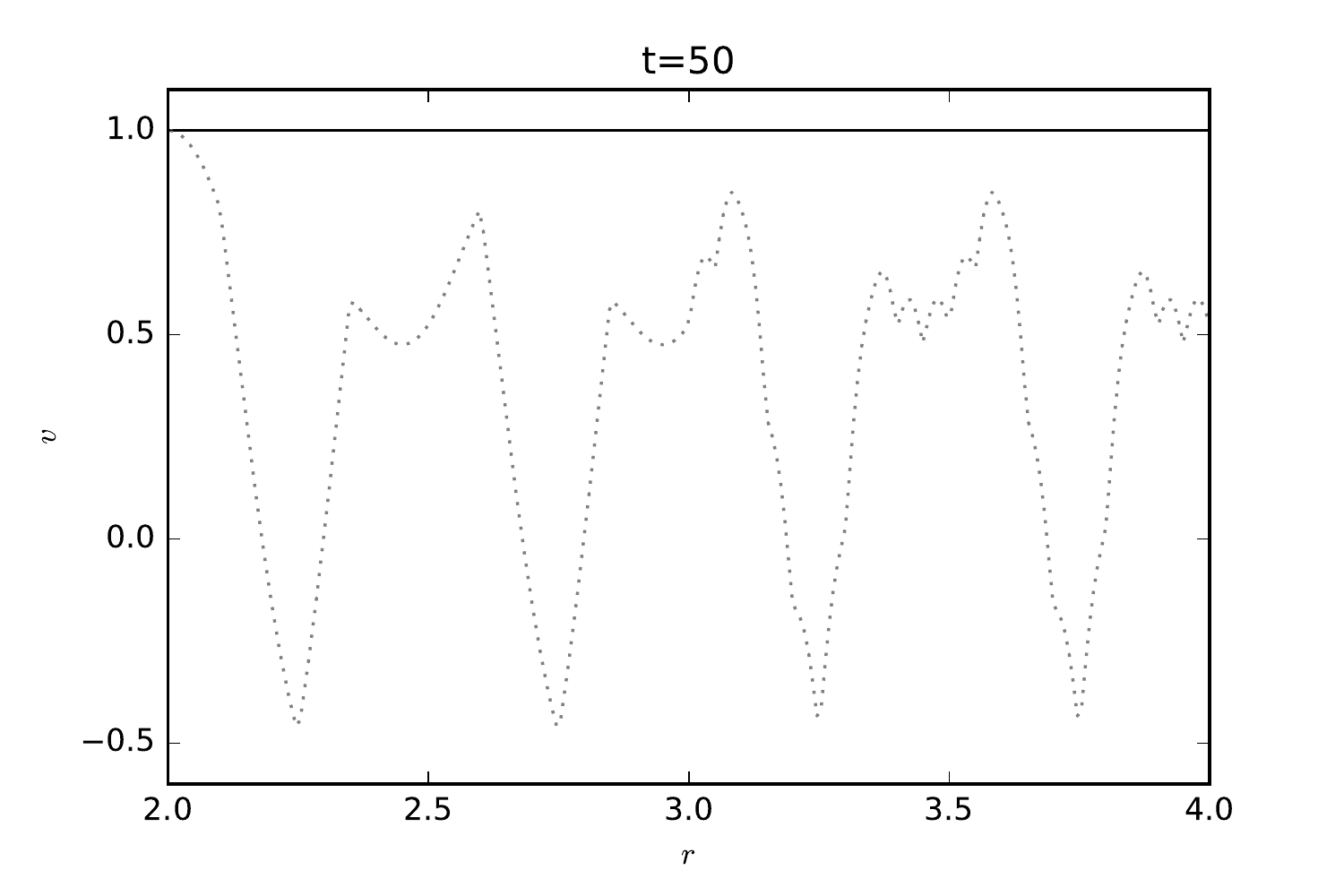,width=2.5in} 
\end{minipage}
\caption{Numerical solution with velocity $1$ at $r=2M$ and at $r=+\infty$, using the Glimm scheme}
\label{FIG-82-1} 
\end{figure}

\begin{figure}
\centering 
\begin{minipage}[t]{0.3\linewidth}
\centering
\epsfig{figure=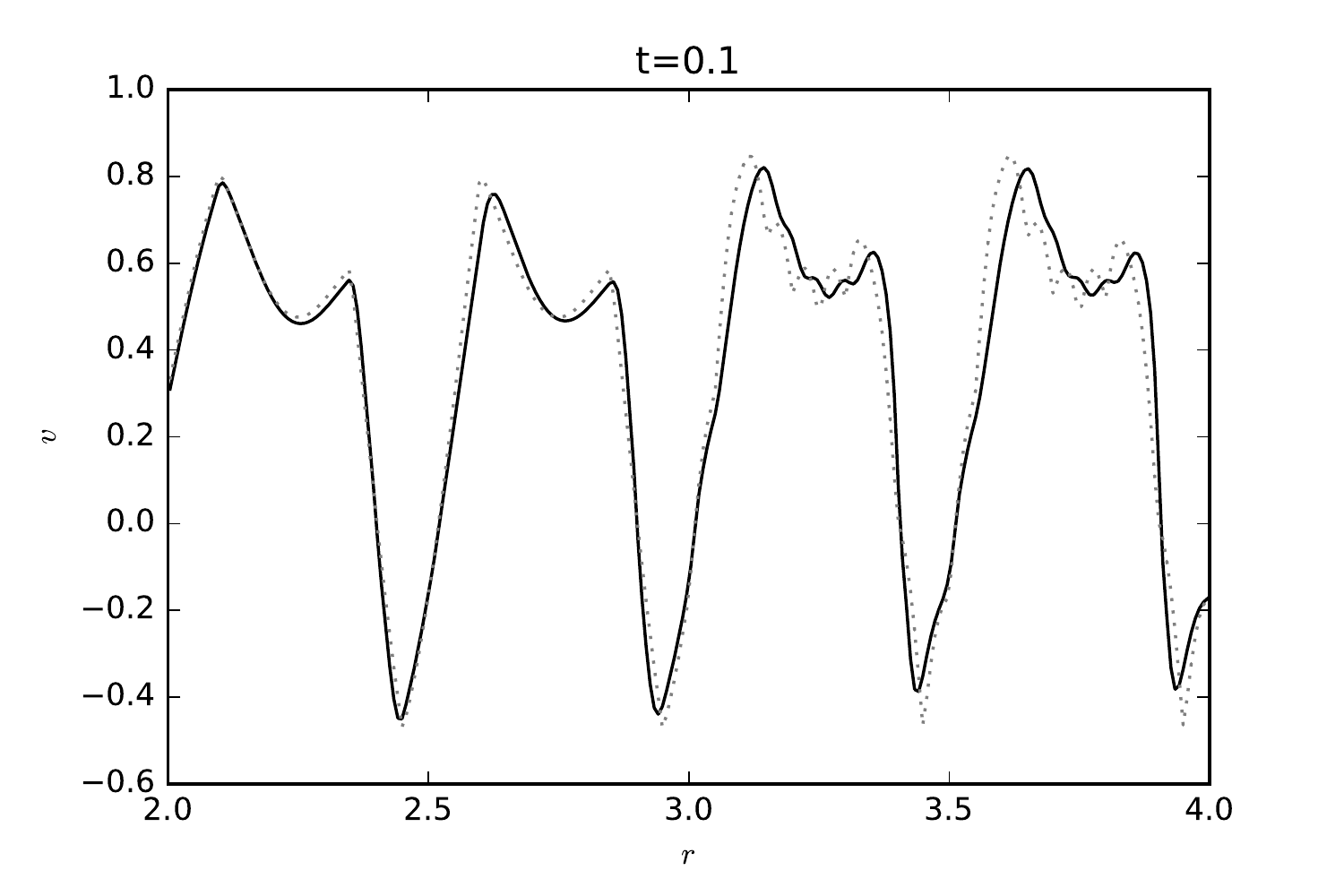,width=2.5in} 
\end{minipage}
\hspace{0.1in}
\begin{minipage}[t]{0.3\linewidth}
\centering
\epsfig{figure=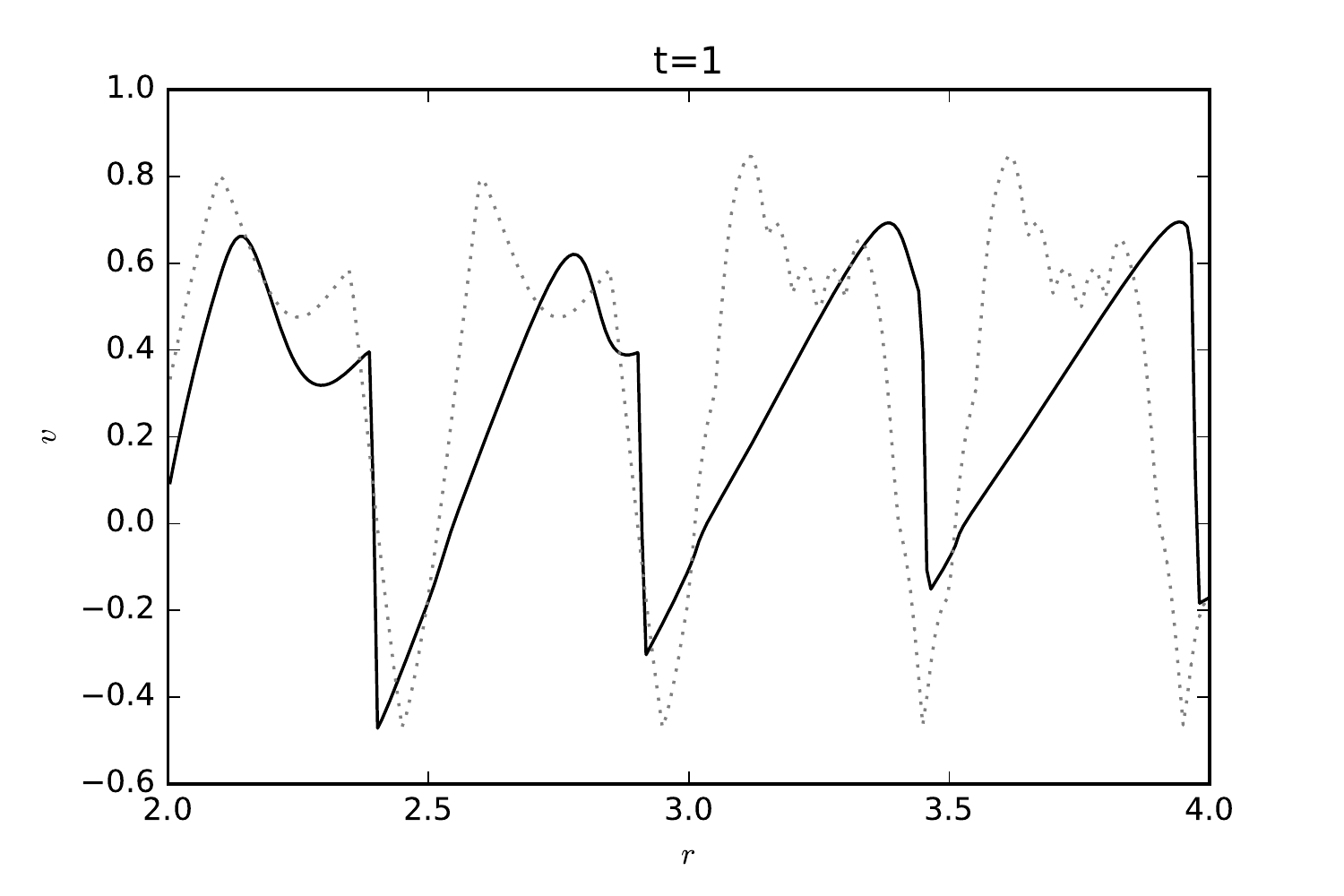,width= 2.5in} 
\end{minipage}
\hspace{0.1in}
\begin{minipage}[t]{0.3\linewidth}
\centering
\epsfig{figure=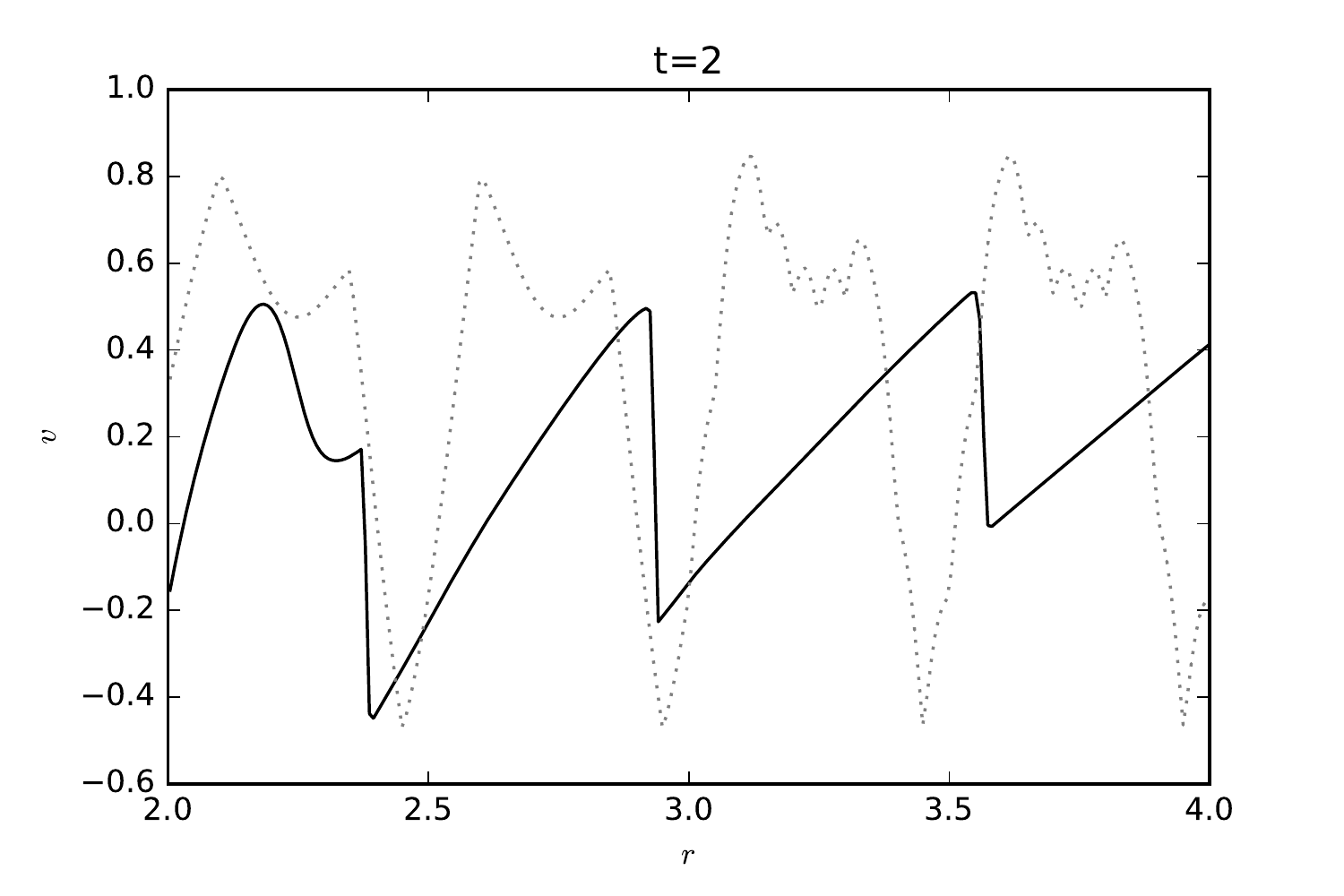,width=2.5in} 
\end{minipage}\\
\begin{minipage}[t]{0.3\linewidth}
\centering
\epsfig{figure=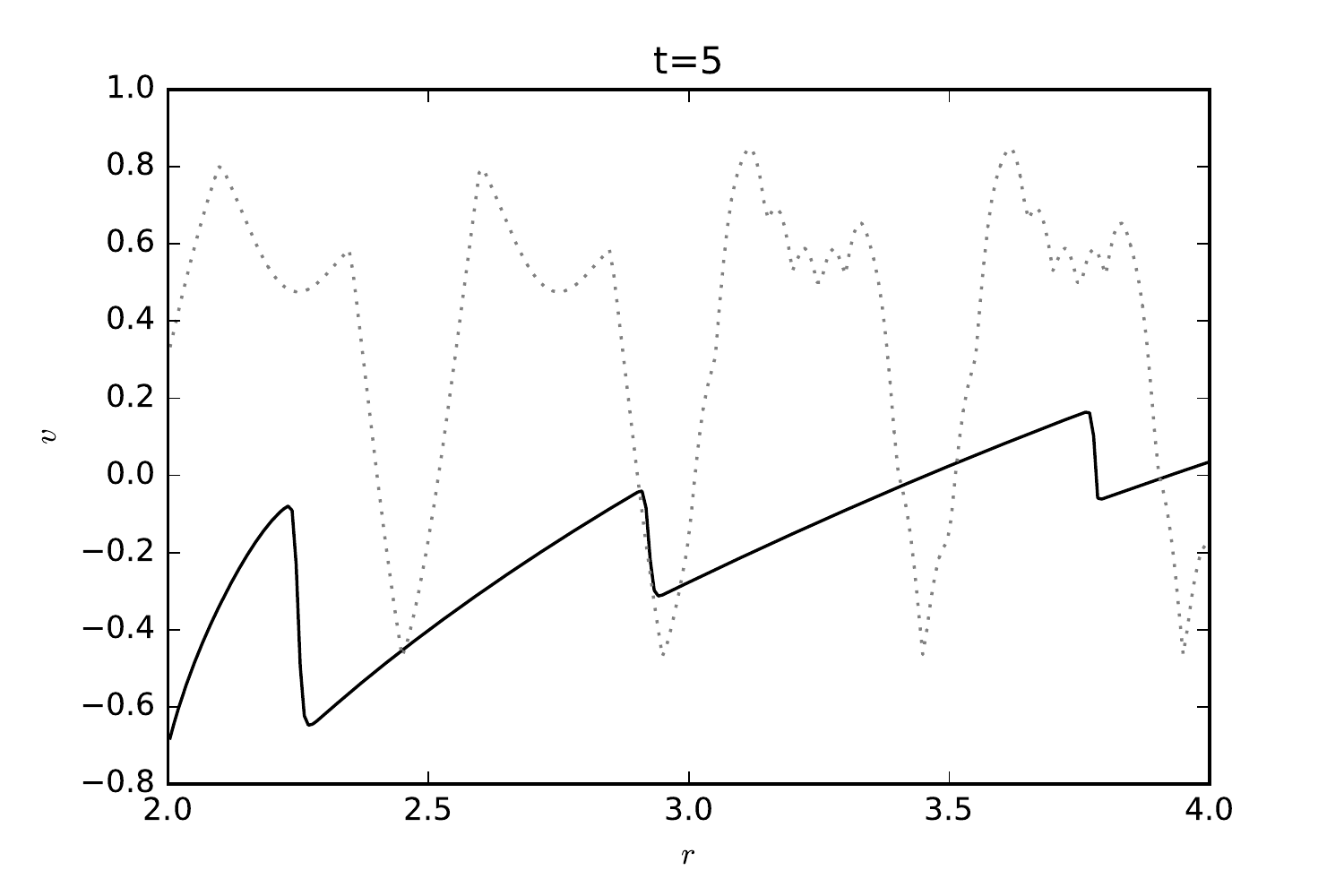,width=2.5in} 
\end{minipage}
\hspace{0.1in}
\begin{minipage}[t]{0.3\linewidth}
\centering
\epsfig{figure=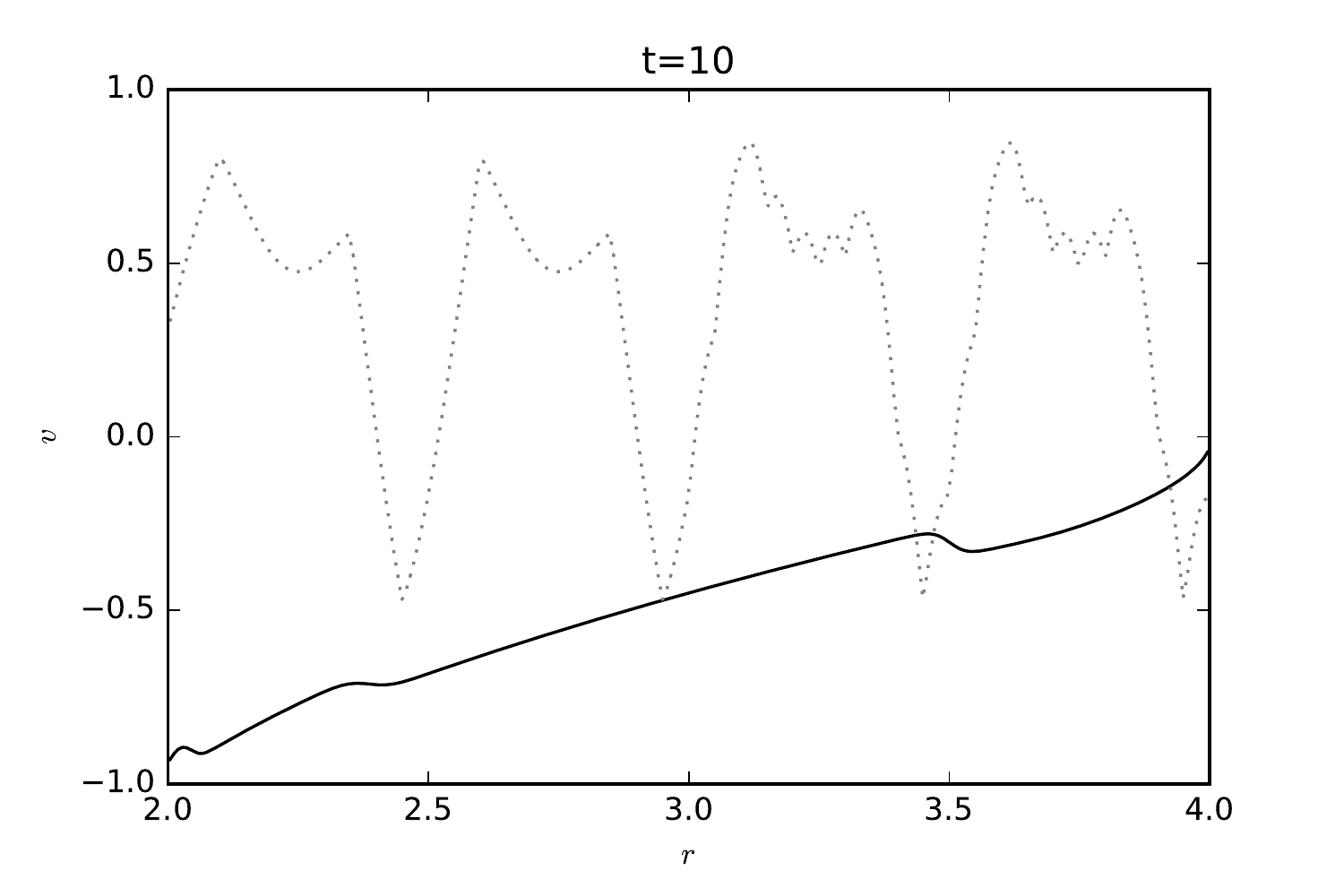,width= 2.5in} 
\end{minipage}
\hspace{0.1in}
\begin{minipage}[t]{0.3\linewidth}
\centering
\epsfig{figure=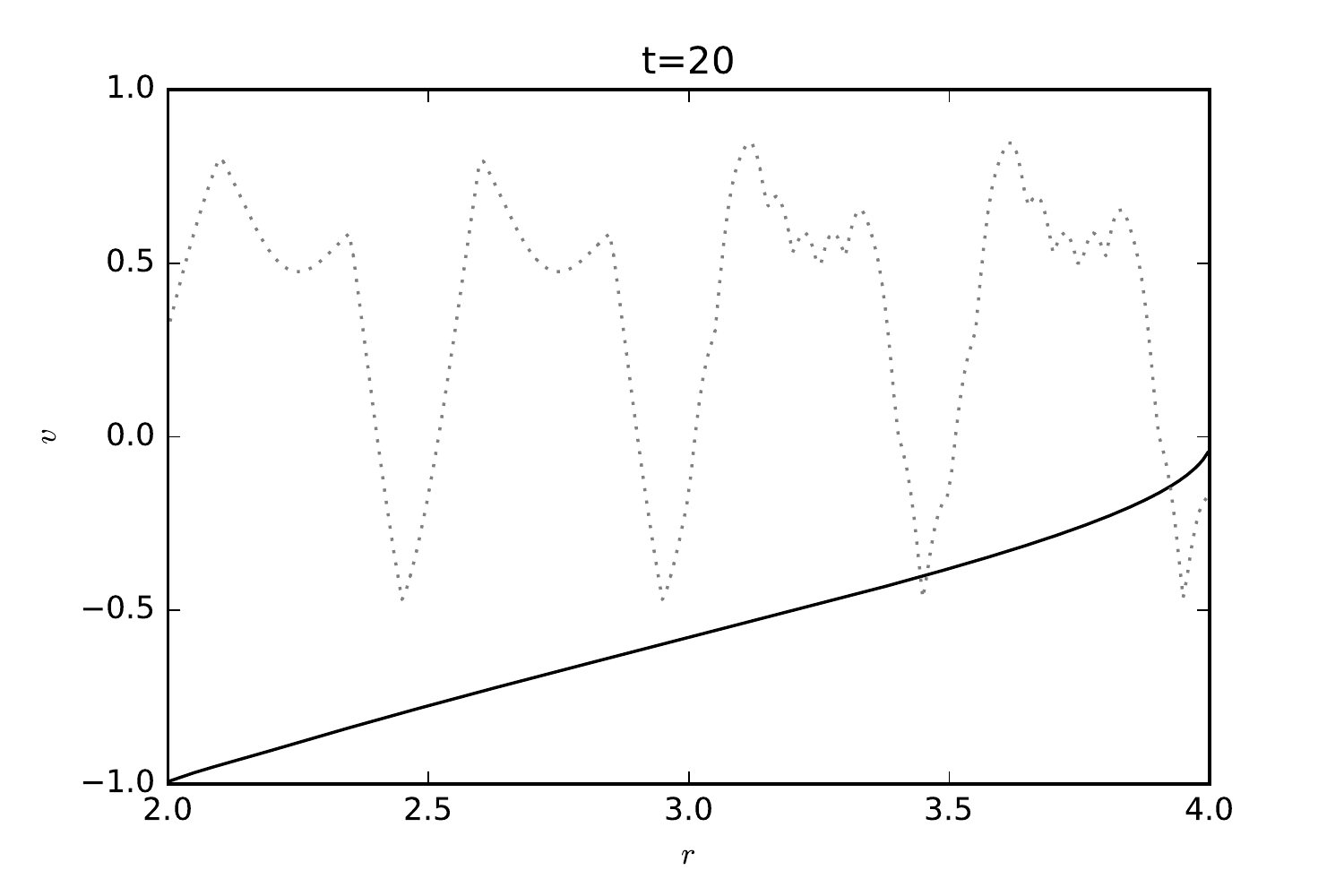,width=2.5in} 
\end{minipage}
\caption{Numerical solutions with positive velocity at $r=2M$ and $r=+\infty$, using the finite volume scheme}
\label{FIG-83} 
\end{figure}

\begin{figure}[!htb] 
\centering 
\begin{minipage}[t]{0.3\linewidth}
\centering
\epsfig{figure=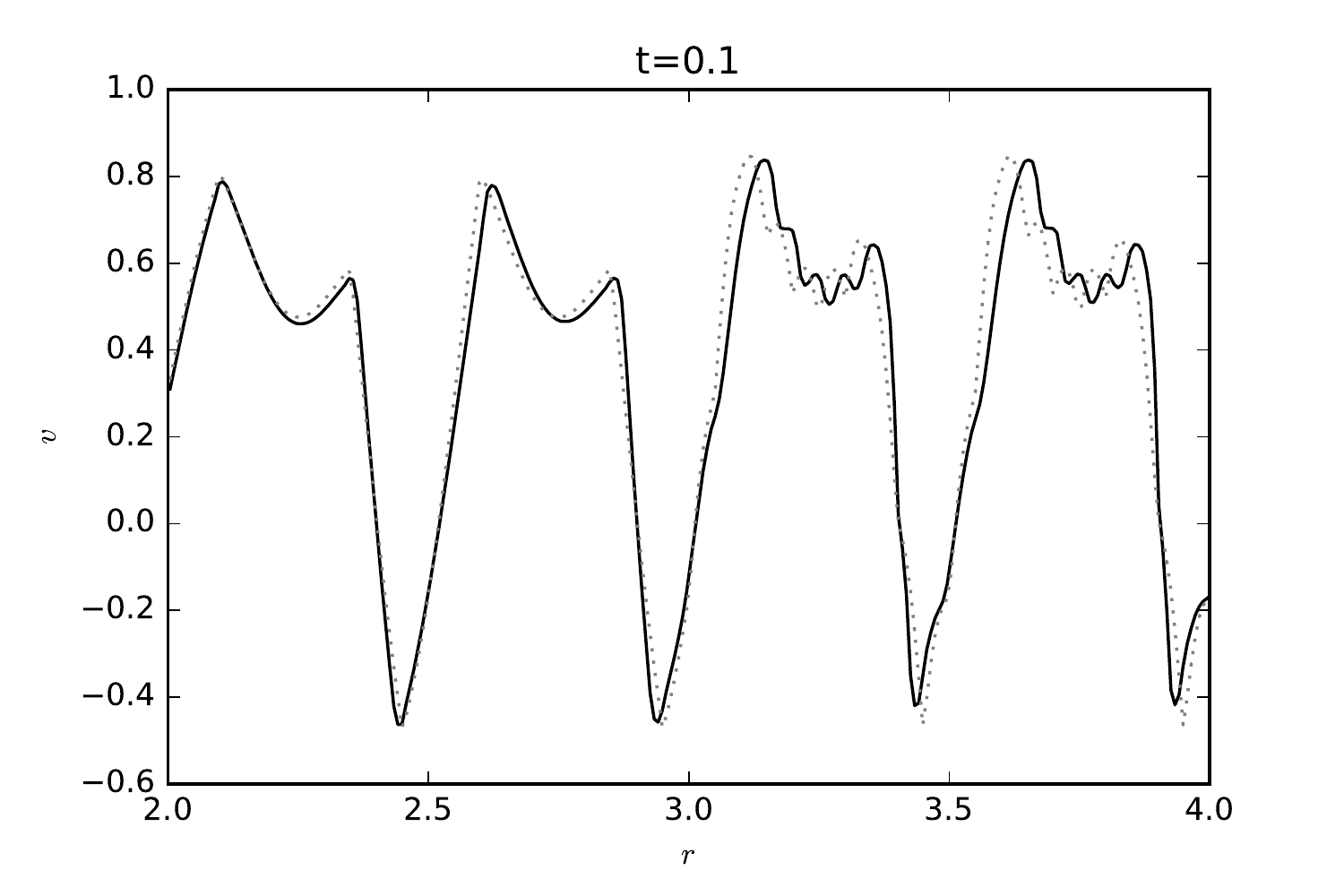,width=2.5in} 
\end{minipage}
\hspace{0.1in}
\begin{minipage}[t]{0.3\linewidth}
\centering
\epsfig{figure=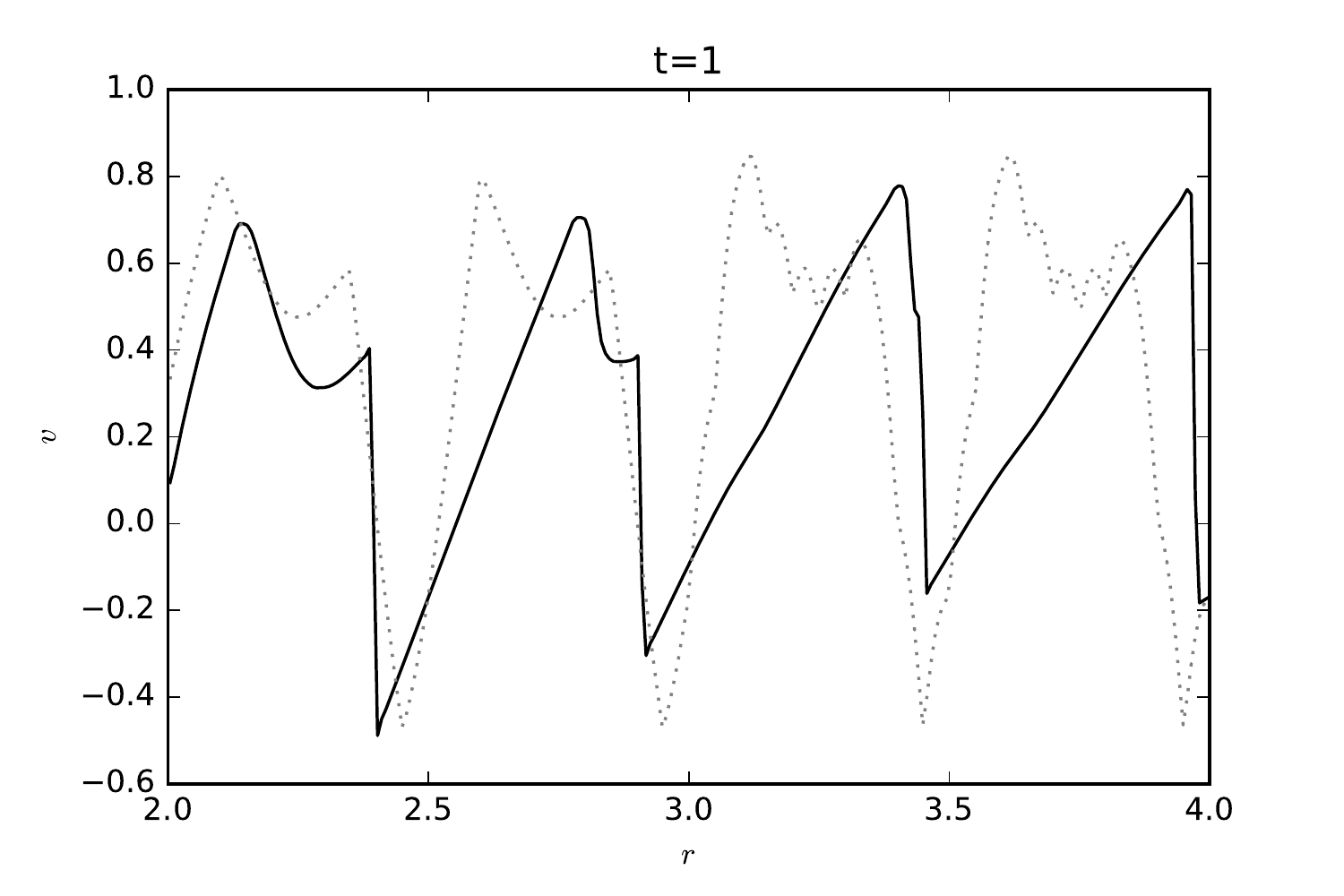,width= 2.5in} 
\end{minipage}
\hspace{0.1in}
\begin{minipage}[t]{0.3\linewidth}
\centering
\epsfig{figure=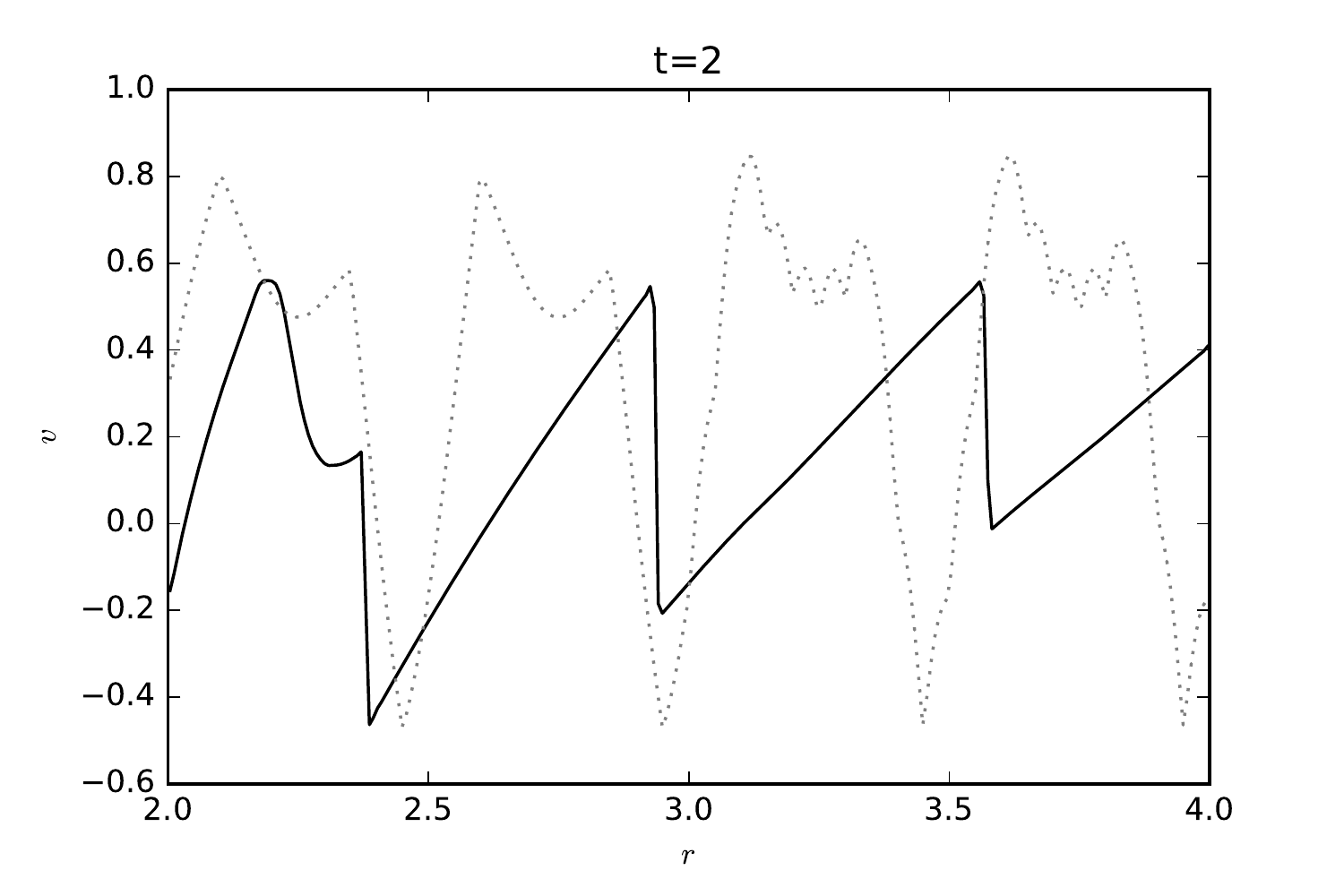,width=2.5in} 
\end{minipage}\\
\begin{minipage}[t]{0.3\linewidth}
\centering
\epsfig{figure=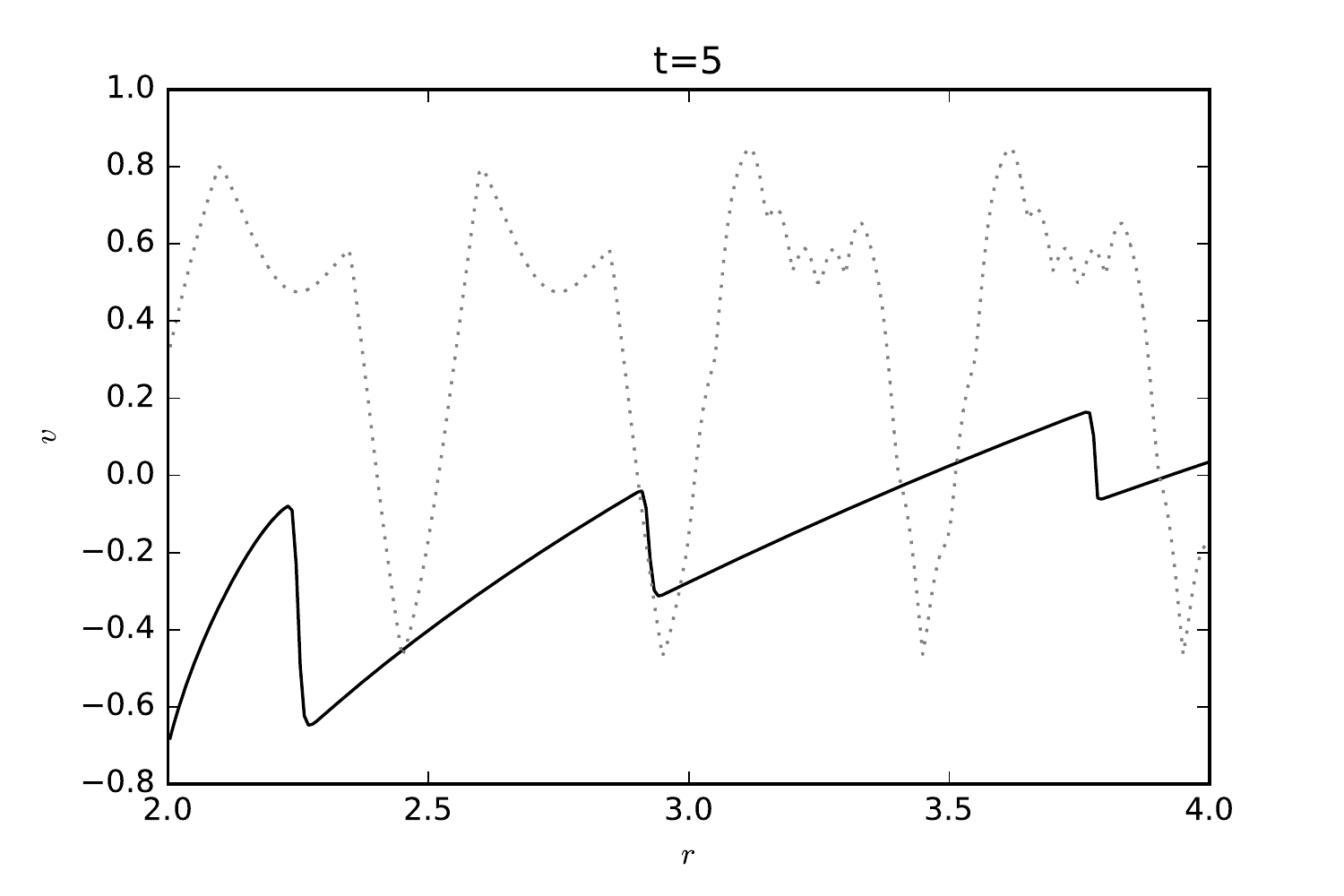,width=2.5in} 
\end{minipage}
\hspace{0.1in}
\begin{minipage}[t]{0.3\linewidth}
\centering
\epsfig{figure=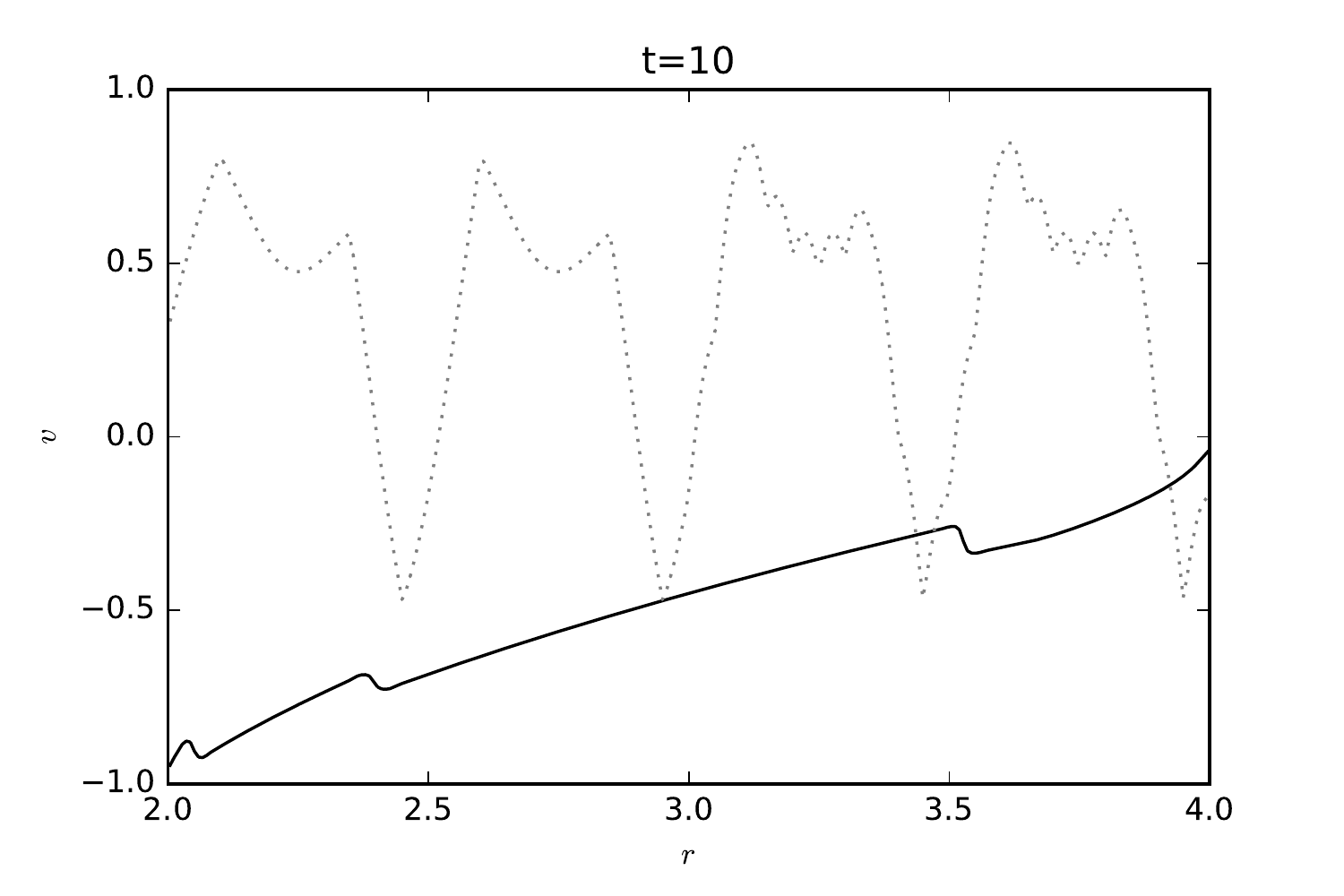,width= 2.5in} 
\end{minipage}
\hspace{0.1in}
\begin{minipage}[t]{0.3\linewidth}
\centering
\epsfig{figure=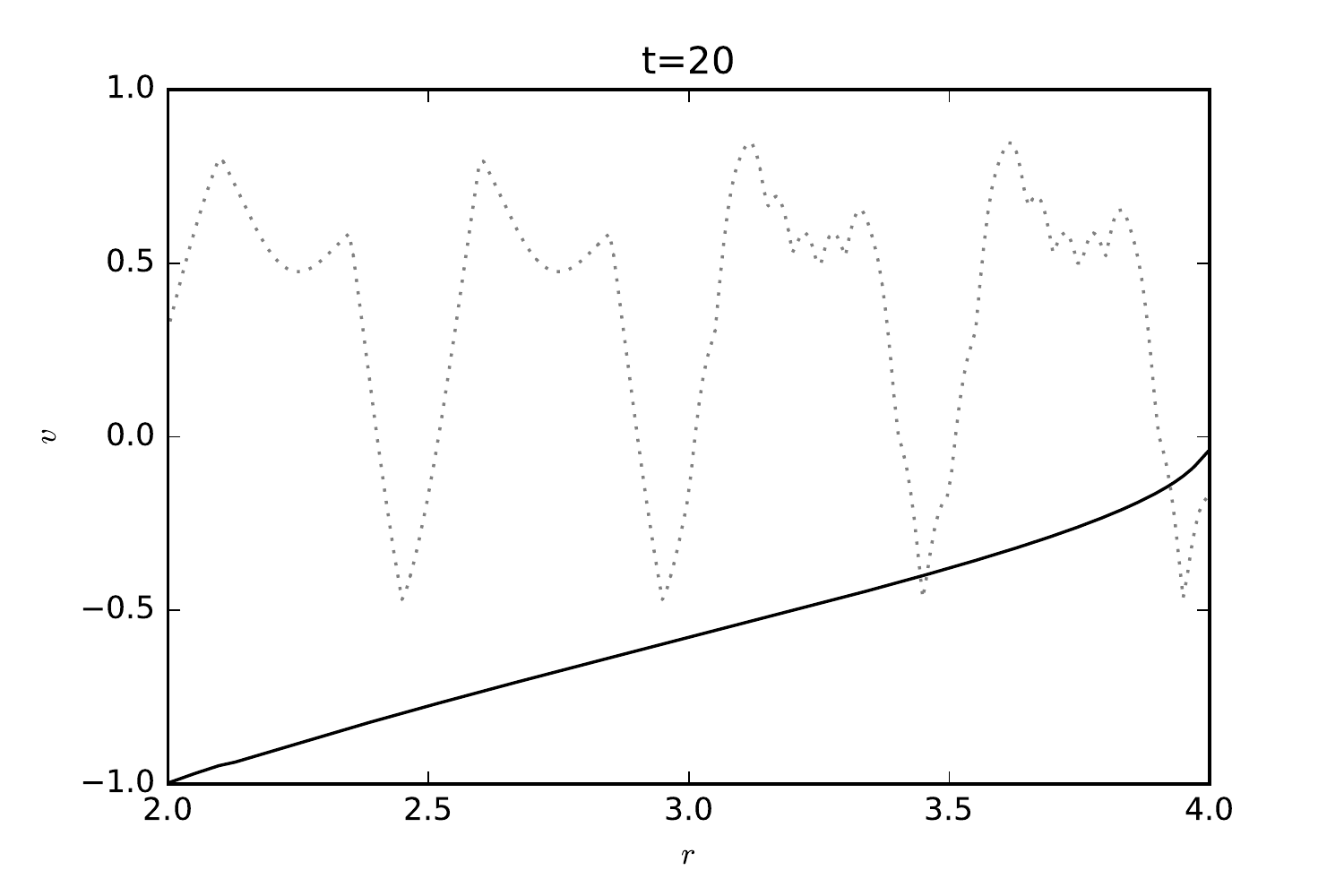,width=2.5in} 
\end{minipage}
\caption{Numerical solution positive velocity at $r= 2M$ and $r=+\infty$, using the Glimm scheme}
\label{FIG-83-bis} 
\end{figure}

\begin{figure}[!htb] 
\centering 
\begin{minipage}[t]{0.3\linewidth}
\centering
\epsfig{figure=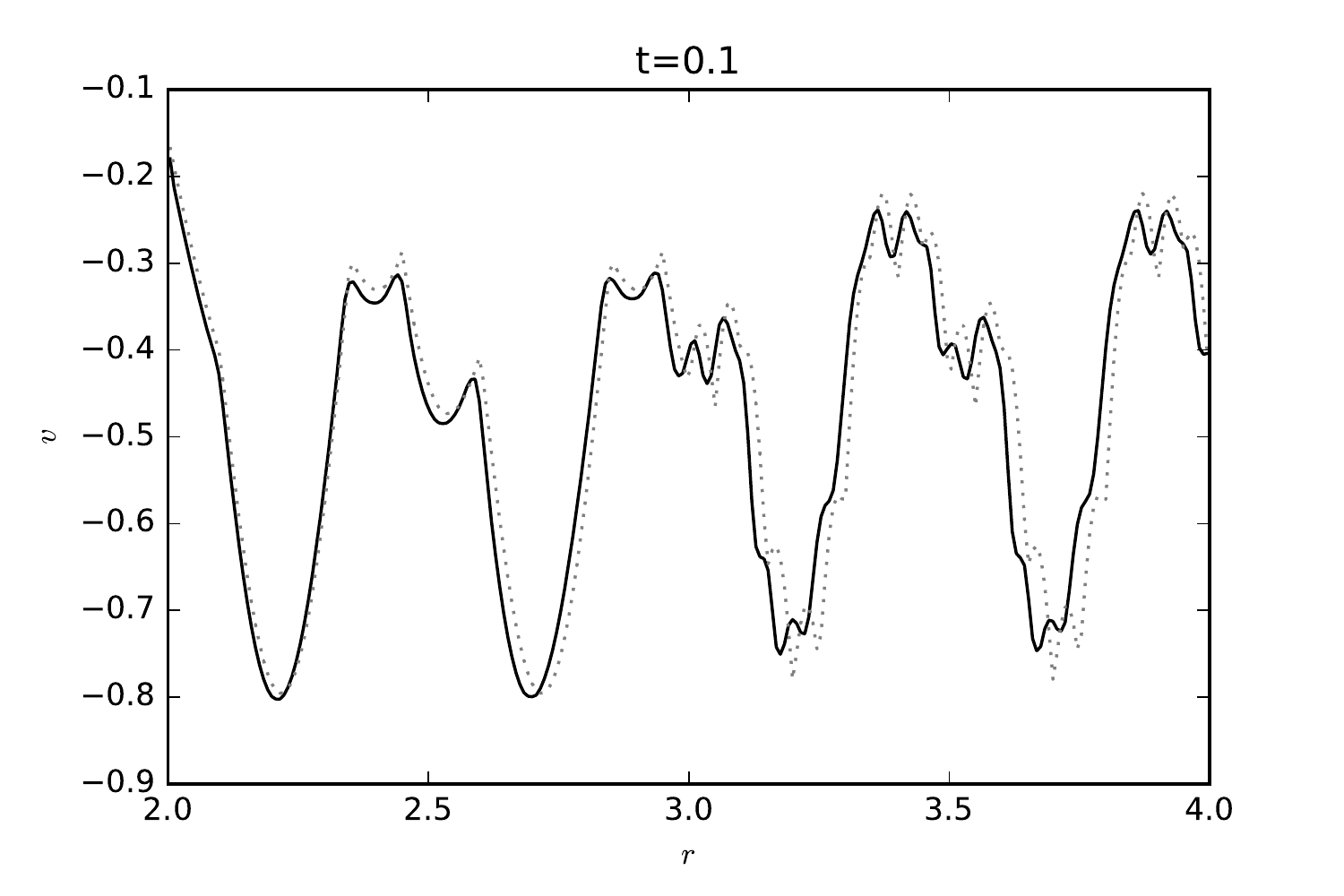,width=2.5in} 
\end{minipage}
\hspace{0.1in}
\begin{minipage}[t]{0.3\linewidth}
\centering
\epsfig{figure=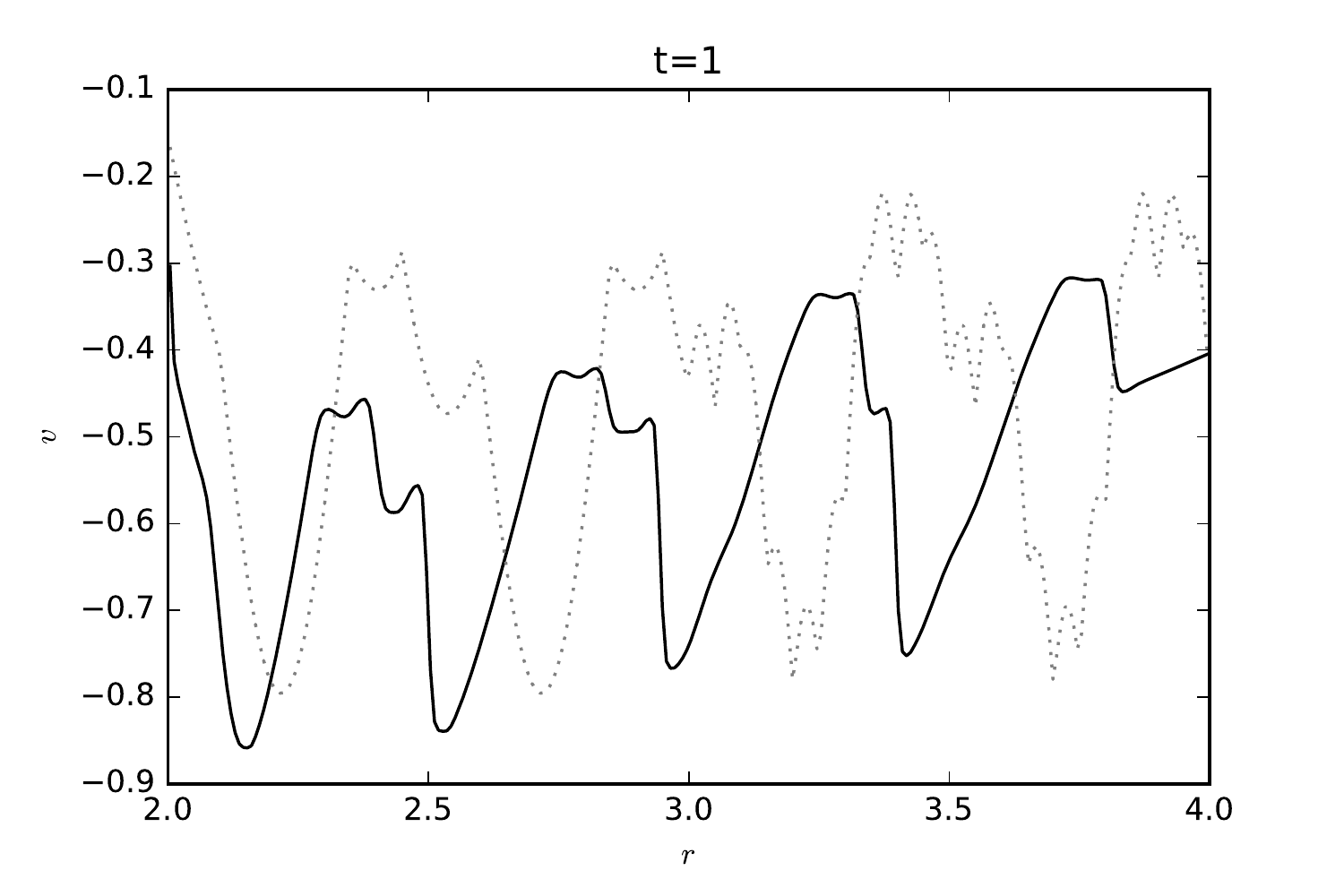,width= 2.5in} 
\end{minipage}
\hspace{0.1in}
\begin{minipage}[t]{0.3\linewidth}
\centering
\epsfig{figure=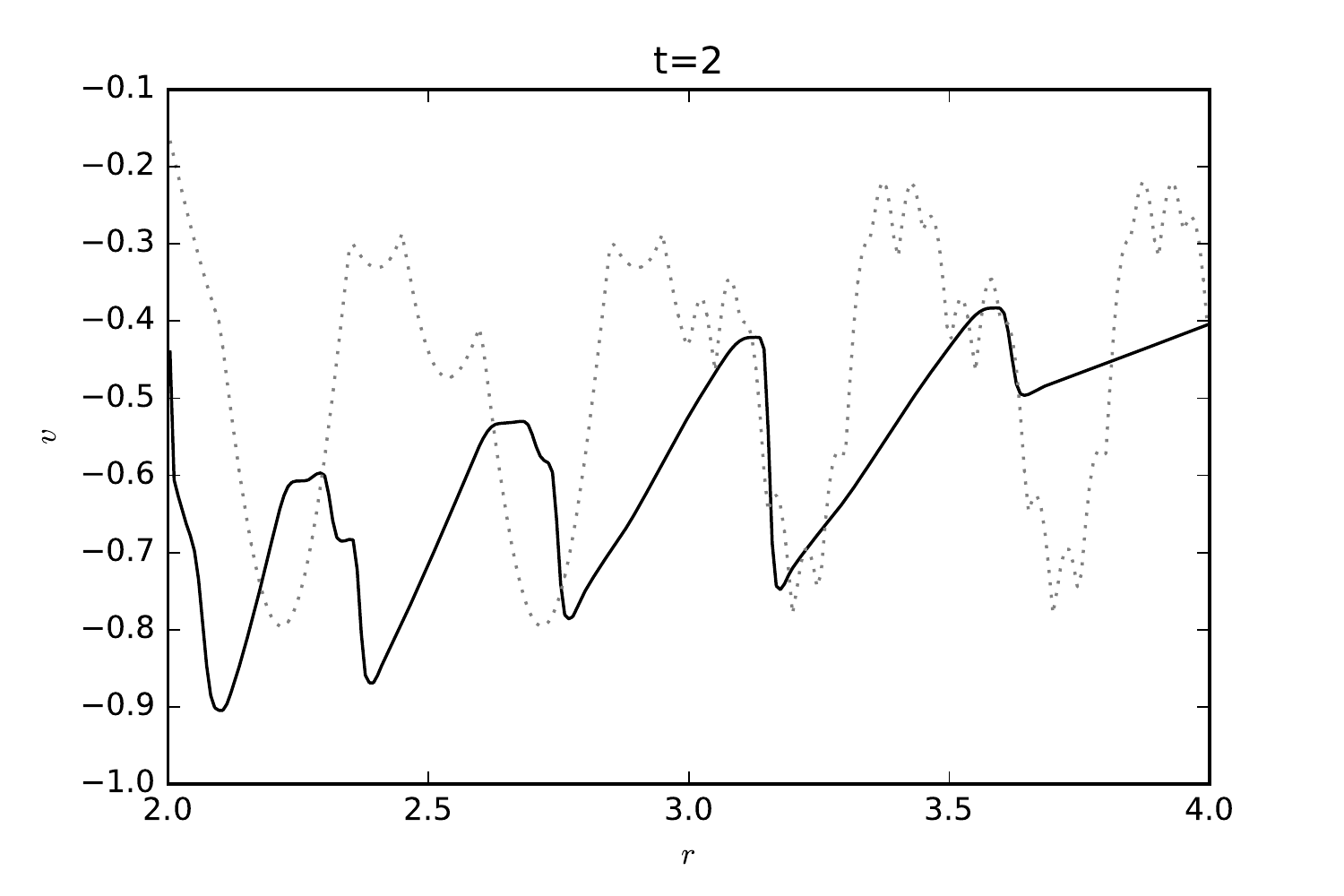,width=2.5in} 
\end{minipage}

\begin{minipage}[t]{0.3\linewidth}
\centering
\epsfig{figure=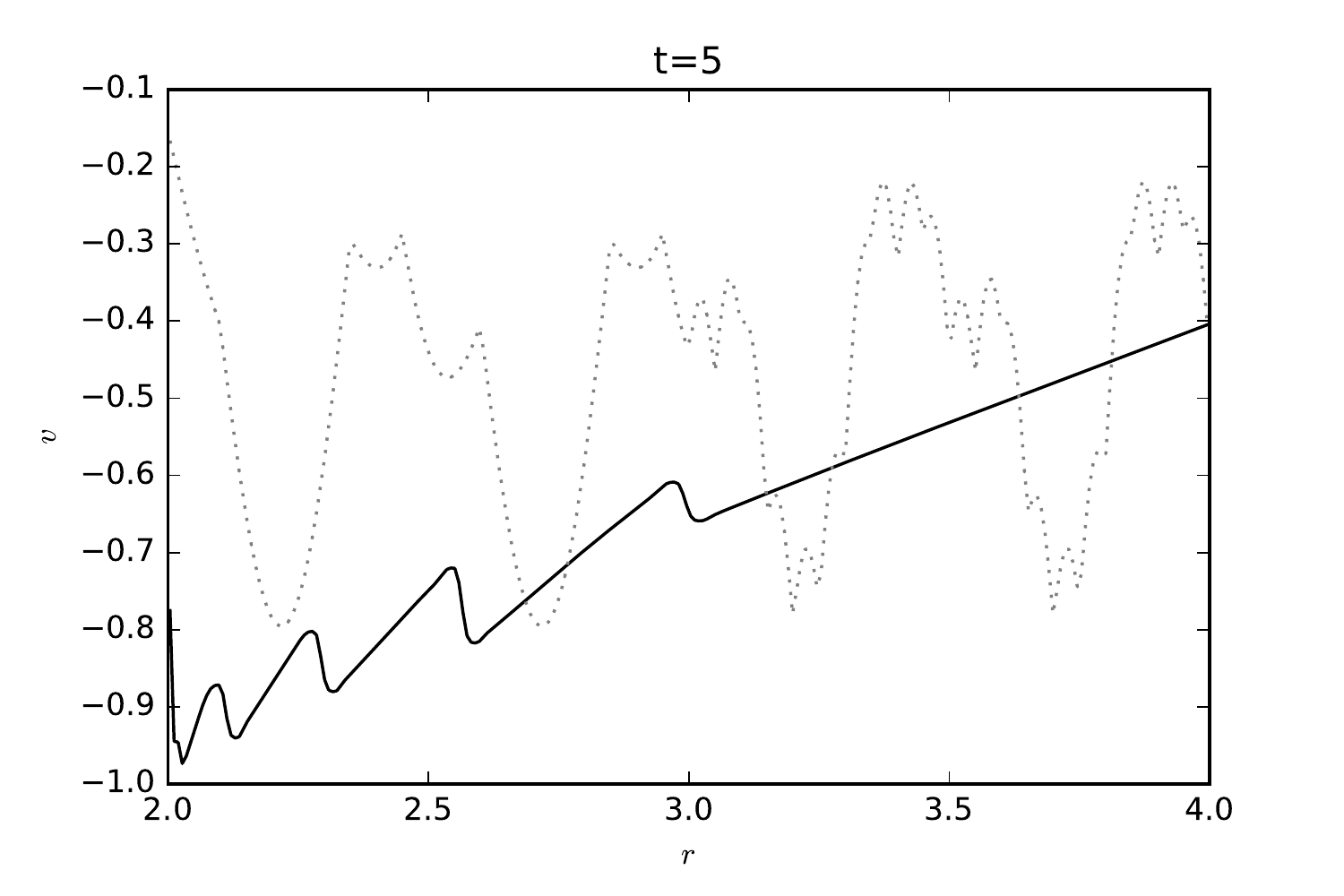,width=2.5in} 
\end{minipage}
\hspace{0.1in}
\begin{minipage}[t]{0.3\linewidth}
\centering
\epsfig{figure=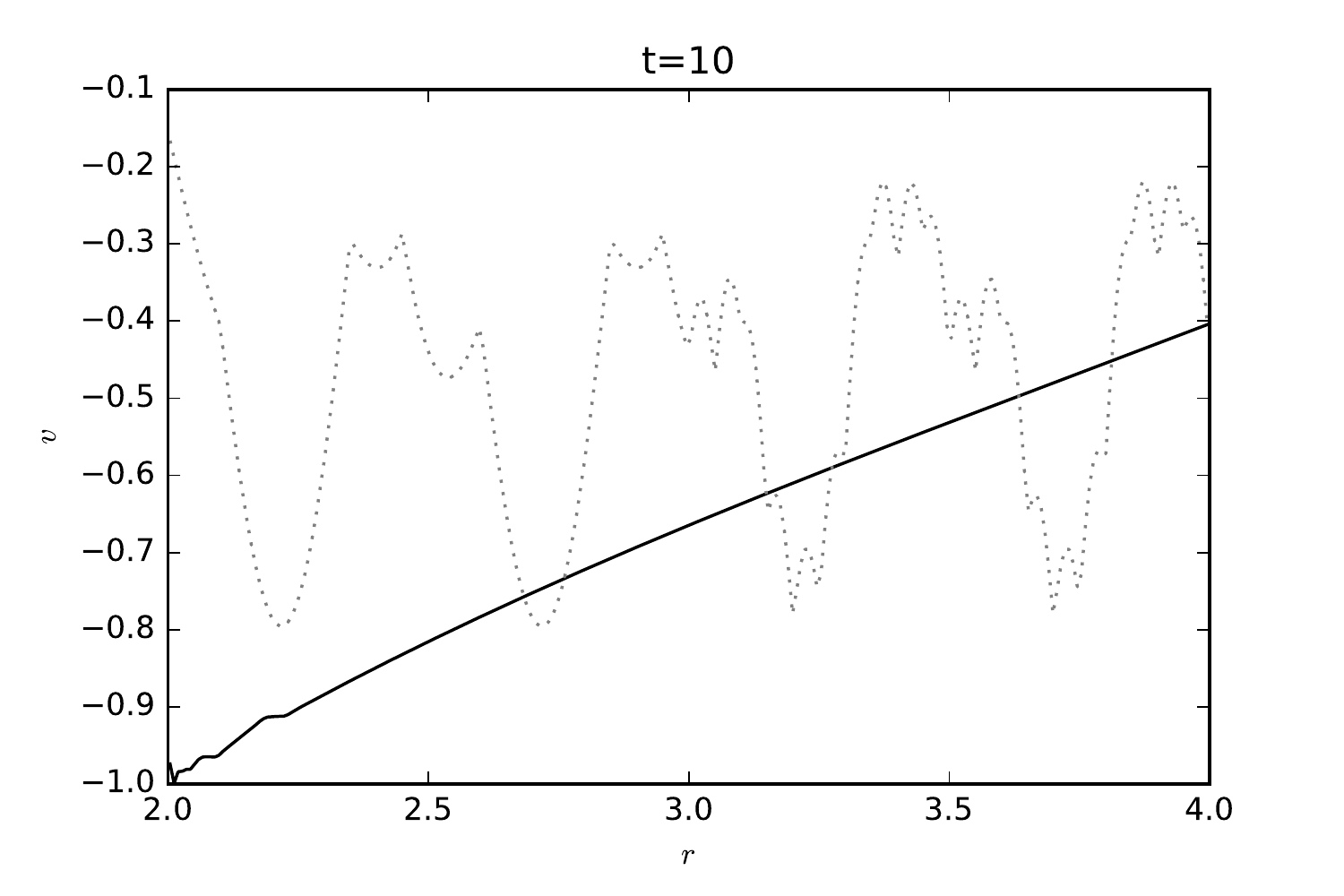,width= 2.5in} 
\end{minipage}
\hspace{0.1in}
\begin{minipage}[t]{0.3\linewidth}
\centering
\epsfig{figure=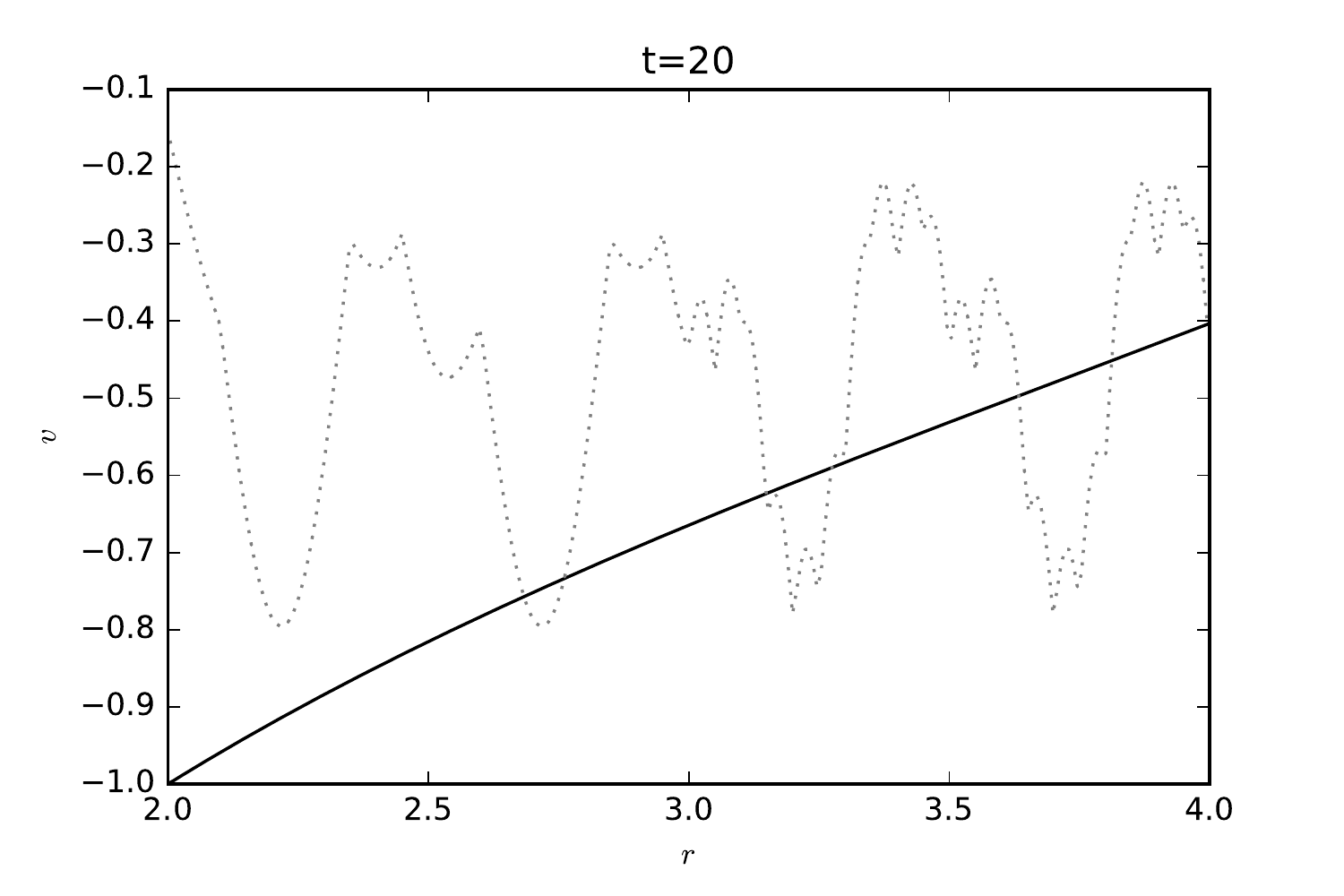,width=2.5in} 
\end{minipage}
\caption{Numerical solution with velocity $-1$ at $r= 2M$ and negative velocity at $r=+\infty$n using the finite volume scheme}
\label{FIG-83-0} 
\end{figure}

\begin{figure}[!htb] 
\centering 
\begin{minipage}[t]{0.3\linewidth}
\centering
\epsfig{figure=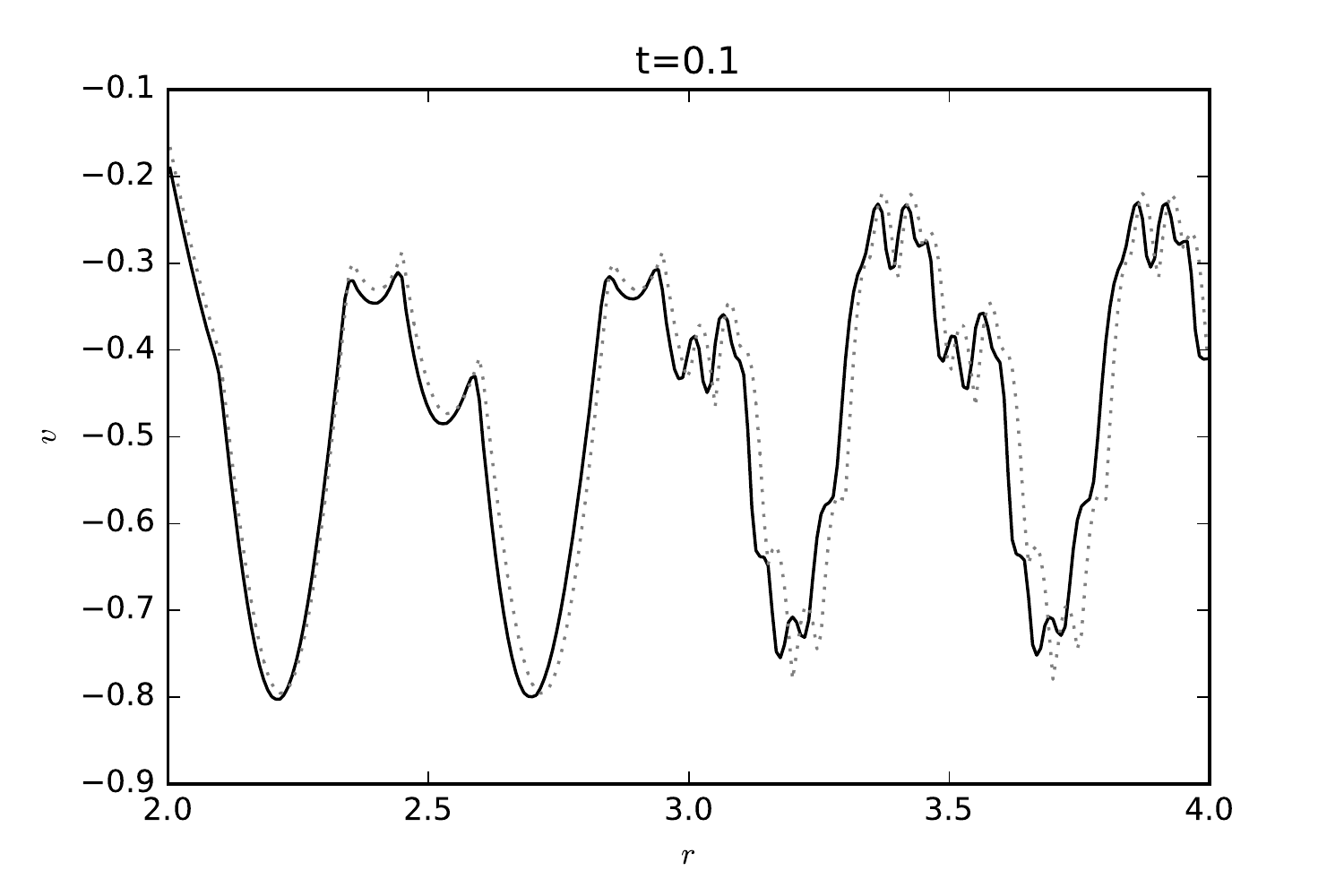,width=2.5in} 
\end{minipage}
\hspace{0.1in}
\begin{minipage}[t]{0.3\linewidth}
\centering
\epsfig{figure=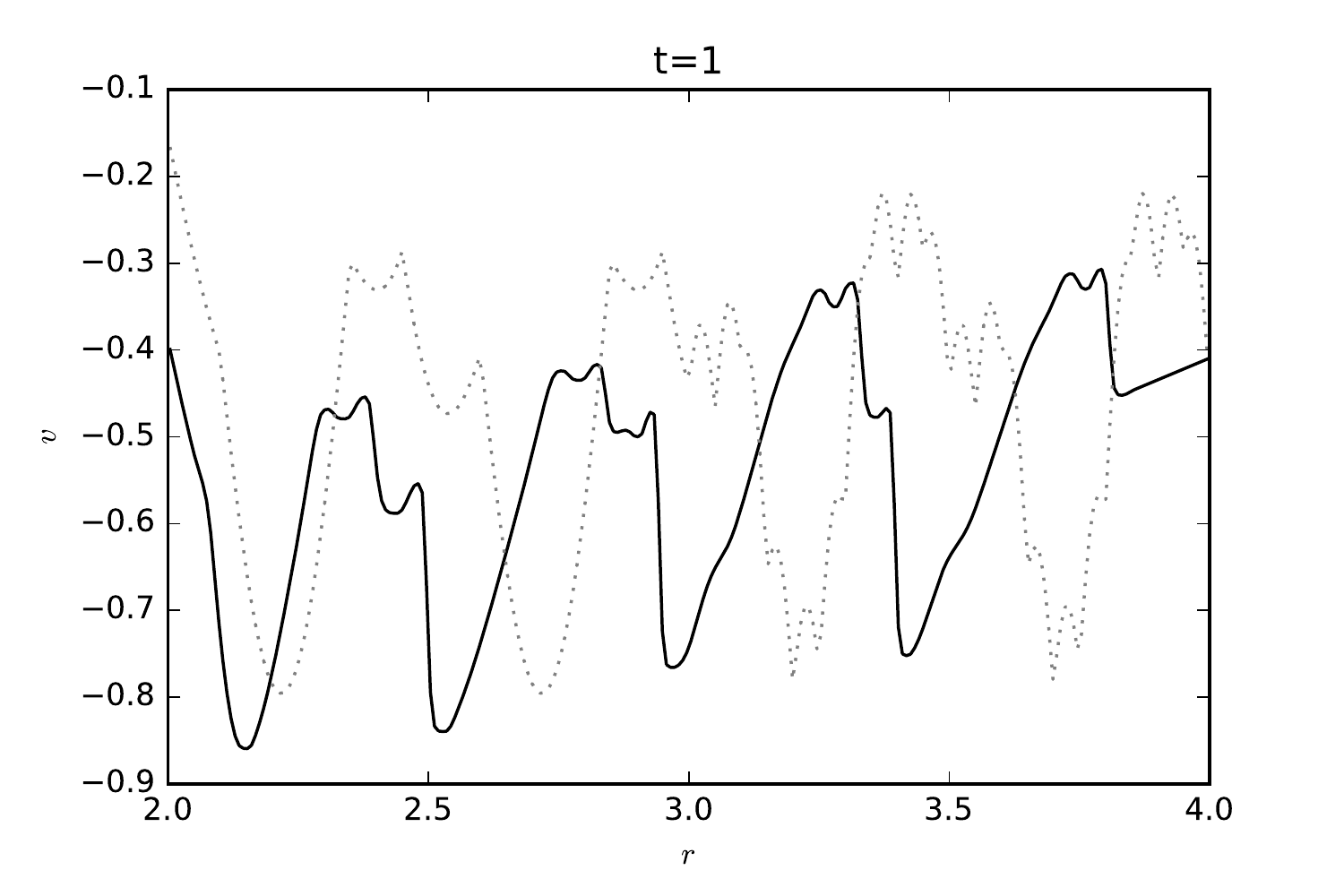,width= 2.5in} 
\end{minipage}
\hspace{0.1in}
\begin{minipage}[t]{0.3\linewidth}
\centering
\epsfig{figure=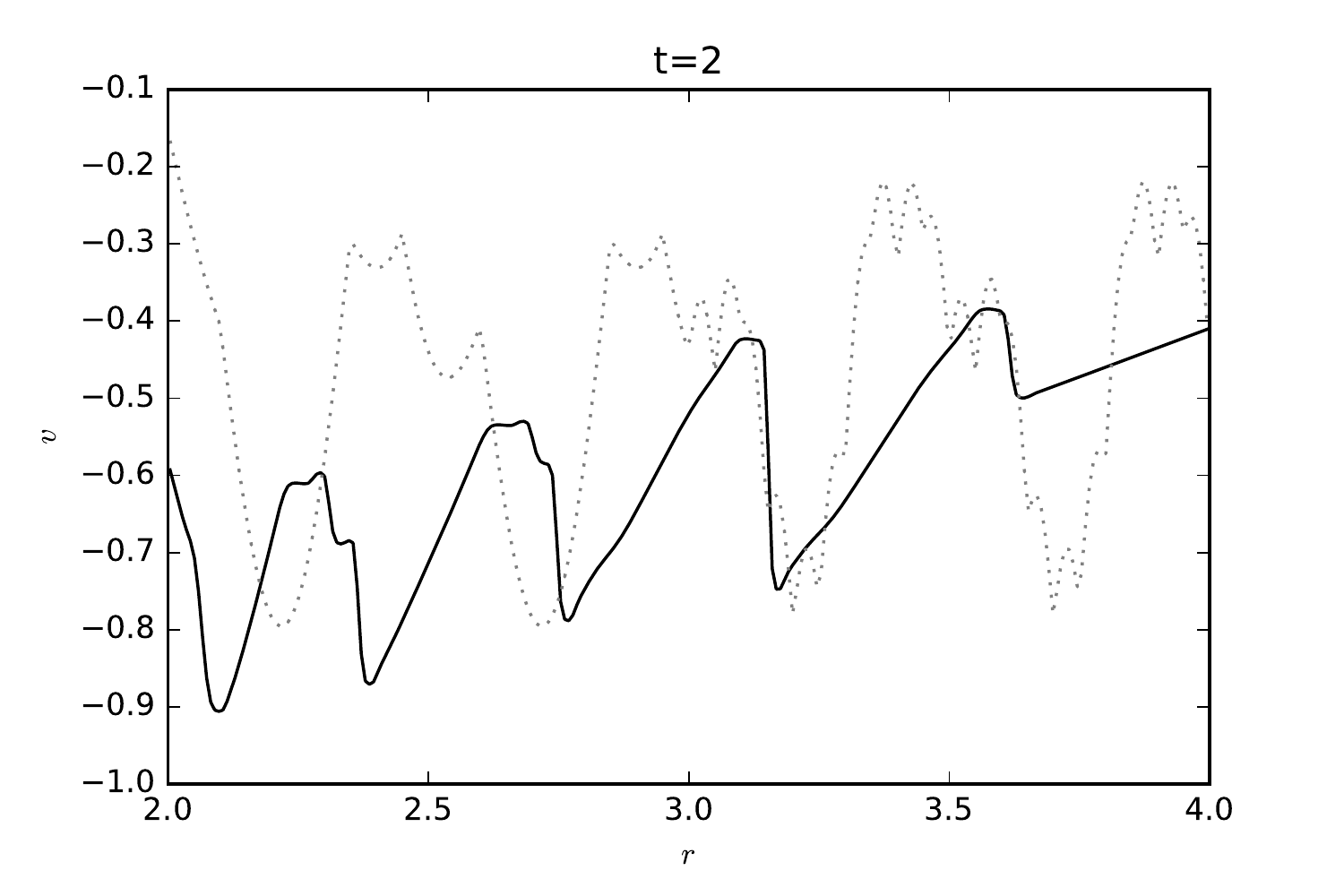,width=2.5in} 
\end{minipage}\\
\begin{minipage}[t]{0.3\linewidth}
\centering
\epsfig{figure=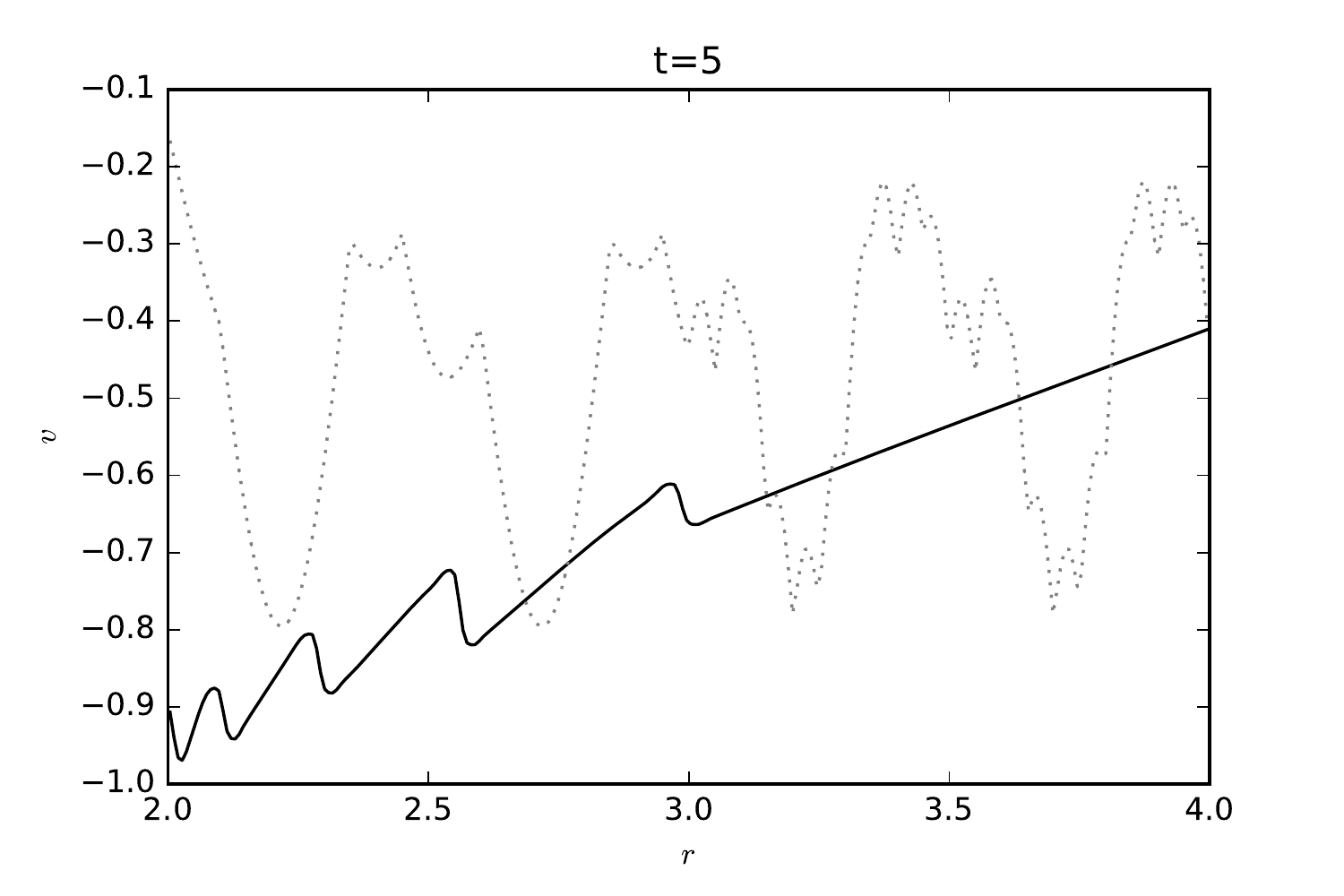,width=2.5in} 
\end{minipage}
\hspace{0.1in}
\begin{minipage}[t]{0.3\linewidth}
\centering
\epsfig{figure=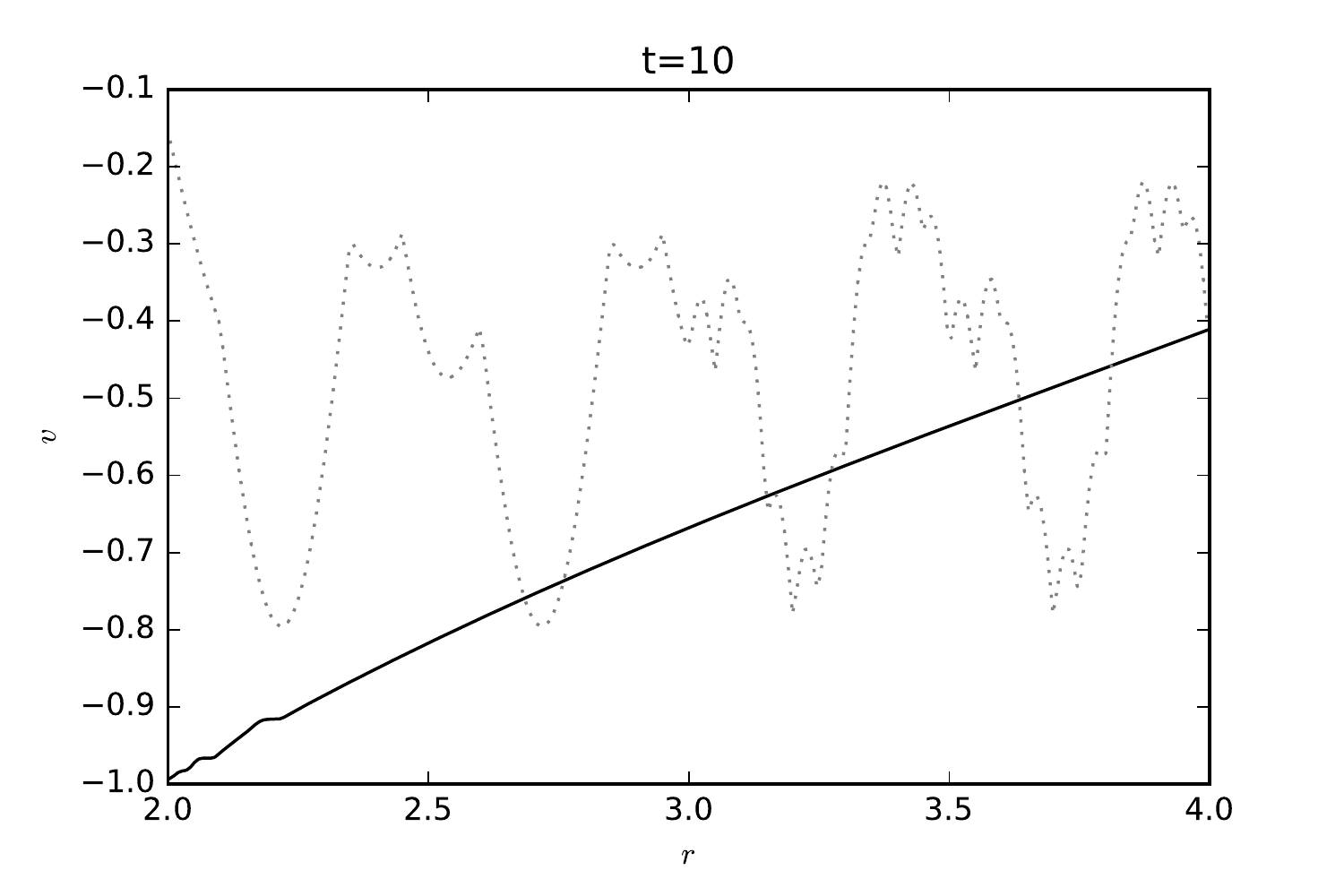,width= 2.5in} 
\end{minipage}
\hspace{0.1in}
\begin{minipage}[t]{0.3\linewidth}
\centering
\epsfig{figure=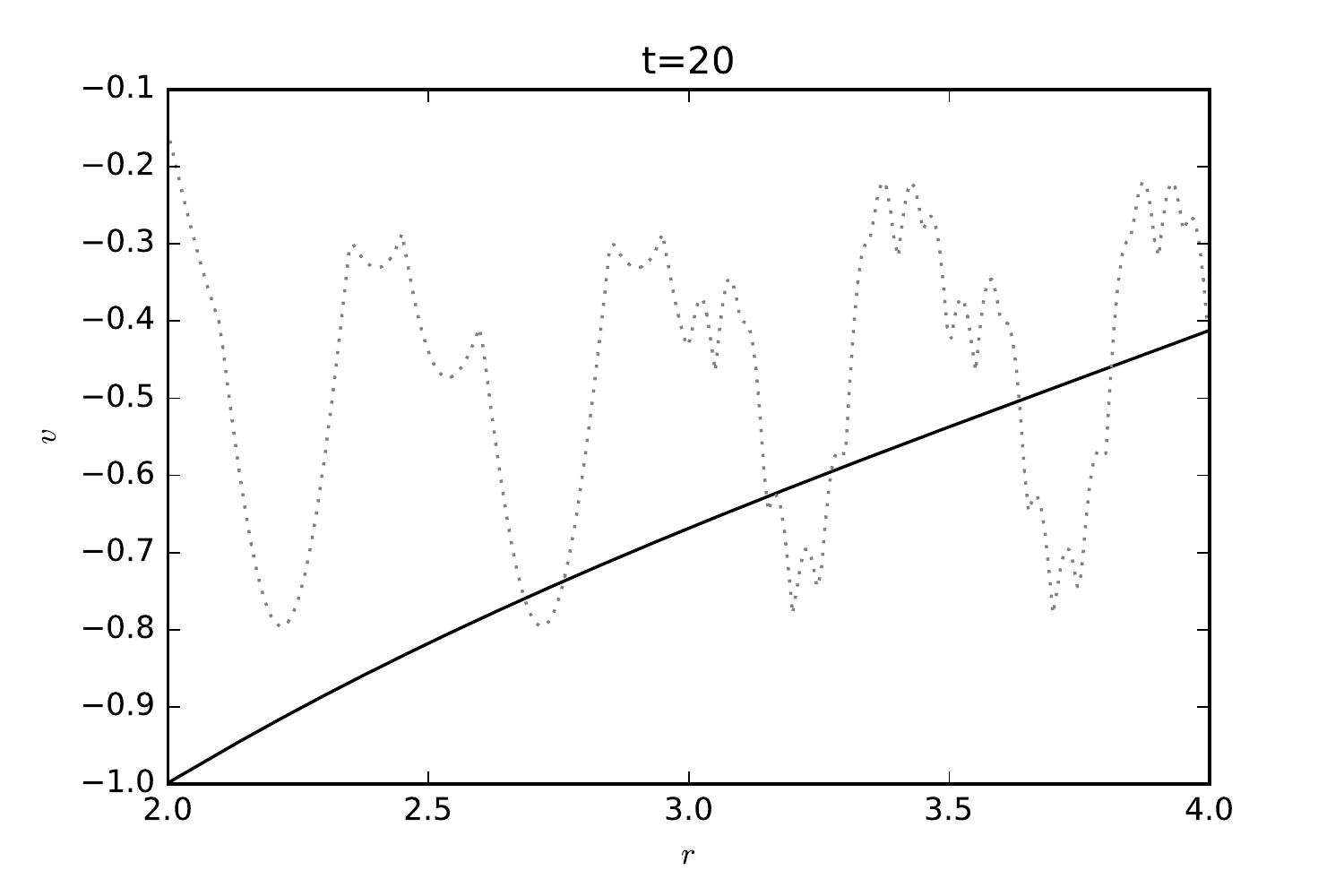,width=2.5in} 
\end{minipage}
\caption{Numerical solution with positive velocity at $r=2M$ and negative velocity at $r=+\infty$, using the Glimm scheme}
\label{FIG-83-1} 
\end{figure}


\begin{figure}[!htb] 

\centering 
\epsfig{figure=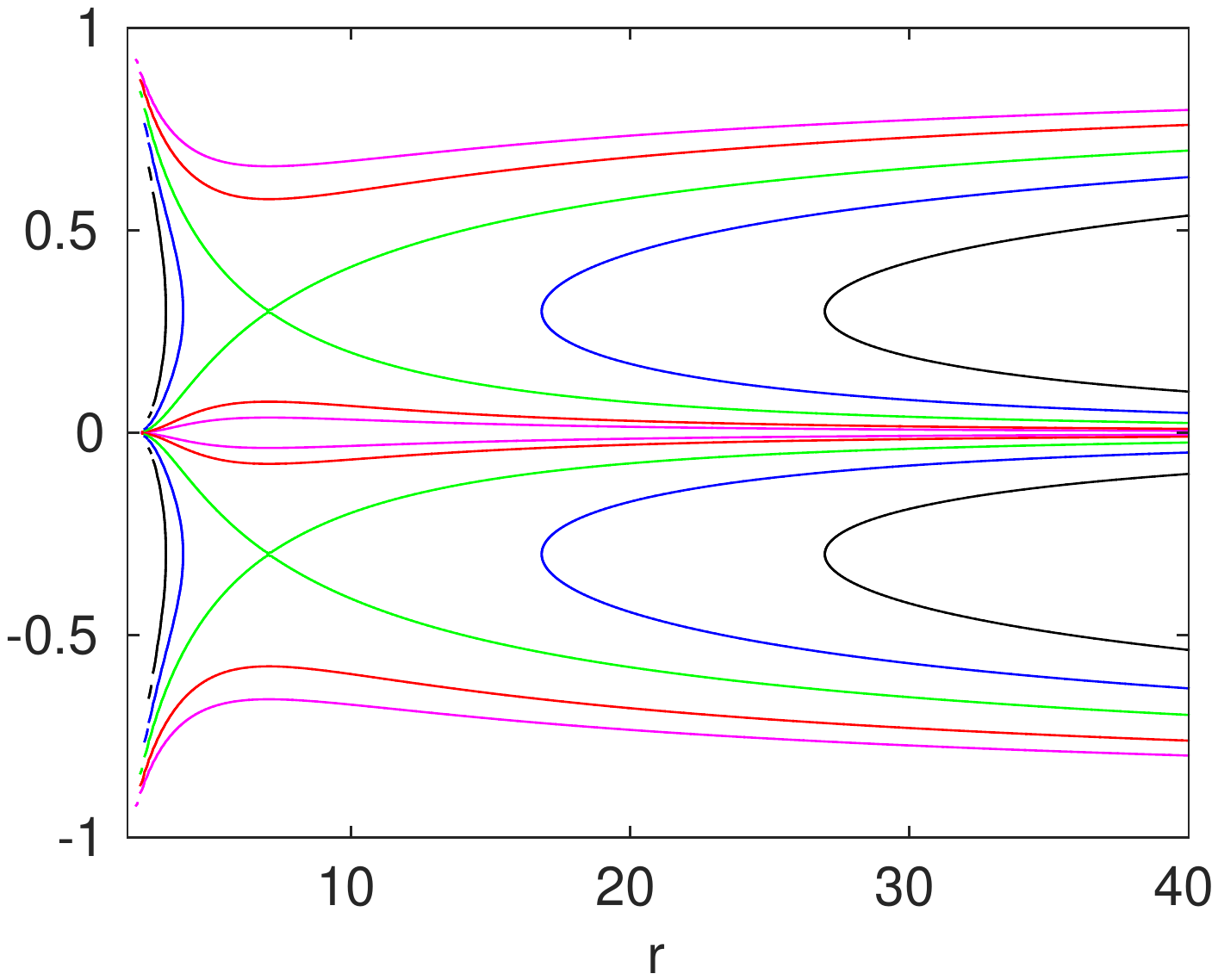,clip, trim=1.2in 3.2in 1.2in 3.2in,height=2.5in} 
\caption{Steady state solutions for the relativistic Euler model} 
\label{FG-30}
\end{figure}


\section{Overview of the theory for the relativistic Euler model} 
\label{Sec:8}

\paragraph{Continuous and discontinuous steady state solutions}

The steady solution to the relativistic Euler model on a Schwarzschild background background \eqref{Euler} is given by the following ordinary differential system: 
\bel{Euler-static}
\aligned
& \del_ r \Big(r(r-2M) {1  \over 1 -  v^2} \rho v \Big) = 0, 
\\
& \del_r  \Big((r-2M)^2 {v^2+ k^2 \over 1 -  v^2} \rho \Big)
= {M \over r} {(r-2M)  \over 1 - v^2}  \Big(3 \rho v^2+ 3 k^2 \rho -  \rho-  k^2 \rho v^2 \Big)  
  + {2k^2 \over r} (r-2M)^2  \rho, 
\endaligned
\ee 
Smooth steady state solutions to the relativistic Euler equation with given radius $r_0>2M$, density $\rho_0 >0$ and velocity $|v_0| < 1$ are given by 
\bel{steady-algebra}
\aligned 
& \sgn (v) (1-v^2) |v|^{2 k^2 \over 1- k^2} r^{4 k^2 \over 1- k^2} /  (1-2M/ r) = \sgn (v_0) (1-v_0^2) |v_0 |^{2 k^2 \over 1- k^2} r_0^{4 k^2 \over 1- k^2} / (1-2M/ r_0), \\  
&  r (r-2M)    \rho  {v   \over 1- v^2} = r _0   (r_0 -2M)   \rho_0   {v_0   \over 1- v_0^2}.
\endaligned 
\ee
We have 
\bel{derivatives}
\aligned 
 & {d \rho \over dr}=-{2 (r-M) \over r(r-2M)}\rho  -  {(1+   v^2) (1- k^2)   \over r  (r-2M)} \rho   \bigg({2  k^2 \over 1-  k^2}   (r-2M)-  M \bigg)\Big/ (v^2- k^2)  \\
 & {d v \over dr}=  v   {(1 - v^2) (1- k^2)   \over r  (r-2M)} \bigg({2  k^2 \over 1-  k^2}   (r-2M)-  M \bigg)\Big/ (v^2- k^2), 
\endaligned
\ee
We denote by the {\em critical steady state solution to the relativistic Euler model} \eqref{Euler} $(\rho, v)$ with its velocity $v=v (r)$ satisfying
\bel{critical-Euler}
{1 - \eps^2 v^2\over 1 -2M/r} (r^2 |v |)^ {2 \eps^2 k^2 \over 1- \eps^2 k^2}=  (1+ 3\eps^2 k^2) k^ {2 \eps^2 k^2 \over 1- \eps^2 k^2} \Big({1+ 3 \eps^2 k^2 \over 2 \eps^2 k^2}  M\Big)^ {4  \eps^2 k^2 \over 1- \eps^2 k^2}. 
\ee
Unlike the static Burgers model \eqref{static-Burgers},  steady state solution to the relativistic Euler model does not have an explicit form. We recall the following from \cite{PLF-SX-one}. 

\begin{theorem}[Smooth steady flows on a Schwarzschild background]
\label{steady-state} 
Let   $k \in [0, 1 ]$ be the sound speed and  $M> 0$ be mass of the blackhole and we consider the relativistic Euler model describing fluid flows on a Schwarzschild background  \eqref{Euler}. 
For any given any radius $r_0 >2M$, density $\rho_0 > 0$, and velocity $|v_0| < 1 $, there exists a smooth unique steady state solution $\rho=\rho(r), v=v(r)$, 
satisfying \eqref{steady-algebra} such that 
the initial condition $\rho(r_0) =\rho_0$ and $v(r_0) =v_0$ holds. 
Moreover, the velocity component satisfies that the signes of $v(r)$ and $|v(r)| -k$ do not change on the domain of definition. We have two different families of solutions: 
\bei
\item If there exists no point  at which the fluid flow is  sonic (referred to the {\em sonic point}), the smooth steady state solution is defined globally on the whole space interval outside of the blackhole $(2M, +\infty)$.

\item Otherwise, the smooth steady state solution cannot be extended once it reaches the sonic point. 
\eei
\end{theorem}

We now turn to  {\em steady shock of the relativistic Euler model} \eqref{Euler}, that is, two steady state solutions connected by a standing shock: 
\bel{steady-shock-Euler1}
 (\rho, v)= \begin{cases}
 (\rho_L, v_L)  (r), 
&  2M < r < r_0, 
 \\ (\rho_R, v_R) (r),
 & r>r_0, 
 \end{cases}
\ee
where $r_0>2M$ is a given radius and $(\rho_L, v_L)$, $ (\rho_R, v_R)$ two steady state solutions 
two steady state solutions satisfying  \eqref{steady-algebra}  such that 
\bel{steady-shock-Euler2}
v_R(r_0) ={k^2 \over v_L (r_0)}, \qquad  \rho_R(r_0) = {v_L(r_0)^2  -  k^4\over  k^2 \big(1 -  v_L(r_0)^2 \big)}   \rho_L (r_0), \qquad v_L (r_0) \in (-k, -k^2) \cup (k, 1). 
\ee
We denote by the {\em steady shock of the relativistic Euler model}  the function given by \eqref{steady-shock-Euler1},\eqref{steady-shock-Euler2} is  a solution to the static Euler equation \eqref{Euler-static} in the distributional sense, satisfying both  the Lax entropy inequality and the Rankine-Hugoniot jump condition.  
Observe that for a fixed radius $r_1 \neq r_0$ and  $(\rho_L, v_L)$, $ (\rho_R, v_R)$ satisfying \eqref{steady-shock-Euler1}, the following function is {\em not} a steady shock of the Euler model \eqref{Euler}: 
$$
(\rho, v)= \begin{cases}
 (\rho_L, v_L)  (r), & 2M <r < r_1, 
 \\ (\rho_R, v_R) (r), & r>r_1. 
 \end{cases}
$$


\paragraph{Generalized Riemann problem and Cauchy problem}

A generalized Riemann problem for the relativistic Euler system \eqref{Euler} is a Cauchy problem with initial data  given as
\bel{initial-steady-Euler}
(\rho_0, v_0) (r)=\begin{cases}
(\rho_L, v_L) (r) &2M  <r<r_0, 
\\ 
(\rho_R, v_R)(r) & r> r_0, 
\end{cases}
\ee
where $r= r_0$ is a fixed radius and $\rho_L=\rho_L(r), v_L=v_L (r), \rho_R= \rho_R(r), v_R= v_R (r)$ are two smooth steady state solutions  satisfying the static Euler equation \eqref{Euler-static}.  Referring to \cite{PLF-SX-one}, we can construct an approximate solver  $\widetilde U =(\widetilde \rho, \widetilde v)= (\widetilde \rho, \widetilde v)(t, r) $
 of  the generalized Riemann problem of the relativistic Euler  model \eqref{Euler} whose initial date is  \eqref{initial-steady-Euler} such that:
 \bei
 \item  $||\widetilde U (t, \cdot) -  U (t, \cdot)||_{L^1}= O (\Delta t^2) $ for any fixed $t>0$ where $U=(\rho, v)= (\rho, v) (t, r)$  satisfying \eqref{Euler},   \eqref{initial-steady-Euler} and $\Delta t $ is the time step in the construction. 
   \item $\widetilde U= (\widetilde \rho, \widetilde v)$ is accurate out of rarefaction fan regions. 
 \item $\widetilde U=  (\widetilde \rho, \widetilde v)$ (so does the accurate solution $U$) contains  at most three steady states: the two states given in the initial data $(\rho_L, v_L)$, $(\rho_R, \rho_R)$  and the uniquely defined intermediate $(\rho_M, v_M)$ connected by a 1-family wave (either 1-shock or 1-rarefaction) and a 2-family wave (either 2-shock or 2-rarefaction). 
 \eei

\begin{theorem}[The existence theory of the relativistic Euler model]
\label{Euler-ex-th}
Consider the Euler system describing fluid flows on a Schwarzschild geometry \eqref{Euler}.  For  any initial density $\rho_0=\rho_0(r) > 0$ and velocity  $|v_0| = |v_0(r)|< 1$ satisfying$$
TV_{[2M+\delta, +\infty)} \big(\ln \rho_0 \big) + TV_{[2M+\delta, +\infty)}  \Bigg(\ln {1- v_0 \over 1 + v_0} \Bigg) < + \infty, 
$$
where  $\delta> 0$ is a constant, there exists a weak solution $(\rho, v) =(\rho,v) (t,r)$defined on $(0, T)$ for any given $T>0$ and satisfying the prescribed initial data at the initial time and, with a constant $C$ independent of time, 
$$
\aligned
&\sup_{t \in [0, T]} \Bigg(
TV_{[2M+\delta, +\infty)} \big(\ln \rho(t,\cdot) \big) 
+ TV_{[2M+\delta, +\infty)} \Bigg(\ln {1- v (t, \cdot) \over 1 +  v(t, \cdot)} \Bigg) \Bigg) \\
& 
\leq TV_{[2M+\delta, +\infty)} \big(\ln \rho_0 \big) + TV_{[2M+\delta, +\infty)}  \Bigg(\ln {1- v_0 \over 1 + v_0} \Bigg)  e^{CT}.
\endaligned 
$$ 
\end{theorem}


\section{A finite volume method for the relativistic Euler model}
\label{Sec:9}

\paragraph{A semi-discretizenumerical scheme}

We consider the relativistic equation on a Schwarzschild background \eqref{Euler1} and we write 
\bel{Euler-form}
\del_ t U+ \del_ r\bigg(\Big(1-{2M \over r} \Big)F (U)\bigg)= S(r, U), 
\qquad 
U =\left (\aligned & U^0\\  & U^1\endaligned \right)= \left (
\aligned 
&  {1 + k^2  v^2 \over 1 - v^2}  \rho \\ 
&{1+k^2  \over 1 -  v^2} \rho v
\endaligned \right),   \qquad 
F(U) =\left (
\aligned 
&  {1+k^2  \over 1 -  v^2} \rho v \\ 
&    {v^2+ k^2  \over 1 -  v^2} \rho
\endaligned \right),  
\ee
and the source term 
$$
S(r, U) =\left (
\aligned 
& - {2\over r} (1-2M/r) {1+k^2  \over 1 -  v^2} \rho v  \\ 
&  {- 2 r+5M \over r^2}{v^2+ k^2  \over 1 - v^2}\rho    -{M \over r^2}  {1+  k^2  v^2 \over 1 -  v^2}  \rho+2{r-2M  \over r^2}k^2 \rho
\endaligned \right).  
$$
We can compute 
\bel{Duf}
D_U F (U)=  \begin{bmatrix}
    0 & 1 \\
 (- v^2 +k^2)\big/  (1- k^2 v^2) & 2  (1-k^2)  v  \big/  (1- k^2v^2)
  \end{bmatrix}, 
\ee 
which gives the two eigenvalues $\mu_\mp =\Big (1-{2M \over r}\Big){v   \mp k \over 1 \mp  k^2 v}  $.  We also have
$ v =  {1+ k^2  -  \sqrt {(1+ k^2)^2 - 4k^2 \big({U^1 \over U^0}\big)^2}\over 2k^2 {U^1 \over U^0}} \in (-1, 1)$
and 
$\rho ={U^1 (1- v^2)\over v (1+ k^2)}$. 
Again, we take  $\Delta t $,  $\Delta r$ 
as the mesh lengths in time and in space respectively with the CFL condition 
\bel{Cfl-2}
{\Delta t \over \Delta x}   \max \big(|\mu_- |, |\mu_+  |\big) \leq {1\over 2}, 
\ee
where $\mu_\mp $ are eigenvalues.   As is done earlier we write 
$t_n = n \Delta t$ and $r_j=2M+  j\Delta r$, 
and we denote the mesh points by $(t_n, r _j)$, $n\geq 0$, $j\geq 0$. We set alsoy $\rho (t_n, r_j)= \rho_ j^n,  v(t_n, r_j)= v_j^n$ and $U (t_n, r_j)= U_n^j$ where $U = U (t,r)$ is given by \eqref{Euler-form}. 

We search for the approximations $U_ j^n ={1\over \Delta r}  \int_  {r_{j-1/2}}^{r_{j+1/2}} U (t_n, r)dr$and 
$ S_j^n = {1\over \Delta r}  \int_  {r_{j-1/2}}^{r_{j+1/2}} S(t,_n  r)dr$  
and  introduce the following finite volume method: 
\bel{Euler-finite}
U_j^{n+1} =  U_j^n- {\Delta t \over \Delta r} (F_{j+1/2}^n  -  F_ {j - 1/2}^n)  +  \Delta t  S_ j^n, 
\ee
where the numerical flux is 
\bel{Euler-flux}
\aligned 
& F_{j- 1/2}^n =\mathcal F_  l (r_{j-1/2}, U_ {j-1}^n, U_j^n)=\Big(1-{2M \over r_{j-1/2}}\Big) \mathcal  F(U_ {j-1/2-}^n, U_ {j -1/2+}^n), \\
\endaligned 
\ee
and  $ U_{j + 1/2 \pm}, U_{j - 1/2 \pm} $ are determined in the forthcoming subsection and  
\bel{L-F-flux}
\mathcal F(U_L, U_R)= {F (U_L) +F (U_R) \over 2} -{1 \over \lambda} {U_ R - U_L \over 2}, 
\ee 
where $\lambda= \Delta r / \Delta t $.  Here, $F $ is the exact flux \eqref{Euler-form} and $S_j^n$ is the discretized source to be determined later. 


\paragraph{Taking the curved geometry into account}

We now give the states $U_{j+ 1/2 \pm}, U_{j-1/2\pm}$ and the discretized source term $S_j^n $ which take into account the geometry of the Schwarzschild spacetime. For a steady state solution $U= U (r)$, the equation 
$\del_ r \Big((1-2M /  r) F(U)\Big)= S(r,U)$ holds,  
where $U, F$ and the source term $S$ are given by \eqref{Euler-form}, or equivalently, the solution $(\rho, v)$ satisfies the static Euler equation \eqref{Euler-static}.  First of all,  we would like to approximate the solution in each cell $(r_ {j-1/2}, r_ {j+1/2})$ by steady state solutions. Hence we expect the following algebraic relations following from the calculations: 
\bel{Euler-states}
\aligned 
& \big(1-{v_ {j+1/2-}^n}^2\big) {v_ {j+1/2-}^n}^{2 k^2 \over 1- k^2}r_{j+1/2}^{4 k^2 \over 1- k^2} /  (1-2M/  r_{j+1/2}) =   \big(1-{v_j^n}^2   \big) {v_j^n}^{2 k^2 \over 1- k^2} r_j^{4 k^2 \over 1- k^2} / (1-2M/   r_j), 
\\  
&   r_{j+1/2} (r_{j+1/2}-2M)   \rho_{j+1/2 -}^n  {v_ {j+1/2-}^n  \over 1- {v_ {j+1/2-}^n}^2}} =  r_j  (r_j  -2M)  {\rho _ j^n  {v_ j^n  \over 1- {v_j^n}^2}, \\
&  \big(1-{v_ {j+1/2+}^n}^2  \big) {v_ {j+1/2+}^n}^{2 k^2 \over 1- k^2} r_{j+1/2}^{4 k^2 \over 1- k^2} /  (1-2M/  r_{j+1/2})  =   \big(1-{v_{j+1}^n}^2 \big) {v_{j+1}^n}^{2 k^2 \over 1- k^2} r_ {j+1}^{4 k^2 \over 1- k^2} / (1-2M/   r_  {j+1}), 
\\
&  r_{j+1/2} (r_{j+1/2}-2M)   \rho_{j+1/2 +}^n  {v_ {j+1/2+}^n  \over 1- {v_ {j+1/2+}^n}^2}} =  {r_j+1} (r_{j+1}  -2M)  {\rho _ {j+1}^n  {v_ {j+1}^n  \over 1- {v_{j+1}^n}^2}. 
\endaligned  
\ee
However, since a steady state solution might not be defined globally on $(2M, +\infty)$, it is possible that \eqref{Euler-states} does not permits a solution.  We simply define $(\rho_ {j+1/2-}^n, v_ {j+1/2-}^n)=(\rho_j^n, v_j^n)  $ if the first two equations in \eqref{Euler-states}  do not have a solution and $(\rho_ {j+1/2-}^n, v_ {j+1/2-}^n)=(\rho_{j+1}^n, v_{j+1}^n)  $ if  the last two equations in \eqref{Euler-states}  do not have a solution. 
Integrating  \eqref{Euler-finite} by parts, we obtain the approximate source term: 
\bel{Euler-source}
\aligned 
 S_j^n = {1\over \Delta r}  \int_  {r_{j-1/2}}^{r_{j+1/2}} S(t_n, r)dr =&   {1\over \Delta r}  \int_  {r_{j-1/2}}^{r_{j+1/2}} \del_ r \Big((1-2M /  r) F\big(U(t_n,  r)\big)\Big) dr \\
 =& {1\over \Delta r} \bigg((1-2M /  r_{j+1/2}) F(U_{j+1/2-}^n) -  (1-2M /  r_{j- 1/2 +}) F(U_{j-1/2+}^n)\bigg), 
\endaligned 
\ee
where $U_ {j+1/2-}^n, U_{j-1/2+}^n$ are two states determined by \eqref{Euler-states} and $F(\cdot)$ the accurate flux of the Euler model given by \eqref{Euler-form}. 
We then  have the following result. 

\begin{theorem}
The finite volume scheme proposed for the relativistic Euler equation on a Schwarzschild background \eqref{Euler1}  satisfies: 
\bei
\item The scheme preserves the steady state solution to the Euler equation \eqref{Euler-static}. 
\item The scheme is consistent, that is,  for an exact  solution $U= U(t,r)$ and  the states $U_L, U_R \to U$, $r_L, r_R \to r$,  we have 
\bel{consistent-Euler}
 \mathcal F_ r  (r_R, U_L, U_R) - \mathcal F_ l (r_L, U_L, U_R) =  S(r, U)(r_R- r_L)+ O \big((r_R- r_L)^2\big), 
 \ee 
where $\mathcal F_l, \mathcal F_r$ are numerical fluxes given by \eqref{Euler-flux} and $S(r, U)$  is the source term given by  \eqref{Euler-form}. 
\item The scheme has  second-order accuracy in space and first-order accuracy in time. 
\eei
\end{theorem}

\begin{proof}
For a steady state given by \eqref{Euler-static}, we have $ U _ {j+1/2+}= U_{j+1/2-}  $. Hence,  the flux of the finite volume method \eqref{Euler-flux} satisfies $F_{j+1/2}= (1-2M /r_{j+1/2})F(U _ {j+1/2+})=  (1-2M /r_{j+1/2})F(U_{j+1/2-})$, which gives: 
$$
{1 \over \Delta r} (F_{j+1/2}^n  - F_ {j - 1/2}^n) =   (1-2M /r_{j+1/2})F(U_{j+1/2-})-   (1-2M /r_{j-1/2})F(U_{j-1/2+})=  S_j^n. 
$$
Therefore, the scheme preserves the steady state solutions.  Next, according to \eqref{Euler-states} and \eqref{Euler-source}, there exist four states $U_L^l, U_R^l, U_L^r,U_R^r $ such that 
$$
\aligned 
&  \mathcal F_ r  (r_R, U_L, U_R) - \mathcal F_ l (r_L, U_L, U_R)=  (1-2M/ r_ R) \mathcal F(U_L^r,U_R^r)-  (1-2M/ r_ L)\mathcal F (U_L^l,U_R^l)\\ 
 =&  \big(1-2M / r + 2M /r^2 (r_R-r)+ O (r_R-r)\big) \big(\mathcal F(U, U)+ \del_ 1 \mathcal F(U, U)(U_R-U) + o(U_R-U)\big)\\
 & -  \big(1-2M / r + 2M /r^2 (r_L -r)+  O (r_L -r)\big) \big(\mathcal F(U, U)+ \del_2  \mathcal F(U, U)(U_L-U) + o(U_L -U)\big). 
\endaligned 
$$
By \eqref{Euler-states}, $U_R- U_L= O(r_R-r_L) S (r, U)$. Moreover, since $U= U (t, r)$ is accurate, we have $\mathcal F(U, U)= F(U)$ and $\del_ 1 \mathcal F(U, U)= \del_ 2 \mathcal F(U, U)=  \del_ U F (U)$.  Therefore, 
$$
\aligned 
& \mathcal F_ r  (r_R, U_L, U_R) - \mathcal F_ l (r_L, U_L, U_R)={2M \over r^2} (r_R- r_L) F (U) + (1-2M/ r) \del_ U F (U)
(U_R- U_L)+ O \big((r_R- r_L)^2\big)\\
= & \del_ r \big((1-2M/r) F (U)\big)(r_R- r_L) +  o (r_R-r_L)=S(r, U)(r_R- r_L)+ O \big((r_R- r_L)^2\big). 
\endaligned 
 $$ 
 
Next, a Taylor expansion with respect to time yields us
$U _ j^{n+1}  =  U _ j^n + \del_ t U _ j^n \Delta t +  \del_ {tt}^2 U _ j^n \Delta t^2 + o (\Delta t^2)$. 
Recall that our scheme gives 
$$
\aligned  
 U _ j^{n+1} & = U _ j^n  -  {\Delta t \over \Delta r}  \big((1-2M/r _{j+1/2}) F_{j+1/2}^n  -(1-2M/r _{j-1/2}) F_ {j - 1/2}^n   -   \Delta r  S_ j^n). 
\\
& = U _ j^n - {1\over \lambda}  \bigg((1-2M/r _{j+1/2}) \Big({F(U_{j+1/2+}) - F(U_{j+ 1/2-})\over 2}-{1 \over  \lambda}   {U_{j+1/2+} - U_{j+ 1/2-}\over 2}\Big)\\ 
& \quad + (1-2M/r _{j-1/2}) \Big({F(U_{j-1/2+})-  F(U_{j-1/2-})\over 2}+ {1 \over  \lambda}   {U_{j-1/2+} - U_{j- 1/2-}\over 2}\Big)\bigg).
\endaligned
$$ 
According our construction, we have 
$$
\aligned 
 & \Big(1-{2M\over r _{j+1/2}}\Big) \Big(F(U_{j+1/2+}) - F(U_{j+ 1/2-}) \Big)
= \Big(1-{2M\over r _{j+1}}  \Big)F(U_{j+1}^n)-  \Big(1-{2M \over r_j}  \Big)F(U_j^n) - \int_ {r_j}^{r_{j+1}} S(r, U(t_n, r)) dr. 
\endaligned 
$$
A Taylor expansion to $\Delta r $ gives us $U_{j+1/2+} - U_{j+ 1/2-}=  O (\Delta  r^3)$ and 
$$
\aligned 
&\Big(1-{2M\over r _{j\pm 1}}\Big)= 1-{2M\over r _j}\pm  {2 M  \over r _ j^2} \Delta r -{2 M \over r_ j^3} \Delta r^2  + O 
 (\Delta r^3), \\
 & F (U_{j\pm 1}^n)= F (U_ j^n)+ \del_ U F (U_ j^n)\Big(\pm \del_ r U_ j^n \Delta r + {1\over 2}\del _{rr}^2  U_ j^n \Delta r^2  \Big)+{1\over 2}(\del_ r U_ j^n)^T \del_ {UU}^2 F(U_ j^n) \del_ r U_j^n \Delta r^2 + O (\Delta r^3), \\ 
 &  \int_ {r_j}^{r_{j+1}} S(r, U(t_n, r)) dr=  S(r_j,U_ j^n)\Delta r + \del_ r S(r_j,,U_ j^n) \Delta r^2 + O (\Delta r^3).
\endaligned
$$
Hence we conclude that 
$\del_ t U _ j^n+ \del_ r \Big((1-2M/ r_ j) F (U_j^n)\Big) - S (r_j, U_j^n)+ O (\Delta t + \Delta r^2)=0$.
\end{proof}


\paragraph{Numerical steady state solution}

Recall that the steady state solution to the relativistic Euler model is given by a static Euler system  \eqref{Euler-static}. Hence, if $U = U (t, r)$ is a steady state solution, it trivially satisfies 
$\int \big|\del_ r F \big((1-2M/r)U\big) - S(r, U)\big|dr =0$, 
where $F= (F^0, F^1)^T$ is the flux and $S= (S^0, S^1)^T$ the source term given by \eqref{Euler-form}. In order to describe the steady state solution numerically, we define the total variation in time: 
\bel{TV-euler}
\aligned 
E^n :=E(t_n)= \sum _ j    \sum _ {i=0, 1} & \bigg|  (1- 2M / r_{j+1/2})\big(F^ i  (U_  {j+1/2+}^n) -    F^i    (U_{j-1/2-}^n)\big)
\\
& - (1-2M /r _{j-1/2})\big(F^ i  (U_ {j-1/2 +}^n)- F^ i  (U_ {j-1/2 -}^n \big)\bigg|. 
\endaligned 
\ee
From our former construction, we have the following result. 

\begin{lemma}
\label{En}
If $U= (t, r)$ is a numerical solution to the relativistic Euler model constructed  by \eqref{Euler-finite}- \eqref{Euler-source}, then $U$ is a steady state solution for $t\geq T$ where $T>0$ is a finite time if and only if there exists a $N<+\infty$ such that for all $n>N$, the total variation  $E^n\equiv0$. 
\end{lemma}


\section{Numerical experiments for the relativistic Euler model}
\label{Sec:10}

\paragraph{Nonlinear stability of steady state solutions}

Before studying the stability of steady state solutions, we check that our scheme preserves smooth steady state solutions to the relativistic Euler model \eqref{Euler1}. Recall that $r>2M$ with $M=1$ being the blackhole mass.   We work on the space interval $(r_{\min,}, r_{\max})$ with $r_{\min}= 2M=2$ and $r_{\max}=10$ and we take $500$ points to discretize this interval. We consider the evolution of two steady state solutions satisfying the algebraic relation \eqref{steady-algebra} of the Euler model with the density $\rho(10)=1.0$, the velocity $v(10)= 0.6$ and  the density $\rho(10)=1.0$, the velocity $v(10)= -0.8$ respectively.   We also provides the evolution of a steady state shock.

\begin{figure}[!htb] 
\centering 
\begin{minipage}[t]{0.3\linewidth}
\centering
\epsfig{figure=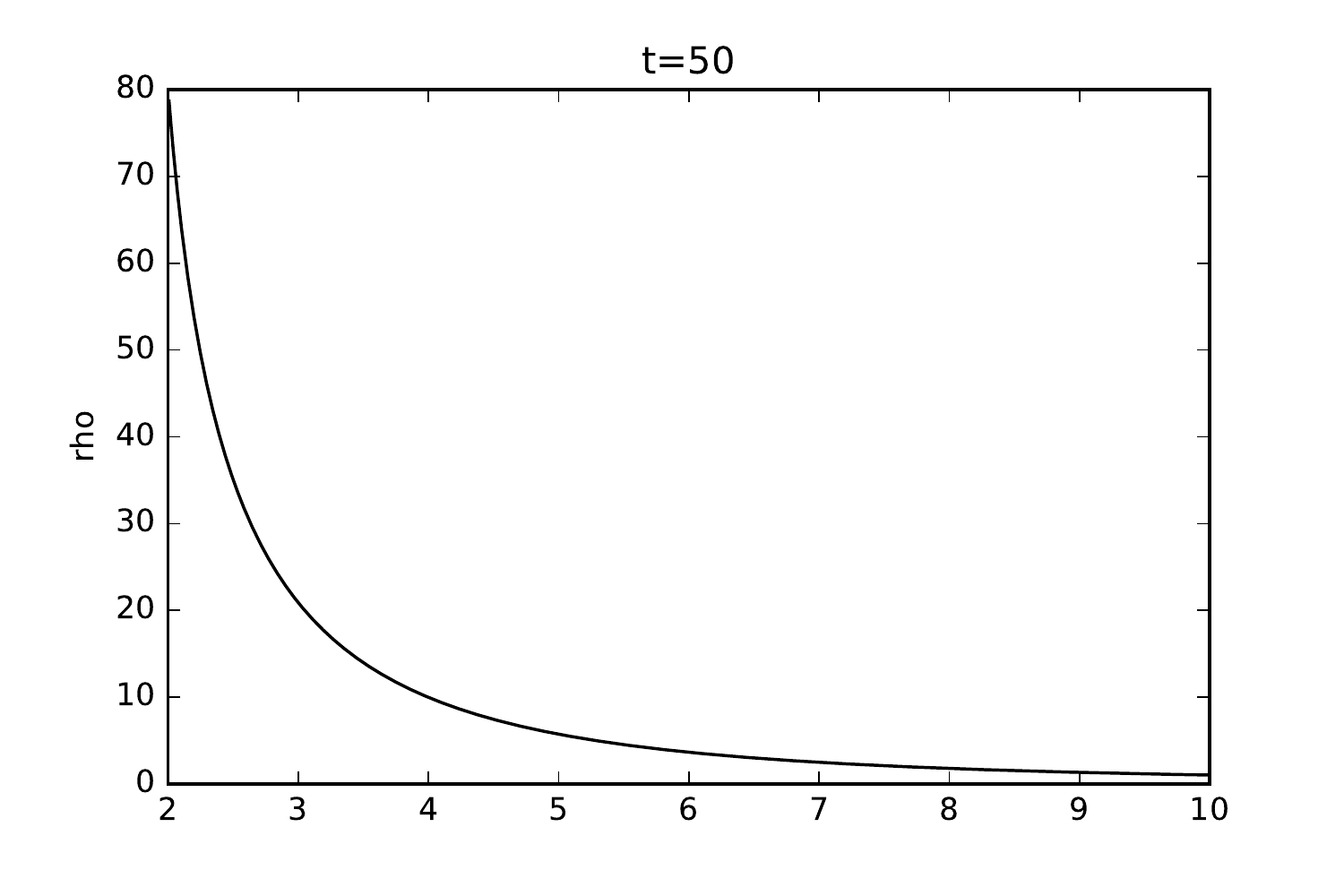,width = 2.5in} 
\end{minipage}
\hspace{0.5in}
\begin{minipage}[t]{0.3\linewidth}
\centering
\epsfig{figure=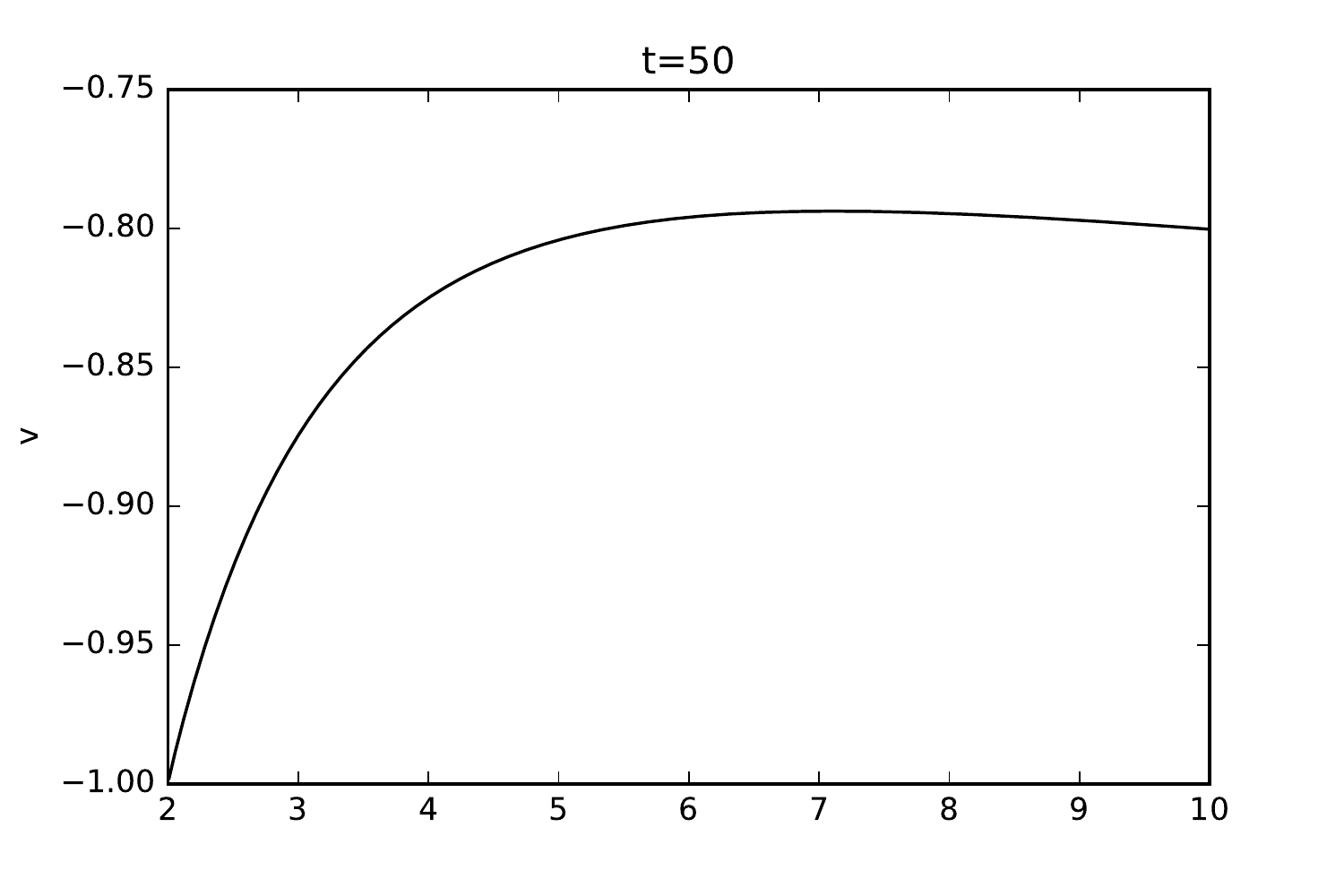,width = 2.5in}
\end{minipage}\\
\begin{minipage}[t]{0.3\linewidth}
\centering
\epsfig{figure=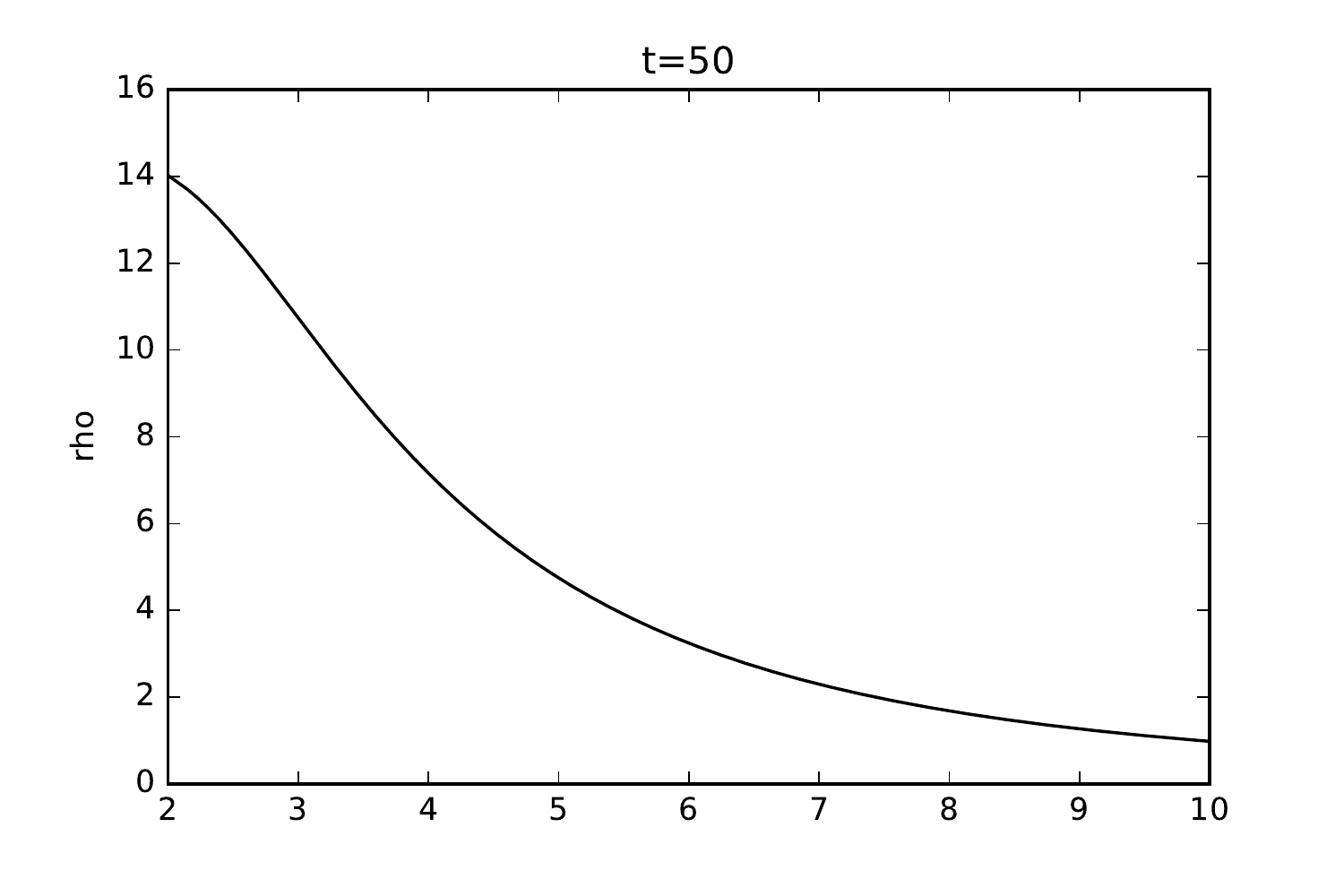,width = 2.5in} 
\end{minipage}
\hspace{0.5in}
\begin{minipage}[t]{0.3\linewidth}
\centering
\epsfig{figure=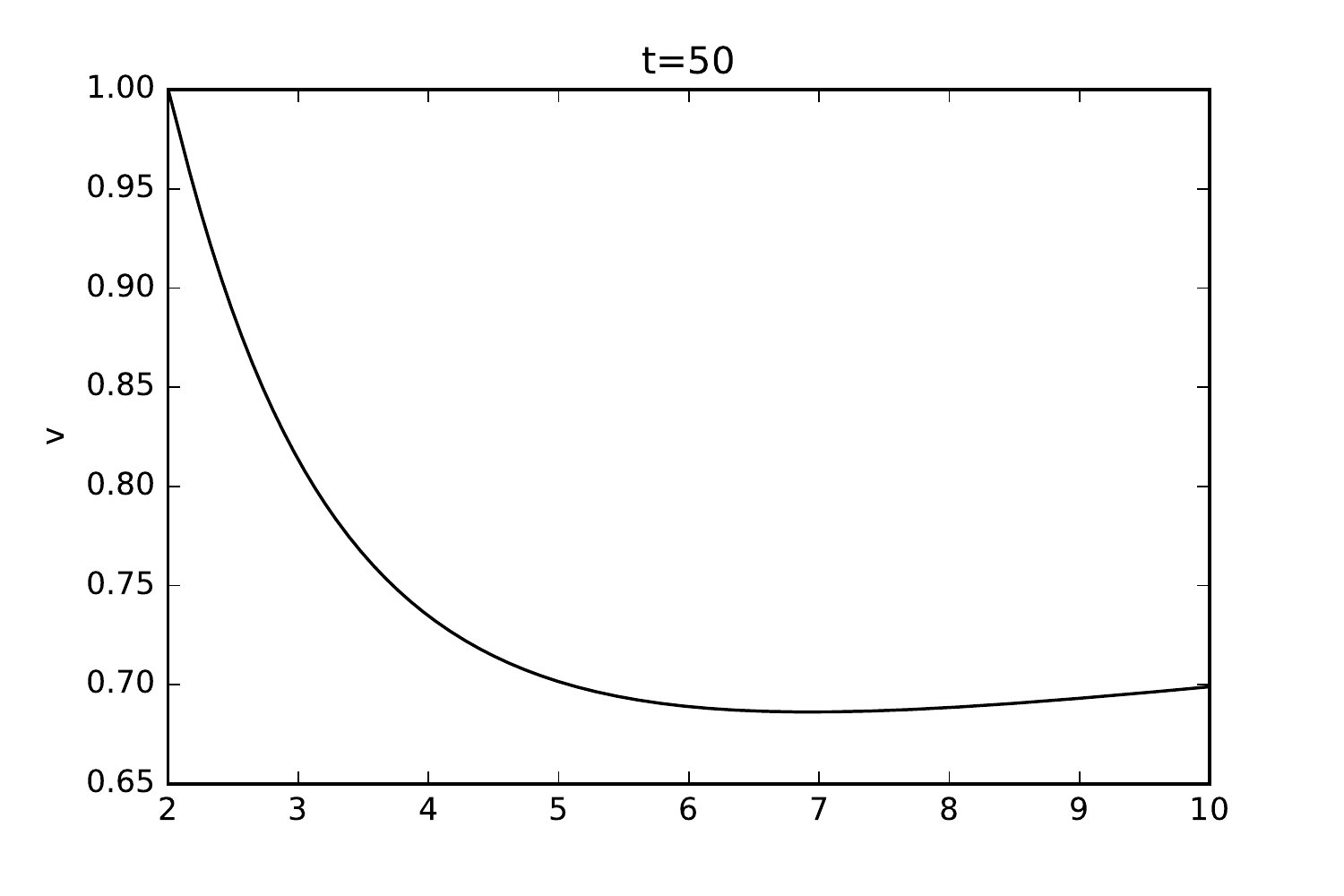,width = 2.5in}
\end{minipage}
\caption{Evolution of steady state solutions, plotted at time $t=50$} 
\end{figure}

\begin{figure}[!htb] 
\centering
\begin{minipage}[t]{0.3\linewidth}
\centering
\epsfig{figure=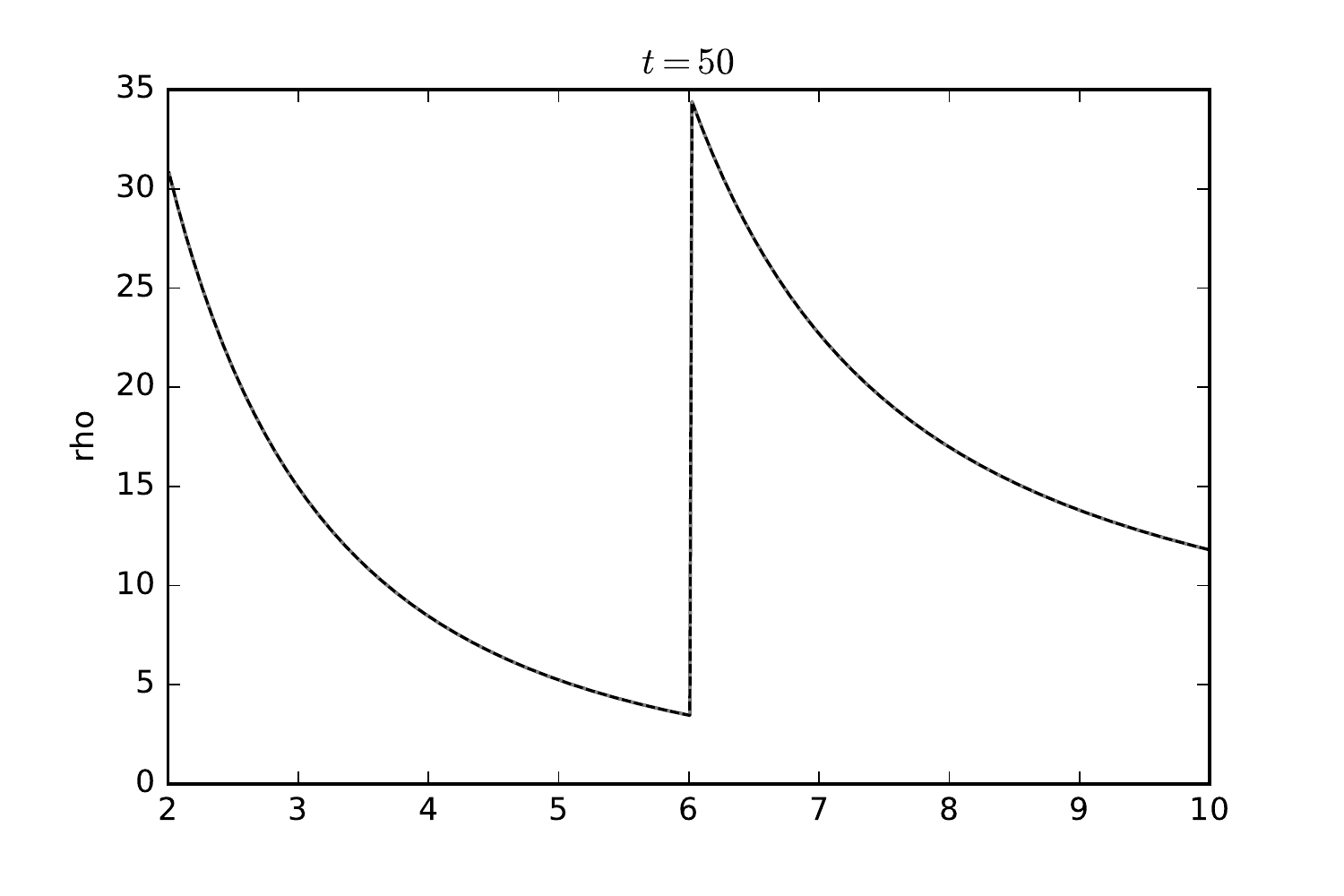,width = 2.5in} 
\end{minipage}
\hspace{0.5in}
\begin{minipage}[t]{0.3\linewidth}
\centering
\epsfig{figure=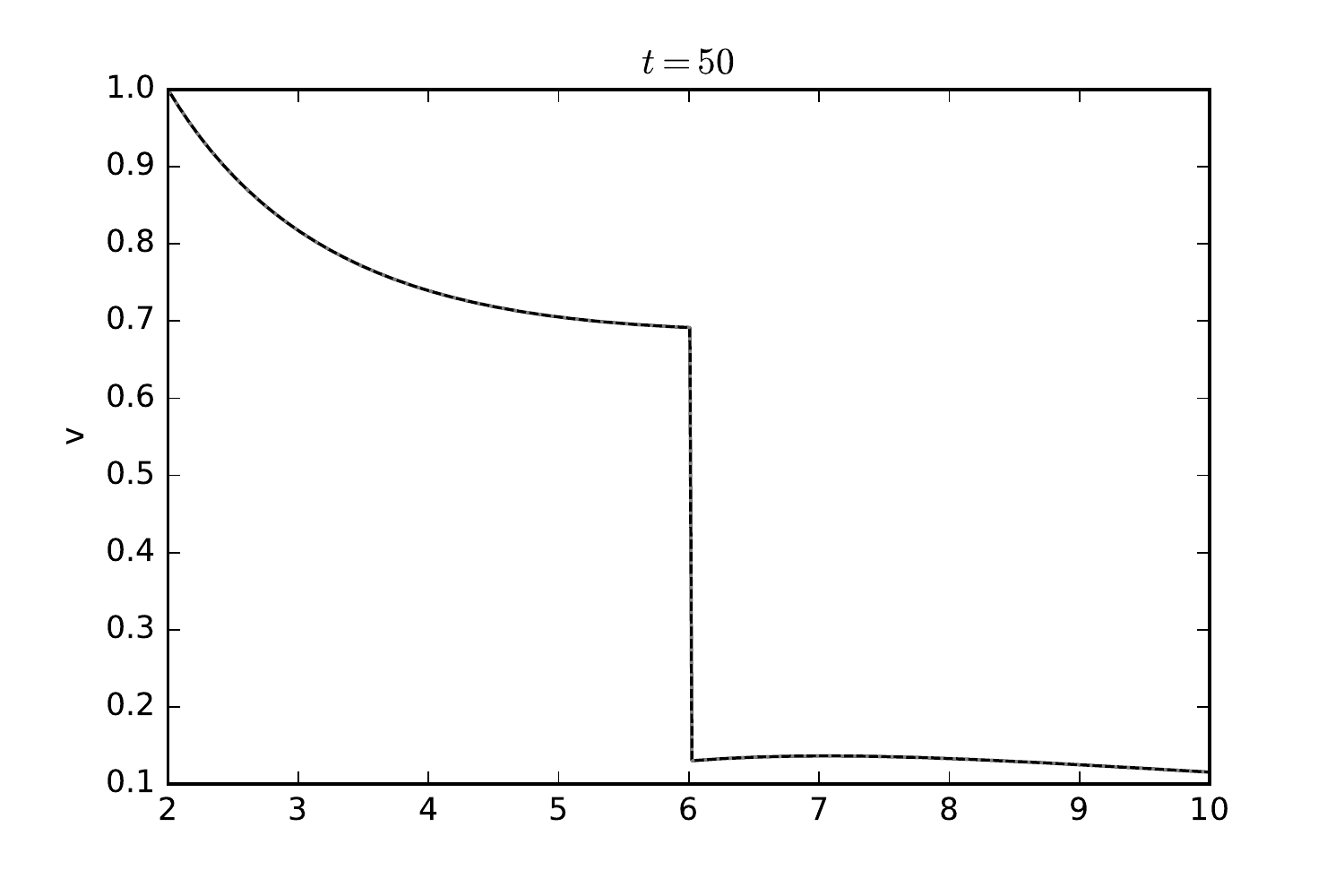,width = 2.5in}
\end{minipage}
\caption{Evolution of a steady shock plotted at time $t=50$} 
\end{figure}


\paragraph{Propagation of discontinuities}

Refering to \cite{PLF-SX-one}, we recall that there exists a solution to the generalized Riemann problem \eqref{Euler}, \eqref{initial-steady-Euler} consisting of at most three steady state solutions.  Figures~\ref{FIG-90}, \ref{FIG-91} show the evolution of two generalized Riemann problem with an initial discontinuity. 
%
Furthermore, we are now interested in the late-time behavior of solutions whose initial data is steady state solution perturbed by a compactly supported solution. Numerical tests lead us to the following result. 

\begin{conclusion}[Stability of smooth steady state solutions to the Euler model]
Let $(\rho_*, v_*) =(\rho_*, v_*)(r)$, $r>2M$ be a smooth steady state solution satisfying the static Euler equation \eqref{Euler-static} and $(\rho_0, v_0)= (\rho_0, v_0) (r)= (\rho_*, v_*)(r)+  (\delta_ {\rho}, \delta_v)(r)$ where $ (\delta_ {\rho}, \delta_v)=  (\delta_ {\rho}, \delta_v)(r)$ is a function with compact support, then  the solution to the relativistic Euler equation on a Schwarzschild background \eqref{Euler1} denoted by $(\rho, v)=(\rho, v)(t, r)$ satisfies that $(\rho, v)(t, \cdot) = (\rho_*, v_*)$ for all $t>t_0$ where $t_0>0$ is a finite time. Numerical experiments show that there exists a finite time $t_0>0$ such that:
\bei

\item If $\int \delta_ {\rho} (r)dr + \int  \delta_v (r)dr = 0$, $(\rho, v) (t, r)= (\rho_*, v_*)(r)$ for all $t>t_0$.

\item If $\int \delta_ {\rho} (r)dr + \int  \delta_v (r)dr \neq 0$, then there exists a time $t_0 >0$ such that $(\rho, v) (t, r)= (\rho_{**}, v_{**})(r)$ for all $t>t_0$ where $(\rho_{**}, v_{**})$ is a steady state solution to the Euler model and   $(\rho_{**}, v_{**})\neq (\rho_*, v_*)$.  

\eei 
\end{conclusion}
 
We observe the phenomena described in Conjecture~\ref{con2} in Figures~\ref{FIG-92} and \ref{FIG-93}.  To check that the numerical solutions in  Figures~\ref{FIG-92}, ~\ref{FIG-93} converge to a steady state solution, we refer to Lemma~\ref{En} and calculate the total variation at each time step. Figure~\ref{FIG-94} shows that these solutions are eventually steady state solutions. 
The steady shock given by \eqref{steady-shock-Euler1} and \eqref{steady-shock-Euler2} is a weak solution satisfying the static Euler equation \eqref{Euler-static}. We are also interested in the behavior of steady shocks with perturbations. We summarize our results as follows; see Figure~\ref{FIG-95}. 

\begin{conclusion}
Consider a steady shock $(\rho_*, v_*) =(\rho_*, v_*)(r)$, $r>2M$ given by  \eqref{steady-shock-Euler1}, \eqref{steady-shock-Euler2}  whose point of discontinuity is at $r=r_*$ and we give the initial data $(\rho_0, v_0)= (\rho_0, v_0) (r)= (\rho_*, v_*)(r)+  (\delta_ {\rho}, \delta_v)(r)$ with  $ (\delta_ {\rho}, \delta_v)=  (\delta_ {\rho}, \delta_v)(r)$ a compactly supported function, then there exists a finite time $t>t_0$ such that  for all $t>t_0$, the solution $ (\rho, v)(t, \cdot)=(\rho_{**}, v_{**}) $ where $(\rho_{**}, v_{**})$ is a steady state shock whose point of discontinuity is at $r=r_{**}$ with $r_{**}\neq r_*$. 
\end{conclusion}

\begin{figure}[!htb] 
\centering 
\begin{minipage}[t]{0.3\linewidth}
\centering
\epsfig{figure=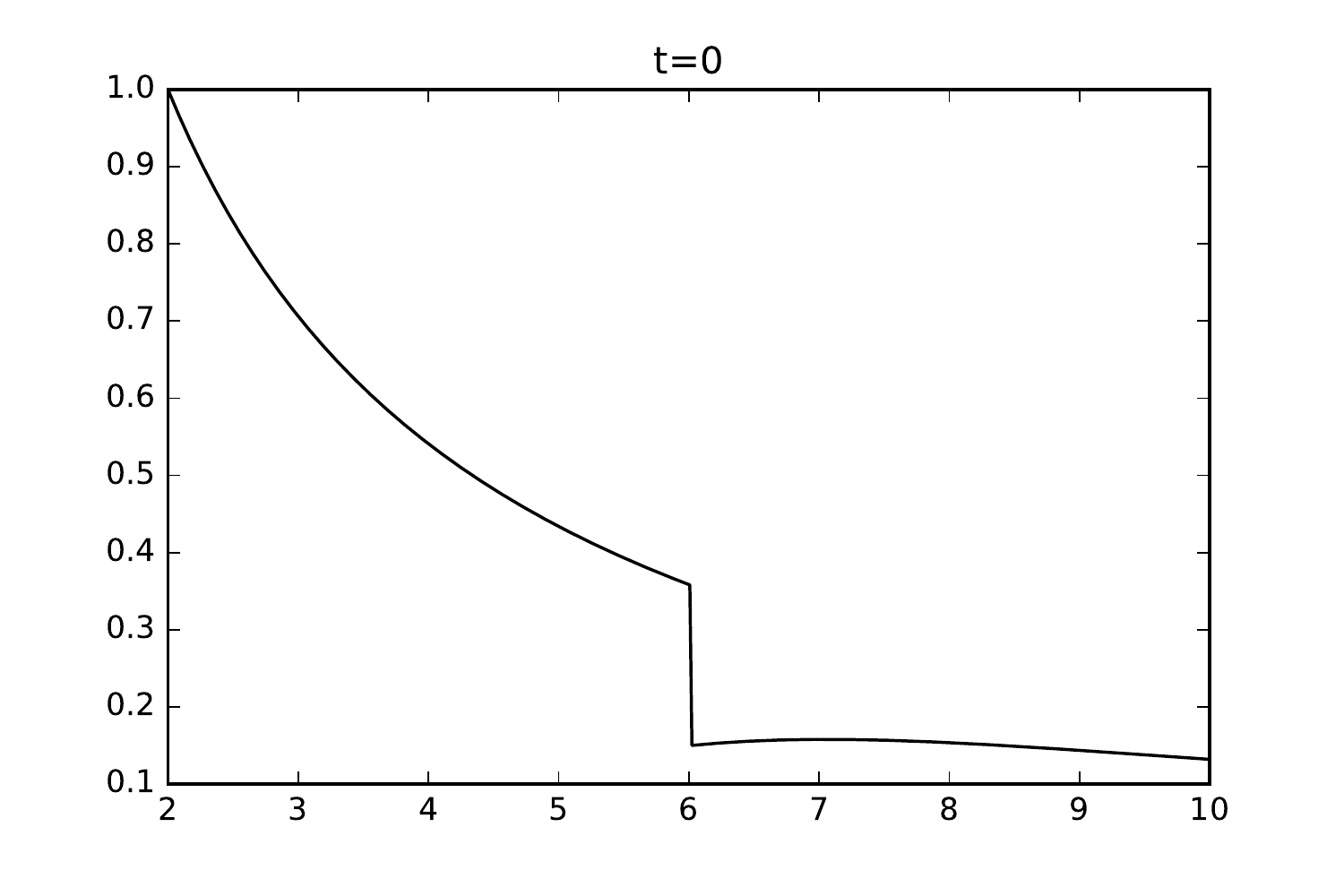,width= 2.5in} 
\end{minipage}
\hspace{0.1in}
\begin{minipage}[t]{0.3\linewidth}
\centering
\epsfig{figure=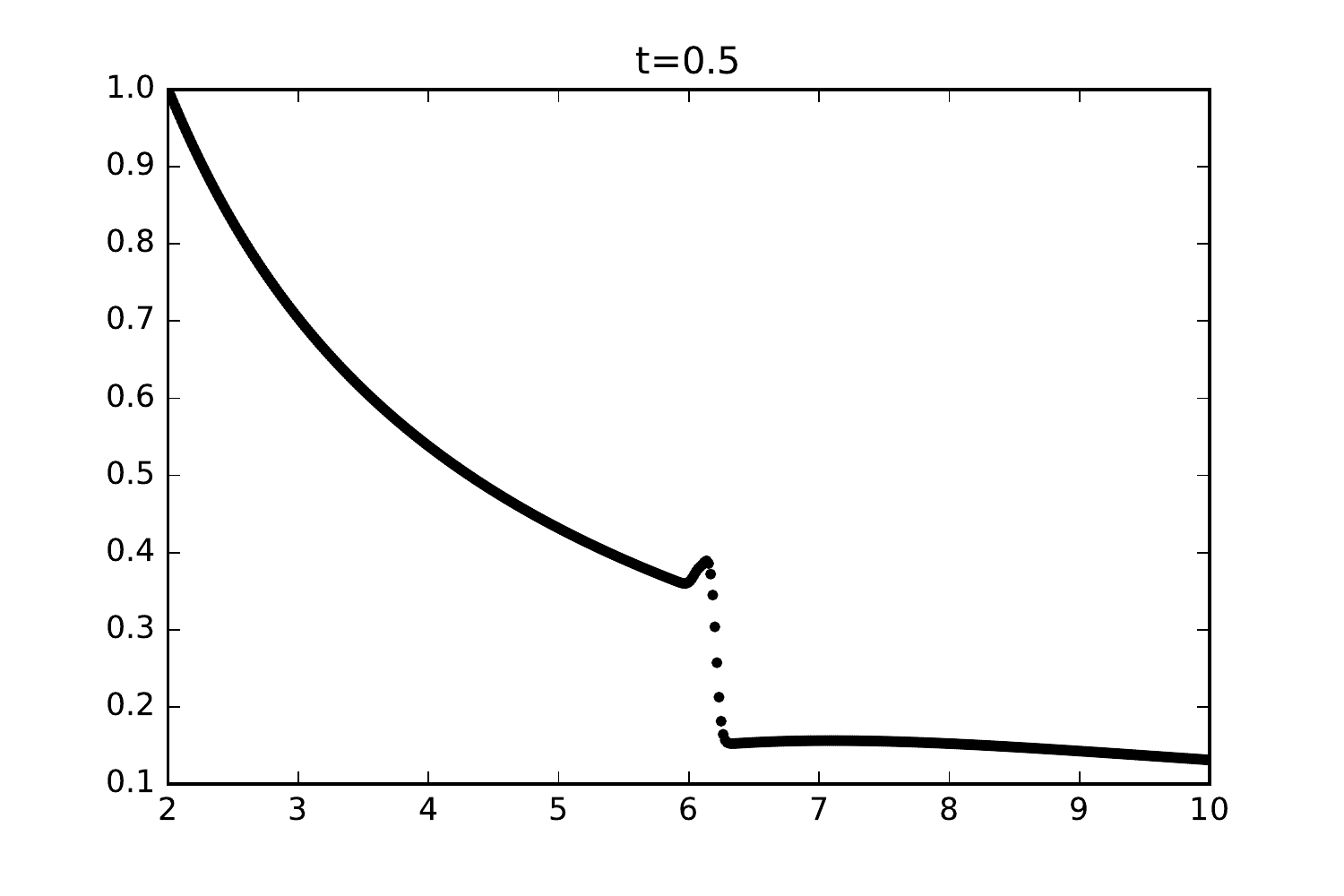,width=2.5in}
\end{minipage}
\hspace{0.1in}
\begin{minipage}[t]{0.3\linewidth}
\centering
\epsfig{figure=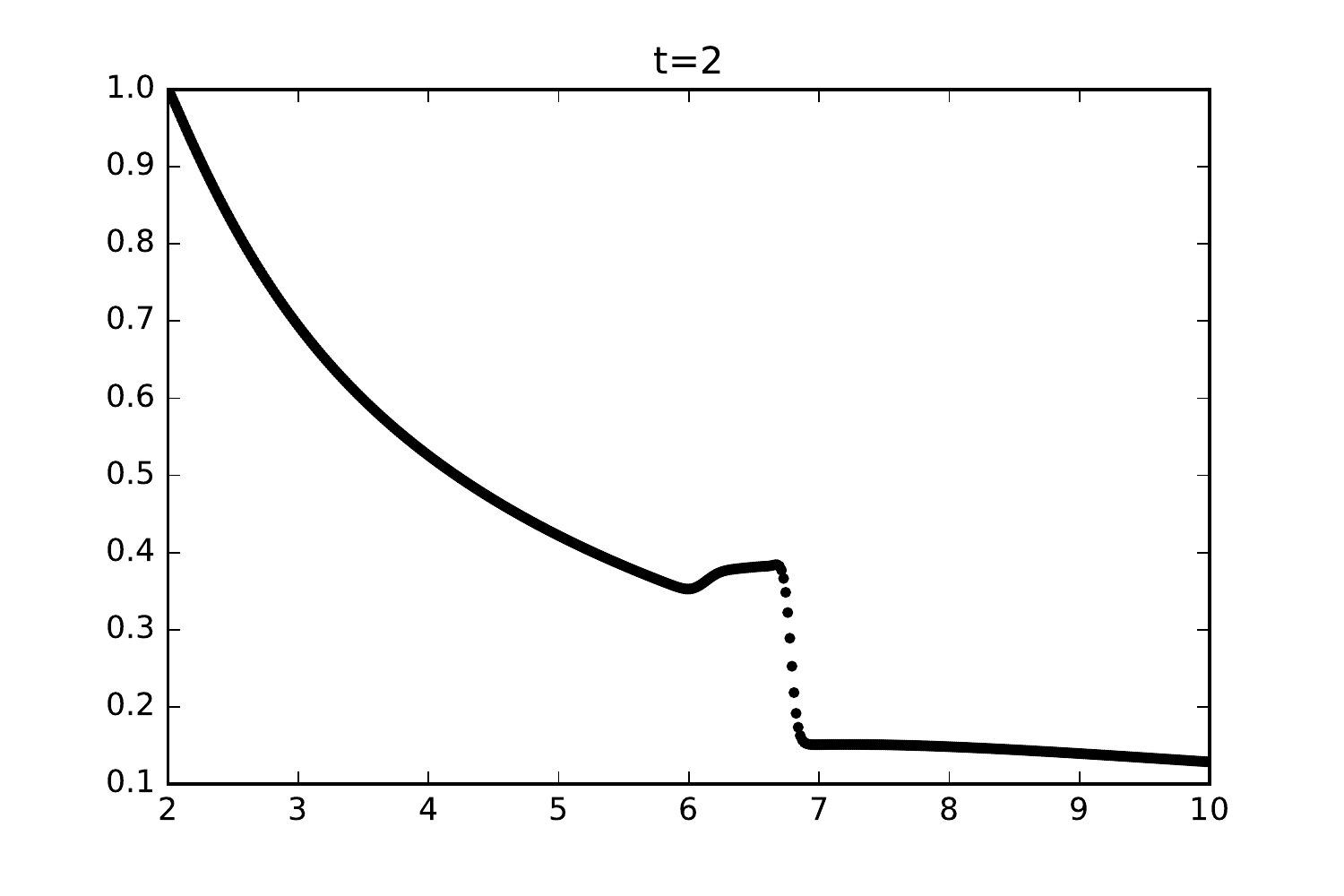,width=2.5in}
\end{minipage}
\caption{Solution to a Riemann problem  (1-rarefaction and 2-shock)} 
\label{FIG-90} 
\end{figure}

\begin{figure}[!htb] 
\centering 
\begin{minipage}[t]{0.3\linewidth}
\centering
\epsfig{figure=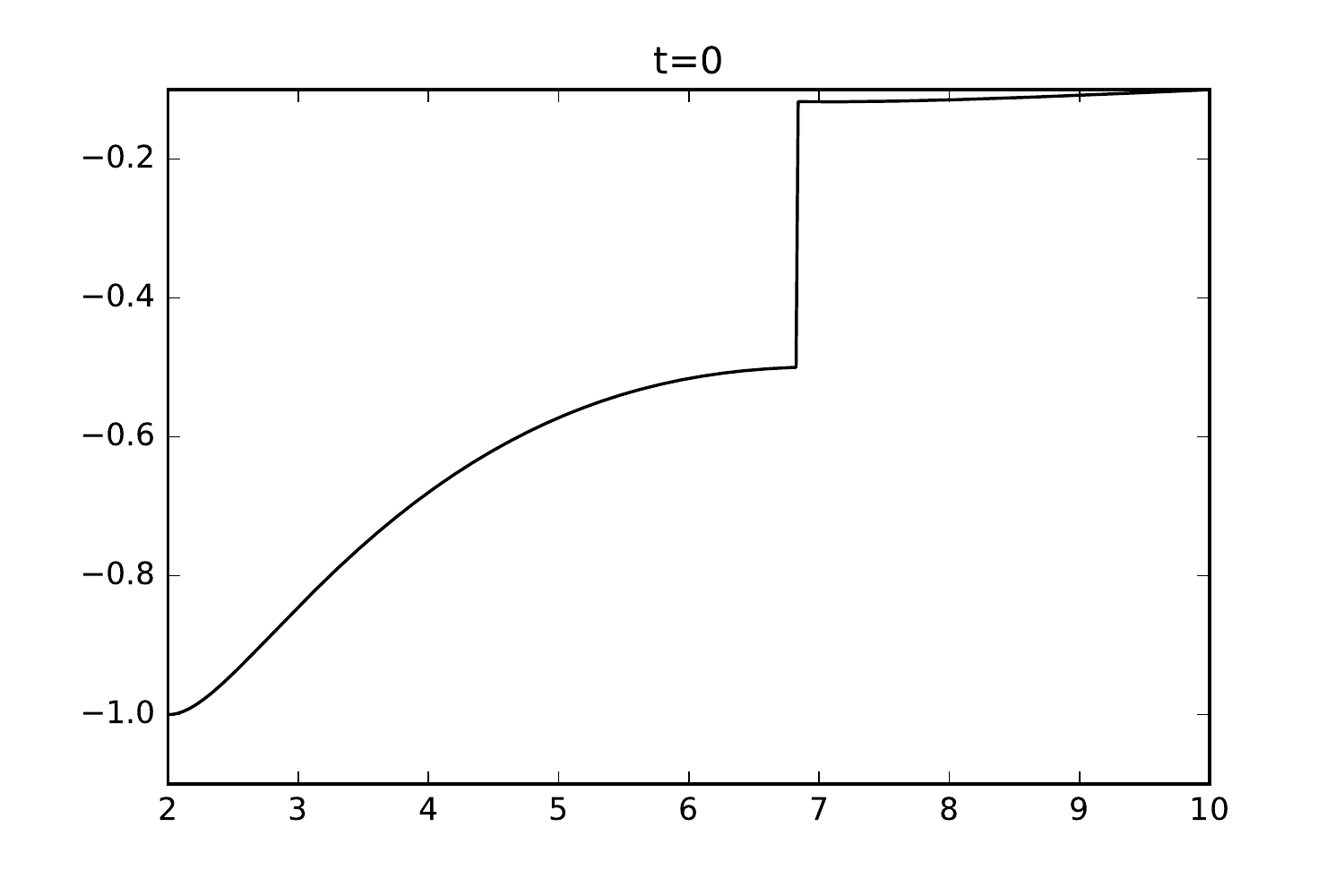,width= 2.5in} 
\end{minipage}
\hspace{0.1in}
\begin{minipage}[t]{0.3\linewidth}
\centering
\epsfig{figure=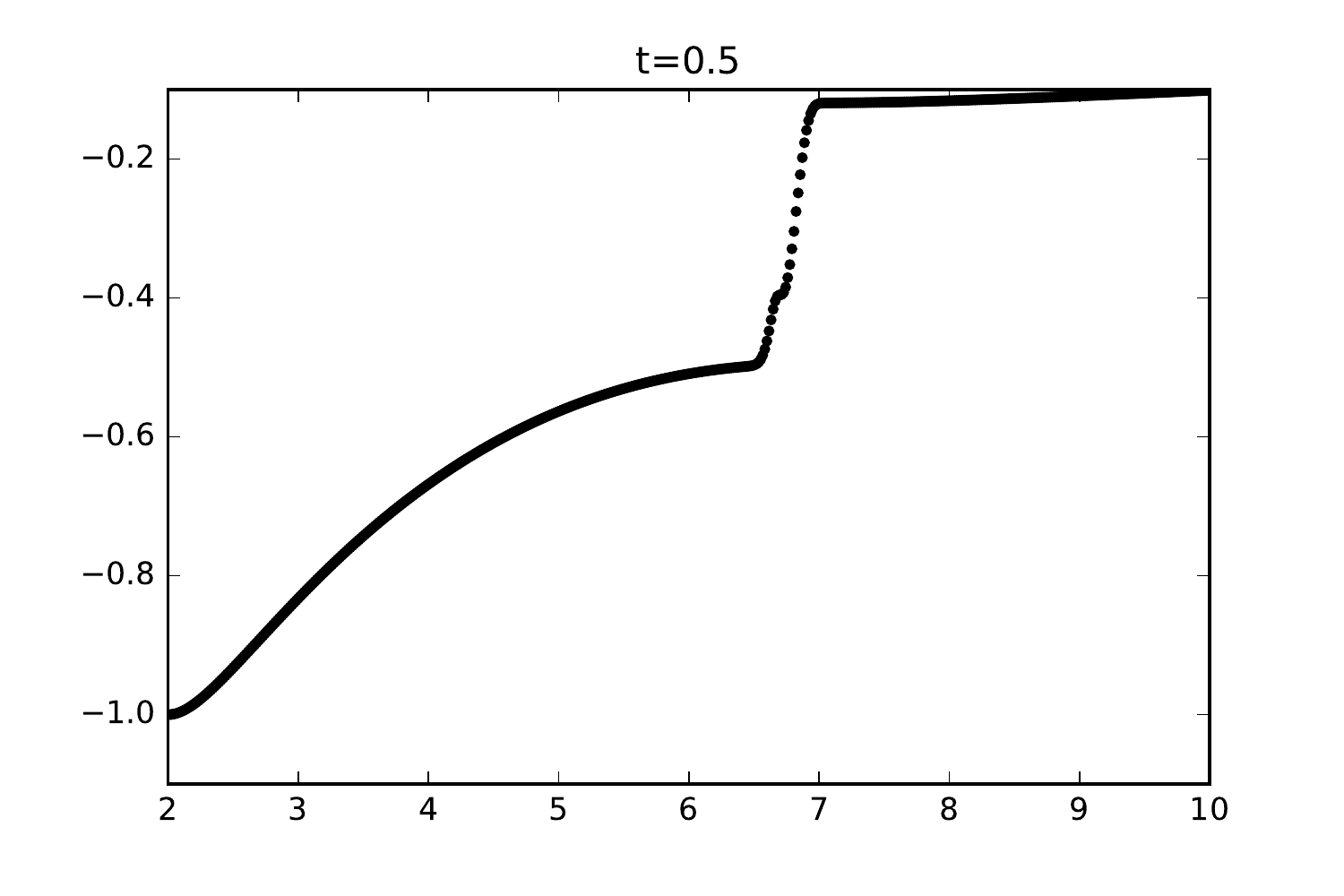,width=2.5in}
\end{minipage}
\hspace{0.1in}
\begin{minipage}[t]{0.3\linewidth}
\centering
\epsfig{figure=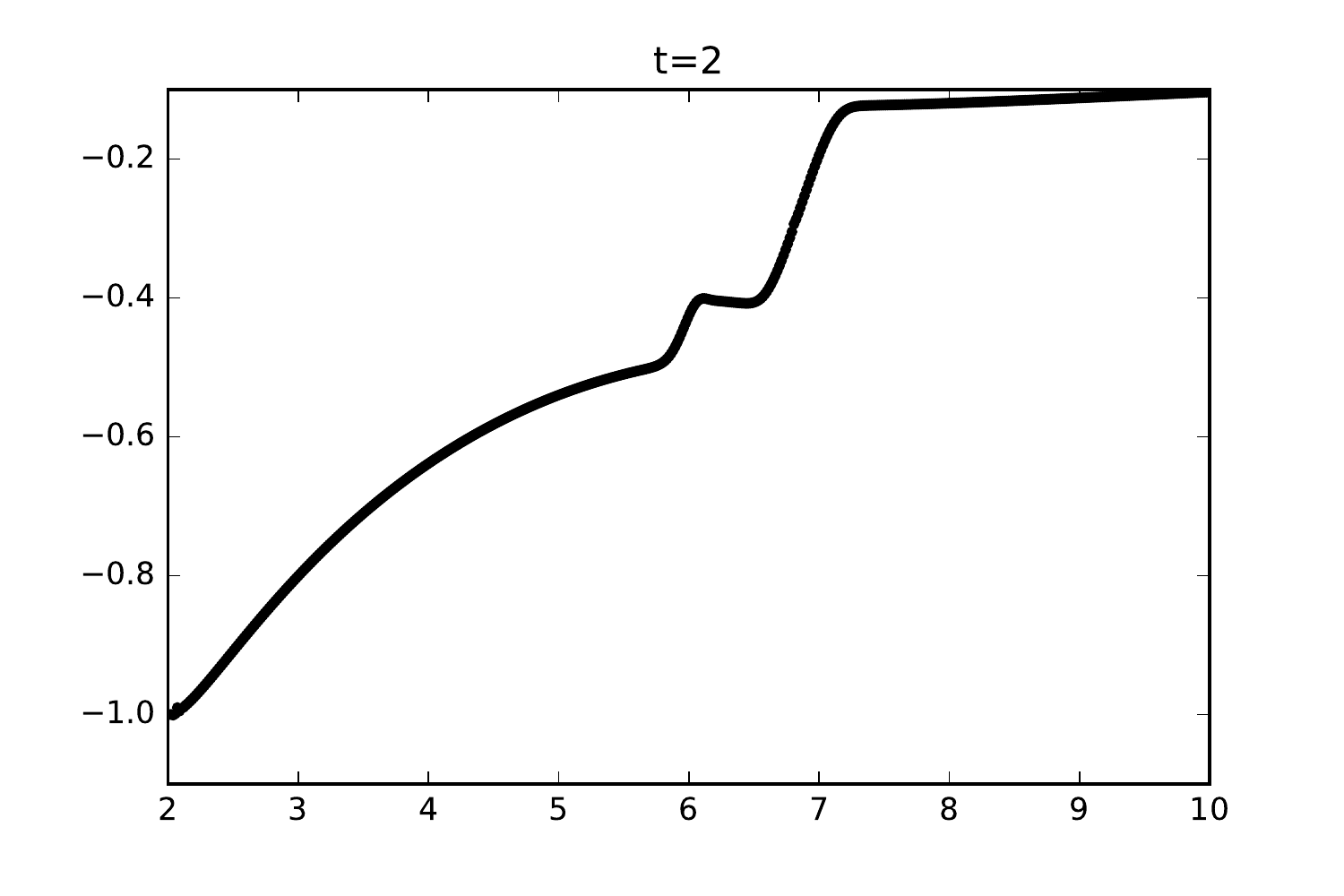,width=2.5in}
\end{minipage}
\caption{Solution to a Riemann problem (1-rarefaction and 2-rarefaction)}
\label{FIG-91} 
\end{figure}


\begin{figure}[!htb] 
\centering 
\begin{minipage}[t]{0.3\linewidth}
\centering
\epsfig{figure=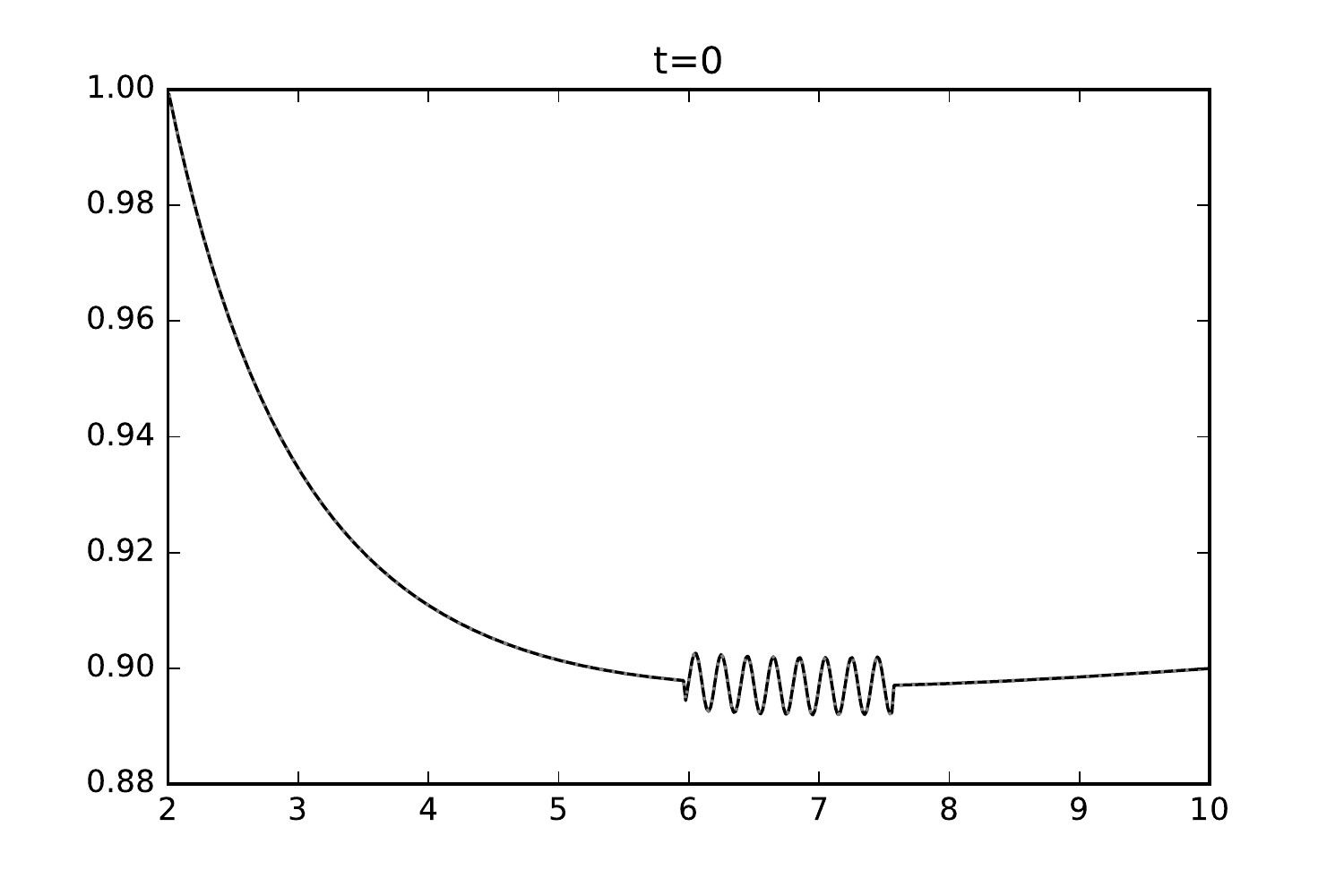,width= 2.5in} 
\end{minipage}
\hspace{0.1in}
\begin{minipage}[t]{0.3\linewidth}
\centering
\epsfig{figure=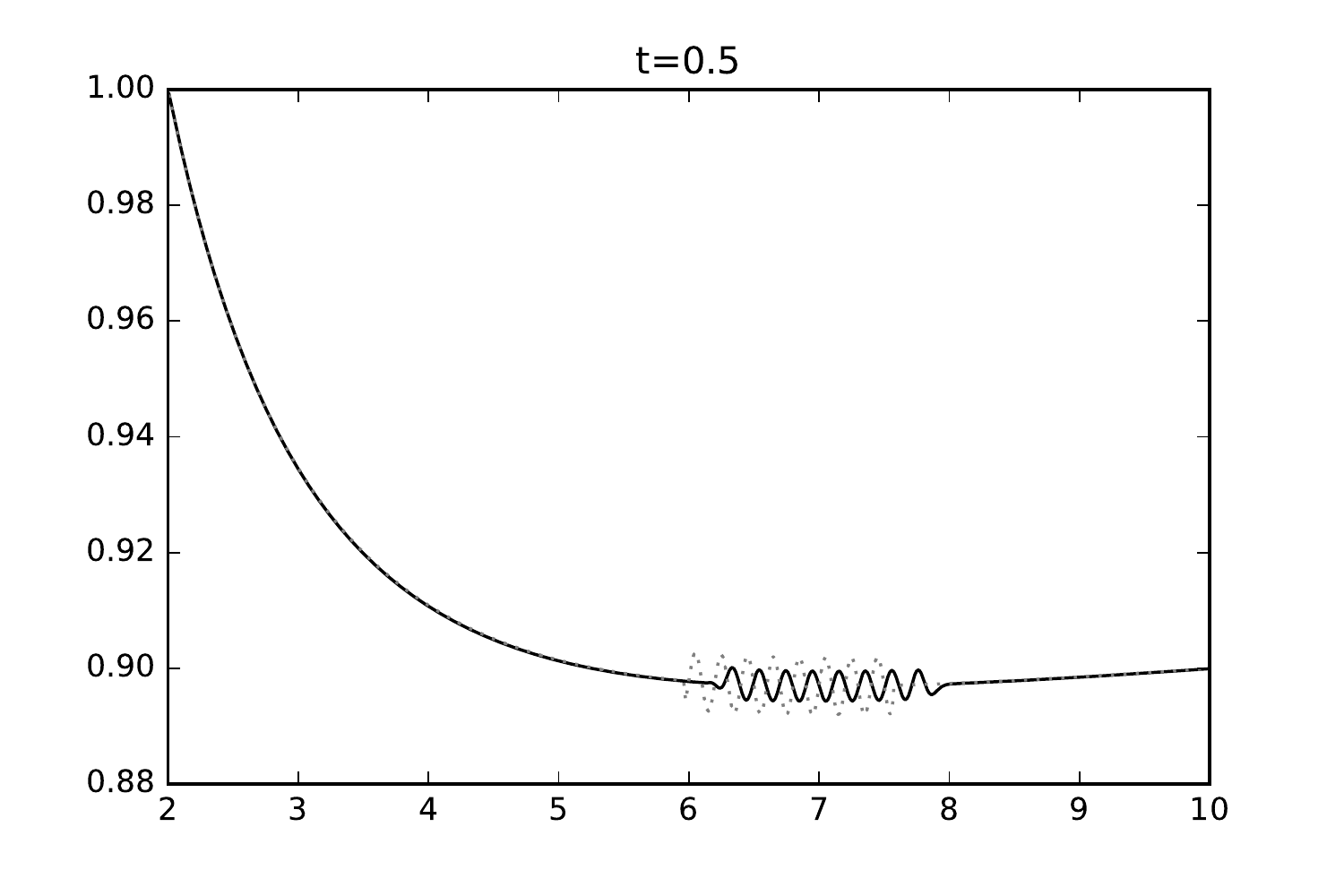,width=2.5in}
\end{minipage}
\hspace{0.1in}
\begin{minipage}[t]{0.3\linewidth}
\centering
\epsfig{figure=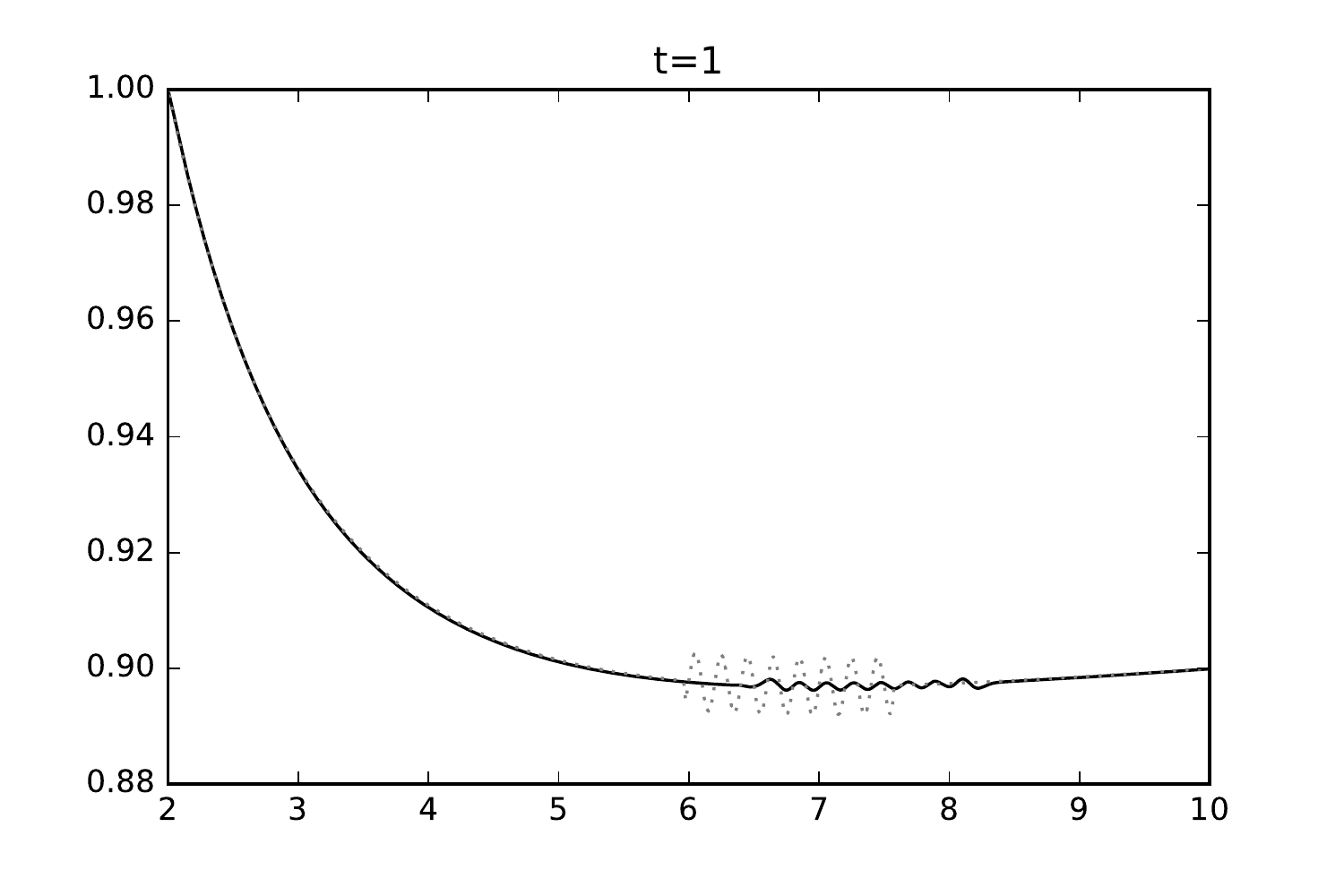,width=2.5in}
\end{minipage}
\begin{minipage}[t]{0.3\linewidth}
\centering
\epsfig{figure=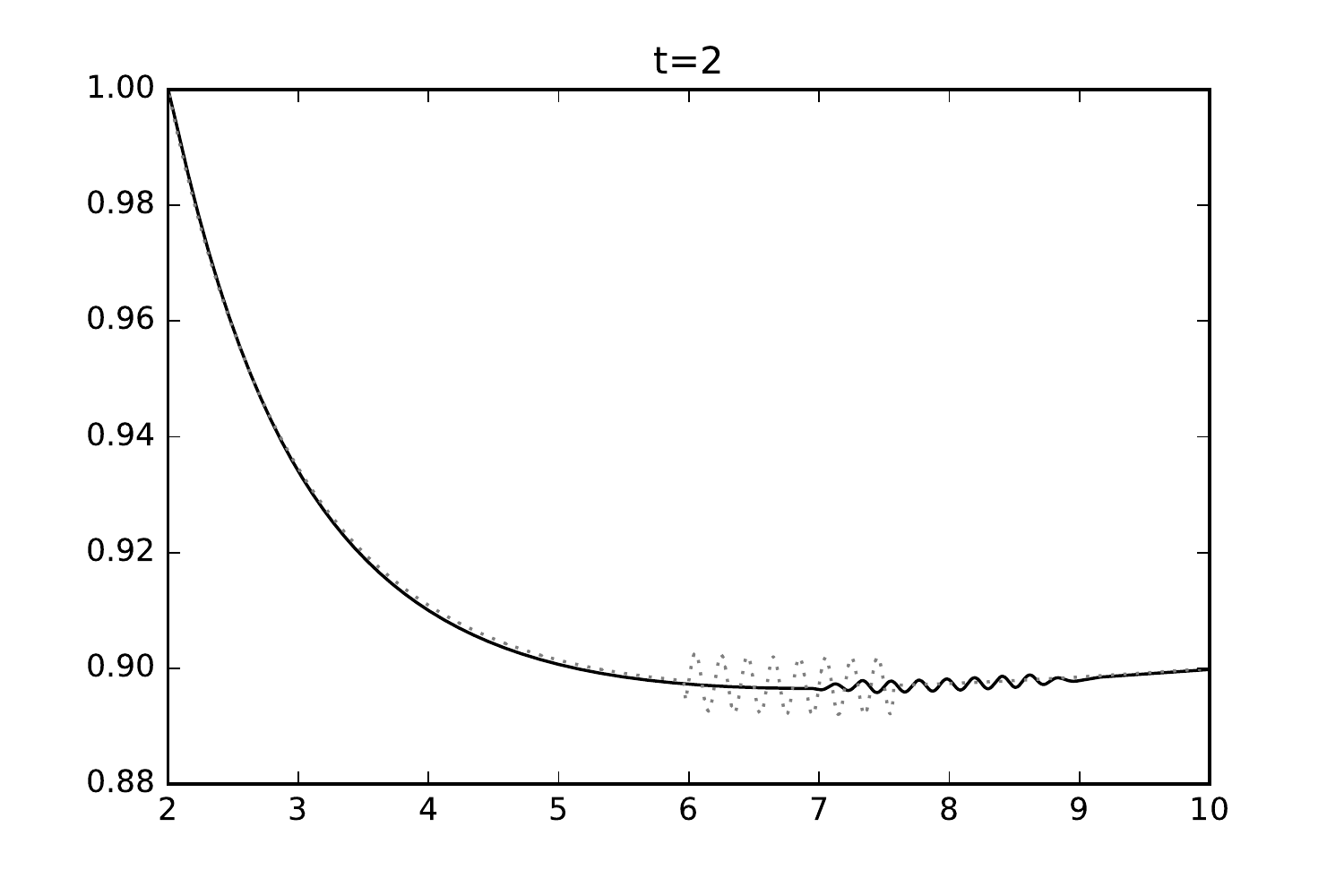,width= 2.5in} 
\end{minipage}
\hspace{0.1in}
\begin{minipage}[t]{0.3\linewidth}
\centering
\epsfig{figure=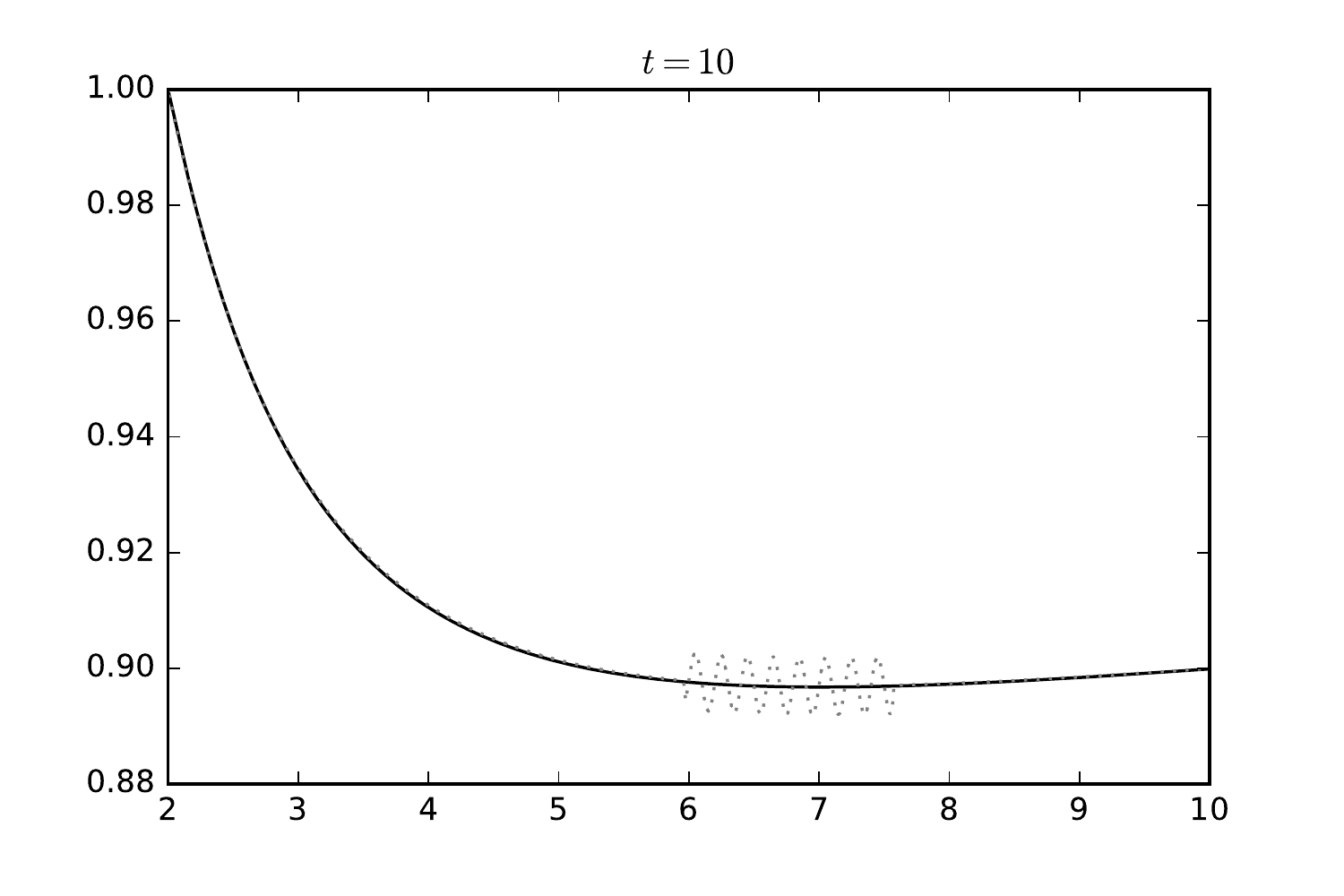,width=2.5in}
\end{minipage}
\hspace{0.2in}
\begin{minipage}[t]{0.3\linewidth}
\centering
\epsfig{figure=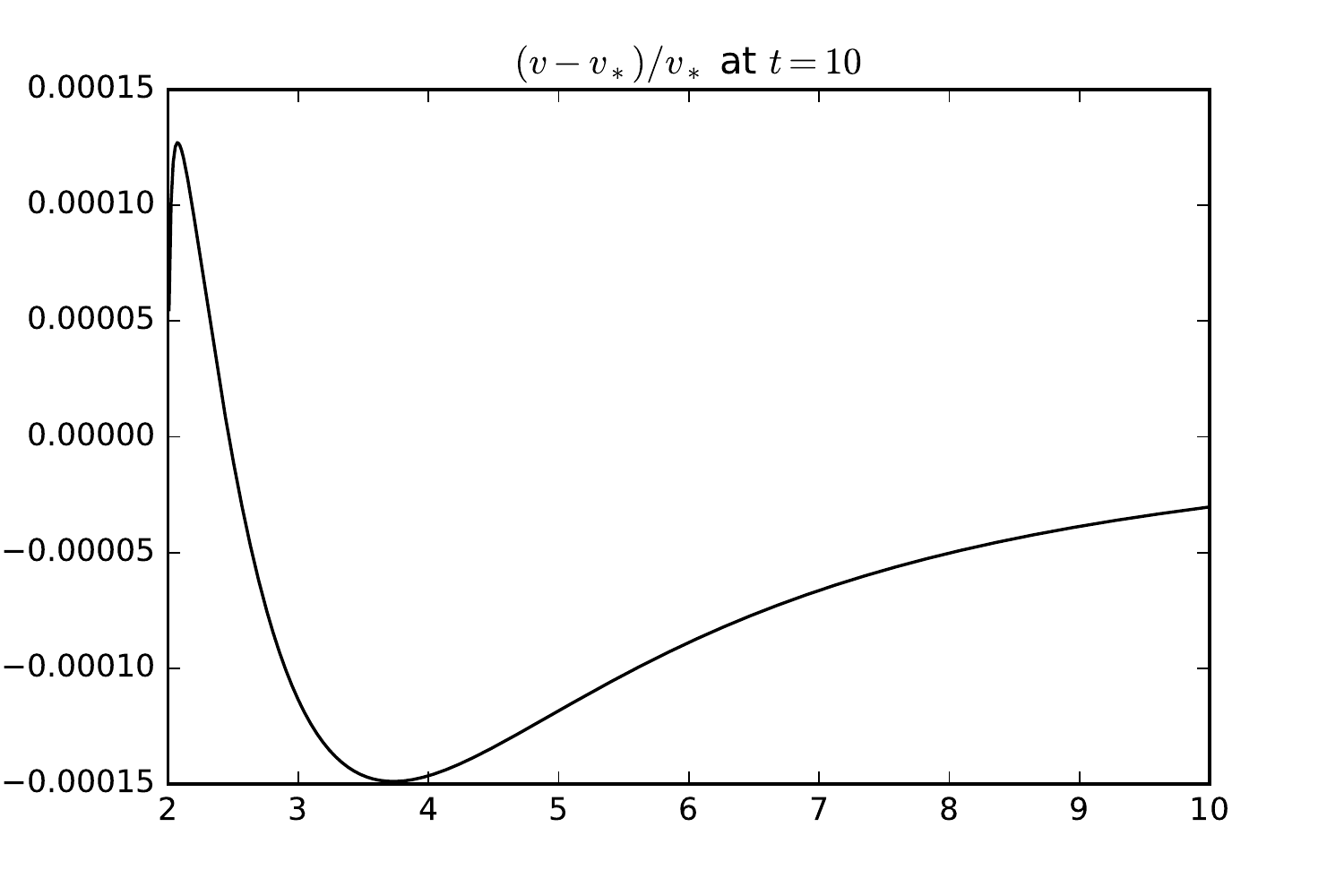,width=2.5in}
\end{minipage}
\caption{Evolution of a steady state with perturbation, converging to the same asymptotic state}
\label{FIG-92} 
\end{figure}

\begin{figure}[!htb] 
\centering 
\begin{minipage}[t]{0.3\linewidth}
\centering
\epsfig{figure=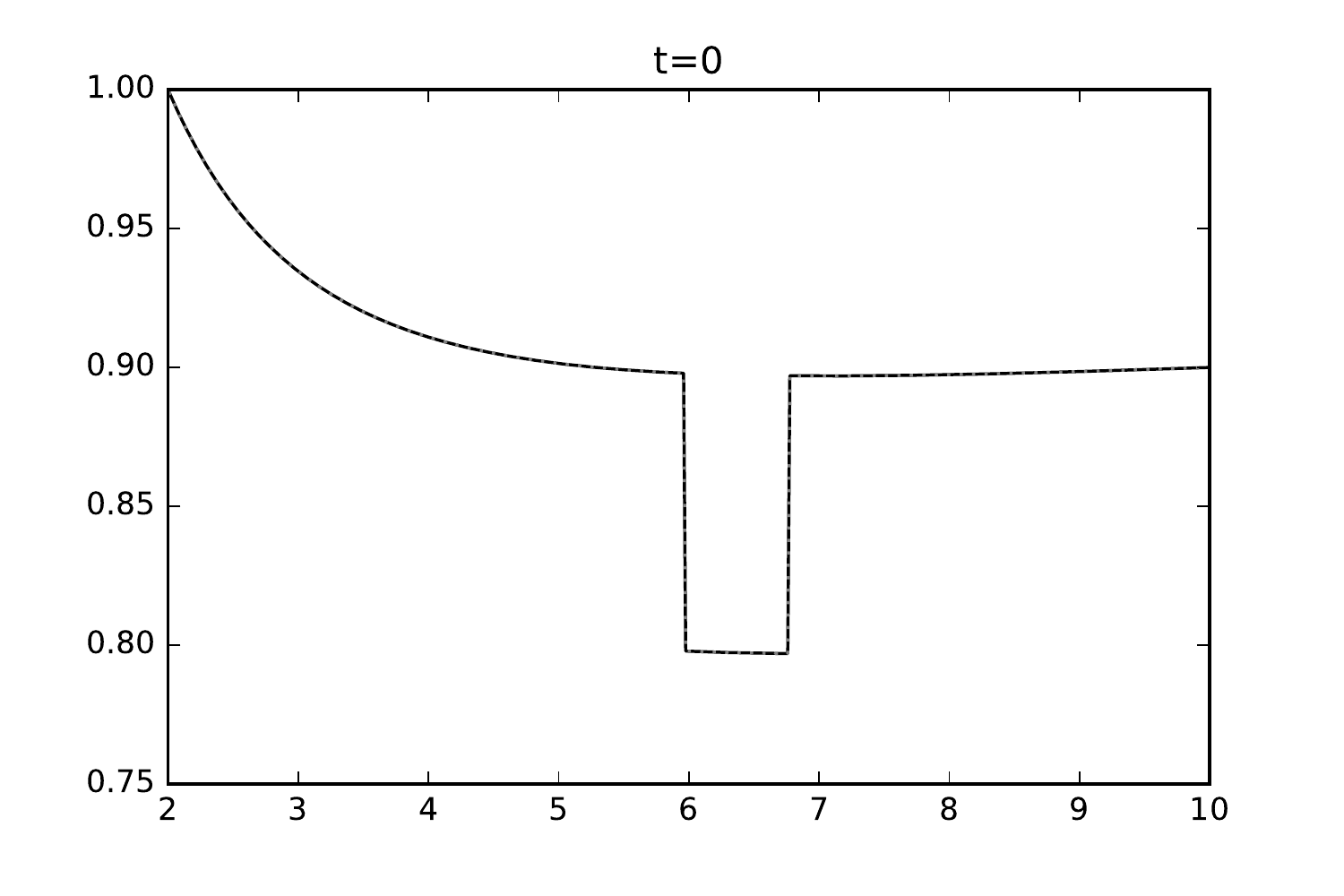,width= 2.5in} 
\end{minipage}
\hspace{0.1in}
\begin{minipage}[t]{0.3\linewidth}
\centering
\epsfig{figure=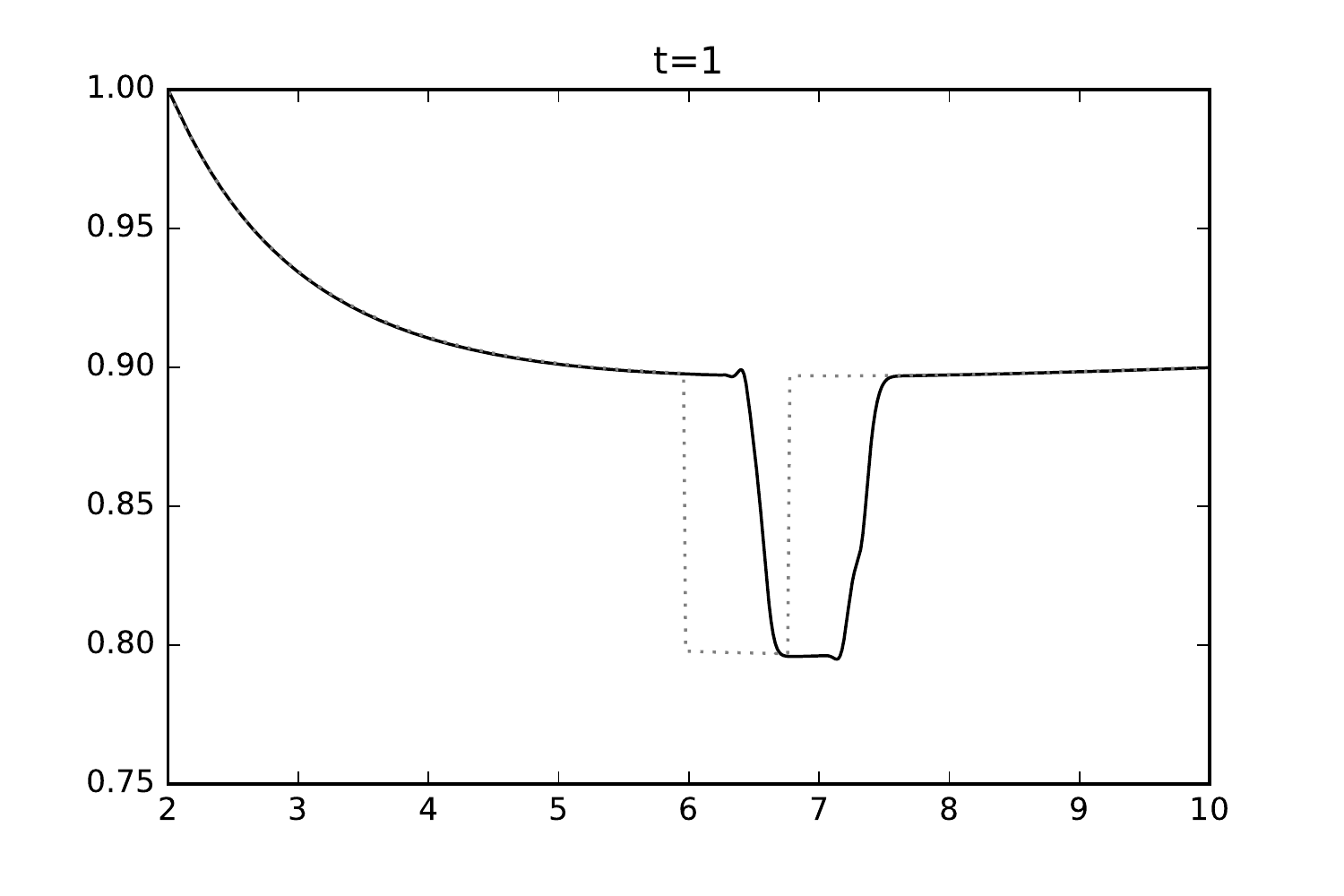,width=2.5in}
\end{minipage}
\hspace{0.1in}
\begin{minipage}[t]{0.3\linewidth}
\centering
\epsfig{figure=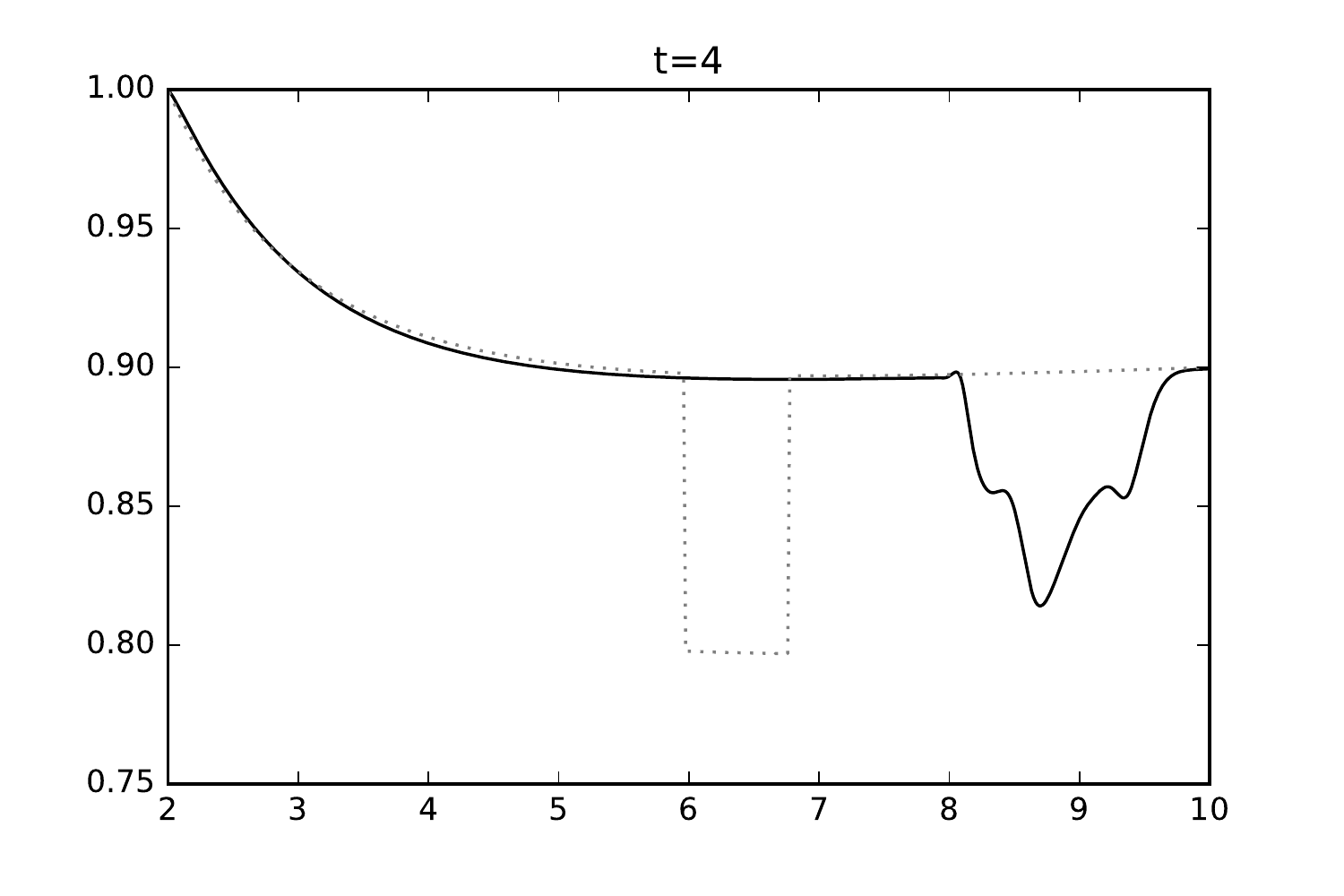,width=2.5in}
\end{minipage}
\begin{minipage}[t]{0.3\linewidth}
\centering
\epsfig{figure=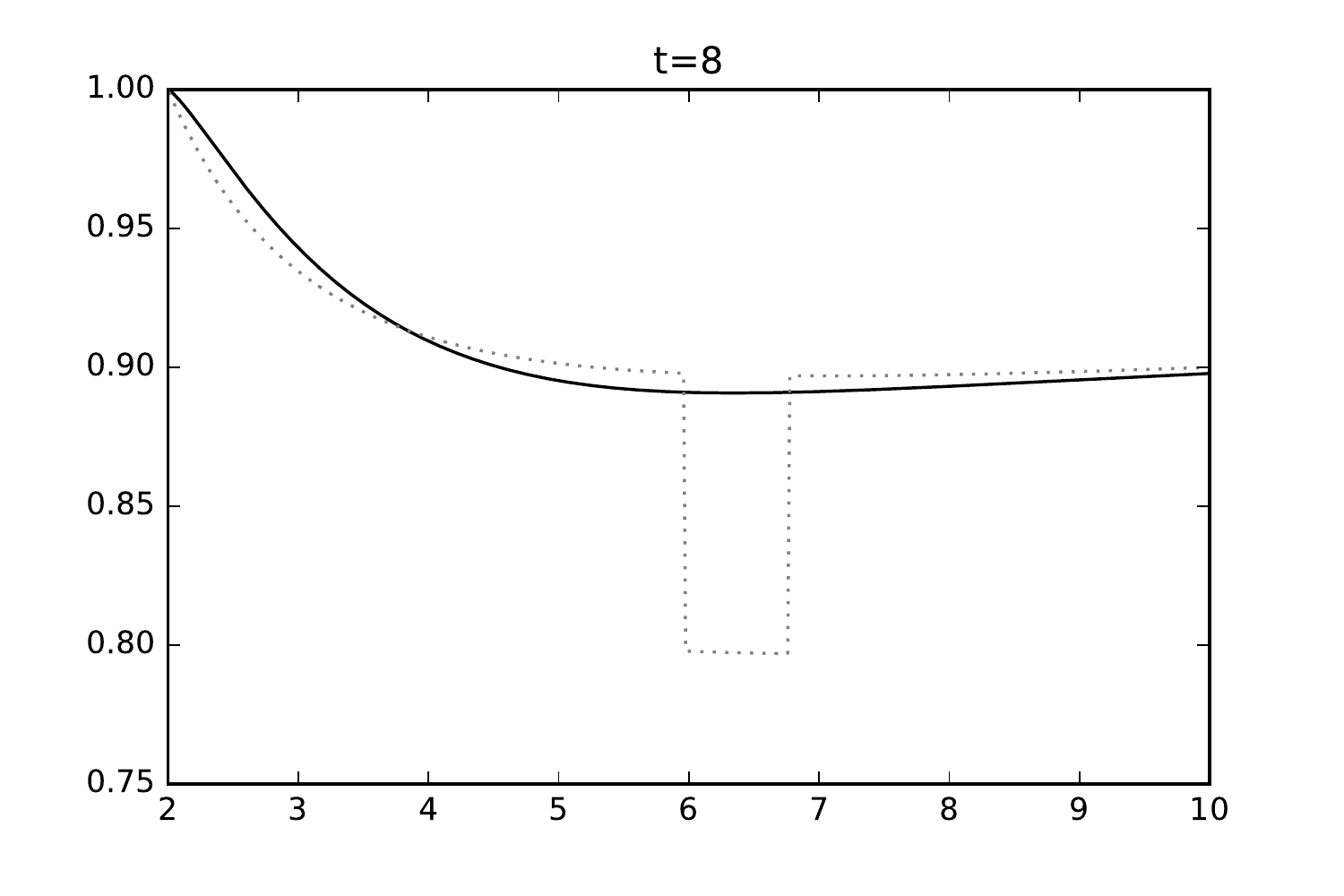,width= 2.5in} 
\end{minipage}
\hspace{0.1in}
\begin{minipage}[t]{0.3\linewidth}
\centering
\epsfig{figure=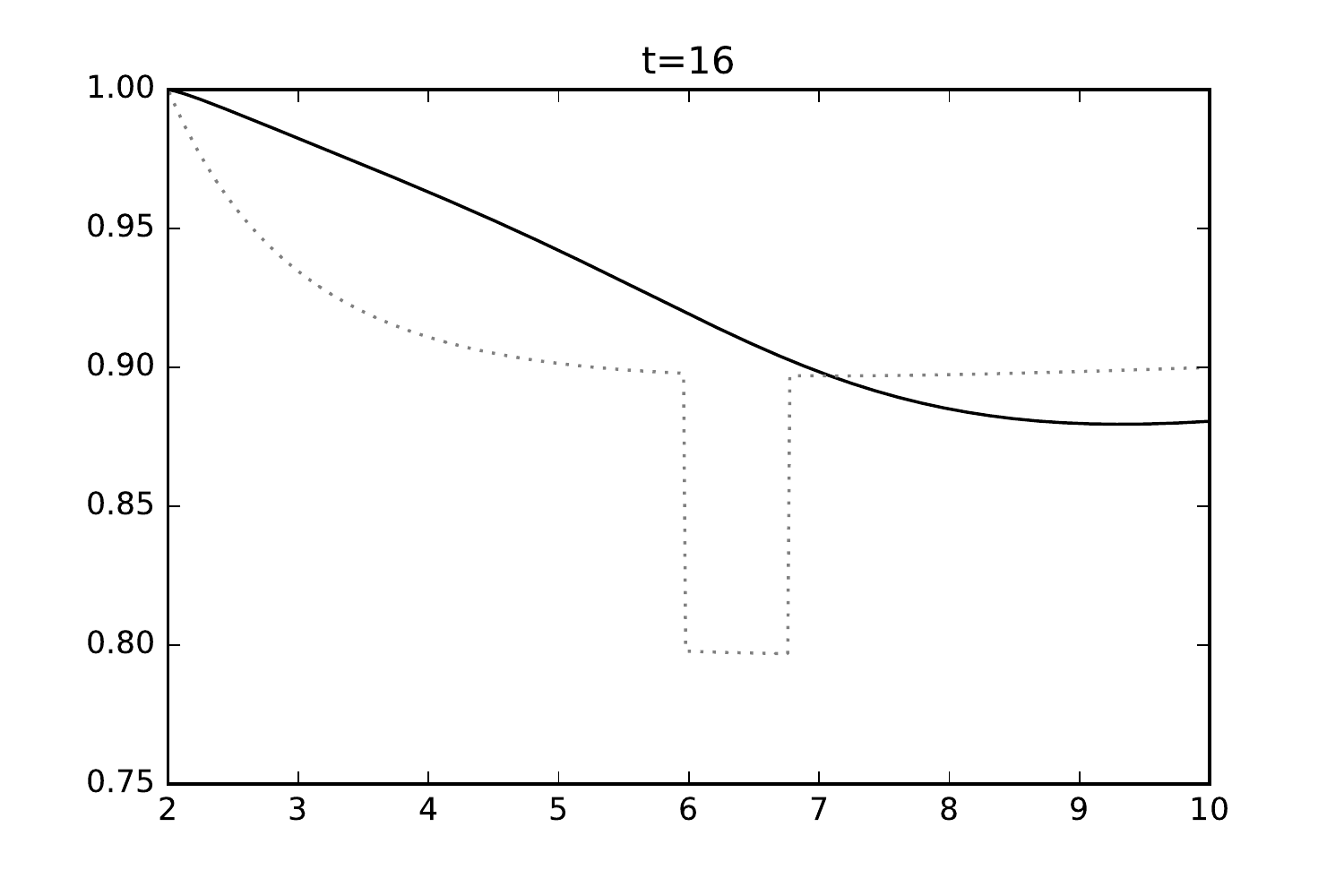,width=2.5in}
\end{minipage}
\hspace{0.2in}
\begin{minipage}[t]{0.3\linewidth}
\centering
\epsfig{figure=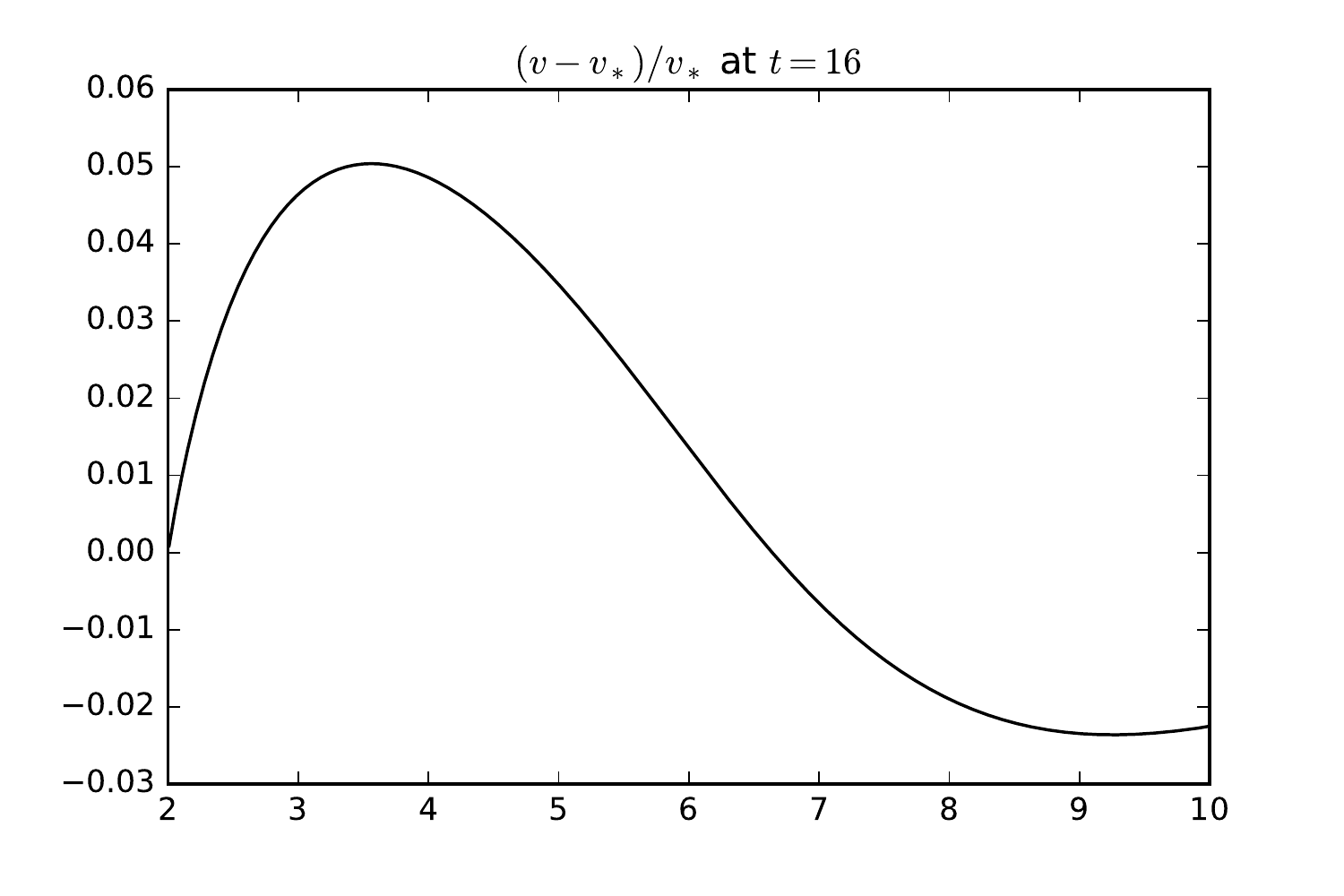,width=2.5in}
\end{minipage}
\caption{Evolution of a steady state with perturbation, converging to a different asymptotic state}
\label{FIG-93} 
\end{figure}

\begin{figure}[!htb] 
\centering 
\begin{minipage}[t]{0.3\linewidth}
\centering
\epsfig{figure=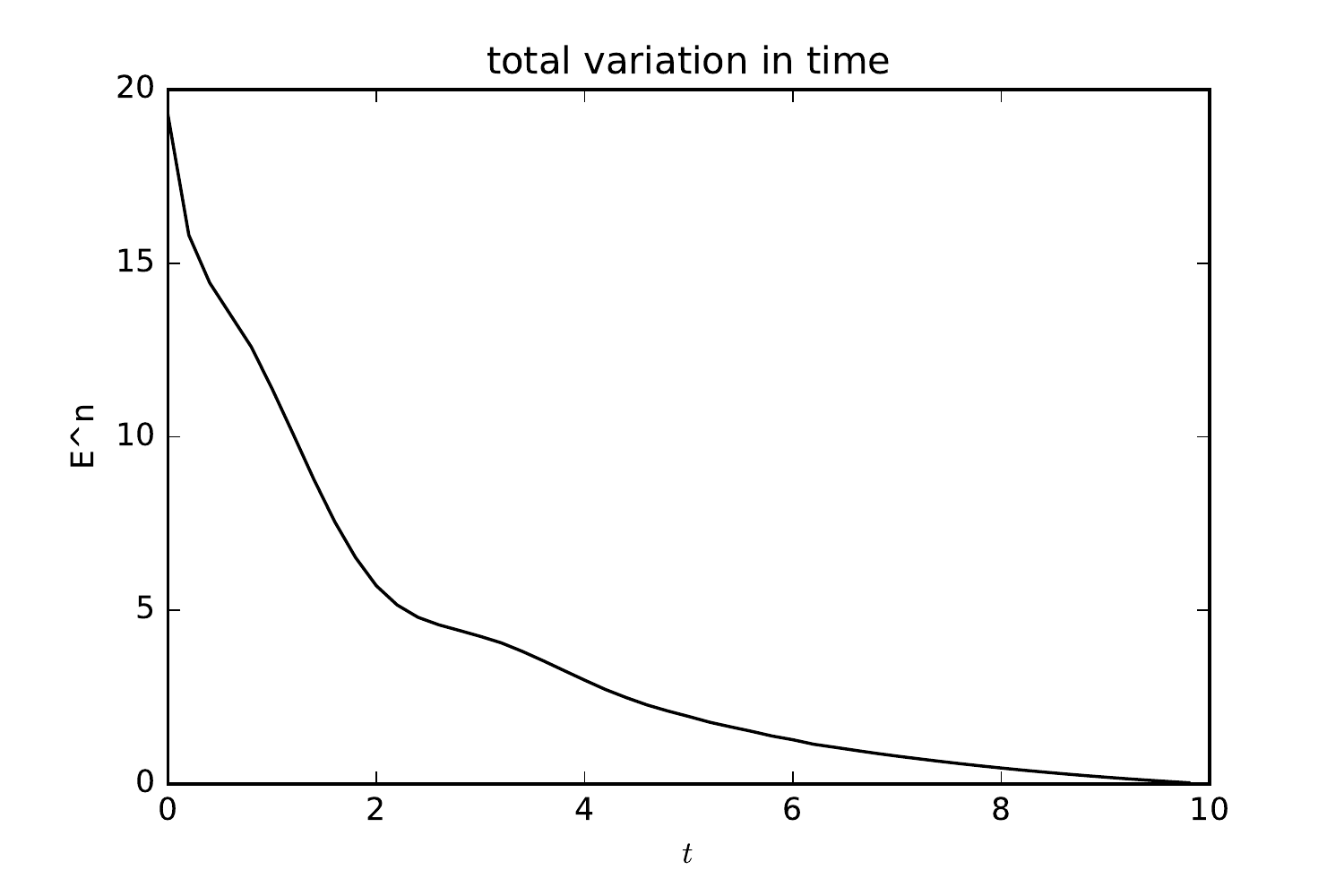,width = 2.5in} 
\end{minipage}
\hspace{1in}
\begin{minipage}[t]{0.3\linewidth}
\centering
\epsfig{figure=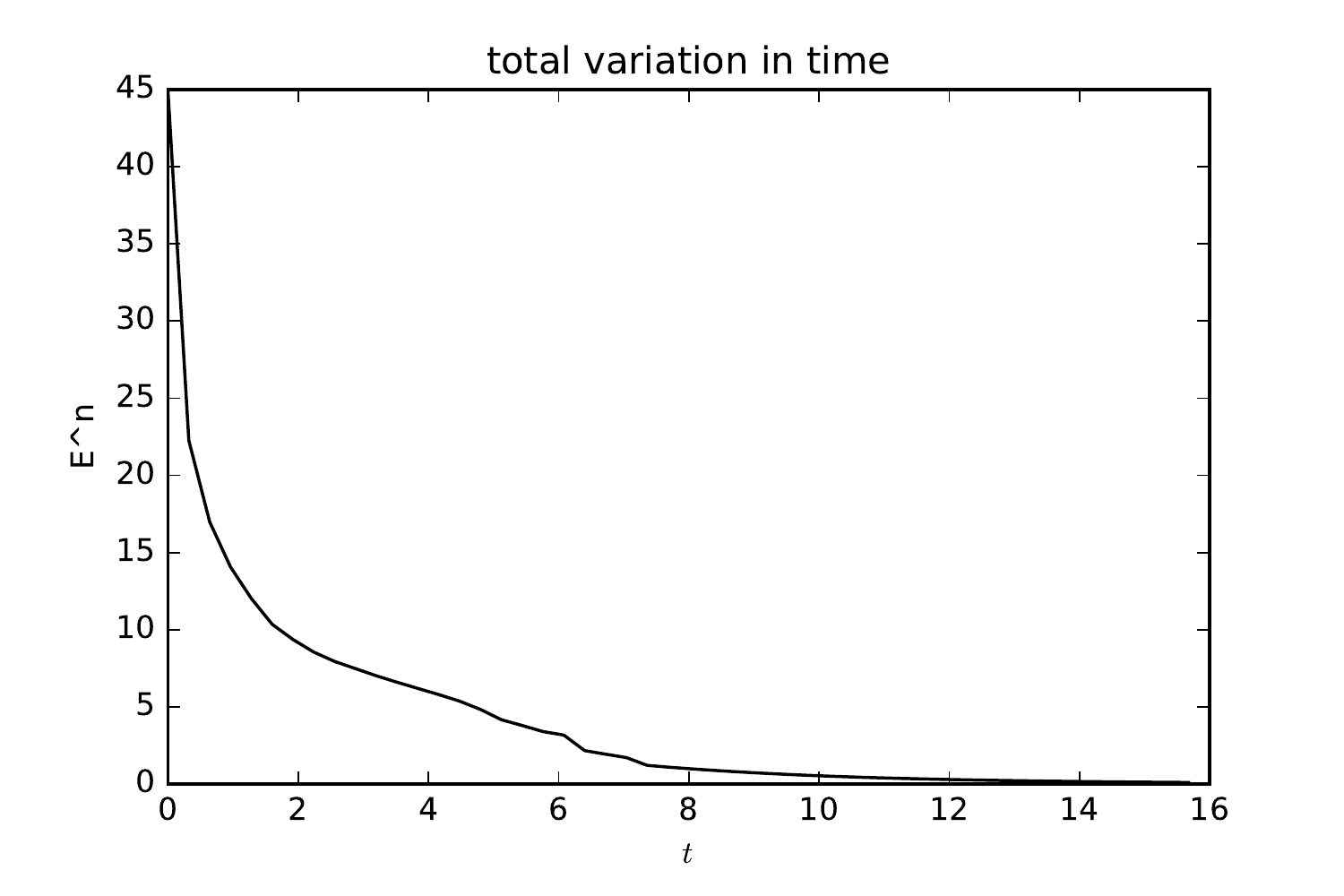,width = 2.5in}
\end{minipage}
\caption{Total variation in time corresponding to Figures~\ref{FIG-92} and \ref{FIG-93}, respectively}
\label{FIG-94}
\end{figure}


\begin{figure}[!htb] 
\centering 
\begin{minipage}[t]{0.3\linewidth}
\centering
\epsfig{figure=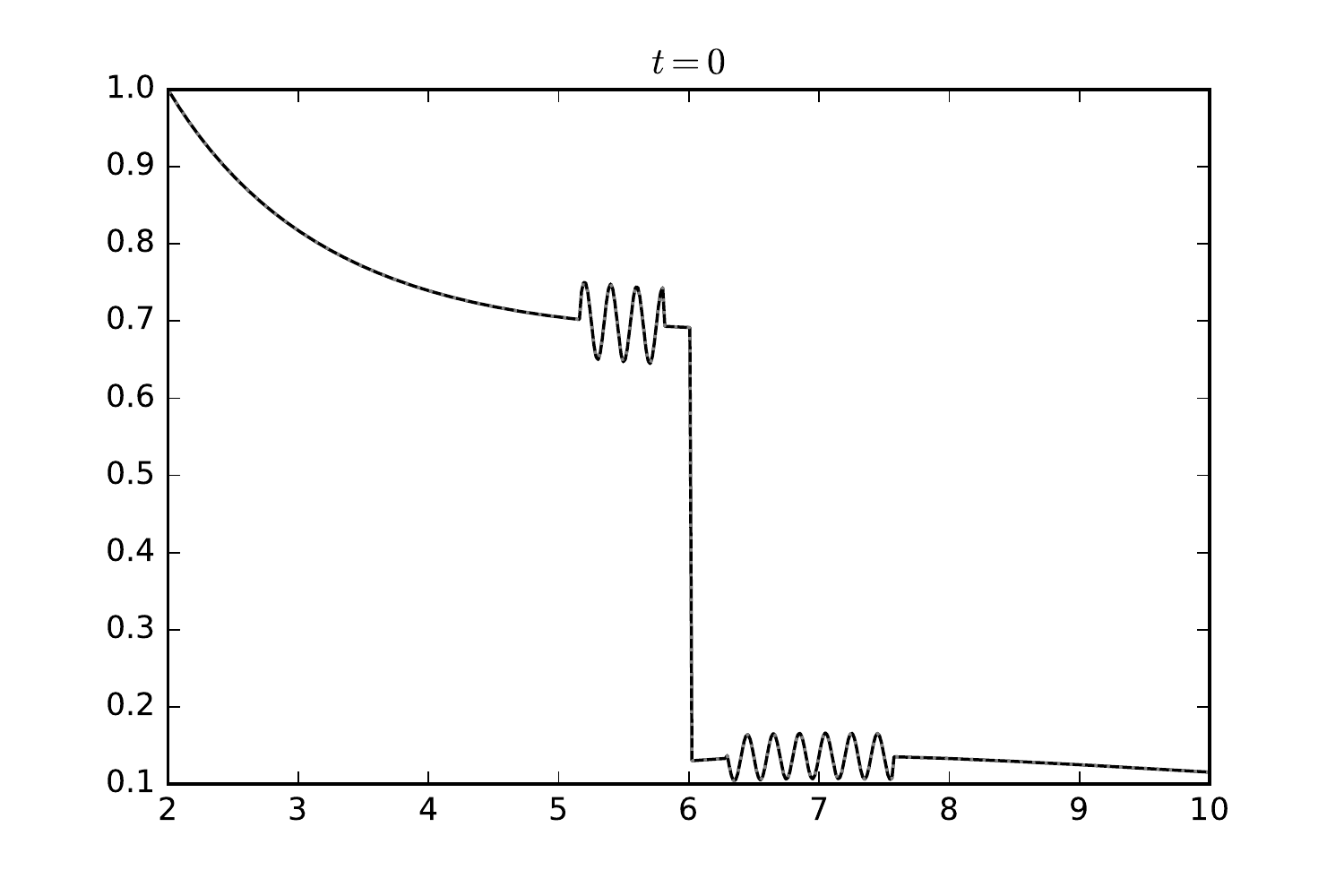,width= 2.5in} 
\end{minipage}
\hspace{0.1in}
\begin{minipage}[t]{0.3\linewidth}
\centering
\epsfig{figure=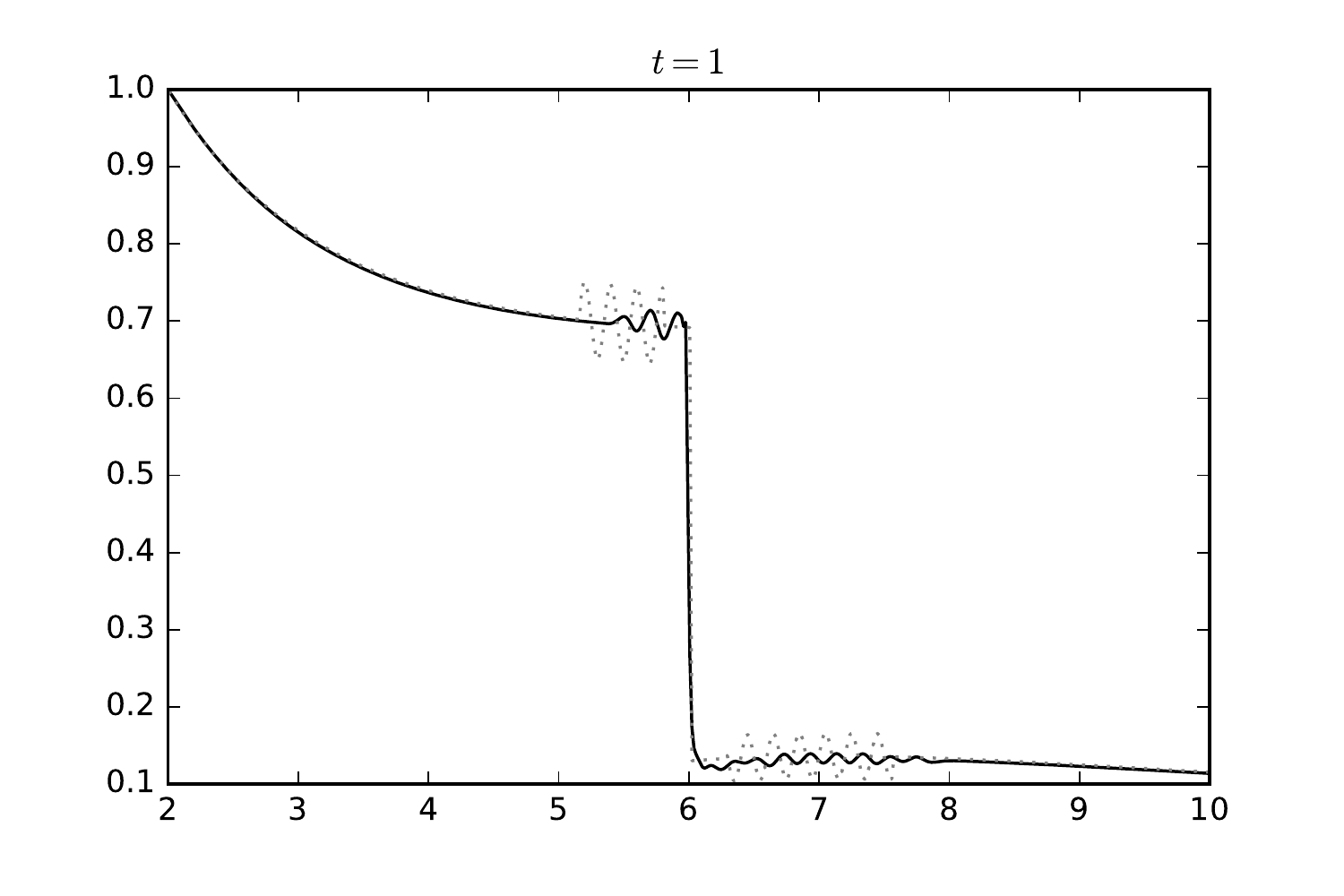,width=2.5in}
\end{minipage}
\hspace{0.1in}
\begin{minipage}[t]{0.3\linewidth}
\centering
\epsfig{figure=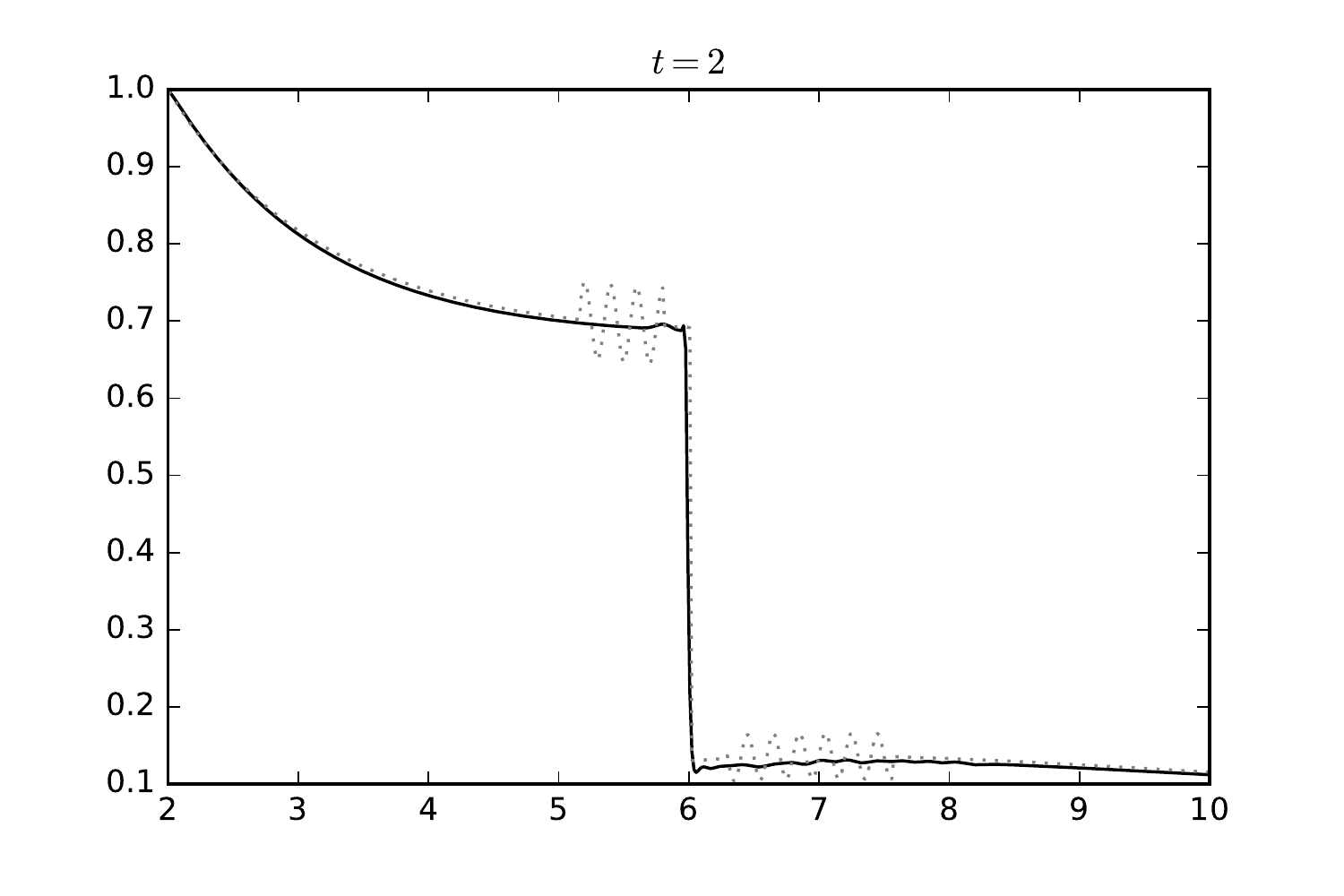,width=2.5in}
\end{minipage}

\centering 
\begin{minipage}[t]{0.3\linewidth}
\centering
\epsfig{figure=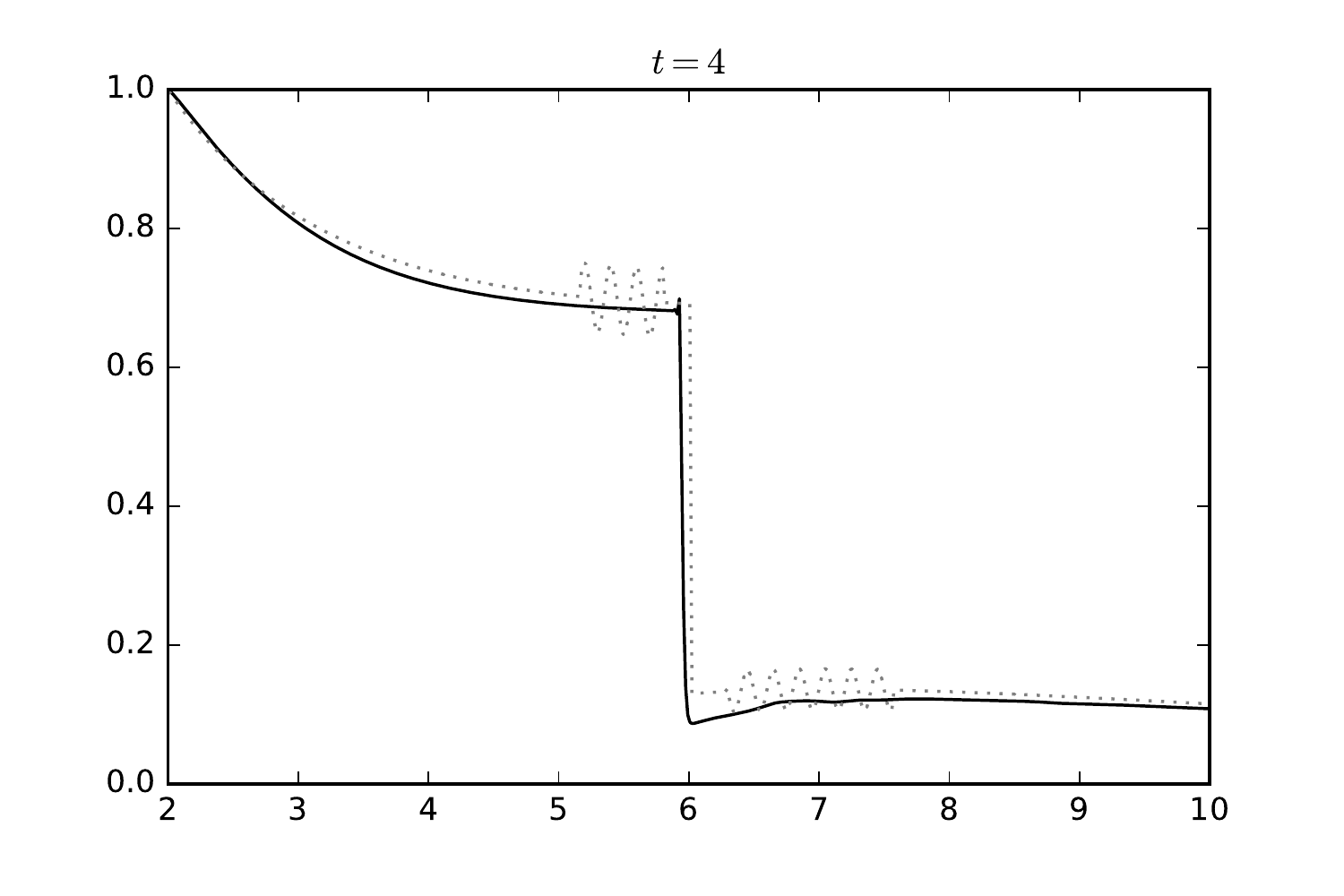,width= 2.5in} 
\end{minipage}
\hspace{0.1in}
\begin{minipage}[t]{0.3\linewidth}
\centering
\epsfig{figure=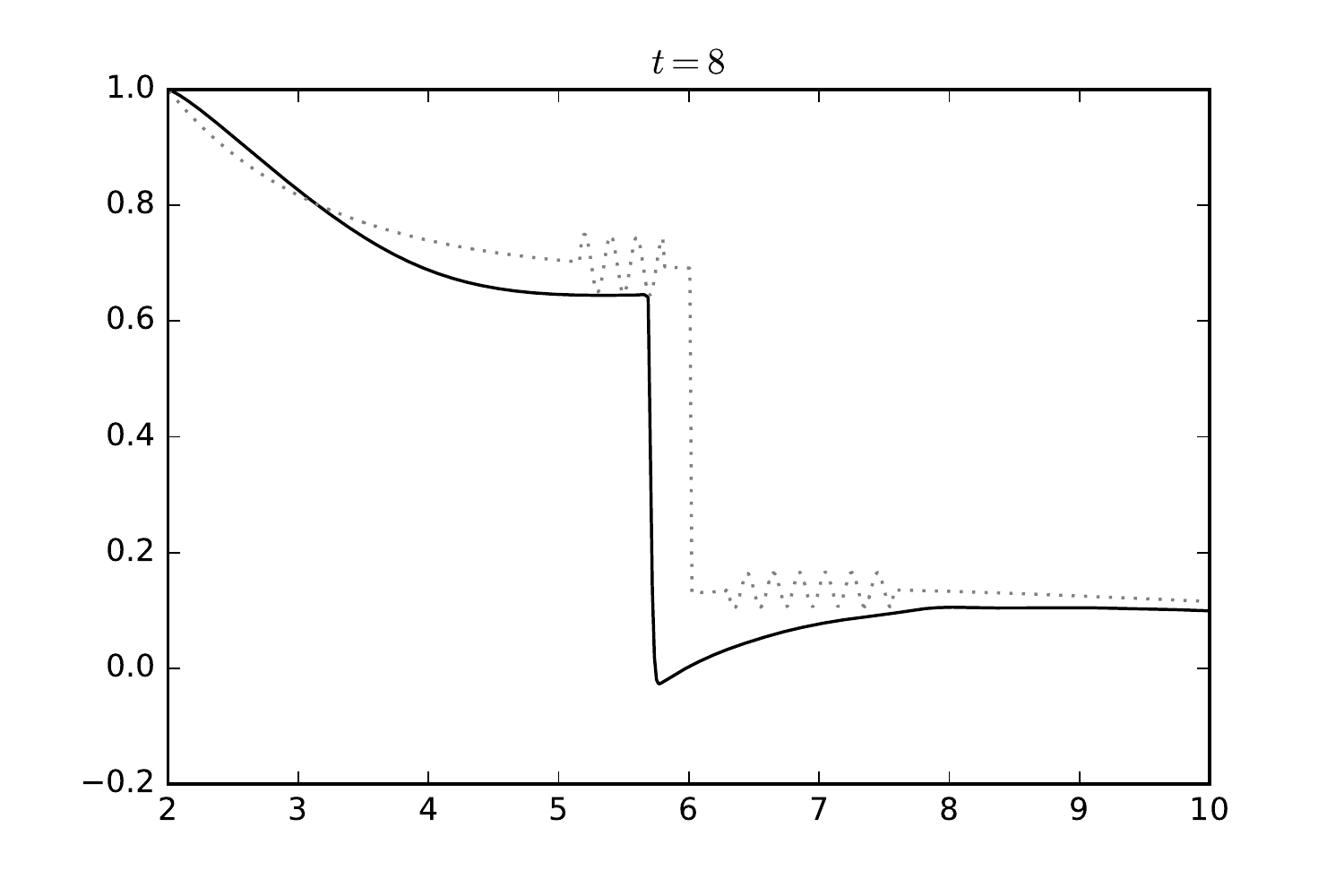,width=2.5in}
\end{minipage}
\hspace{0.1in}
\begin{minipage}[t]{0.3\linewidth}
\centering
\epsfig{figure=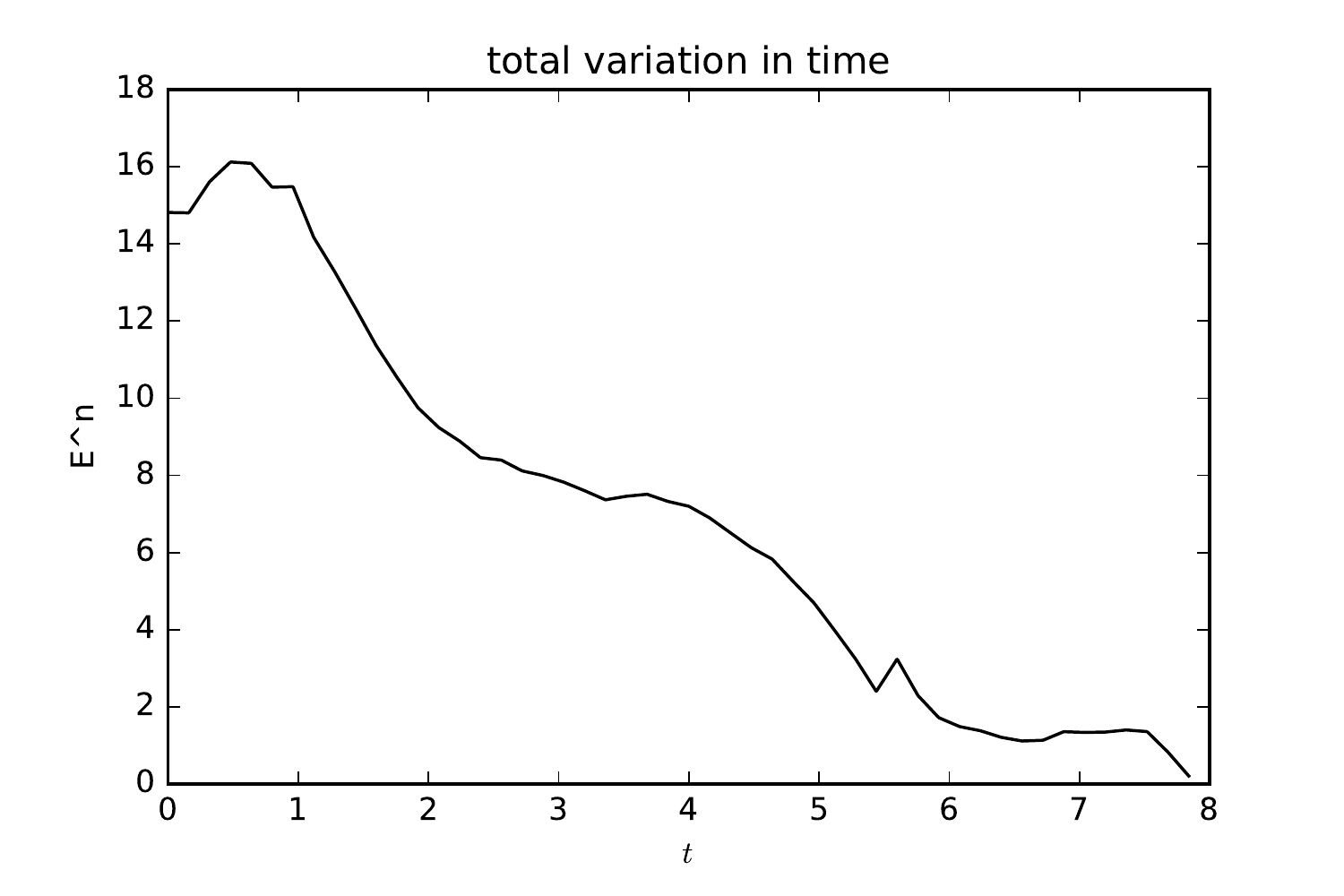,width=2.5in}
\end{minipage}
\caption{Evolution of an initially perturbed steady shock and its total variation in time}
\label{FIG-95}
\end{figure}


\end{document}